\documentclass[11pt,reqno]{amsart}

\usepackage{
    amsmath,
    amsfonts,
    amssymb,
    amsthm,
    amscd,
    comment,
    enumitem,
    etoolbox,
    textcomp,
    gensymb,
    mathtools,
    mathdots,
    stmaryrd,
}
\usepackage[dvipsnames]{xcolor}
\usepackage[all]{xy}
\usepackage[T1]{fontenc}
\usepackage{dsfont}
\usepackage[colorlinks=true, linkcolor=black, citecolor=black, urlcolor=purple, breaklinks=true]{hyperref}

\usepackage{zref-clever}

\zcsetup{
  cap,                      
  nameinlink=true,          
}

\zcRefTypeSetup{assumption}{
    Name-sg = Assumption ,
    name-sg = assumption ,
    Name-pl = Assumptions ,
    name-pl = assumptions ,
}
\zcRefTypeSetup{conv}{
    Name-sg = Convention ,
    name-sg = convention ,
    Name-pl = Conventions ,
    name-pl = conventions ,
}
\zcRefTypeSetup{cor}{
    Name-sg = Corollary ,
    name-sg = corollary ,
    Name-pl = Corollaries ,
    name-pl = corollaries ,
}
\zcRefTypeSetup{defin}{
    Name-sg = Definition ,
    name-sg = definition ,
    Name-pl = Definitions ,
    name-pl = definitions ,
}
\zcRefTypeSetup{eg}{
    Name-sg = Example ,
    name-sg = example ,
    Name-pl = Examples ,
    name-pl = examples ,
}
\zcRefTypeSetup{equation}{
    Name-sg = {} ,
    name-sg = {} ,
    Name-pl = {} ,
    name-pl = {} ,
}
\zcRefTypeSetup{lem}{
    Name-sg = Lemma ,
    name-sg = lemma ,
    Name-pl = Lemmas ,
    name-pl = lemmas ,
}
\zcRefTypeSetup{prop}{
    Name-sg = Proposition ,
    name-sg = proposition ,
    Name-pl = Propositions ,
    name-pl = propositions ,
}
\zcRefTypeSetup{rem}{
    Name-sg = Remark ,
    name-sg = remark ,
    Name-pl = Remarks ,
    name-pl = remarks ,
}
\zcsetup{countertype={subsection=subsection}} 
\zcRefTypeSetup{subsection}{
    Name-sg = \S ,
    name-sg = \S ,
    Name-pl = \S\S ,
    name-pl = \S\S ,
    namesep = {} ,      
}
\zcRefTypeSetup{table}{
    Name-sg = Table ,
    name-sg = table ,
    Name-pl = Tables ,
    name-pl = tables ,
}
\zcRefTypeSetup{theo}{
    Name-sg = Theorem ,
    name-sg = theorem ,
    Name-pl = Theorems ,
    name-pl = theorems ,
}

\zcsetup{
  countertype = {
    enumi=item, enumii=item, enumiii=item, enumiv=item
  },
  counterresetby = {
    enumii=enumi, enumiii=enumii, enumiv=enumiii
  }
}
\zcRefTypeSetup{item}{
    Name-sg = {} ,
    name-sg = {} ,
    Name-pl = {} ,
    name-pl = {} ,
}

\AddToHook{env/assumption/begin}{\zcsetup{countertype={theo=assumption}}}
\AddToHook{env/conv/begin}{\zcsetup{countertype={theo=conv}}}
\AddToHook{env/cor/begin}{\zcsetup{countertype={theo=cor}}}
\AddToHook{env/defin/begin}{\zcsetup{countertype={theo=defin}}}
\AddToHook{env/eg/begin}{\zcsetup{countertype={theo=eg}}}
\AddToHook{env/lem/begin}{\zcsetup{countertype={theo=lem}}}
\AddToHook{env/prop/begin}{\zcsetup{countertype={theo=prop}}}
\AddToHook{env/rem/begin}{\zcsetup{countertype={theo=rem}}}
\AddToHook{env/table/begin}{\zcsetup{countertype={theo=table}}}

\newcommand{\cref}[1]{\zcref{#1}}
\newcommand{\Cref}[1]{\zcref[S]{#1}}

\setenumerate[1]{label=(\arabic*)}

\newcommand\minush{\ell}
\newcommand\C{\mathbb{C}}
\newcommand\N{\mathbb{N}}
\newcommand\OO{\mathcal{X}}

\newcommand\R{\mathbb{R}}
\newcommand\Z{\mathbb{Z}}
\newcommand\kk{\Bbbk}
\newcommand\one{\mathds{1}}
\newcommand\tT{\mathtt{T}}
\newcommand\tR{\mathtt{R}}

\newcommand\0{{\bar 0}}
\renewcommand\1{{\bar 1}}
\newcommand{\dash}{\operatorname{-}}
\newcommand\catA{\mathbf{A}}
\newcommand\catB{\mathbf{B}}
\newcommand\catC{\mathbf{C}}
\newcommand\catR{\mathbf{R}}
\newcommand\fg{\mathfrak{g}}

\newcommand\fq{\mathfrak{q}}
\newcommand\fU{\mathfrak{U}}
\newcommand\fV{\mathfrak{V}}

\newcommand\eps{\varepsilon}

\newcommand\Cl{C}
\newcommand\op{\mathrm{op}}

\newcommand\rev{\mathrm{rev}}

\newcommand{\smod}{\textup{-smod}}
\newcommand\transpose{\textup{T}}
\newcommand\p{{\mathtt{p}}}
\newcommand\cg{c}
\newcommand\form{\tau}

\newcommand\lround{(\!(}
\newcommand\rround{)\!)}
\newcommand{\cev}[1]{\reflectbox{\ensuremath{\vec{\reflectbox{\ensuremath{#1}}}}}}

\newcommand\rmGL{\operatorname{GL}}
\newcommand\rmOSp{\operatorname{OSp}}
\newcommand\rmP{\operatorname{P}}
\newcommand\rmQ{\operatorname{Q}}

\newcommand\Heis{\mathbf{Heis}}
\newcommand\sEnd{\mathbf{End}}
\newcommand\sHom{\mathbf{Hom}}
\newcommand\sCat{\mathfrak{sCat}}
\newcommand\sVec{\mathbf{s\hspace{-.8pt}Vec}}

\DeclareMathOperator{\End}{End}

\DeclareMathOperator{\Hom}{Hom}
\DeclareMathOperator{\id}{id}

\DeclareMathOperator{\pr}{pr}

\usepackage{tikz}
\usetikzlibrary{arrows,patterns}
\usetikzlibrary{arrows.meta}
\usetikzlibrary{shapes.geometric}
\usetikzlibrary{decorations.markings,decorations.pathreplacing}
\usetikzlibrary{calc}
\usepackage{tikz-cd}

\newcommand*\getscale[1]{%
\begingroup
\pgfgettransformentries{\scaleA}{\scaleB}{\scaleC}{\scaleD}{\whatevs}{\whatevs}%
\pgfmathsetmacro{#1}{sqrt(abs(\scaleA*\scaleD-\scaleB*\scaleC))}%
\expandafter
\endgroup
\expandafter\def\expandafter#1\expandafter{#1}%
}

\tikzset{anchorbase/.style={baseline={([yshift=-0.5ex]current bounding box.center)}}}
\tikzset{ 
centerzero/.style={baseline={([yshift=-0.5ex](#1))}},
centerzero/.default={0,0}
}
\tikzset{wipe/.style={white,line width=4pt}}

\tikzset{->-/.style={decoration={
markings,
mark=at position #1 with {\arrow{>}}},postaction={decorate}}
}
\tikzset{-<-/.style={decoration={
markings,
mark=at position #1 with {\arrow{<}}},postaction={decorate}}
}

\tikzset{gcolor/.style={green!60!black,semithick}}
\tikzset{H/.style={>=to,semithick,purple}}
\tikzset{Hred/.style={>=to,semithick,red}}
\tikzset{KM/.style={>=to,thin,black}}
\tikzset{IKM/.style={>=to,thin,black}}
\tikzset{pinhead/.style={fill=yellow!40!white}}
\tikzset{fakebubble/.style={fill=red!20!white}}
\tikzset{internalbubble/.style={fill=green!20!white}}
\tikzset{bulb/.style={fill=yellow!40!white}}
\tikzset{shadow/.style={double,thin}}

\newcommand{\circledbar}[3][black]{\node[circle,draw,pinhead,inner sep=.85pt] at (#2) {$\hspace{0.06pt}\color{#1}\scriptscriptstyle #3$};\draw[#1,line width=.6pt] ($(#2)+(-.06,.07)$) to ($(#2)+(.06,.07)$)}

\newcommand{\circled}[3][black]{\node[circle,draw,pinhead,inner sep=.85pt] at (#2) {$\hspace{0.06pt}\color{#1}\scriptscriptstyle #3$}}

\newcommand{\pin}[4][black]{
\path (#3) node[rectangle,rounded corners,draw,pinhead,inner sep=2.5pt](y) {$\color{#1}\scriptstyle#4$};
\draw[Triangle Cap-,thick,#1!40!white] (#2)--(y);
\singdot{#2}
}

\newcommand{\pinpin}[5][black]{
\path (#4) node[rectangle,rounded corners,draw,pinhead,inner sep=2.5pt](y) {$\color{#1}\scriptstyle#5$};
\draw[Triangle Cap-,thick,#1!40!white] (#2)--(#3)--(y);
\singdot{#3}; \singdot{#2}
}
\newcommand{\pinpinpin}[6][black]{
\path (#5) node[rectangle,rounded corners,draw,pinhead,inner sep=2.5pt](y) {$\color{#1}\scriptstyle#6$};
\draw[Triangle Cap-,thick,#1!40!white] (#2)--(#3)--(#4)--(y);
\singdot{#4}; \singdot{#3}; \singdot{#2}
}

\newcommand{\pinPinpin}[6][black]{
\path (#5) node[rectangle,rounded corners,draw,pinhead,inner sep=2.5pt](y) {$\color{#1}\scriptstyle#6$};
\draw[Triangle Cap-,thick,#1!40!white] (#2)--(#3)--(#4)--(#3)--(y);
\singdot{#4}; \singdot{#3}; \singdot{#2}
}

\newcommand\braidup{to[out=up,in=down]}
\newcommand\braiddown{to[out=down,in=up]}

\newcommand\objlabel[1]{$\color{black}\scriptstyle{#1}$}
\newcommand\dotlabel[1]{$\color{teal}\scriptstyle{#1}$}
\newcommand\regionlabel[1]{$\color{gray}\scriptstyle{#1}$}
\newcommand\strandlabel[1]{$\color{olive}\scriptstyle{#1}$}
\newcommand\stringlabel[2]{\node at (#1) {\strandlabel{#2}}} 
\newcommand\botlabel[1]{node[inner sep=0.5pt,anchor=north] {\strandlabel{#1}}}
\newcommand\toplabel[1]{node[anchor=south,inner sep=0.5pt] {\strandlabel{#1}}}
\newcommand\region[2]{\node at (#1) {\regionlabel{#2}}} 
\newcommand\strand[2]{\node at (#1) {\strandlabel{#2}}} 

\newcommand{\token}[1]{
\node at (#1) {$\scriptstyle{\bullet}$}
}

\newcommand\singdot[1]{
\node at (#1) {\color{white} $\scriptstyle{\bullet}$};
\node at (#1) {$\scriptstyle{\circ}$}
}
\newcommand\multdot[3]{
\singdot{#1};
\draw (#1) node[anchor=#2] {\dotlabel{#3}}
}
\newcommand\notch[2][0]{
\draw[thin] (#2)++(#1:0.06) -- ++(#1:-0.12)
}
\newcommand\gnotch[2][0]{
\draw[gcolor,thin] (#2)++(#1:0.06) -- ++(#1:-0.12)
}
\newcommand\projcr[1]{
\draw (#1) node[rotate=45] {$\scriptstyle{\pmb\Box}$}
}

\newcommand\rightbub[1]{
\getscale{\scalefactor};
\draw[->] (#1)++(0.24/\scalefactor,0) arc(360:0:0.24/\scalefactor)
}

\newcommand\rightbubdot[2]{
\getscale{\scalefactor};
\draw[->] (#1)++(0,0.24/\scalefactor) arc(90:-270:0.24/\scalefactor);
\filldraw[fill=white, draw=black] (#1)++(-0.24/\scalefactor,0) circle (1.5pt/\scalefactor) node[anchor=east] {\dotlabel{#2}}
}
\newcommand\rightbubgen[2][u]{
\getscale{\scalefactor};
\draw[->] (#2)++(-0.75/\scalefactor,0)arc(180:-180:0.24/\scalefactor);
\node[black] at (#2) {$(#1)$}
}

\newcommand\leftbub[1]{
\getscale{\scalefactor};
\draw[->] (#1)++(-0.24/\scalefactor,0) arc(-180:180:0.24/\scalefactor)
}
\newcommand\leftbubdot[2]{
\getscale{\scalefactor};
\draw[->] (#1)++(0,0.24/\scalefactor) arc(90:450:0.24/\scalefactor);
\filldraw[fill=white, draw=black] (#1)++(0.24/\scalefactor,0) circle (1.5pt/\scalefactor) node[anchor=west] {\dotlabel{#2}}
}
\newcommand\leftbubgen[2][u]{
\getscale{\scalefactor};
\draw[->] (#2)++(-0.75/\scalefactor,0)arc(-180:180:0.24/\scalefactor);
\node[black] at (#2) {$(#1)$}
}

\newtheorem{theo}{Theorem}[section]

\newtheorem{lem}[theo]{Lemma}
\newtheorem{cor}[theo]{Corollary}

\theoremstyle{definition}

\newtheorem{conv}[theo]{Convention}
\newtheorem{rem}[theo]{Remark}

\numberwithin{equation}{section}
\allowdisplaybreaks

\setenumerate[1]{label=(\arabic*)}

\newtoggle{comments}
\togglefalse{comments}

\iftoggle{comments}{
\usepackage[notref,notcite]{showkeys}
\newcommand{\acomments}[1]{{\color{red}\noindent\textbf{AS:} #1}}
\newcommand{\jcomments}[1]{{\color{purple}\noindent\textbf{JB:} #1}}

}
{
\newcommand{\acomments}[1]{}
\newcommand{\jcomments}[1]{}
}

\begin{document}

\begin{abstract}
We develop a general framework for studying Abelian categories arising in isomeric representation theory, that is, representation theory broadly related to the supergroup $\rmQ(n)$. In this first part, we introduce notions of {\em isomeric Heisenberg categorification} and {\em isomeric Kac--Moody categorication}, and explain how to pass from the former to the latter. This is analogous to the passage from Heisenberg categorification to Kac--Moody categorification developed in our previous work with Webster.
\end{abstract}

\title{Isomeric Heisenberg and Kac--Moody categorification I}

\author{Jonathan Brundan}
\address[J.B.]{
Department of Mathematics \\
University of Oregon \\
Eugene, OR, USA
}
\urladdr{\href{https://pages.uoregon.edu/brundan}{pages.uoregon.edu/brundan}, \textrm{\textit{ORCiD}:} \href{https://orcid.org/0009-0009-2793-216X}{orcid.org/0009-0009-2793-216X}}
\email{brundan@uoregon.edu}

\author{Alistair Savage}
\address[A.S.]{
Department of Mathematics and Statistics \\
University of Ottawa \\
Ottawa, ON K1N 6N5, Canada
}
\urladdr{\href{https://alistairsavage.ca}{alistairsavage.ca}, \textrm{\textit{ORCiD}:} \href{https://orcid.org/0000-0002-2859-0239}{orcid.org/0000-0002-2859-0239}}
\email{alistair.savage@uottawa.ca}


\thanks{J.B.\ was supported in part by NSF grant DMS-2348840.
A.S.\ was supported by Discovery Grant RGPIN-2023-03842 from the Natural Sciences and Engineering Research Council of Canada.
}

\maketitle

\section{Introduction}\label{s1-intro}

This is the first of two papers in which we develop the {\em isomeric} analog of the theory of Heisenberg and Kac--Moody categorification from our joint work with 
Webster \cite{BSW-HKM}.
The word ``isomeric'' used here was suggested by Nagpal, Sam and Snowden in \cite[Sec.~1.5]{NSS}. It indicates a connection with
the {\em isomeric supergroup} $\rmQ(n)$, which is one of the
four families $\rmGL(m|n), \rmOSp(m|2n), \rmP(n)$ and $\rmQ(n)$ of
classical algebraic supergroups.

The existing theory of Heisenberg and Kac--Moody categorification gives a unifying framework for studying many of the Abelian categories which arise in $\rmGL$-type representation theory, including representations of symmetric groups, general linear groups over finite fields,
cyclotomic Hecke algebras, rational representations of general linear groups and quantum linear groups, and related categories like the BGG category $\mathcal O$ for the general linear Lie algebra. All of these categories admit an action of the Heisenberg category $\Heis_\kappa$ from \cite{K-heis,MS-heis,B-heis}
(or its quantum analog $q\text{-}\hspace{-1pt}\mathbf{Heis}_\kappa$
from \cite{LS-qheis,BSW-qheis})
for some central charge $\kappa \in \Z$, 
and a corresponding action of one of the Kac--Moody 2-categories
from \cite{Rou,KL3}. The
Cartan matrix of the underlying Kac--Moody algebra
has connected components of type $A_{\infty}$ if the characteristic (or quantum characteristic) 
$p$ of the ground field is $0$,
or $A_{p-1}^{(1)}$ if $p > 0$. The main result of \cite{BSW-HKM}
constructs the required bridge to pass from the Heisenberg action,
which is usually in plain view, to the Kac--Moody action, which is hidden. This bridge gives access to many powerful structural results about Kac--Moody categorifications established by Chuang and Rouquier \cite{CR}, Rouquier \cite{Rou,Rou2}, Kang and Kashiwara \cite{KK},
Webster \cite{Web-memoir},
Losev and Webster \cite{LW}, and others. These 
are all expressed in terms of the rich combinatorics of integrable representations and crystal bases of the affine Lie algebras
of these Cartan types.

\begin{table}
$$
\begin{array}{rcl}
A_\infty\:(p=0):&
\begin{tikzpicture}[anchorbase]
\node at (-2,-.3) {$\scriptstyle \phantom{k}$};
\node at (-2.8,0) {$\cdots$};
\node at (2.9,0) {$\cdots$};
\node at (-2,0) {$\bullet$};
\node at (-2,.3) {$\scriptstyle k-2$};
\node at (-1,0) {$\bullet$};
\node at (-1,.3) {$\scriptstyle k-1$};
\node at (0,0) {$\bullet$};
\node at (0,.3) {$\scriptstyle k$};
\node at (1,0) {$\bullet$};
\node at (1,.3) {$\scriptstyle k+1$};
\node at (2,0) {$\bullet$};
\node at (2,.3) {$\scriptstyle k+2$};
\draw (-2.4,0) to (2.4,0);
\end{tikzpicture}
&k \notin \{\cdots,-\frac{1}{2},0,\frac{1}{2},\dots\}
\\
B_\infty\:(p=0):
&\begin{tikzpicture}[anchorbase]
\node at (-1,-.3) {$\scriptstyle \phantom{0}$};
\node at (1.9,0) {$\cdots$};
\node at (0,0) {$\bullet$};
\node at (0,.3) {$\scriptstyle 1$};
\node at (1,0) {$\bullet$};
\node at (1,.3) {$\scriptstyle 2$};
\draw (-.931,.03) to (0,.03);
\draw (-.931,-.03) to (0,-.03);
\draw (0,0) to (1.4,0);
\node at (-1,0) {$\scriptscriptstyle\otimes$};
\node at (-1,.3) {$\scriptstyle 0$};
\node at (-.5,0.01) {$\scriptscriptstyle <$};
\end{tikzpicture}
\\
C_\infty\:(p=0):&
\begin{tikzpicture}[anchorbase]
\node at (1,-.35) {$\scriptstyle \phantom{\frac{1}{2}}$};
\node at (-1.9,0) {$\cdots$};
\node at (0,0) {$\bullet$};
\node at (-0.1,.35) {$\scriptstyle -\frac{3}{2}$};
\node at (-1,0) {$\bullet$};
\node at (-1.1,.35) {$\scriptstyle -\frac{5}{2}$};
\draw (1,.03) to (0,.03);
\draw (1,-.03) to (0,-.03);
\draw (0,0) to (-1.4,0);
\node at (1,0) {$\bullet$};
\node at (.9,.35) {$\scriptstyle -\frac{1}{2}$};
\node at (.5,0.01) {$\scriptscriptstyle <$};
\end{tikzpicture}
\\
A_{p-1}^{(1)}\:(p > 2):&
\begin{tikzpicture}[anchorbase]
\node at (-1,-.3) {$\scriptstyle \phantom{\frac{p-3}{2}}$};
\node at (1.95,0) {$\cdots$};
\draw (-1,0) to (0,0);
\draw (2.5,0) to (3.5,0);
\draw (0,0) to (1.4,0);
\draw (-1,0) to [out=-45,in=-135,looseness=.6] (3.5,0);
\node at (-1,0) {$\bullet$};
\node at (-1,.3) {$\scriptstyle k$};
\node at (0,0) {$\bullet$};
\node at (0,.3) {$\scriptstyle k+1$};
\node at (1,0) {$\bullet$};
\node at (1,.3) {$\scriptstyle k+2$};
\node at (3.5,0) {$\bullet$};
\node at (2.4,.3) {$\scriptstyle k+p-2$};
\node at (2.5,0) {$\bullet$};
\node at (3.7,.3) {$\scriptstyle k+p-1$};
\end{tikzpicture}
&k \notin \left\{0,1,\dots,p-1\right\}
\\
A_{2}^{(2)}\:(p =3):&\begin{tikzpicture}[anchorbase]
\node at (-1,-.3) {$\scriptstyle \phantom{-\frac{1}{2}}$};
\draw (-.93,.0166) to (0,.0166);
\draw (-.93,-.0166) to (0,-.0166);
\draw (-.945,.05) to (0,.05);
\draw (-.945,-.05) to (0,-.05);
\node at (-1,0) {$\scriptscriptstyle\otimes$};
\node at (-1,.3) {$\scriptstyle 0$};
\token{0,0};
\node at (-.1,.3) {$\scriptstyle -\frac{1}{2}$};
\node at (-.5,0.01) {$\scriptstyle <$};
\end{tikzpicture}\\
A_{p-1}^{(2)}\:(p > 3):&\begin{tikzpicture}[anchorbase]
\node at (-1,-.3) {$\scriptstyle \phantom{\frac{p-3}{2}}$};
\node at (1.95,0) {$\cdots$};
\draw (-.931,.03) to (0,.03);
\draw (-.931,-.03) to (0,-.03);
\draw (2.5,0) to (3.5,0);
\draw (4.5,.03) to (3.5,.03);
\draw (4.5,-.03) to (3.5,-.03);
\draw (0,0) to (1.4,0);
\node at (-1,0) {$\scriptscriptstyle\otimes$};
\node at (-1,.3) {$\scriptstyle 0$};
\node at (0,0) {$\bullet$};
\node at (0,.3) {$\scriptstyle 1$};
\node at (1,0) {$\bullet$};
\node at (1,.3) {$\scriptstyle 2$};
\node at (2.5,0) {$\bullet$};
\node at (2.4,.3) {$\scriptstyle -\frac{5}{2}$};
\node at (3.5,0) {$\bullet$};
\node at (3.4,.3) {$\scriptstyle -\frac{3}{2}$};
\node at (4.5,0) {$\bullet$};
\node at (4.4,.3) {$\scriptstyle -\frac{1}{2}$};
\node at (-.5,0.01) {$\scriptscriptstyle <$};
\node at (4,0.01) {$\scriptscriptstyle <$};
\end{tikzpicture}
\end{array}
$$
\caption{Root systems}\label{dynkintable}
\end{table}

In this paper and its sequel, we will establish analogs of these
results for $\rmQ$-type representation theory, including representations of the spin symmetric groups, rational representations of the supergroup $\rmQ(n)$, and category $\mathcal O$ for the Lie superalgebra $\mathfrak{q}_n(\C)$.
On the Kac--Moody side, the Cartan matrices that emerge have connected components of 
types $A_\infty, B_\infty$ and $C_\infty$ in characteristic $0$, or
$A_{p-1}^{(1)}$ and $A_{p-1}^{(2)}$ in characteristic $p > 2$ (see
\cref{dynkintable}). In fact, these are {\em super} Cartan matrices
in the sense of Kang, Kashiwara and Tsuchioka \cite{KKT16}, with the simple root labelled by $0$ being odd and all other simple roots being even. The same super Cartan types can already be seen in \cite{KKT16}, which constructed
Morita equivalences (perhaps with an additional Clifford twist)
between completions of affine Sergeev superalgebras and the {\em quiver Hecke superalgebras} introduced in their work.
Our results extend \cite{KKT16} to categorical actions involving adjoint pairs of functors in the spirit of \cite{Rou}, replacing affine Sergeev superalgebras with the monoidal {\em isomeric Heisenberg category}, and the quiver Hecke superalgebras with the corresponding {\em super Kac--Moody 2-categories} from \cite{BE-SKM}.

There is an additional complication in the isomeric case in that the passage from
an isomeric Heisenberg action to a super Kac--Moody action
involves an intermediate third object, the {\em isomeric Kac--Moody 2-category}. This is a new 2-category introduced in this paper, although the definition is not hard to guess since it is the 2-categorical counterpart of the {\em quiver Hecke--Clifford superalgebras} which already appeared in \cite{KKT16}.
In other words, unlike for the existing $\rmGL$ theory, the bridge for the $\rmQ$ theory has two spans. The first goes from isomeric Heisenberg to isomeric Kac--Moody (this paper), and involves a remarkable change-of-variable. The second span of the bridge goes from isomeric Kac--Moody to super Kac--Moody (the sequel), and requires some Clifford twists.
It is not until both spans are constructed that we are able to access the known structural results about super Kac--Moody categorifications, such as the derived equivalences from \cite{BK-odd} which extend Chuang--Rouquier's Rickard equivalences from \cite{CR}. 
These results and applications to the representation theory of spin symmetric groups and cyclotomic Sergeev superalgebras
extending \cite{BK01}, and 
to category $\mathcal{O}$ for the Lie superalgebra $\fq_n(\C)$
extending \cite{BD17,BD19}, will be explained more fully in Part II.

The work of Kang, Kashiwara and Tsuchioka underpinning our construction 
also applies to affine Hecke--Clifford superalgebras, which are the
$q$-analogs of affine Sergeev superalgebras. As well as the Cartan types
$A_\infty, B_\infty$ and $C_\infty$ when the parameter $q$ is not a
root of unity, there are connected components of type $A_{\ell}^{(1)}$ when $q^2$ is a
primitive $(\ell+1)$th root of unity, type $A_{2\ell}^{(2)}$ when
$q^2$ is a primitive $(2\ell+1)$th root of unity, and
$C_{\ell}^{(1)}$ and $D_{\ell}^{(2)}$ when $q^2$ is a primitive
$2\ell$th root of unity ($\ell \geq 2$). We have not included the quantum case in the
present paper partly because the applications seem less significant,
but also because we do not at present know how to define a suitable quantum
analog of the isomeric Heisenberg category for all choices of central
charge. (The appropriate category for central charge 0 is known---it is
the {\em quantum affine isomeric category} from \cite{Sav24}.)

The article is organized as follows. In \cref{s2-prelim}, we explain our conventions and review some of the general language of superalgebra.
In \cref{s3-iheis}, we define the {\em isomeric Heisenberg supercategory}, which is a special case of the Frobenius Heisenberg supercategories from \cite{Sav19}.
We are mainly interested in {\em isomeric Heisenberg categorifications}, which are Abelian supercategories equipped with a suitable action of the isomeric Heisenberg category. The precise definition can be found at the end of this section. 
In \cref{s4-wsd}, we explain how to
decompose any isomeric Heisenberg categorification into ``blocks'' labelled in a natural way by weights of an underlying root system whose Cartan 
type is as in \cref{dynkintable}. In \cref{s5-ikm}, we introduce the {\em isomeric Kac-Moody 2-category} attached to these and more general Cartan types, 
leading to the definition of an {\em isomeric Kac-Moody categorification} formulated at the end of the section.
Finally, in \cref{s6-iheis2ikm}, we prove the main
\cref{maintheorem2}, which gives a general construction 
making any isomeric Heisenberg categorification into an isomeric Kac-Moody categorification.

\vspace{2mm}
\noindent
{\em Acknowledgements.}
We thank Ben Webster for teaching us about 
bubbles during our previous collaboration on \cite{BSW-HKM}.

\setcounter{section}{1}
\section{Reminders about supercategories}\label{s2-prelim}

Throughout the paper we work over an algebraically closed field $\kk$ of characteristic $p \neq 2$.
We use the shorthand $\hbar$ for the element $-\frac{1}{2} \in \kk$.
For a proposition $P$, we use the notation $\delta_P$ to denote 
$1$ if $P$ is true or $0$ if $P$ is false.

\subsection{Superalgebras}

Almost everything in the paper 
will be enriched over the closed symmetric monoidal category $\sVec$ of vector superspaces over the ground field $\kk$, morphisms being parity-preserving linear maps.
We denote the parity of a homogeneous vector
$v$ in a vector superspace $V = V_\0 \oplus V_\1$
by $\p(v) \in \{\0,\1\}$.

A {\em superalgebra} is an associative, unital algebra $A$ that is also a vector superspace
such that $\p(ab) = \p(a) +\p(b)$ for all  $a, b \in A$.
Here, and subsequently, when we write formulae involving parities, we assume implicitly that the elements in question are homogeneous.
For superalgebras $A$ and $B$, the superspace 
$A\otimes B$ is a superalgebra with multiplication  
\begin{equation}
(a' \otimes b) (a \otimes b') = (-1)^{\p(a) \p(b)} a'a \otimes bb'
\end{equation}
for $a,a' \in A$, $b,b' \in B$.
The \emph{opposite} superalgebra $A^\op$ is a copy $\{a^\op : a \in A\}$ of the vector superspace $A$ with multiplication defined from
\begin{equation}\label{mups}
a^\op  b^\op := (-1)^{\p(a) \p(b)} (ba)^{\op}.
\end{equation}

To give a relevant example, the
polynomial algebra $\kk[x]$ can be viewed as a 
superalgebra by declaring that $x$ is odd.
It is commutative but not supercommutative.
There is a 
unique superalgebra isomorphism
$\tT:\kk[x] \stackrel{\sim}{\rightarrow} \kk[x]^\op$
mapping $x$ to $x^\op$.
Since $(x^\op)^n = (-1)^{\binom{n}{2}}
(x^n)^\op$, $\tT$ maps $f(x) =\sum_{r=0}^n c_r x^r \in \kk[x]$
to $\tilde f(x)^\op$ where $\tilde f(x) \in \kk[x]$ is defined by
\begin{equation}\label{toshis}
\tilde f(x) := \sum_{r=0}^n (-1)^{\binom{r}{2}}c_r x^r.
\end{equation}

\subsection{Monoidal supercategories and 2-supercategories}

We will work with
\emph{strict monoidal supercategories} 
and (strict) \emph{2-supercategories}
in the sense 
of \cite{BE17}. We will not repeat these definitions in full here, but recall some basic notions
since the language is not completely standard.
\begin{itemize}
\item A \emph{supercategory} means a category $\catA$ whose morphism spaces are vector superspaces, with composition of morphisms being bilinear and parity-preserving.  We denote the opposite supercategory by $\catA^\op$. Composition of morphisms in $\catA^\op$ involves a sign analogous to \cref{mups}. Also $\underline{\catA}$ denotes the underlying category, which has the same objects but only the even morphisms.
\item A \emph{superfunctor} $F : \catA \rightarrow \catB$ between supercategories is a functor which induces a parity-preserving linear map between morphism superspaces.
The underlying functor $\underline{F}: \underline{\catA}\rightarrow \underline{\catB}$
is its restriction to underlying categories.
\item  A \emph{supernatural transformation $\alpha : F \Rightarrow G$ of parity $r\in\Z/2$} between two superfunctors $F,G : \catA \to \catB$ is the data of morphisms $\alpha_X\in \Hom_{\catB}(FX, GX)$ of parity $r$ for each $X \in \catA$, such that $Gf \circ \alpha_X = (-1)^{r \p(f)}\alpha_Y\circ Ff$ for each $f \in \Hom_{\catA}(X, Y)$.
Note when $r$ is odd that $\alpha$ is \emph{not} a natural transformation in the usual sense due to the sign. Then a general \emph{supernatural transformation} $\alpha : F \Rightarrow G$ is of the form $\alpha = \alpha_{\0} + \alpha_{\1}$, with each $\alpha_r$ being a supernatural transformation of parity $r$.
If $\alpha$ is an even supernatural transformation, the same data defines a natural transformation $\alpha:\underline{F}\rightarrow \underline{G}$
between the underlying functors.
\end{itemize}
For supercategories $\catA$ and $\catB$, we write $\sHom(\catA,\catB)$
for the supercategory of superfunctors and supernatural transformations. In particular, $\sEnd(\catA) := \sHom(\catA,\catA)$
is a strict monoidal supercategory, with monoidal product defined on objects by composition of functors and on morphisms by horizontal composition of supernatural transformations.
There is a 2-supercategory $\sCat$ consisting of 
(small) supercategories, superfunctors and supernatural transformations.

In any monoidal supercategory $\catC$, morphisms satisfy the \emph{super interchange law}:
\begin{equation}\label{interchange}
(f' \otimes g) \circ (f \otimes g')
= (-1)^{\p(f) \p(g)} (f' \circ f) \otimes (g \circ g').
\end{equation}
We denote the unit object by $\one$ and the identity endomorphism of an object $X$ by $\id_X$.
The {\em reverse} of $\catC$
is denoted $\catC^\rev$, i.e., we reverse the order of the tensor product with appropriate signs.
We will use the usual calculus of string diagrams, representing the horizontal composition $f \otimes g$ (resp., vertical composition $f \circ g$) of morphisms $f$ and $g$ diagrammatically by drawing $f$ to the left of $g$ (resp., drawing $f$ above $g$).  Care is needed with horizontal levels in such diagrams due to the signs implied by \cref{interchange}:
\begin{equation}\label{intlaw}
\begin{tikzpicture}[anchorbase]
\draw (-0.5,-0.5) -- (-0.5,0.5);
\draw (0.5,-0.5) -- (0.5,0.5);
\node[draw,fill=white,rounded corners] at (-0.5,0.15) {$\scriptstyle{f}$};
\node[draw,fill=white,rounded corners] at (0.5,-0.15) {$\scriptstyle{g}$};
\end{tikzpicture}
\quad=\quad
\begin{tikzpicture}[anchorbase]
\draw (-0.5,-0.5) -- (-0.5,0.5);
\draw (0.5,-0.5) -- (0.5,0.5);
\node[draw,fill=white,rounded corners] at (-0.5,0) {$\scriptstyle{f}$};
\node[draw,fill=white,rounded corners] at (0.5,0) {$\scriptstyle{g}$};
\end{tikzpicture}
\quad=\quad
(-1)^{\p(f)\p(g)}\
\begin{tikzpicture}[anchorbase]
\draw (-0.5,-0.5) -- (-0.5,0.5);
\draw (0.5,-0.5) -- (0.5,0.5);
\node[draw,fill=white,rounded corners] at (-0.5,-0.15) {$\scriptstyle{f}$};
\node[draw,fill=white,rounded corners] at (0.5,0.15) {$\scriptstyle{g}$};
\end{tikzpicture}\ .
\end{equation}

\subsection{$\Pi$-Supercategories}

Roughly speaking, a {\em $\Pi$-supercategory} is a 
supercategory $\catA$ equipped with a parity switching functor $\Pi$.
Formally, it is a triple $(\catA, \Pi, \zeta)$ consisting of a supercategory $\catA$, 
a superfunctor $\Pi:\catA \rightarrow \catA$, and an odd supernatural isomorphism $\zeta:\Pi\stackrel{\sim}{\Rightarrow} \id_\catA$.
The basic example is the $\Pi$-supercategory $A\smod$ of left $A$-supermodules over a superalgebra $A$:
\begin{itemize}
\item
A {\em left $A$-supermodule} $V$ is a left $A$-module which is also a vector superspace such that $A_i V_j \subseteq V_{i+j}$.
These are the objects in the category $A\smod$.
\item
A {\em left $A$-supermodule homomorphism} $f:V \rightarrow W$ between two left $A$-supermodules is a linear map such that
$f(av) = (-1)^{\p(f) \p(a)} a f(v)$
for $a \in A, v \in V$, where
$\p(f) = \0$ if $f$ is parity-preserving and
$\p(f) = \1$ if $f$ is parity-reversing.
We use the notation $\Hom_{A\dash}(V,W)$ to denote the superspace of all left
$A$-supermodule homomorphisms. These are the morphism superspaces in the category $A\smod$.
\item
The superfunctor $\Pi:A \smod \rightarrow A \smod$
making $A\smod$ into a $\Pi$-supercategory
is the usual parity switching 
functor.
In particular, the action of $A$ on $\Pi V$
is defined by $a \cdot v := (-1)^{\p(a)} av$.
For a morphism $f\in \Hom_{A\dash}(V,W)$,
$\Pi f\in \Hom_{A\dash}(\Pi V,\Pi W)$
is the linear map $(-1)^{\p(f)} f$.
\item
The odd supernatural isomorphism $\zeta: \Pi \stackrel{\sim}{\Rightarrow} \id_{A\smod}$ is defined
by letting $\zeta_V: \Pi V \rightarrow V$ be the
$A$-supermodule homomorphism defined by the identity function. 
\end{itemize}
For more details, and
the related definitions of
{\em monoidal $\Pi$-supercategory} and {\em $\Pi$-2-supercategory},
we refer to \cite{BE17}.

Any supercategory $\catA$ can be upgraded to a $\Pi$-supercategory by formally adjoining a parity shift functor $\Pi$. The resulting category is the {\em $\Pi$-envelope} $\catA_\pi$ of $\catA$. There is a canonical embedding $J:\catA \rightarrow \catA_\pi$, and a universal property asserting that any superfunctor from $\catA$ to a $\Pi$-supercategory factors through $J$.
If $\catC$ is a monoidal supercategory, its $\Pi$-envelope
$\catC_\pi$ is monoidal, and similarly for 2-supercategories.
Again all of this is discussed in detail in \cite{BE17}.

\subsection{Abelian supercategories\label{ssas}}

By an {\em Abelian supercategory}, we mean a $\Pi$-supercategory $\catA$ whose underlying category $\underline{\catA}$
is Abelian in the usual sense. This is not a standard piece of language.
For example, for a superalgebra $A$, the $\Pi$-supercategory
$A\smod$ is an Abelian supercategory.

We will need to impose an additional finiteness
condition: we say that an Abelian supercategory $\catA$
is a {\em locally finite Abelian supercategory} 
if the underlying category is locally finite in the usual sense, that is, all of its objects are of finite length and all of its morphism spaces are finite-dimensional.

\setcounter{section}{2}
\section{Isomeric Heisenberg categorications}\label{s3-iheis}

In this section, we define the \emph{isomeric Heisenberg category} $\Heis_\kappa(\Cl)$ of central charge $\kappa \in \Z$, leading to the notion of an {\em isomeric Heisenberg categorification}.
In fact,  $\Heis_\kappa(\Cl)$ is a strict monoidal supercategory, although we usually omit the word ``super'' for brevity.
In the special case $\kappa = 0$, it is the
degenerate affine oriented Brauer--Clifford supercategory introduced in \cite{BCK19} (our Clifford token is the one there scaled by $\sqrt{-1}$); see also \cite{GRSS19}.
For general central charge,
$\Heis_\kappa(\Cl)$ is a special case of the \emph{Frobenius Heisenberg supercategories} introduced in \cite[Def.~1.1]{Sav19} taking the Frobenius superalgebra, denoted $F$ there, to be the rank one Clifford superalgebra
\begin{equation}\label{clifforddef}
\Cl := \kk\langle \cg : \cg^2=-1\rangle
\end{equation}
with the generator $\cg$ being odd. 
We choose the even Frobenius form $\form : \Cl \to \kk$ determined by
$$
\form(1) = 1,\qquad \form(\cg) = 0.
$$
When comparing to the presentation of \cite{Sav19}, one should take the basis of $\Cl$ to be $\{1,\cg\}$, in which case $1^\vee = 1$, $\cg^\vee = -\cg$.  The Nakayama automorphism
$\psi:\Cl \rightarrow \Cl$, which maps $a$ to
the unique $\psi(a)$ such that $\form(ab) = (-1)^{\p(a) \p(b)}\form (b\psi(a))$
for all $b \in \Cl$, is given by
\begin{equation}
\psi(1) = 1,\qquad \psi(\cg) = - \cg.
\end{equation}

\subsection{Definition of isomeric Heisenberg category}

The \emph{isomeric Heisenberg category} $\Heis_\kappa(\Cl)$ of central charge $\kappa \in \Z$ is the strict monoidal supercategory generated by objects $P$ and $Q$, whose identity endomorphisms are represented by 
$\begin{tikzpicture}[H,baseline=-2mm,scale=.75]
\draw[->] (0,-0.2) \botlabel{i} -- (0,0.2);
\end{tikzpicture}$
and
$\begin{tikzpicture}[H,baseline=-2mm,scale=.75]
\draw[<-] (0,-0.2) \botlabel{i} -- (0,0.2);
\end{tikzpicture}$,
and morphisms
\begin{align*}
\begin{tikzpicture}[H,centerzero]
\draw[->] (0,-0.3) -- (0,0.3);
\token{0,0};
\end{tikzpicture} &: P \to P,&
\begin{tikzpicture}[H,centerzero]
\draw[->] (0,-0.3) -- (0,0.3);
\singdot{0,0};
\end{tikzpicture} &: P \to P,&
\begin{tikzpicture}[H,centerzero]
\draw[->] (0.3,-0.3) -- (-0.3,0.3);
\draw[->] (-0.3,-0.3) -- (0.3,0.3);
\end{tikzpicture}\ &: P \otimes P \to P \otimes P,&
\;\begin{tikzpicture}[H,anchorbase]
\draw[->] (-0.25,0.25) -- (-0.25,0) arc(180:360:0.25) -- (0.25,0.25);
\end{tikzpicture}\ &: \one \to Q \otimes P,&
\;\begin{tikzpicture}[H,anchorbase]
\draw[->] (-0.25,-0.25) -- (-0.25,0) arc(180:0:0.25) -- (0.25,-0.25);
\end{tikzpicture}\ &: P \otimes Q \to \one,
\end{align*}
subject to the relations recorded shortly.
We refer to the morphism $\begin{tikzpicture}[H,centerzero]
\draw[->] (0,-0.25) -- (0,0.25);
\token{0,0};
\end{tikzpicture}$ as the \emph{Clifford token} and the morphism 
$\begin{tikzpicture}[H,centerzero]
\draw[->] (0,-0.25) -- (0,0.25);
\singdot{0,0};
\end{tikzpicture}$ 
as the \emph{dot}.
The $\Z/2$-grading is defined by declaring that the Clifford token is odd and all other generating morphisms are even.

Before formulating the relations, we say a bit more about our diagrammatic conventions.
We use the following to denote the morphisms obtained by ``rotating'' the generating morphisms:
\begin{align}
\begin{tikzpicture}[H,centerzero,scale=1.1]
\draw[<-] (0,-0.4) -- (0,0.4);
\token{0,0};
\end{tikzpicture}
&:=
\begin{tikzpicture}[H,centerzero,scale=1.1]
\draw[<-] (0.3,-0.4) -- (0.3,0) arc(0:180:0.15) arc(360:180:0.15) -- (-0.3,0.4);
\token{0,0};
\end{tikzpicture}\ ,&
\begin{tikzpicture}[H,centerzero,scale=1.1]
\draw[<-] (0,-0.4) -- (0,0.4);
\singdot{0,0};
\end{tikzpicture}
&:=
\begin{tikzpicture}[H,centerzero,scale=1.1]
\draw[<-] (0.3,-0.4) -- (0.3,0) arc(0:180:0.15) arc(360:180:0.15) -- (-0.3,0.4);
\singdot{0,0};
\end{tikzpicture}\ ,&
\begin{tikzpicture}[H,centerzero]
\draw[<-] (0.3,-0.3)  -- (-0.3,0.3);
\draw[->] (-0.3,-0.3) -- (0.3,0.3);
\end{tikzpicture}
&:=
\begin{tikzpicture}[H,centerzero,scale=1.2]
\draw[->] (0.1,-0.3) \braidup (-0.1,0.3);
\draw[->] (-0.4,0.3) -- (-0.4,0.1) to[out=down,in=left] (-0.2,-0.2) to[out=right,in=left] (0.2,0.2) to[out=right,in=up] (0.4,-0.1) -- (0.4,-0.3);
\end{tikzpicture}\ ,&
\begin{tikzpicture}[H,centerzero]
\draw[<-] (0.3,-0.3) -- (-0.3,0.3) ;
\draw[<-] (-0.3,-0.3) -- (0.3,0.3) ;
\end{tikzpicture}
&:=
\begin{tikzpicture}[H,centerzero,scale=1.2]
\draw[<-] (0.1,-0.3) \braidup (-0.1,0.3);
\draw[->] (-0.4,0.3)  -- (-0.4,0.1) to[out=down,in=left] (-0.2,-0.2) to[out=right,in=left] (0.2,0.2) to[out=right,in=up] (0.4,-0.1) -- (0.4,-0.3);
\end{tikzpicture}
\ .\label{hrightpivot}
\end{align}
When a dot is labelled by a multiplicity,
we mean to take its power under vertical composition.
For a polynomial
$f(x) = \sum_{r=0}^n c_r x^r$, we use
shorthand
\begin{equation}\label{singlepin}
\begin{tikzpicture}[H,centerzero]
\draw[-] (0,-0.3) -- (0,0.3);
\pin{0,0}{-.7,0}{f(x)};
\end{tikzpicture}
= \begin{tikzpicture}[H,centerzero]
\draw[-] (0,-0.3) -- (0,0.3);
\pin{0,0}{.7,0}{f(x)};
\end{tikzpicture}\
:=
\sum_{r=0}^n c_r\
\begin{tikzpicture}[H,centerzero]
\draw[-] (0,-0.3) -- (0,0.3);
\multdot{0,0}{west}{r};
\end{tikzpicture}
\end{equation}
to ``pin'' $f(x)$ to a dot on a string (which may be oriented either upward or downward).
Similarly, for
$f(x,y) = \sum_{r=0}^n\sum_{s=0}^m c_{r,s} x^r y^s$, we use
\begin{align}\label{doublepin}
\begin{tikzpicture}[H,centerzero]
\draw[-] (0,-0.3) -- (0,0.3);
\draw[-] (0.4,-0.3) -- (0.4,0.3);
\pinpin{.4,0}{0,0}{-.9,0}{f(x,y)};
\end{tikzpicture}
=
\begin{tikzpicture}[H,centerzero]
\draw[-] (0,-0.3)-- (0,0.3);
\draw[-] (0.4,-0.3) -- (0.4,0.3);
\pinpin{0,0}{0.4,0}{1.3,0}{f(x,y)};
\end{tikzpicture}
&:=
\sum_{r=0}^n\sum_{s=0}^m c_{r,s}
\begin{tikzpicture}[H,centerzero]
\draw[-] (0,-0.3)-- (0,0.3);
\draw[-] (0.4,-0.3) -- (0.4,0.3);
\multdot{.4,0}{west}{s};
\multdot{0,0}{east}{r};
\end{tikzpicture}
\ .
\end{align}
This notation extends in the obvious way to polynomials $f(x,y,z)$
in three variables pinned to three dots.

\begin{conv}\label{jonisdotty}
In \cref{doublepin}, the first variable $x$ corresponds to the left dot and the second variable $y$ corresponds to the right dot. 
When we use this notation in more general situations, the first variable $x$ corresponds to the dot whose Cartesian coordinate is the smallest in the lexicographic ordering on $\R^2$.
Thus, $x$ corresponds the left dot unless the two dots lie in the same vertical line, in which case $x$ is the lower dot.
When a polynomial in $x,y,z$ is pinned to three dots, $x$ corresponds to the dot with the lexicographically smallest Cartesian coordinate and $z$ corresponds to the dot with the lexicographically largest one.
\end{conv}

There are four families of defining relations.
First, we have the {\em zig-zag relations} asserting that $Q$ is right dual to $P$:
\begin{align}
\label{adjright}
\begin{tikzpicture}[H,centerzero,scale=1.2]
\draw[->] (-0.3,0.4) -- (-0.3,0) arc(180:360:0.15) arc(180:0:0.15) -- (0.3,-0.4);
\end{tikzpicture}
&=
\begin{tikzpicture}[H,centerzero,scale=1.2]
\draw[<-] (0,-0.4) -- (0,0.4);
\end{tikzpicture}
\ ,&
\begin{tikzpicture}[H,centerzero,scale=1.2]
\draw[->] (-0.3,-0.4) -- (-0.3,0) arc(180:0:0.15) arc(180:360:0.15) -- (0.3,0.4);
\end{tikzpicture}
&=
\begin{tikzpicture}[H,centerzero,scale=1.2]
\draw[->] (0,-0.4) -- (0,0.4);
\end{tikzpicture}\ .
\end{align}
Next, we have the {\em affine Sergeev superalgebra relations}:
\begin{align} \label{sergeev1}
\begin{tikzpicture}[H,centerzero]
\draw[->] (0.2,-0.4) to[out=135,in=down] (-0.15,0) to[out=up,in=225] (0.2,0.4);
\draw[->] (-0.2,-0.4) to[out=45,in=down] (0.15,0) to[out=up,in=-45] (-0.2,0.4);
\end{tikzpicture}
&=
\begin{tikzpicture}[H,centerzero]
\draw[->] (-0.2,-0.4) -- (-0.2,0.4);
\draw[->] (0.2,-0.4) -- (0.2,0.4);
\end{tikzpicture}
\ ,&
\begin{tikzpicture}[H,centerzero]
\draw[->] (0.4,-0.4) -- (-0.4,0.4);
\draw[->] (0,-0.4) to[out=135,in=down] (-0.32,0) to[out=up,in=225] (0,0.4);
\draw[->] (-0.4,-0.4) -- (0.4,0.4);
\end{tikzpicture}
&=
\begin{tikzpicture}[H,centerzero]
\draw[->] (0.4,-0.4) -- (-0.4,0.4);
\draw[->] (0,-0.4) to[out=45,in=down] (0.32,0) to[out=up,in=-45] (0,0.4);
\draw[->] (-0.4,-0.4) -- (0.4,0.4);
\end{tikzpicture}
\ ,\\\label{sergeev2}
\begin{tikzpicture}[H,centerzero]
\draw[->] (0,-0.3) -- (0,0.3);
\token{0,-0.1};
\token{0,0.1};
\end{tikzpicture}
&= -\ 
\begin{tikzpicture}[H,centerzero]
\draw[->] (0,-0.3) -- (0,0.3);
\end{tikzpicture}
\ ,&
\begin{tikzpicture}[H,centerzero]
\draw[->] (0.3,-0.3) -- (-0.3,0.3);
\draw[->] (-0.3,-0.3) -- (0.3,0.3);
\token{-0.15,-0.15};
\end{tikzpicture}
&=
\begin{tikzpicture}[H,centerzero]
\draw[->] (0.3,-0.3) -- (-0.3,0.3);
\draw[->] (-0.3,-0.3) -- (0.3,0.3);
\token{0.15,0.15};
\end{tikzpicture}
\ ,
\\ \label{affsergeev}
\begin{tikzpicture}[H,centerzero,scale=1.1]
\draw[->] (0,-0.3) -- (0,0.3);
\token{0,-0.1};
\singdot{0,0.1};
\end{tikzpicture}
&= -
\begin{tikzpicture}[H,centerzero,scale=1.1]
\draw[->] (0,-0.3) -- (0,0.3);
\token{0,0.1};
\singdot{0,-0.1};
\end{tikzpicture}
,&
\begin{tikzpicture}[H,centerzero,scale=1.1]
\draw[->] (0.3,-0.3) -- (-0.3,0.3);
\draw[->] (-0.3,-0.3) -- (0.3,0.3);
\singdot{-0.15,-0.15};
\end{tikzpicture}
-
\begin{tikzpicture}[H,centerzero,scale=1.1]
\draw[->] (0.3,-0.3) -- (-0.3,0.3);
\draw[->] (-0.3,-0.3) -- (0.3,0.3);
\singdot{0.15,0.15};
\end{tikzpicture}
&=
\begin{tikzpicture}[H,centerzero,scale=1.1]
\draw[->] (-0.2,-0.3) -- (-0.2,0.3);
\draw[->] (0.2,-0.3) -- (0.2,0.3);
\end{tikzpicture}\ 
-
\begin{tikzpicture}[H,centerzero,scale=1.1]
\draw[->] (-0.2,-0.3) -- (-0.2,0.3);
\draw[->] (0.2,-0.3) -- (0.2,0.3);
\token{0.2,0};
\token{-0.2,0};
\end{tikzpicture}\,.
\end{align}
Third, we have the {\em inversion relation},
which asserts that there are additional generators satisfying
the relations needed to ensure that the matrix
\begin{equation}\label{vikings}
M_\kappa :=
\begin{dcases}
\begin{pmatrix}
\begin{tikzpicture}[H,centerzero]
\draw[->] (-0.25,-0.25) -- (0.25,0.25);
\draw[<-] (0.25,-0.25) -- (-0.25,0.25);
\end{tikzpicture} &
\begin{tikzpicture}[H,anchorbase]
\draw[->] (-0.2,0.3) -- (-0.2,0) arc(180:360:0.2) -- (0.2,0.3);
\token{0.2,-.05};
\end{tikzpicture}
&
\begin{tikzpicture}[H,anchorbase]
\draw[->] (-0.2,0.3) -- (-0.2,0) arc(180:360:0.2) -- (0.2,0.3);
\end{tikzpicture}
&
\begin{tikzpicture}[H,anchorbase]
\draw[->] (-0.2,0.3) -- (-0.2,0) arc(180:360:0.2) -- (0.2,0.3);
\singdot{0.2,0.1};
\token{0.2,-.05};
\end{tikzpicture}
&
\begin{tikzpicture}[H,anchorbase]
\draw[->] (-0.2,0.3) -- (-0.2,0) arc(180:360:0.2) -- (0.2,0.3);
\singdot{0.2,0.1};
\end{tikzpicture}
&
 \!\!\!\!     \cdots\!\!
&
\begin{tikzpicture}[H,anchorbase]
\draw[->] (-0.2,0.3) -- (-0.2,0) arc(180:360:0.2) -- (0.2,0.3);
\multdot{0.2,0.1}{west}{-\kappa-1};
\token{0.2,-.05};
\end{tikzpicture}
&\!\!\!
\begin{tikzpicture}[H,anchorbase]
\draw[->] (-0.2,0.3) -- (-0.2,0) arc(180:360:0.2) -- (0.2,0.3);
\multdot{0.2,0.1}{west}{-\kappa-1};
\end{tikzpicture}
\end{pmatrix}
&\text{if } \kappa \leq 0\\
\begin{pmatrix}
\begin{tikzpicture}[H,centerzero]
\draw[->] (-0.25,-0.25) -- (0.25,0.25);
\draw[<-] (0.25,-0.25) -- (-0.25,0.25);
\end{tikzpicture} &
\begin{tikzpicture}[H,anchorbase]
\draw[->] (-0.2,-0.3) -- (-0.2,0) arc(180:0:0.2) -- (0.2,-0.3);
\token{-0.2,-0.15};
\end{tikzpicture}
&
\begin{tikzpicture}[H,anchorbase]
\draw[->] (-0.2,-0.3) -- (-0.2,0) arc(180:0:0.2) -- (0.2,-0.3);
\end{tikzpicture}
&
\begin{tikzpicture}[H,anchorbase]
\draw[->] (-0.2,-0.3) -- (-0.2,0) arc(180:0:0.2) -- (0.2,-0.3);
\singdot{-0.2,0};
\token{-0.2,-0.15};
\end{tikzpicture}
&
\begin{tikzpicture}[H,anchorbase]
\draw[->] (-0.2,-0.3) -- (-0.2,0) arc(180:0:0.2) -- (0.2,-0.3);
\singdot{-0.2,0};
\end{tikzpicture}
&
\!\!\!\cdots\!\!\!
&
\begin{tikzpicture}[H,anchorbase]
\draw[->] (-0.2,-0.3) -- (-0.2,0) arc(180:0:0.2) -- (0.2,-0.3);
\multdot{-0.2,0}{east}{\kappa-1};
\token{-0.2,-0.15};
\end{tikzpicture}
&\!\!\!
\begin{tikzpicture}[H,anchorbase]
\draw[->] (-0.2,-0.3) -- (-0.2,0) arc(180:0:0.2) -- (0.2,-0.3);
\multdot{-0.2,0}{east}{\kappa-1};
\end{tikzpicture}\ 
\end{pmatrix}^\transpose &\text{if } \kappa \geq 0
\end{dcases}
\end{equation}
is an isomorphism.
We introduce the following shorthands for morphisms arising from
the entries of the two-sided inverse of the matrix $M_\kappa$:
\begin{itemize}
\item
Let $\ \begin{tikzpicture}[H,centerzero]
\draw[->] (0.25,-0.25) -- (-0.25,0.25);
\draw[<-] (-0.25,-0.25) -- (0.25,0.25);
\end{tikzpicture}\ $ be the first entry of
the inverse matrix $M_\kappa^{-1}$.
\item
Let $\ \begin{tikzpicture}[H,anchorbase]
\draw[<-] (-0.2,-0.2) -- (-0.2,0) arc(180:0:0.2) -- (0.2,-0.2);
\end{tikzpicture}\ $ be the last entry of $M_\kappa^{-1}$
if $\kappa < 0$
or $\begin{tikzpicture}[H,anchorbase,scale=1.1]
\draw[->] (.2,-.2) to [out=90,in=-90] (-.2,.2) to[out=90,in=180] (0,.4) to[out=0,in=90] (.2,.2) to [out=-90,in=90] (-.2,-.2);
\multdot{-0.2,0.2}{east}{\kappa};
\end{tikzpicture}\ $ if $\kappa \geq 0$.
\item Let
$\ \begin{tikzpicture}[H,anchorbase]
\draw[<-] (-0.2,0.2) -- (-0.2,0) arc(180:360:0.2) -- (0.2,0.2);
\end{tikzpicture}\ $ be the last entry of $-M_\kappa^{-1}$
if $\kappa > 0$ or $\ \begin{tikzpicture}[H,anchorbase,scale=1.1]
\draw[->] (.2,.2) to [out=-90,in=90] (-.2,-.2) to[out=-90,in=180] (0,-.4) to[out=0,in=-90] (.2,-.2) to [out=90,in=-90] (-.2,.2);
\multdot{0.2,-0.2}{west}{-\kappa};
\end{tikzpicture}$ if $\kappa \leq 0$.
\end{itemize}
All of these morphisms are even.
Finally, we have the {\em odd bubble relation},
which asserts that
\begin{equation}\label{hoddbubble}
\begin{tikzpicture}[H,baseline=-1mm,scale=.8]
\draw[-] (-0.25,0) arc(180:-180:0.25);
\draw[-] (-.18,.18) to (.18,-.18);
\draw[-] (-.18,-.18) to (.18,.18);
\end{tikzpicture} :=
\begin{dcases}
\begin{tikzpicture}[H,centerzero]
\leftbub{0,0};
\token{0.22,0.11};
\multdot{0.22,-0.11}{west}{-\kappa};
\end{tikzpicture}&\text{if $\kappa \leq 0$}\\
\begin{tikzpicture}[H,centerzero]
\rightbub{0,0};
\token{-0.22,-0.11};
\multdot{-0.22,0.11}{east}{\kappa};
\end{tikzpicture}&\text{if $\kappa > 0$}
\end{dcases}
\end{equation}
equals 0.

\begin{rem}\label{likeinhere}
We refer to the supercategory defined in the same way as
$\Heis_\kappa(\Cl)$ but with the final odd bubble relation omitted as the {\em non-reduced isomeric Heisenberg category}, denoted $\widehat{\Heis}_\kappa(\Cl)$. 
This is what was studied in \cite{Sav19}.
We have added the odd bubble relation since it leads to significant simplifications and is satisfied in all of the applications that we are interested in.
When checking that the odd bubble relation holds in a categorical action, it is useful to know in $\widehat{\Heis}_\kappa(\Cl)$
that 
$\begin{tikzpicture}[H,baseline=-1mm,scale=.8]
\draw[-] (-0.25,0) arc(180:-180:0.25);
\draw[-] (-.18,.18) to (.18,-.18);
\draw[-] (-.18,-.18) to (.18,.18);
\end{tikzpicture}$ slides freely across strings, i.e.,
\begin{align}\label{strictlycentral}
\begin{tikzpicture}[H,baseline=-1mm,scale=.8]
\draw[->] (-.5,-.4) to (-.5,.4);
\draw[-] (-0.25,0) arc(180:-180:0.25);
\draw[-] (-.18,.18) to (.18,-.18);
\draw[-] (-.18,-.18) to (.18,.18);
\end{tikzpicture}
&=
\begin{tikzpicture}[H,baseline=-1mm,scale=.8]
\draw[->] (.5,-.4) to (.5,.4);
\draw[-] (-0.25,0) arc(180:-180:0.25);
\draw[-] (-.18,.18) to (.18,-.18);
\draw[-] (-.18,-.18) to (.18,.18);
\end{tikzpicture}\ ,&
\begin{tikzpicture}[H,baseline=-1mm,scale=.8]
\draw[<-] (-.5,-.4) to (-.5,.4);
\draw[-] (-0.25,0) arc(180:-180:0.25);
\draw[-] (-.18,.18) to (.18,-.18);
\draw[-] (-.18,-.18) to (.18,.18);
\end{tikzpicture}
&=
\begin{tikzpicture}[H,baseline=-1mm,scale=.8]
\draw[<-] (.5,-.4) to (.5,.4);
\draw[-] (-0.25,0) arc(180:-180:0.25);
\draw[-] (-.18,.18) to (.18,-.18);
\draw[-] (-.18,-.18) to (.18,.18);
\end{tikzpicture}\ .
\end{align}
This follows from
\cite[(37)--(38)]{Sav19}.
\end{rem}

\subsection{Two natural symmetries}

The Clifford token on a downward string satisfies
\begin{equation} \label{torch}
\begin{tikzpicture}[H,anchorbase]
\draw[<-] (0,-0.4) -- (0,0.7);
\token{0,0};
\token{0,0.3};
\end{tikzpicture}
\overset{\cref{hrightpivot}}{=}\
\begin{tikzpicture}[H,anchorbase]
\draw[->] (-0.9,0.7) -- (-0.9,0.3) arc(180:360:0.15) arc(180:0:0.15) -- (-0.3,0) arc(180:360:0.15) arc(180:0:0.15) -- (0.3,-0.4);
\token{0,0};
\token{-0.6,0.3};
\end{tikzpicture}
\overset{\cref{intlaw}}{=} -\
\begin{tikzpicture}[H,anchorbase]
\draw[->] (-0.9,0.7) -- (-0.9,-0.1) arc(180:360:0.15) -- (-0.6,0.1) arc(180:0:0.15) arc(180:360:0.15) -- (0,0.4) arc(180:0:0.15) -- (0.3,-0.4);
\token{-0.6,-0.1};
\token{0,0.4};
\end{tikzpicture}
\overset{\cref{adjright}}{=} -\
\begin{tikzpicture}[H,anchorbase]
\draw[->] (-0.6,0.55) -- (-0.6,-0.2) arc(180:360:0.15) -- (-0.3,0.2) arc(180:0:0.15) -- (0,-0.55);
\token{-0.3,-0.1};
\token{-0.3,0.1};
\end{tikzpicture}
\overset{\cref{sergeev2}}{=}
\begin{tikzpicture}[H,anchorbase]
\draw[->] (-0.6,0.55) -- (-0.6,-0.2) arc(180:360:0.15) -- (-0.3,0.2) arc(180:0:0.15) -- (0,-0.55);
\end{tikzpicture}
\overset{\cref{adjright}}{=}
\begin{tikzpicture}[H,anchorbase]
\draw[<-] (0,-0.4) -- (0,0.7);
\end{tikzpicture}
\ .
\end{equation}
From this and similar arguments for the other defining relations, it follows that
there is an isomorphism
of 2-supercategories, which we call the \emph{Chevalley involution},
\begin{equation} \label{mirror}
\tT : \Heis_\kappa(\Cl) \rightarrow \Heis_{-\kappa}(\Cl)^\op
\end{equation}
defined on a string diagram by reflecting in a horizontal axis, also multiplying by $(-1)^{n+\binom{m}{2}}$ where $n$ is the total number of crossings, leftward cups and leftward caps
in the diagram,
and $m$ is the total number of Clifford tokens.
Before applying this rule, the Clifford tokens should be arranged so that no two are at the same horizontal level.
For example:
\begin{equation}\label{backthen}
\tT\left(\ 
\begin{tikzpicture}[H,centerzero]
\draw[->] (-0.2,-0.3) -- (-0.2,0.3);
\draw[->] (0.2,-0.3) -- (0.2,0.3);
\token{-.2,0};\token{.2,0};
\end{tikzpicture}\!\right)=
\tT\left(\ 
\begin{tikzpicture}[H,centerzero]
\draw[->] (-0.2,-0.3) -- (-0.2,0.3);
\draw[->] (0.2,-0.3) -- (0.2,0.3);
\token{-.2,.1};\token{.2,-.1};
\end{tikzpicture}
\!\right)=-\ 
\begin{tikzpicture}[H,centerzero]
\draw[<-] (-0.2,-0.3) -- (-0.2,0.3);
\draw[<-] (0.2,-0.3) -- (0.2,0.3);
\token{-.2,-.1};\token{.2,.1};
\end{tikzpicture}=
\begin{tikzpicture}[H,centerzero]
\draw[<-] (-0.2,-0.3) -- (-0.2,0.3);
\draw[<-] (0.2,-0.3) -- (0.2,0.3);
\token{-.2,0};\token{.2,0};
\end{tikzpicture}.
\end{equation}

There also an isomorphism of strict monoidal supercategories
\begin{equation}\label{taraa}
\tR : \Heis_\kappa(\Cl) \to \Heis_{-\kappa}(\Cl)^\rev 
\end{equation}
defined on a string diagram by reflecting in a vertical axis, then multiplying by $(-1)^{n + t \kappa}$, where $n$ is the number of crossings, and $t$ is the number of Clifford tokens on downward strands.  Again, Clifford tokens should be arranged so that no two are at the same horizontal level.  To prove this, one checks that the images of the defining relations of $\Heis_\kappa(\Cl)$ under $\tR$ all hold, as follows from the alternative presentation of $\Heis_{-\kappa}(\Cl)$ established in \cite[Th.~1.2]{Sav19}. 

\subsection{Further relations}

Many further relations are
derived from the defining relations
in \cite[Th.~1.3]{Sav19} (without assuming the odd bubble relation).
For example, we have that
\begin{align}\label{sergeev12}
\begin{tikzpicture}[H,centerzero]
\draw[->] (0.3,-0.3) -- (-0.3,0.3);
\draw[->] (-0.3,-0.3) -- (0.3,0.3);
\token{-0.15,0.15};
\end{tikzpicture}
&=
\begin{tikzpicture}[H,centerzero]
\draw[->] (0.3,-0.3) -- (-0.3,0.3);
\draw[->] (-0.3,-0.3) -- (0.3,0.3);
\token{0.15,-0.15};
\end{tikzpicture},
&
\begin{tikzpicture}[H,centerzero,scale=1.1]
\draw[->] (0.3,-0.3) -- (-0.3,0.3);
\draw[->] (-0.3,-0.3) -- (0.3,0.3);
\singdot{-0.15,0.15};
\end{tikzpicture}
 -
\begin{tikzpicture}[H,centerzero,scale=1.1]
\draw[->] (0.3,-0.3) -- (-0.3,0.3);
\draw[->] (-0.3,-0.3) -- (0.3,0.3);
\singdot{0.15,-0.15};
\end{tikzpicture}
& =\ 
\begin{tikzpicture}[H,centerzero,scale=1.1]
\draw[->] (-0.2,-0.3) -- (-0.2,0.3);
\draw[->] (0.2,-0.3) -- (0.2,0.3);
\end{tikzpicture}\ 
+
\begin{tikzpicture}[H,centerzero,scale=1.1]
\draw[->] (-0.2,-0.3) -- (-0.2,0.3);
\draw[->] (0.2,-0.3) -- (0.2,0.3);
\token{0.2,0};
\token{-0.2,0};
\end{tikzpicture},
\end{align}
as is easily seen by composing the last relations in \cref{sergeev2,affsergeev} on the top and bottom with a crossing.
From the definitions \cref{hrightpivot} and \cref{adjright}, it follows that Clifford tokens, dots and crossings
slide over rightward cups and caps:
\begin{align} \label{ruby}
\begin{tikzpicture}[H,anchorbase,scale=1.2]
\draw[->] (-0.2,-0.2) -- (-0.2,0) arc (180:0:0.2) -- (0.2,-0.2);
\token{-0.2,0};
\end{tikzpicture}\ &=\
\begin{tikzpicture}[H,anchorbase,scale=1.2]
\draw[->] (-0.2,-0.2) -- (-0.2,0) arc (180:0:0.2) -- (0.2,-0.2);
\token{0.2,0};
\end{tikzpicture}
\ ,&
\begin{tikzpicture}[H,anchorbase,scale=1.2]
\draw[->] (-0.2,0.2) -- (-0.2,0) arc (180:360:0.2) -- (0.2,0.2);
\token{-0.2,0};
\end{tikzpicture}
\ &=\
\begin{tikzpicture}[H,anchorbase,scale=1.2]
\draw[->] (-0.2,0.2) -- (-0.2,0) arc (180:360:0.2) -- (0.2,0.2);
\token{0.2,0};
\end{tikzpicture}\ ,&
\begin{tikzpicture}[H,anchorbase,scale=1.2]
\draw[->] (-0.2,-0.2) -- (-0.2,0) arc (180:0:0.2) -- (0.2,-0.2);
\singdot{-0.2,0};
\end{tikzpicture}
\ &=\
\begin{tikzpicture}[H,anchorbase,scale=1.2]
\draw[->] (-0.2,-0.2) -- (-0.2,0) arc (180:0:0.2) -- (0.2,-0.2);
\singdot{0.2,0};
\end{tikzpicture}
\ ,&
\begin{tikzpicture}[H,anchorbase,scale=1.2]
\draw[->] (-0.2,0.2) -- (-0.2,0) arc (180:360:0.2) -- (0.2,0.2);
\singdot{-0.2,0};
\end{tikzpicture}
\ &=\
\begin{tikzpicture}[H,anchorbase,scale=1.2]
\draw[->] (-0.2,0.2) -- (-0.2,0) arc (180:360:0.2) -- (0.2,0.2);
\singdot{0.2,0};
\end{tikzpicture}
\ ,\\
\begin{tikzpicture}[H,centerzero,scale=1.2]
\draw[->] (-0.2,0.3) -- (-0.2,0.1) arc(180:360:0.2) -- (0.2,0.3);
\draw[->] (-0.3,-0.3) to[out=up,in=down] (0,0.3);
\end{tikzpicture}
&=
\begin{tikzpicture}[H,centerzero,scale=1.2]
\draw[->] (-0.2,0.3) -- (-0.2,0.1) arc(180:360:0.2) -- (0.2,0.3);
\draw[->] (0.3,-0.3) to[out=up,in=down] (0,0.3);
\end{tikzpicture}
\ ,&
\begin{tikzpicture}[H,centerzero,scale=1.2]
\draw[->] (-0.2,-0.3) -- (-0.2,-0.1) arc(180:0:0.2) -- (0.2,-0.3);
\draw[->] (-0.3,0.3) \braiddown (0,-0.3);
\end{tikzpicture}
&=
\begin{tikzpicture}[H,centerzero,scale=1.2]
\draw[->] (-0.2,-0.3) -- (-0.2,-0.1) arc(180:0:0.2) -- (0.2,-0.3);
\draw[->] (0.3,0.3) \braiddown (0,-0.3);
\end{tikzpicture}
\ ,&
\begin{tikzpicture}[H,centerzero,scale=1.2]
\draw[->] (-0.2,0.3) -- (-0.2,0.1) arc(180:360:0.2) -- (0.2,0.3);
\draw[<-] (-0.3,-0.3) to[out=up,in=down] (0,0.3);
\end{tikzpicture}
&=
\begin{tikzpicture}[H,centerzero,scale=1.2]
\draw[->] (-0.2,0.3) -- (-0.2,0.1) arc(180:360:0.2) -- (0.2,0.3);
\draw[<-] (0.3,-0.3) to[out=up,in=down] (0,0.3);
\end{tikzpicture}
\ ,&
\begin{tikzpicture}[H,centerzero,scale=1.2]
\draw[->] (-0.2,-0.3) -- (-0.2,-0.1) arc(180:0:0.2) -- (0.2,-0.3);
\draw[<-] (-0.3,0.3) \braiddown (0,-0.3);
\end{tikzpicture}
&=
\begin{tikzpicture}[H,centerzero,scale=1.2]
\draw[->] (-0.2,-0.3) -- (-0.2,-0.1) arc(180:0:0.2) -- (0.2,-0.3);
\draw[<-] (0.3,0.3) \braiddown (0,-0.3);
\end{tikzpicture}\ .\label{hwax}
\end{align}
In \cite[Lem.~2.12]{Sav19}, it is shown that
\begin{align}\label{adjleft}
\begin{tikzpicture}[H,centerzero,scale=1.2]
\draw[<-] (-0.3,0.4) -- (-0.3,0) arc(180:360:0.15) arc(180:0:0.15) -- (0.3,-0.4);
\end{tikzpicture}
&=
\begin{tikzpicture}[H,centerzero,scale=1.2]
\draw[->] (0,-0.4) -- (0,0.4);
\end{tikzpicture}
\ ,&
\begin{tikzpicture}[H,centerzero,scale=1.2]
\draw[<-] (-0.3,-0.4) -- (-0.3,0) arc(180:0:0.15) arc(180:360:0.15) -- (0.3,0.4);
\end{tikzpicture}
&=
\begin{tikzpicture}[H,centerzero,scale=1.2]
\draw[<-] (0,-0.4) -- (0,0.4);
\end{tikzpicture}\ ,
\end{align}
that is, $Q$ is also left dual to $P$.
Moreover, by \cite[Lemmas~2.4, 2.7, 2.13]{Sav19},
\begin{align} \label{ruby2}
\begin{tikzpicture}[H,anchorbase,scale=1.2]
\draw[<-] (-0.2,-0.2) -- (-0.2,0) arc (180:0:0.2) -- (0.2,-0.2);
\token{-0.2,0};
\end{tikzpicture}
\ &=\ (-1)^\kappa\ 
\begin{tikzpicture}[H,anchorbase,scale=1.2]
\draw[<-] (-0.2,-0.2) -- (-0.2,0) arc (180:0:0.2) -- (0.2,-0.2);
\token{0.2,0};
\end{tikzpicture}
,&
\begin{tikzpicture}[H,anchorbase,scale=1.2]
\draw[<-] (-0.2,0.2) -- (-0.2,0) arc (180:360:0.2) -- (0.2,0.2);
\token{-0.2,0};
\end{tikzpicture}
\ &=\ (-1)^\kappa\ \ 
\begin{tikzpicture}[H,anchorbase,scale=1.2]
\draw[<-] (-0.2,0.2) -- (-0.2,0) arc (180:360:0.2) -- (0.2,0.2);
\token{0.2,0};
\end{tikzpicture},&
\begin{tikzpicture}[H,anchorbase,scale=1.2]
\draw[<-] (-0.2,-0.2) -- (-0.2,0) arc (180:0:0.2) -- (0.2,-0.2);
\singdot{-0.2,0};
\end{tikzpicture}
\ &=\
\begin{tikzpicture}[H,anchorbase,scale=1.2]
\draw[<-] (-0.2,-0.2) -- (-0.2,0) arc (180:0:0.2) -- (0.2,-0.2);
\singdot{0.2,0};
\end{tikzpicture}
,&
\begin{tikzpicture}[H,anchorbase,scale=1.2]
\draw[<-] (-0.2,0.2) -- (-0.2,0) arc (180:360:0.2) -- (0.2,0.2);
\singdot{-0.2,0};
\end{tikzpicture}
\ &=\
\begin{tikzpicture}[H,anchorbase,scale=1.2]
\draw[<-] (-0.2,0.2) -- (-0.2,0) arc (180:360:0.2) -- (0.2,0.2);
\singdot{0.2,0};
\end{tikzpicture}
,\\
\begin{tikzpicture}[H,centerzero,scale=1.2]
\draw[<-] (-0.2,0.3) -- (-0.2,0.1) arc(180:360:0.2) -- (0.2,0.3);
\draw[->] (-0.3,-0.3) to[out=up,in=down] (0,0.3);
\end{tikzpicture}
&=
\begin{tikzpicture}[H,centerzero,scale=1.2]
\draw[<-] (-0.2,0.3) -- (-0.2,0.1) arc(180:360:0.2) -- (0.2,0.3);
\draw[->] (0.3,-0.3) to[out=up,in=down] (0,0.3);
\end{tikzpicture}
\ ,&
\begin{tikzpicture}[H,centerzero,scale=1.2]
\draw[<-] (-0.2,-0.3) -- (-0.2,-0.1) arc(180:0:0.2) -- (0.2,-0.3);
\draw[->] (-0.3,0.3) \braiddown (0,-0.3);
\end{tikzpicture}
&=
\begin{tikzpicture}[H,centerzero,scale=1.2]
\draw[<-] (-0.2,-0.3) -- (-0.2,-0.1) arc(180:0:0.2) -- (0.2,-0.3);
\draw[->] (0.3,0.3) \braiddown (0,-0.3);
\end{tikzpicture}
\ ,&
\begin{tikzpicture}[H,centerzero,scale=1.2]
\draw[<-] (-0.2,0.3) -- (-0.2,0.1) arc(180:360:0.2) -- (0.2,0.3);
\draw[<-] (-0.3,-0.3) to[out=up,in=down] (0,0.3);
\end{tikzpicture}
&=
\begin{tikzpicture}[H,centerzero,scale=1.2]
\draw[<-] (-0.2,0.3) -- (-0.2,0.1) arc(180:360:0.2) -- (0.2,0.3);
\draw[<-] (0.3,-0.3) to[out=up,in=down] (0,0.3);
\end{tikzpicture}
\ ,&
\begin{tikzpicture}[H,centerzero,scale=1.2]
\draw[<-] (-0.2,-0.3) -- (-0.2,-0.1) arc(180:0:0.2) -- (0.2,-0.3);
\draw[<-] (-0.3,0.3) \braiddown (0,-0.3);
\end{tikzpicture}
&=
\begin{tikzpicture}[H,centerzero,scale=1.2]
\draw[<-] (-0.2,-0.3) -- (-0.2,-0.1) arc(180:0:0.2) -- (0.2,-0.3);
\draw[<-] (0.3,0.3) \braiddown (0,-0.3);
\end{tikzpicture}\ .\label{hwax2}
\end{align}
The dot slides in \cref{ruby2} are simpler than in \cite{Sav19} thanks to the odd bubble relation.

\begin{lem}\label{msc}
In $\Heis_\kappa(\Cl)$, all odd endomorphisms of $\one$ are 0.
In particular,
\begin{equation}\label{odddottedbubbles}
\begin{tikzpicture}[H,centerzero]
\leftbub{0,0};
\token{0.22,0.11};
\multdot{0.22,-0.11}{west}{n};
\end{tikzpicture}
=
\begin{tikzpicture}[H,centerzero]
\rightbub{0,0};
\token{-0.22,-0.11};
\multdot{-0.22,0.11}{east}{n};
\end{tikzpicture}
= 0
\end{equation}
for all $n \geq 0$.
\end{lem}

\begin{proof}
We just go through the argument in the case that $\kappa \leq 0$; the result for $\kappa \geq 0$ can then be deduced 
by applying $\tT$.
By a standard straightening argument which proceeds by induction on the number of crossings, using the infinite Grassmannian relation of \cite[(27)--(29)]{Sav19} and other relations therein, any endomorphism of $\one$
is a polynomial in the counterclockwise dotted bubbles
$\begin{tikzpicture}[H,centerzero]
\leftbub{0,0};
\multdot{0.24,0}{west}{n};
\end{tikzpicture}$ (which are even) and
$\begin{tikzpicture}[H,centerzero]
\leftbub{0,0};
\token{0.22,0.11};
\multdot{0.22,-0.11}{west}{n};
\end{tikzpicture}$ (which are odd)
for $n \geq -\kappa$.
To complete the proof, we show that
$\begin{tikzpicture}[H,centerzero]
\leftbub{0,0};
\token{0.22,0.11};
\multdot{0.22,-0.11}{west}{n};
\end{tikzpicture} = 0$ for all $n \geq -\kappa$. 
The case $n=-\kappa$ follows by the odd bubble relation, i.e., \cref{hoddbubble} is zero.
When $n > -\kappa$, the result follows because
$$
\begin{tikzpicture}[H,centerzero]
\leftbub{0,0};
\token{0.22,0.11};
\multdot{0.22,-0.11}{west}{n};
\end{tikzpicture}
\overset{\cref{ruby}}{=}
\begin{tikzpicture}[H,centerzero]
\leftbub{0,0};
\token{0.22,0.11};
\multdot{0.22,-0.11}{west}{n-1};
\singdot{-.22,-.11};
\end{tikzpicture}
\overset{\cref{ruby2}}{=}
\begin{tikzpicture}[H,centerzero]
\leftbub{0,0};
\singdot{.2,.15};
\token{0.25,0.0};
\multdot{0.2,-0.15}{west}{n-1};
\end{tikzpicture}
\overset{\cref{affsergeev}}{=}
-\ \begin{tikzpicture}[H,centerzero]
\leftbub{0,0};
\token{0.22,0.11};
\multdot{0.22,-0.11}{west}{n};
\end{tikzpicture}\ .
$$
\end{proof}

\begin{lem}\label{darkness}
For $n \geq 0$, we have that \begin{equation}
\begin{tikzpicture}[H,centerzero]
\leftbub{0,0};
\multdot{0.24,0}{west}{n};
\end{tikzpicture}
= 
\begin{tikzpicture}[H,centerzero]
\rightbub{0,0};
\multdot{-0.24,0}{east}{n};
\end{tikzpicture}
= 0
\end{equation}
when $n \equiv\kappa\pmod{2}$.
\end{lem}

\begin{proof}
We just prove the result for counterclockwise bubbles; it then follows for clockwise bubbles by applying \cref{mirror}.
We have that
\[
\begin{tikzpicture}[H,centerzero]
\leftbub{0,0};
\multdot{0.24,0}{west}{n};
\end{tikzpicture}
\overset{\cref{sergeev2}}{=} -\
\begin{tikzpicture}[H,centerzero]
\draw[->] (0,0.4) arc(90:180:0.2) -- (-0.2,-0.2) arc(180:360:0.2) -- (0.2,0.2) arc(0:90:0.2);
\multdot{0.2,-0.2}{west}{n};
\token{0.2,0.2};
\token{0.2,0};
\end{tikzpicture}
\stackrel{\cref{ruby2}}{=}
 -(-1)^{\kappa}\
\begin{tikzpicture}[H,centerzero]
\draw[->] (0,0.4) arc(90:180:0.2) -- (-0.2,-0.2) arc(180:360:0.2) -- (0.2,0.2) arc(0:90:0.2);
\multdot{0.2,-.2}{west}{n};
\token{-0.2,0.2};
\token{0.2,0};
\end{tikzpicture}
\stackrel{\cref{intlaw}}{=} (-1)^\kappa\
\begin{tikzpicture}[H,centerzero]
\draw[->] (0,0.4) arc(90:180:0.2) -- (-0.2,-0.2) arc(180:360:0.2) -- (0.2,0.2) arc(0:90:0.2);
\multdot{0.2,0}{west}{n};
\token{-0.2,-0.2};
\token{0.2,0.2};
\end{tikzpicture}
\overset{\cref{ruby}}{\underset{\cref{affsergeev}}{=}} (-1)^{\kappa+n}\
\begin{tikzpicture}[H,centerzero]
\draw[->] (0,0.4) arc(90:180:0.2) -- (-0.2,-0.2) arc(180:360:0.2) -- (0.2,0.2) arc(0:90:0.2);
\multdot{0.2,-0.2}{west}{n};
\token{0.2,0};
\token{0.2,0.2};
\end{tikzpicture}
\overset{\cref{sergeev2}}{=} (-1)^{\kappa+n+1}\ 
\begin{tikzpicture}[H,centerzero]
\leftbub{0,0};
\multdot{0.24,0}{west}{n};
\end{tikzpicture}.
\]
This implies that $\begin{tikzpicture}[H,centerzero]
\leftbub{0,0};
\multdot{0.24,0}{west}{n};
\end{tikzpicture}
=0$ when $\kappa+n$ is even.
\end{proof}

It is helpful to work systematically with generating functions, which in general will be formal Laurent series in $u^{-1}$.
If $f(u)$ is such a series, we use the notation $[f(u)]_{u:r}$
for its $u^r$-coefficient, $[f(u)]_{u:\geq 0}$ for its polynomial part,
$[f(u)]_{u:<0}$ for $f(u) - [f(u)]_{u:\geq 0}$, etc..
We view
\begin{equation}\label{favoritegf}
\frac{1}{u-x} = \sum_{n \geq 0} x^n u^{-n-1}=u^{-1} + u^{-2} x + u^{-3} x^2 + \dotsb \in \kk[x] \llbracket u^{-1} \rrbracket
\end{equation}
as a generating function for multiple dots on a string.  We introduce the shorthand notation
\begin{align}\label{idgf}
\begin{tikzpicture}[H,centerzero]
\draw[-] (0,-0.3) -- (0,0.3);
\circled{0,0}{u};
\end{tikzpicture}
&:=
\begin{tikzpicture}[H,centerzero]
\draw[-] (0,-0.3) -- (0,0.3);
\pin{0,0}{.8,0}{\frac{1}{u-x}};
\end{tikzpicture}\
,&
\begin{tikzpicture}[H,centerzero]
\draw[-] (0,-0.3) -- (0,0.3);
\circledbar{0,0}{u};
\end{tikzpicture}
&:=
\begin{tikzpicture}[H,centerzero]
\draw[-] (0,-0.3) -- (0,0.3);
\pin{0,0}{.8,0}{\frac{1}{u+x}};
\end{tikzpicture}\ .
\end{align}
For a polynomial $f(x) \in \kk[x]$, we have the useful tricks
\begin{align} \label{trick}
\begin{tikzpicture}[H,centerzero,scale=1.1]
\draw[-] (0,-0.25) -- (0,0.25);
\pin{0,0}{-0.6,0}{f(x)};
\end{tikzpicture}
&=
\left[
f(u)\  \begin{tikzpicture}[H,centerzero,scale=1.1]
\draw[-] (0,-0.25) -- (0,0.25);
\circled{0,0}{u};
\end{tikzpicture}\ 
\right]_{u:-1},&
\begin{tikzpicture}[H,centerzero,scale=1.1]
\draw[-] (0,-0.25) -- (0,0.25);
\pin{0,0}{-0.7,0}{f(-x)};\region{.25,0}{\lambda};
\end{tikzpicture}
&=
\left[ f(u)\
\begin{tikzpicture}[H,centerzero,scale=1.1]
\draw[-] (0,-0.25) -- (0,0.25);
\circledbar{0,0}{u};
\end{tikzpicture}\ 
\right]_{u:-1}
\end{align}
for $f(x) \in \kk[x]$.

From \cref{ruby,ruby2}, we get that
\begin{align}\label{slippery}
\begin{tikzpicture}[H,anchorbase,scale=1.2]
\draw[->] (-0.2,-0.35) -- (-0.2,-0.1) arc(180:0:0.2) -- (0.2,-0.35);
\circled{-0.2,-0.1}{u};
\end{tikzpicture}
&=
\begin{tikzpicture}[H,anchorbase,scale=1.2]
\draw[->] (-0.2,-0.35) -- (-0.2,-0.1) arc(180:0:0.2) -- (0.2,-0.35);
\circled{0.2,-0.1}{u};
\end{tikzpicture}
\ ,&
\begin{tikzpicture}[H,anchorbase,scale=1.2]
\draw[->] (-0.2,-0.35) -- (-0.2,-0.1) arc(180:0:0.2) -- (0.2,-0.35);
\circledbar{-0.2,-0.1}{u};
\end{tikzpicture}
&=
\begin{tikzpicture}[H,anchorbase,scale=1.2]
\draw[->] (-0.2,-0.35) -- (-0.2,-0.1) arc(180:0:0.2) -- (0.2,-0.35);
\circledbar{0.2,-0.1}{u};
\end{tikzpicture}
\ ,&
\begin{tikzpicture}[H,anchorbase,scale=1.2]
\draw[<-] (-0.2,-0.35) -- (-0.2,-0.1) arc(180:0:0.2) -- (0.2,-0.35);
\circled{-0.2,-0.1}{u};
\end{tikzpicture}
&=
\begin{tikzpicture}[H,anchorbase,scale=1.2]
\draw[<-] (-0.2,-0.35) -- (-0.2,-0.1) arc(180:0:0.2) -- (0.2,-0.35);
\circled{0.2,-0.1}{u};
\end{tikzpicture}
\ ,&
\begin{tikzpicture}[H,anchorbase,scale=1.2]
\draw[<-] (-0.2,-0.35) -- (-0.2,-0.1) arc(180:0:0.2) -- (0.2,-0.35);
\circledbar{-0.2,-0.1}{u};
\end{tikzpicture}
&=
\begin{tikzpicture}[H,anchorbase,scale=1.2]
\draw[<-] (-0.2,-0.35) -- (-0.2,-0.1) arc(180:0:0.2) -- (0.2,-0.35);
\circledbar{0.2,-0.1}{u};
\end{tikzpicture}
\ ,
\end{align}
and similarly for cups.
It also follows from \cref{affsergeev} that
\begin{align} \label{frog}
\begin{tikzpicture}[H,centerzero]
\draw[->] (0,-0.4) -- (0,0.4);
\token{0,-0.2};
\circled{0,0.1}{u};
\end{tikzpicture}
&=
\begin{tikzpicture}[H,centerzero]
\draw[->] (0,-0.4) -- (0,0.4);
\token{0,0.13};
\circledbar{0,-0.17}{u};
\end{tikzpicture}\ ,&
\begin{tikzpicture}[H,centerzero]
\draw[->] (0,-0.4) -- (0,0.4);
\token{0,0.13};
\circled{0,-0.17}{u};
\end{tikzpicture}
&=
\begin{tikzpicture}[H,centerzero]
\draw[->] (0,-0.4) -- (0,0.4);
\token{0,-0.17};
\circledbar{0,0.13}{u};
\end{tikzpicture}\ ,\\
\label{wave}
\begin{tikzpicture}[H,anchorbase]
\draw[->] (-0.4,-0.5) to (0.4,0.5);
\draw[->] (0.4,-0.5) to (-0.4,0.5);
\circled{-0.2,-0.25}{u};
\end{tikzpicture}
-
\begin{tikzpicture}[H,anchorbase]
\draw[->] (-0.4,-0.5) to (0.4,0.5);
\draw[->] (0.4,-0.5) to (-0.4,0.5);
\circled{0.2,0.25}{u};
\end{tikzpicture}
&=
\begin{tikzpicture}[H,anchorbase]
\draw[->] (-0.2,-0.5) to (-0.2,0.5);
\draw[->] (0.2,-0.5) to (0.2,0.5);
\circled{-0.2,0}{u};
\circled{0.2,0}{u};
\end{tikzpicture}
-
\begin{tikzpicture}[H,anchorbase]
\draw[->] (-0.2,-0.5) to (-0.2,0.5);
\draw[->] (0.2,-0.5) to (0.2,0.5);
\circled{-0.2,-0.25}{u};
\circled{0.2,0.25}{u};
\token{-0.2,0};
\token{0.2,0};
\end{tikzpicture}
\ ,&
\begin{tikzpicture}[H,anchorbase]
\draw[->] (-0.4,-0.5) to (0.4,0.5);
\draw[->] (0.4,-0.5) to (-0.4,0.5);
\circled{-0.2,0.25}{u};
\end{tikzpicture}
-
\begin{tikzpicture}[H,anchorbase]
\draw[->] (-0.4,-0.5) to (0.4,0.5);
\draw[->] (0.4,-0.5) to (-0.4,0.5);
\circled{0.2,-0.25}{u};
\end{tikzpicture}
&=
\begin{tikzpicture}[H,anchorbase]
\draw[->] (-0.2,-0.5) to (-0.2,0.5);
\draw[->] (0.2,-0.5) to (0.2,0.5);
\circled{-0.2,0}{u};
\circled{0.2,0}{u};
\end{tikzpicture}
+
\begin{tikzpicture}[H,anchorbase]
\draw[->] (-0.2,-0.5) to (-0.2,0.5);
\draw[->] (0.2,-0.5) to (0.2,0.5);
\circled{-0.2,0.25}{u};
\circled{0.2,-0.25}{u};
\token{-0.2,0};
\token{0.2,0};
\end{tikzpicture}
\ .
\end{align}

The next important consequence of the defining relations is
the {\em infinite Grassmannian relation},
which follows from \cite[(27)--(29)]{Sav19}, using also \cref{msc,darkness}.
It asserts that there are unique
formal Laurent series, the 
{\em bubble generating functions}
\begin{align}\label{loud1}
\begin{tikzpicture}[H,centerzero]
\leftbubgen{0,0};
\end{tikzpicture}&
\in u^\kappa \id_\one + u^{\kappa-2} \End_{\Heis_\kappa(\Cl)}(\one) \llbracket u^{-2} \rrbracket,
\\
\begin{tikzpicture}[H,centerzero]
\rightbubgen{0,0};
\end{tikzpicture}
& \in -u^{-\kappa} \id_\one + u^{-\kappa-2} \End_{\Heis_\kappa(\Cl)}(\one) \llbracket u^{-2} \rrbracket,\label{loud2}
\end{align}
such that 
\begin{align}
\left[\ 
\begin{tikzpicture}[H,centerzero]
\leftbubgen{0,0};
\end{tikzpicture}
\right]_{u:<0}&=
\sum_{n \geq 0}
\begin{tikzpicture}[H,centerzero]
\leftbub{0,0};
\multdot{0.24,0}{west}{n};
\end{tikzpicture}
u^{-n-1},&
\left[\ 
\begin{tikzpicture}[H,centerzero]
\rightbubgen{0,0};
\end{tikzpicture}
\right]_{u:<0}&=
\sum_{n \geq 0}
\begin{tikzpicture}[H,centerzero]
\rightbub{0,0};
\multdot{-0.24,0}{east}{n};
\end{tikzpicture}
\  u^{-n-1},&
\end{align}
and
\begin{equation} \label{eyes}
\begin{tikzpicture}[H,centerzero]
\rightbubgen{0,0};
\end{tikzpicture}\ 
\begin{tikzpicture}[H,centerzero]
\leftbubgen{0,0};
\end{tikzpicture} 
= -\id_\one.
\end{equation}
Hidden in the form of the Laurent series specified in
\cref{loud1,loud2} is the implicit relation that
\begin{align}
\label{infweed1}
\begin{tikzpicture}[H,centerzero]
\leftbub{0,0};
\multdot{0.24,0}{west}{n};
\end{tikzpicture}
&= \delta_{n=-\kappa-1}\id_\one \text{ for } 0\leq n \leq -\kappa-1,&
\begin{tikzpicture}[H,centerzero]
\rightbub{0,0};
\multdot{-0.24,0}{east}{n};
\end{tikzpicture}
&= - \delta_{n=\kappa-1}\id_\one \text{ for }0\leq n \leq \kappa-1.
\end{align}
By \cref{trick}, for a polynomial $f(x) \in \kk[x]$, we have that
\begin{align}
\label{tadpole}
\begin{tikzpicture}[H,centerzero]
\rightbub{0,0};
\pin{-0.25,0}{-0.85,0}{f(x)};
\end{tikzpicture}
&=
\left[f(u)\
\begin{tikzpicture}[H,centerzero]
\rightbubgen{0,0};
\end{tikzpicture}
\right]_{u:-1},
&
\begin{tikzpicture}[H,centerzero]
\leftbub{0,0};
\pin{0.25,0}{0.85,0}{f(x)};
\end{tikzpicture}
&=
\left[f(u)\
\begin{tikzpicture}[H,centerzero]
\leftbubgen{0,0};
\end{tikzpicture}
\right]_{u:-1}.
\end{align}
Note also
that $\tT\left(\ 
\begin{tikzpicture}[H,centerzero]
\rightbubgen{0,0};
\end{tikzpicture}
\right) = -\ 
\begin{tikzpicture}[H,centerzero]
\leftbubgen{0,0};
\end{tikzpicture}$
and
$\tT\left(\ 
\begin{tikzpicture}[H,centerzero]
\leftbubgen{0,0};
\end{tikzpicture}
\right) = -\ 
\begin{tikzpicture}[H,centerzero]
\rightbubgen{0,0};
\end{tikzpicture}$.

By equating coefficients, the following relation
follows from \cite[(18)]{Sav19}, remembering also that odd bubbles are 0 now:
\begin{align}
\label{pizza2}
\begin{tikzpicture}[H,anchorbase,scale=1.1]
\draw[->] (-0.25,-0.6) \braidup (0.25,0) \braidup (-0.25,0.6);
\draw[<-] (0.25,-0.6) \braidup (-0.25,0) \braidup (0.25,0.6);
\end{tikzpicture}
&=
\begin{tikzpicture}[H,anchorbase,scale=1.1]
\draw[->] (-0.2,-0.6) to (-0.2,0.6);
\draw[<-] (0.2,-0.6) to (0.2,0.6);
\end{tikzpicture}
\ +
\left[\ 
\begin{tikzpicture}[H,anchorbase,scale=1.1]
\draw[<-] (-0.2,0.6) to (-0.2,0.3) arc(180:360:0.2) to (0.2,0.6);
\draw[->] (-0.2,-0.6) to (-0.2,-0.3) arc(180:0:0.2) to (0.2,-0.6);
\leftbubgen{-0.6,0};
\circled{-0.2,0.35}{u};
\circled{-0.2,-0.35}{u};
\end{tikzpicture}
-
\begin{tikzpicture}[H,anchorbase,scale=1.1]
\draw[<-] (-0.2,0.7) to (-0.2,0.3) arc(180:360:0.2) to (0.2,0.7);
\draw[->] (-0.2,-0.7) to (-0.2,-0.3) arc(180:0:0.2) to (0.2,-0.7);
\leftbubgen{-0.6,0};
\token{-0.2,0.53};
\token{-0.2,-0.53};
\circled{-0.2,0.3}{u};
\circled{-0.2,-0.3}{u};
\end{tikzpicture}\ 
\right]_{u:-1}.
\end{align}
Applying $\tT$ and simplifying gives also that
\begin{align}
\label{pizza1}
\begin{tikzpicture}[H,anchorbase,scale=1.1]
\draw[<-] (-0.25,-0.6) \braidup (0.25,0) \braidup (-0.25,0.6);
\draw[->] (0.25,-0.6) \braidup (-0.25,0) \braidup (0.25,0.6);
\end{tikzpicture}
&=
\begin{tikzpicture}[H,anchorbase,scale=1.1]
\draw[<-] (-0.2,-0.6) to (-0.2,0.6);
\draw[->] (0.2,-0.6) to (0.2,0.6);
\end{tikzpicture}
\ +
\left[\ 
\begin{tikzpicture}[H,anchorbase,scale=1.1]
\draw[->] (-0.2,0.6) to (-0.2,0.3) arc(180:360:0.2) to (0.2,0.6);
\draw[<-] (-0.2,-0.6) to (-0.2,-0.3) arc(180:0:0.2) to (0.2,-0.6);
\rightbubgen{1.1,0};
\circled{0.2,0.35}{u};
\circled{0.2,-0.35}{u};
\end{tikzpicture}
-
\begin{tikzpicture}[H,anchorbase,scale=1.1]
\draw[->] (-0.2,0.7) to (-0.2,0.3) arc(180:360:0.2) to (0.2,0.7);
\draw[<-] (-0.2,-0.7) to (-0.2,-0.3) arc(180:0:0.2) to (0.2,-0.7);
\rightbubgen{1.1,0};
\circled{0.2,0.3}{u};
\circled{0.2,-0.3}{u};
\token{0.2,0.53};
\token{0.2,-0.53};
\end{tikzpicture}\ 
\right]_{u:-1}.
\end{align}
(To simplify the last term in this argument, first apply \cref{ruby,ruby2,slippery,frog}, then replace $u$ by $-u$, using that
$
\begin{tikzpicture}[H,anchorbase]
\draw[->] (-0.68,0) arc(180:-180:0.2);
\node[black] at (0.14,0) {$(-u)$};
\end{tikzpicture}
=(-1)^\kappa\ 
\begin{tikzpicture}[H,anchorbase]
\draw[->] (-0.68,0) arc(180:-180:0.2);
\node[black] at (0,0) {$(u)$};
\end{tikzpicture}
$.)

The following \emph{curl relations} are equivalent to
\cite[(35)--(36)]{Sav19}:
\begin{equation} \label{fancycurls}
\begin{tikzpicture}[H,anchorbase,scale=1.1]
\draw[->] (0,-0.5) to[out=up,in=0] (-0.3,0.2) to[out=180,in=up] (-0.45,0) to[out=down,in=180] (-0.3,-0.2) to[out=0,in=down] (0,0.5);
\circled{-0.45,0}{u};
\end{tikzpicture}
=
\left[
\ \begin{tikzpicture}[H,anchorbase,scale=1.1]
\draw[->] (0,-0.5) -- (0,0.5);
\leftbubgen{-0.5,0};
\circled{0,0}{u};
\end{tikzpicture}
\ \right]_{u:<0}
\ ,\qquad
\begin{tikzpicture}[H,anchorbase,scale=1.1]
\draw[->] (0,-0.5) to[out=up,in=180] (0.3,0.2) to[out=0,in=up] (0.45,0) to[out=down,in=0] (0.3,-0.2) to[out=180,in=down] (0,0.5);
\circled{0.45,0}{u};
\end{tikzpicture}
=-
\left[
\begin{tikzpicture}[H,anchorbase,scale=1.1]
\draw[->] (0,-0.5) -- (0,0.5);
\rightbubgen{1,0};
\circled{0,0}{u};
\end{tikzpicture}
\right]_{u:<0}.
\end{equation}
Applying \cref{trick,fancycurls} and using the symmetry $\tT$,
we deduce that
\begin{align}
\begin{tikzpicture}[H,anchorbase,scale=1.1]
\draw[->] (0,-0.5) to[out=up,in=180] (0.3,0.2) to[out=0,in=up] (0.45,0) to[out=down,in=0] (0.3,-0.2) to[out=180,in=down] (0,0.5);
\pin{.45,0}{1.1,0}{f(x)};
\end{tikzpicture}
&=-
\left[f(u)\
\begin{tikzpicture}[H,anchorbase,scale=1.1]
\draw[->] (0,-0.5) -- (0,0.5);
\rightbubgen{.95,0};
\circled{0,0}{u};
\end{tikzpicture}
\right]_{u:-1},&
\begin{tikzpicture}[H,anchorbase,scale=1.1]
\draw[<-] (0,-0.5) to[out=up,in=180] (0.3,0.2) to[out=0,in=up] (0.45,0) to[out=down,in=0] (0.3,-0.2) to[out=180,in=down] (0,0.5);
\pin{.45,0}{1.1,0}{f(x)};
\end{tikzpicture}
&=
\left[f(u)\
\begin{tikzpicture}[H,anchorbase,scale=1.1]
\draw[<-] (0,-0.5) -- (0,0.5);
\leftbubgen{.95,0};
\circled{0,0}{u};
\end{tikzpicture}
\right]_{u:-1}\label{junesun}
\end{align}
for a polynomial $f(x)\in \kk[x]$.

\subsection{Bubble slides}

The formal power series
\begin{equation}\label{divergent}
p(u,x) :=
1-\frac{1}{(u-x)^{2}} - \frac{1}{(u+x)^{2}} \in \kk[x^2] \llbracket u^{-2}\rrbracket
\end{equation}
arises naturally from the proof of the next lemma.

\begin{lem}\label{bubbleslide}
We have
\begin{align} \label{bubslide}
\begin{tikzpicture}[H,anchorbase]
\draw[->] (0,-0.5) to (0,0.5);
\leftbubgen{0.9,0};
\end{tikzpicture}
&= \begin{tikzpicture}[H,anchorbase]
\draw[->] (0,-0.5) to (0,0.5);
\leftbubgen{-0.4,0};
\pin{0,0}{.8,0}{p(u,x)};
\end{tikzpicture}
\ ,&
\begin{tikzpicture}[H,anchorbase]
\draw[<-] (0,-0.5) to (0,0.5);
\rightbubgen{0.9,0};
\end{tikzpicture}&= \begin{tikzpicture}[H,anchorbase]
\draw[<-] (0,-0.5) to (0,0.5);
\rightbubgen{-0.4,0};
\pin{0,0}{.8,0}{p(u,x)};
\end{tikzpicture}
\ ,\\\label{bubslide2}
\begin{tikzpicture}[H,anchorbase]
\draw[->] (0,-0.5) to (0,0.5);
\rightbubgen{-0.4,0};
\end{tikzpicture}  &=\begin{tikzpicture}[H,anchorbase]
\draw[->] (0,-0.5) to (0,0.5);
\rightbubgen{0.9,0};
\pin{0,0}{-.8,0}{p(u,x)};
\end{tikzpicture}
\ ,&
\begin{tikzpicture}[H,anchorbase]
\draw[<-] (0,-0.5) to (0,0.5);
\leftbubgen{-0.4,0};
\end{tikzpicture}&=\begin{tikzpicture}[H,anchorbase]
\draw[<-] (0,-0.5) to (0,0.5);
\leftbubgen{0.9,0};
\pin{0,0}{-.8,0}{p(u,x)};
\end{tikzpicture}
\ .
\end{align}
\end{lem}

\begin{proof}
First suppose $\kappa \leq 0$.  Then
$\begin{tikzpicture}[H,centerzero]
\leftbubgen{0,0};
\end{tikzpicture}
=
\delta_{\kappa=0}\id_\one +\,
\begin{tikzpicture}[H,centerzero,scale=1.1]
\draw[->] (-0.2,0) arc(-180:180:0.2);
\circled{0.2,0}{u};
\end{tikzpicture}
\in \End_{\Heis_\kappa(\Cl)} \llbracket u^{-1} \rrbracket$
and so
\begin{equation} \label{Houston}
\begin{tikzpicture}[H,anchorbase,scale=1.1]
\draw[->] (0,-0.5) to[out=up,in=0] (-0.3,0.2) to[out=180,in=up] (-0.45,0) to[out=down,in=180] (-0.3,-0.2) to[out=0,in=down] (0,0.5);
\circled{-0.45,0}{u};
\end{tikzpicture}
\overset{\cref{fancycurls}}{=}
\begin{tikzpicture}[H,anchorbase,scale=1.1]
\draw[->] (0,-0.5) -- (0,0.5);
\leftbubgen{-0.5,0};
\circled{0,0}{u};
\end{tikzpicture}
\qquad \text{and} \qquad
\begin{tikzpicture}[H,anchorbase,scale=1.1]
\draw[->] (-0.2,-0.5) \braidup (0.2,0) \braidup (-0.2,0.5);
\draw[<-] (0.2,-0.5) \braidup (-0.2,0) \braidup (0.2,0.5);
\end{tikzpicture}
\overset{\cref{pizza2}}{=}
\begin{tikzpicture}[H,anchorbase,scale=1.1]
\draw[->] (-0.2,-0.5) to (-0.2,0.5);
\draw[<-] (0.2,-0.5) to (0.2,0.5);
\end{tikzpicture}
\ .
\end{equation}
Thus,
\begin{align*}
\begin{tikzpicture}[H,anchorbase,scale=1.1]
\draw[->] (0,-0.5) to (0,0.5);
\leftbubgen{0.9,0};
\end{tikzpicture}
&=
\delta_{\kappa=0} \
\begin{tikzpicture}[H,centerzero,scale=1.1]
\draw[->] (0,-0.5) -- (0,0.5);
\end{tikzpicture}
\, +
\begin{tikzpicture}[H,centerzero,scale=1.1]
\draw[->] (-0.3,0) arc(-180:180:0.3);
\circled{0.3,0}{u};
\draw[->] (-0.5,-0.5) to[out=45,in=down] (0,0) to[out=up,in=-45] (-0.5,0.5);
\end{tikzpicture}\\&\!\!\!
\overset{\cref{wave}}{=}
\delta_{k=0} \
\begin{tikzpicture}[H,centerzero,scale=1.1]
\draw[->] (0,-0.5) -- (0,0.5);
\end{tikzpicture}
\, +
\begin{tikzpicture}[H,centerzero,scale=1.1]
\draw[->] (0.3,0) arc(0:360:0.3);
\circled{-0.3,0}{u};
\draw[->] (0.5,-0.5) to[out=135,in=down] (0,0) to[out=up,in=-135] (0.5,0.5);
\end{tikzpicture}
-
\begin{tikzpicture}[H,anchorbase,scale=1.1]
\draw[->] (0,-0.5) to[out=up,in=0] (-0.3,0.2) to[out=180,in=up] (-0.45,0) to[out=down,in=180] (-0.3,-0.2) to[out=0,in=down] (0,0.5);
\circled{-0.45,0}{u};
\circled{0,-0.3}{u};
\end{tikzpicture}
+
\begin{tikzpicture}[H,anchorbase,scale=1.1]
\draw[->] (0,-0.7) -- (0,-0.5) to[out=up,in=0] (-0.3,0.2) to[out=180,in=up] (-0.45,0) to[out=down,in=180] (-0.3,-0.2) to[out=0,in=down] (0,0.5);
\circled{-0.45,0}{u};
\circled{0,-0.3}{u};
\token{0,-0.55};
\token{0,0.25};
\end{tikzpicture}
\overset{\cref{sergeev1}}{\underset{\cref{Houston}}{=}}
\begin{tikzpicture}[H,anchorbase,scale=1.1]
\draw[->] (0,-0.5) to (0,0.5);
\leftbubgen{-0.4,0};
\end{tikzpicture}
-
\begin{tikzpicture}[H,anchorbase,scale=1.1]
\draw[->] (0.1,-0.5) to (0.1,0.5);
\leftbubgen{-0.4,0};
\circled{0.1,0.15}{u};
\circled{0.1,-0.15}{u};
\end{tikzpicture}
+
\begin{tikzpicture}[H,anchorbase,scale=1.1]
\draw[->] (0.1,-0.5) to (0.1,0.5);
\leftbubgen{-0.4,0};
\circled{0.1,0.11}{u};
\circled{0.1,-0.17}{u};
\token{0.1,0.33};
\token{0.1,-0.39};
\end{tikzpicture}
\\
&\!\!\!\overset{\cref{sergeev2}}{\underset{\cref{frog}}{=}}
\begin{tikzpicture}[H,anchorbase,scale=1.1]
\draw[->] (0,-0.5) to (0,0.5);
\leftbubgen{-0.4,0};
\end{tikzpicture}
-
\begin{tikzpicture}[H,anchorbase,scale=1.1]
\draw[->] (0.1,-0.5) to (0.1,0.5);
\leftbubgen{-0.4,0};
\circled{0.1,0.15}{u};
\circled{0.1,-0.15}{u};
\end{tikzpicture}
-
\begin{tikzpicture}[H,anchorbase,scale=1.1]
\draw[->] (0.1,-0.5) to (0.1,0.5);
\leftbubgen{-0.4,0};
\circledbar{.1,0.15}{u};
\circledbar{0.1,-0.15}{u};
\end{tikzpicture}
=
\begin{tikzpicture}[H,anchorbase,scale=1.1]
\draw[->] (0,-0.5) to (0,0.5);
\leftbubgen{-0.4,0};
\pin{0,0}{1.5,0}{1-\frac{1}{(u-x)^{2}}-\frac{1}{(u+x)^{2}}};
\end{tikzpicture}
\ .
\end{align*}
This proves first relation in \cref{bubslide} for $\kappa \leq 0$.
Then we attach a rightward cap at the top, a rightward cup at the bottom and simplify using \cref{adjright,intlaw}
to obtain the second relation in \cref{bubslide} for $\kappa \leq 0$. The relations \cref{bubslide} for $\kappa > 0$ then follow from the ones proved so far by applying
the functor $\tT$ of \cref{mirror}.
Finally, to deduce \cref{bubslide2}, we tensor on the left and right by the inverses of the bubble generating functions using \cref{eyes}.
\end{proof}

If we replace $x^2$ by $y(y+1)$, the formula for $p(u,x)$ can be simplified to obtain
\begin{equation}\label{insurgent}
\frac{\big( u^2-(y-1)y \big) \big( u^2-(y+1)(y+2) \big)}{\big( u^2-y(y+1) \big)^2} \in \kk[y]\llbracket u^{-2}\rrbracket.
\end{equation}
This change-of-variables will play an important role
subsequently.

\subsection{Definition of isomeric Heisenberg categorification}\label{ihcdef}

An {\em isomeric Heisenberg categorification} of central charge $\kappa \in \Z$ is
a locally finite Abelian supercategory $\catR$ plus an adjoint pair $(P,Q)$ of endofunctors (super, of course) such that:
\begin{itemize}
\item[(IH1)] The adjoint pair $(P,Q)$ has a prescribed adjunction with unit and counit of adjunction
denoted $\;\begin{tikzpicture}[H,anchorbase]
\draw[->] (-0.25,0.25) -- (-0.25,0) arc(180:360:0.25) -- (0.25,0.25);
\end{tikzpicture}\ : \id \Rightarrow Q P$ and
$\;\begin{tikzpicture}[H,anchorbase]
\draw[->] (-0.25,-0.25) -- (-0.25,0) arc(180:0:0.25) -- (0.25,-0.25);
\end{tikzpicture}\ : P Q \Rightarrow \id$. Both of these should be {\em even}.
\item[(IH2)] 
There are given supernatural transformations
$\begin{tikzpicture}[H,centerzero]
\draw[->] (0,-0.3) -- (0,0.3);
\token{0,0};
\end{tikzpicture} : P \Rightarrow P$,
$\begin{tikzpicture}[H,centerzero]
\draw[->] (0,-0.3) -- (0,0.3);
\singdot{0,0};
\end{tikzpicture} :P \Rightarrow P$ and
$\begin{tikzpicture}[H,centerzero]
\draw[->] (0.3,-0.3) -- (-0.3,0.3);
\draw[->] (-0.3,-0.3) -- (0.3,0.3);
\end{tikzpicture}\ : P^2 \Rightarrow P^2$
satisfying the affine Sergeev superalgebra relations from \cref{sergeev1,sergeev2,affsergeev}. These should be odd, even and even, respectively.
\item[(IH3)] 
Defining the rightward crossing like in \cref{hrightpivot}, the matrix $M_\kappa$ from \cref{vikings}, viewed now as a matrix of supernatural transformations, is an isomorphism.
\item[(IH4)] 
There exists a family of objects $V \in \catR$
such that the {\em supercenter}
\begin{equation}\label{supercenter}
Z_V := \{z \in \End_{\catR}(V):z \circ f = (-1)^{\p(z)\p(f)} f \circ z\text{ for all }f \in \End_{\catR}(V)\}
\end{equation}
of $\End_{\catR}(V)$
is purely even for each $V$ in the family, and
the objects obtained from these objects by applying sequences of the functors $P$ and $Q$ are a generating family for $\catR$.
\end{itemize}

The properties (IH1)--(IH3) are equivalent to saying that
the locally finite Abelian supercategory 
$\catR$ is a
strict left $\widehat{\Heis}_\kappa(\Cl)$-module supercategory.
In view of the relations \cref{strictlycentral}, the property (IH4) implies\footnote{The reader may wonder why in place of (IH4) we have not simply required that the odd bubble acts as zero on any object of $\catR$. We have done it this way because we will need the slightly stronger hypothesis (IH4) (and the corresponding hypothesis (IKM4) formulated at the end of 
\cref{s5-ikm}) in order to prove the main  \cref{maintheorem2} below. Thus, (IH4) is something of a compromise, although we believe it is easy to check in all of the examples of interest.} that the odd bubble
$\begin{tikzpicture}[H,baseline=-1mm,scale=.8]
\draw[-] (-0.25,0) arc(180:-180:0.25);
\draw[-] (-.18,.18) to (.18,-.18);
\draw[-] (-.18,-.18) to (.18,.18);
\end{tikzpicture}$ acts as 0 on any object of $\catR$. Hence, $\catR$ is actually a strict left
$\Heis_\kappa(\Cl)$-module supercategory.
In other words, there is a strict monoidal superfunctor
\begin{equation}\label{Psi}
\Psi:\Heis_\kappa(\Cl) \rightarrow \sEnd(\catR)
\end{equation}
induced by the categorical action.

\setcounter{section}{3}
\section{Spectral analysis of isomeric Heisenberg categorifications}\label{s4-wsd}

In this section, we start to investigate the structure of 
isomeric Heisenberg categorifications.
Our analysis is similar to that of \cite[Sec.~4]{BSW-HKM}, but several more root systems are needed since the bubble slide relation \cref{bubslide} is more complicated in the isomeric case.
The relevant ones are introduced in the first subsection.
After that, we assume we are given an isomeric Heisenberg categorification $\catR$, decompose the associated endofunctors $P$ and $Q$ into eigenfunctors denoted $P_i$ and $Q_i$ for $i \in \kk$, and prove a series of lemmas about induced supernatural transformations between these eigenfunctors.
The first important theorem in the section, \cref{lostboys}, 
explains how to use the weight lattice $X$ attached to the root system to index central characters of irreducible objects of $\catR$. The second important theorem, \cref{theworstplace}, 
establishes commutation relations between the eigenfunctors $P_i$ and $Q_i$.

\subsection{Super Cartan datum}\label{seccd}

Recall that a symmetrizable generalized Cartan matrix
$(c_{ij})_{i,j \in I}$
is a matrix such that  $c_{ii} = 2$ for all $i \in I$,
$c_{ij} \in -\N$ for $i\neq j$ in $I$,
$c_{ij} = 0\Leftrightarrow c_{ji}=0$, and
there are given positive rational numbers $d_i\:(i \in I)$
such that $d_i c_{ij} = d_j c_{ji}$ for all $i,j \in I$.
We do not insist that the set $I$ is finite, but
the number of non-zero entries in each 
row and each column of the Cartan matrix should be finite.
An additional piece of data required in the super case
is a {\em parity function} $\p:I \rightarrow \Z/2$
such that
\begin{equation}\label{constraint}
\p(i) = \1
\quad\Rightarrow\quad c_{ij} \text{ is even for all }j \in I.
\end{equation}
By a {\em realization} of such a super Cartan matrix we mean:
\begin{itemize}
\item
A free Abelian group $X$, the {\em weight lattice},
containing elements $\alpha_i\:(i \in I)$, called {\em simple roots}, 
and $\varpi_i\:(i \in I)$, called {\em fundamental weights}.
\item
Homomorphisms $h_i:X \rightarrow \Z\:(i \in I)$
such that $h_i(\alpha_j) = c_{ij}$ and $h_i(\varpi_j) = \delta_{i=j}$
for all $i,j \in I$;
\item
A function $\p:X \rightarrow \Z/2$
such that 
\begin{equation}\label{parityf}
\p(\lambda+\alpha_i) = \p(\lambda)+\p(i)
\end{equation}
for $\lambda \in X$ and $i \in I$.
When the simple roots are linearly independent, it is always possible to choose such a function, but it might not be possible if there is some dependency.
\end{itemize}

For the remainder of the section, we will be working with a specific choice of Cartan matrix which depends on the algebraically closed ground field $\kk$. To introduce this,
let $\sim$ be the equivalence relation on $\kk$ defined by
$i \sim j$ if $j = n \pm i$ for some $n \in \Z$.
Remembering our convention that $\hbar = -\frac{1}{2}$,
we have that $\hbar\not\sim 0$ when $p=0$,
but in positive characteristic it is the case that $\hbar\sim 0$.
Let $A$ be a choice of $\sim$-equivalence class representatives with $0 \in A$ always, and also $\hbar \in A$ when $p=0$.
Then let
$I := \bigsqcup_{k \in A} I_k$ where
\begin{equation}\label{components}
I_k
:=
\begin{cases}
\{\dots,k-1,k, k + 1, \dots\}&\text{if $p=0$, $k \neq 0$ and $k \neq\hbar$}\\
\{0,1,2,\dots\}&\text{if $p=0$ and $k = 0$}\\
\left\{\dots,-\frac{5}{2},-\frac{3}{2},-\frac{1}{2}\right\}&\text{if $p=0$ and $k = \hbar$}\\
\{k, k+1,\dots,k+p-2,k+p-1\}&\text{if $p > 2$ and $k \neq 0$}\\
\left\{0,1,\dots,\frac{p-3}{2},\frac{p-1}{2}\right\} = \left\{-\frac{p}{2},\dots,-\frac{3}{2},-\frac{1}{2}\right\}&\text{if $p > 2$ and $k = 0$,}
\end{cases}
\end{equation}
viewed as a subset of $\kk$.
We have that
\begin{align}\label{goodintersections}
I \cup (-I) &= \kk,&
I \cap (-I) &= \{0\}.
\end{align}
The set $B := A -\hbar$ is another set of $\sim$-equivalence class representatives, and $J := I-\hbar$ is the disjoint union
$J = \bigsqcup_{k \in B} J_k$ where
\begin{equation}\label{components2}
\hspace{10mm}
J_k
:=
\begin{cases}
\{\dots,k-1,k, k + 1, \dots\}&\text{if $p=0$, $k \neq 0$ and $k \neq -\hbar$}\\
\{\dots,-2,-1,0\}&\text{if $p=0$ and $k = 0$}\\
\left\{\frac{1}{2},\frac{3}{2},\frac{5}{2},\dots\right\}&\text{if $p=0$ and $k = -\hbar$}\\
\{k, k+1,\dots,k+p-2,k+p-1\}&\text{if $p > 2$ and $k \neq -\hbar$}\\
\left\{\frac{1-p}{2},\dots,-1,0\right\} = \left\{\frac{1}{2},\frac{3}{2},\dots, \frac{p}{2}\right\}\hspace{15mm}&\text{if $p > 2$ and $k =-\hbar$.}
\end{cases}
\end{equation}
This set also satisfies
\begin{align}\label{goodintersections2}
J \cup (-J) &= \kk, &
J \cap (-J) &= \{0\}.
\end{align}
The definition of the sets $I$ and $J$ has its
origins in the change-of-variables
$x^2=y(y+1)$ used to derive \cref{insurgent}
from \cref{divergent}.

In view of \cref{goodintersections2},
each $i \in \kk$ has a unique square root belonging to the set $J = I-\hbar$.
We denote this distinguished choice of square root
simply by $\sqrt{i}$.
For example,  $\sqrt{\frac{1}{4}} = \frac{1}{2}$
 and
$\sqrt{1} = -1$.

\begin{lem}\label{importantfunction}
The function
\begin{align*}
b &:\kk \rightarrow \kk,
&i &\mapsto
\begin{cases}
\sqrt{i(i+1)}&\text{if $i \in I$}\\
-\sqrt{i(i-1)}&\text{if $i \in -I$}
\end{cases}
\end{align*}
is a bijection such that $b(-i) = -b(i)$ for each $i \in \kk$. It restricts to a bijection $b : I \stackrel{\sim}{\rightarrow} J$ whose inverse takes $j \in J$ to $\sqrt{j^2+\frac{1}{4}}-\frac{1}{2}$.
\end{lem}

\begin{proof}
Exercise.
\end{proof}

\begin{lem}\label{whenisitzero}
Let $p(u,x)$ be the formal power series from \cref{divergent}. For $i \in \kk$, we have that
$$
p(u,b(i)) = \frac{(u^2-(i-1)i)(u^2-(i+1)(i+2))}{(u^2-i(i+1))^2}.
$$
\end{lem}

\begin{proof}
Since $b(i)^2 = i(i+1)$,
this follows
from \cref{insurgent} by replacing $y$ with $i$.
\end{proof}

\begin{lem}\label{rainisback}
The following hold for $i, j \in I$.
\begin{enumerate}
\item\label{kona1} If $i(i+1) = j(j+1)$ then $i=j$.\label{rainisback-a}
\item If $i(i+1) = (j-1)j$ then $i+1 = j$ or $i=j=0$.\label{rainisback-b}
\item If $i(i+1) = (j+1)(j+2)$ then $i=j+1$ or $i+1=j=\hbar$.\label{rainisback-c}
\end{enumerate}
\end{lem}

\begin{proof}
(1)
If $i(i+1)=j(j+1)$ then $b(i) = b(j)$, whence, $i=j$ since the function $b$ in \cref{importantfunction}
is injective.

\vspace{2mm}
\noindent
(2)
If $j-1 \in I$ we get that $i=j-1$ by \cref{kona1}.
Otherwise, by the nature of \cref{components}, we must have that $j=0$. Then we have that
$(j-1)j = j(j+1)$, so
$i=j$ by \cref{kona1} again.

\vspace{2mm}
\noindent
(3)
If $j+1 \in I$ we get that $i=j+1$ by \cref{kona1}.
Otherwise, by the nature of \cref{components}, we must have that $j = \hbar$. Then we have that
$(j+1)(j+2) = (j-1)j$, so $i=j-1$ by \cref{kona1} again.
\end{proof}

For $i,j \in I$, we define
\begin{align}\label{cartan}
d_i &:= 2^{\delta_{i=\hbar}-\delta_{i=0}} \in \left\{\textstyle\frac{1}{2},1,2\right\},
&
c_{ij} &:=
\begin{cases}
2&\text{if $i=j$}\\
-2^{\delta_{i=0}+\delta_{j=\hbar}}&\text{if $i = j \pm 1$}\\
0&\text{otherwise.}
\end{cases}
\end{align}
We have that $d_i c_{ij} = d_j c_{ji}$ for each $i,j \in I$.
Thus, $(c_{ij})_{i,j \in I}$ is a symmetrizable generalized Cartan matrix.
Its indecomposable components are indexed by the subsets $I_k \subset I$ for $k \in A$, and
the corresponding Dynkin diagrams
are as in \cref{dynkintable} in the introduction.
Since
$c_{0i}$ is even for each $i\in I$,
the parity function $\p:I \rightarrow \Z/2$
defined by letting $\p(0) := \1$ and $\p(i) := \0$ for all other $i \in I$
satisfies \cref{constraint}.
So now we have in our hands a super symmetrizable Cartan matrix.

We choose the {\em minimal realization}
of this super Cartan matrix, which has weight lattice
\begin{equation}\label{minreal}
X = \bigoplus_{i \in I} \Z \varpi_i,
\end{equation}
defining $h_i:X \rightarrow \Z$ by $h_i(\varpi_j) = \delta_{i=j}$,
and setting $\alpha_j := \sum_{i \in I} c_{ij} \varpi_i$ so that $h_i(\alpha_j) = c_{ij}$.
In this realization, when $p> 0$, the simple roots are
linearly dependent. Nevertheless, it is always possible to define a parity function $\p:X \rightarrow \Z / 2$ satisfying \cref{parityf}.
This is clear when $p=0$ since the simple roots are linearly independent in that case. When $p > 0$, one can take
\begin{equation}
\p(\lambda) := h_1(\lambda)+h_3(\lambda)+\cdots+h_{(p-1)/2}(\lambda)\pmod 2.
\end{equation}

\subsection{The eigenfunctors $P_i$ and $Q_i$}\label{seigenfunctors}
Now we assume that we are given an isomeric Heisenberg categorification $\catR$ in the sense of \cref{ihcdef}. So $\catR$ is a locally finite Abelian supercategory,
and there is a strict monoidal
superfunctor $\Psi$ as in \cref{Psi}.
We will use string diagrams to denote supernatural transformations between endofunctors of $\catR$,
using the same diagram for a morphism in $\Heis_\kappa(\Cl)$ and for the supernatural transformation that is its image under $\Psi$.
Given also an object $V$ of $\catR$,
we draw a green string labelled by $V$ on the right-hand side of this string diagram in order
to denote the morphism obtained by evaluating the supernatural transformation on $V$.
Morphisms in $\catR$ can be represented diagrammatically by adding an additional coupon to this green string labelled by the morphism in question.
Recalling that Abelian supercategories as defined in \cref{ssas} 
are $\Pi$-supercategories, for an object $V$ of $\catR$, we use the {\em tags}
\begin{align}\label{gnocchi}
\begin{tikzpicture}[anchorbase]
\draw[gcolor,thick] (0.35,-0.5) \botlabel{\Pi V} -- (0.35,0.5)\toplabel{V};
\gnotch{.35,0};
\end{tikzpicture}
&:=
\begin{tikzpicture}[anchorbase]
\draw[gcolor,thick] (0.35,-0.5) \botlabel{\Pi V} -- (0.35,0.5)\toplabel{V};
\node[draw,gcolor,thick,fill=white!90!green,inner sep=1.5pt,rounded corners] at (.35,0) {\objlabel{\zeta_V}};
\end{tikzpicture}\ ,&
\begin{tikzpicture}[anchorbase]
\draw[gcolor,thick] (0.35,-0.5) \botlabel{V} -- (0.35,0.5)\toplabel{\Pi V};
\gnotch{.35,0};
\end{tikzpicture}
&:=
\begin{tikzpicture}[anchorbase]
\draw[gcolor,thick] (0.35,-0.5) \botlabel{V} -- (0.35,0.4)\toplabel{\Pi V};
\node[draw,gcolor,thick,fill=white!90!green,inner sep=1.5pt,rounded corners] at (.35,0) {\objlabel{\zeta_V^{-1}}};
\end{tikzpicture}
\end{align}
to denote the mutually inverse odd isomorphisms arising from the odd supernatural transformation
$\zeta:\Pi\stackrel{\sim}{\Rightarrow}\id_\catR$
that is part of the $\Pi$-supercategory structure.

Let $b:\kk\rightarrow \kk$ be the bijection from \cref{importantfunction}.
For $i \in \kk$, we define the {\em eigenfunctors}
$P_i$ and $Q_i$ to
be the subfunctors of the superfunctors
$P,Q:\catR\rightarrow\catR$ defined on $V \in \catR$ by declaring that $P_i V$ and $Q_i V$ are the generalized $b(i)$-eigenspaces of the even endomorphisms
\begin{equation}\label{wombat}
\begin{tikzpicture}[H,centerzero,scale=1.1]
\draw[->] (-0.15,-0.25) -- (-0.15,0.25);
\draw[gcolor,thick] (0.15,-0.25) \botlabel{V} -- (0.15,0.25);
\singdot{-0.15,0};
\end{tikzpicture} \qquad\text{ and }\qquad
\begin{tikzpicture}[H,centerzero,scale=1.1]
\draw[<-] (-0.15,-0.25) -- (-0.15,0.25);
\draw[gcolor,thick] (0.15,-0.25) \botlabel{V} -- (0.15,0.25);
\singdot{-0.15,0};
\end{tikzpicture} ,
\end{equation}
respectively. Thus, we have that
\begin{align}\label{piqi}
P &=\bigoplus_{i \in \kk} P_i,&
Q &=\bigoplus_{i \in \kk} Q_i.
\end{align}
As $\catR$ is locally finite, 
the endomorphism superalgebras $\End_{\catR}(PV)$ and $\End_{\catR}(QV)$ are finite-dimensional.  So it makes sense to 
define $m_V(x), n_V(x) \in \kk[x]$ to be the (monic) \emph{minimal polynomials} of the endomorphisms \cref{wombat}.
Then there are injective homomorphisms
\begin{align}\label{cradle1}
\kk[x]/m_V(x) &\hookrightarrow \End_\catR(PV),
&f(x) &\mapsto
\begin{tikzpicture}[H,centerzero]
\draw[->] (-0.15,-0.25) -- (-0.15,0.25);
\draw[gcolor,thick] (0.15,-0.25) \botlabel{V} -- (0.15,0.25);
\pin{-0.15,0}{-0.75,0}{f(x)};
\end{tikzpicture},\\
\kk[x]/n_V(x) &\hookrightarrow \End_\catR(QV),
&g(x) &\mapsto
\begin{tikzpicture}[H,centerzero]
\draw[<-] (-0.15,-0.25) -- (-0.15,0.25);
\draw[gcolor,thick] (0.15,-0.25) \botlabel{V} -- (0.15,0.25);
\pin{-0.15,0}{-0.75,0}{g(x)};
\end{tikzpicture}.\label{cradle2}
\end{align}
Let $\eps_i(V)$ and $\phi_i(V)$ denote the multiplicities of $b(i)$ as a root of the polynomials $m_V(u)$ and $n_V(u)$, respectively.
Since $\kk$ is algebraically closed and $b$ is a bijection, the Chinese Remainder Theorem implies that
\begin{align}\label{switchybitchy}
\kk[x]/m_V(x) &\cong \prod_{i \in \kk} \kk[x] \big/ (x-b(i))^{\eps_i(V)},&
\kk(x)/n_V(x) &\cong \prod_{i \in \kk} \kk[x] \big/ (x-b(i))^{\phi_i(V)}.
\end{align}
There are corresponding decompositions $1 = \sum_{i \in \kk} e_i$ and $1 = \sum_{i \in \kk} f_i$ of the identity elements of these algebras as a sum of mutually orthogonal idempotents.  
The summand $P_iV$ of $PV$
(resp., $Q_iV$ of $QV$) is the image of
$e_i$ (resp., $f_i$) viewed as an idempotent endomorphism of $PV$ 
(resp., $QV$) via the embedding \cref{cradle1} (resp., \cref{cradle2}).

\begin{lem}\label{birdie}
For $V \in \catR$, we have that
\begin{align*}
m_V(x) &= (-1)^{\deg m_V(x)} m_V(-x),
&
n_V(x) &= (-1)^{\deg n_V(x)} n_V(-x).
\end{align*}
Since $b(-i)=-b(i)$, it follows that
$\eps_i(V) = \eps_{-i}(V)$
and $\phi_i(V) = \phi_{-i}(V)$
for each $i \in \kk$.
\end{lem}

\begin{proof}
The Clifford token defines an odd endomorphism
$c_V:P V \rightarrow P V$, and the dot defines
an even endomorphism $x_V:PV \rightarrow PV$.
Since $c_V^2 = -\id_{PV}$, $c_V$ is an automorphism.
Also $c_V \circ x_V^n \circ c_V^{-1}
= (-1)^{n} x_V^n$ for each $n \in \N$.
So if $f(x) \in \kk[x]$ is a monic polynomial
with $f(x_V) = 0$ then $g(x) := (-1)^{\deg f(x)} f(-x)$
is another monic polynomial with
$$
g(x_V) = (-1)^{\deg f(x)} f(-x_V)
=
(-1)^{\deg f(x)}
c_V \circ f(x_V) \circ c_V^{-1} = 0.
$$
The claim that $m_V(x) = (-1)^{\deg m_V(x)}
m_V(-x)$ follows.
The proof for $n_V(x)$ is similar.
\end{proof}

\begin{cor}\label{pointless}
We have that 
$m_V(x) \in \kk[x] \cap x^{\deg m_V(x)} \kk[x^{-2}]$
and 
$n_V(x) \in \kk[x] \cap x^{\deg n_V(x)} \kk[x^{-2}]$.
\end{cor}

We will represent the identity endomorphisms of $P_i$ and $Q_i$ by vertical strings labeled by $i$:
\begin{align*}
\begin{tikzpicture}[H,centerzero]
\draw[->] (0,-0.2) \botlabel{i} -- (0,0.2);
\end{tikzpicture}\
&: P_i \Rightarrow P_i,
&\begin{tikzpicture}[H,centerzero]
\draw[<-] (0,-0.2) \botlabel{i} -- (0,0.2);
\end{tikzpicture}\
&: Q_i \Rightarrow Q_i.
\end{align*}
We depict the inclusions $P_i \hookrightarrow P$, $Q_i \hookrightarrow Q$ and the projections $P \twoheadrightarrow P_i$, $Q \twoheadrightarrow Q_i$ respectively, by
\begin{align*}
\begin{tikzpicture}[H,baseline=-2mm]
\draw[->] (0,-0.2) \botlabel{i} -- (0,0.2);
\notch{0,0};
\end{tikzpicture}\
&: P_i \Rightarrow P,
&
\begin{tikzpicture}[H,baseline=-2mm]
\draw[<-] (0,-0.2) \botlabel{i} -- (0,0.2);
\notch{0,0};
\end{tikzpicture}\
&: Q_i \Rightarrow Q,
&
\begin{tikzpicture}[H,centerzero]
\draw[->] (0,-0.2) -- (0,0.2) \toplabel{i};
\notch{0,0};
\end{tikzpicture}\
&: P \Rightarrow P_i,
&
\begin{tikzpicture}[H,centerzero]
\draw[<-] (0,-0.2) -- (0,0.2) \toplabel{i};
\notch{0,0};
\end{tikzpicture}\
&: Q \Rightarrow Q_i.
\end{align*}

\vspace{-4mm}
\noindent
Thus,
$\begin{tikzpicture}[H,centerzero]
\draw[->] (0,-0.3) -- (0,0.3);
\strand{0.1,0}{i};
\notch{0,-0.15};
\notch{0,0.15};
\end{tikzpicture}
: P \Rightarrow P$
is the projection of $P$ onto its summand $P_i$, and
$\begin{tikzpicture}[H,anchorbase]
\draw[->] (0,-0.25) \botlabel{i} -- (0,0.25) \toplabel{j};
\notch{0,-0.1};
\notch{0,0.1};
\end{tikzpicture}
= \delta_{i=j}\ 
\begin{tikzpicture}[H,centerzero]
\draw[->] (0,-0.25) \botlabel{i} -- (0,0.25);
\end{tikzpicture}\ $.

\subsection{Projected dots and tokens}
It is clear from the definitions that the endomorphisms of $P$ and $Q$ defined by the dots restrict to even endomorphisms of the summands $P_i$ and $Q_i$.
We represent these restrictions by drawing dots on a string colored by $i$. So for $V \in \catR$ we have the morphisms
\begin{equation}\label{wombat2}
\begin{tikzpicture}[H,centerzero,scale=1.1]
\draw[->] (-0.15,-0.25) \botlabel{i} -- (-0.15,0.25);
\draw[gcolor,thick] (0.15,-0.25) \botlabel{V} -- (0.15,0.25);
\singdot{-0.15,0};
\end{tikzpicture} \qquad\text{ and }\qquad
\begin{tikzpicture}[H,centerzero,scale=1.1]
\draw[<-] (-0.15,-0.25)\botlabel{i} -- (-0.15,0.25);
\draw[gcolor,thick] (0.15,-0.25) \botlabel{V} -- (0.15,0.25);
\singdot{-0.15,0};
\end{tikzpicture}.
\end{equation}
The minimal polynomials of these endomorphisms are
$(x-b(i))^{\eps_i(V)}$ and $(x-b(i))^{\phi_i(V)}$,
respectively.
We have that
\begin{align} \label{sizzle}
\begin{tikzpicture}[H,centerzero={0,-0.05}]
\draw[->] (0,-0.3) \botlabel{i} -- (0,0.3);
\notch{0,0.1};
\singdot{0,-0.13};
\end{tikzpicture}
&=
\begin{tikzpicture}[H,centerzero={0,-0.05}]
\draw[->] (0,-0.3) \botlabel{i}-- (0,0.3);
\notch{0,-0.13};
\singdot{0,0.1};
\end{tikzpicture}
\ ,&
\begin{tikzpicture}[H,centerzero={0,-0.05}]
\draw[<-] (0,-0.3) \botlabel{i} -- (0,0.3);
\notch{0,0.13};
\singdot{0,-0.1};
\end{tikzpicture}
&=
\begin{tikzpicture}[H,centerzero={0,-0.05}]
\draw[<-] (0,-0.3) \botlabel{i}-- (0,0.3);
\notch{0,-0.1};
\singdot{0,0.13};
\end{tikzpicture}
\ ,&
\begin{tikzpicture}[H,centerzero={0,0.05}]
\draw[->] (0,-0.3) -- (0,0.3) \toplabel{i};
\notch{0,0.1};
\singdot{0,-0.13};
\end{tikzpicture}
&=
\begin{tikzpicture}[H,centerzero={0,0.05}]
\draw[->] (0,-0.3)-- (0,0.3) \toplabel{i};
\notch{0,-0.13};
\singdot{0,0.1};
\end{tikzpicture}\
,&
\begin{tikzpicture}[H,centerzero={0,0.05}]
\draw[<-] (0,-0.3) -- (0,0.3) \toplabel{i};
\notch{0,0.13};
\singdot{0,-0.1};
\end{tikzpicture}
&=
\begin{tikzpicture}[H,centerzero={0,0.05}]
\draw[<-] (0,-0.3)-- (0,0.3) \toplabel{i};
\notch{0,-0.1};
\singdot{0,0.13};
\end{tikzpicture}
\ .
\end{align}
Also, using again that $b(-i)=-b(i)$,
the Clifford token induces odd isomorphisms $P_i \stackrel{\sim}{\Rightarrow} P_{-i}$ and $Q_i \stackrel{\sim}{\Rightarrow} Q_{-i}$, which we denote by
\begin{align}\label{cliffordisos}
\begin{tikzpicture}[H,centerzero]
\draw[->] (0,-0.3) \botlabel{i} -- (0,0.3) \toplabel{-i};
\token{0,0};
\end{tikzpicture}
&=
\begin{tikzpicture}[H,centerzero]
\draw[->] (0,-0.3) \botlabel{i} -- (0,0.3);
\token{0,0};
\end{tikzpicture}
=
\begin{tikzpicture}[H,centerzero]
\draw[->] (0,-0.3) -- (0,0.3) \toplabel{-i};
\token{0,0};
\end{tikzpicture},&
\begin{tikzpicture}[H,centerzero]
\draw[<-] (0,-0.3) \botlabel{i} -- (0,0.3) \toplabel{-i};
\token{0,0};
\end{tikzpicture}
&=
\begin{tikzpicture}[H,centerzero]
\draw[<-] (0,-0.3) \botlabel{i} -- (0,0.3);
\token{0,0};
\end{tikzpicture}
=
\begin{tikzpicture}[H,centerzero]
\draw[<-] (0,-0.3) -- (0,0.3) \toplabel{-i};
\token{0,0};
\end{tikzpicture}.
\end{align}
We then have that
\begin{align} \label{moreslipperyness}
\begin{tikzpicture}[H,centerzero={0,-0.05}]
\draw[->] (0,-0.3) \botlabel{i} -- (0,0.3);
\notch{0,0.1};
\token{0,-0.13};
\end{tikzpicture}
&=
\begin{tikzpicture}[H,centerzero={0,-0.05}]
\draw[->] (0,-0.3) \botlabel{i}-- (0,0.3);
\notch{0,-0.13};
\token{0,0.1};
\end{tikzpicture}\
,&
\begin{tikzpicture}[H,centerzero={0,-0.05}]
\draw[<-] (0,-0.3) \botlabel{i} -- (0,0.3);
\notch{0,0.13};
\token{0,-0.1};
\end{tikzpicture}
&=
\begin{tikzpicture}[H,centerzero={0,-0.05}]
\draw[<-] (0,-0.3) \botlabel{i}-- (0,0.3);
\notch{0,-0.1};
\token{0,0.13};
\end{tikzpicture}\
,&
\begin{tikzpicture}[H,centerzero={0,0.05}]
\draw[->] (0,-0.3) -- (0,0.3) \toplabel{i};
\notch{0,0.1};
\token{0,-0.13};
\end{tikzpicture}
&=
\begin{tikzpicture}[H,centerzero={0,0.05}]
\draw[->] (0,-0.3)-- (0,0.3) \toplabel{i};
\notch{0,-0.13};
\token{0,0.1};
\end{tikzpicture}\
,&
\begin{tikzpicture}[H,centerzero={0,0.05}]
\draw[<-] (0,-0.3) -- (0,0.3) \toplabel{i};
\notch{0,0.13};
\token{0,-0.1};
\end{tikzpicture}
&=
\begin{tikzpicture}[H,centerzero={0,0.05}]
\draw[<-] (0,-0.3)-- (0,0.3) \toplabel{i};
\notch{0,-0.1};
\token{0,0.13};
\end{tikzpicture}\ .
\end{align}

\subsection{Projected cups and caps}
The relation \cref{adjright} means that
the rightward cup and rightward cap define the unit and counit of an adjunction $(P,Q)$. Similarly, thanks to \cref{adjleft},
the leftward cup and leftward cap define the unit and counit of an adjunction $(Q,P)$.
Thus, the endofunctors $\underline{P}$ and $\underline{Q}$ of $\underline{\catR}$
are biadjoint. In particular, they are both exact.

It follows from the last two relations in \cref{ruby} that the rightward cup and cap induce adjunctions $(P_i,Q_i)$
for all $i \in \kk$.
Similarly,
from the last two relations in \cref{ruby2},
the leftward cup and cap induce adjunctions
$(Q_i,P_i)$ for all $i$.  We draw the units and counits of these adjunctions using cups and caps colored by $i$:
$\begin{tikzpicture}[H,centerzero,scale=1]
\draw[->] (-0.2,-0.2) \botlabel{i} -- (-0.2,0) arc(180:0:0.2) -- (0.2,-0.2);
\end{tikzpicture}\ $,
$\begin{tikzpicture}[H,centerzero,scale=1]
\draw[->] (-0.2,0.2) \toplabel{i} -- (-0.2,0) arc(180:360:0.2) -- (0.2,0.2);
\end{tikzpicture}\ $,
$\begin{tikzpicture}[H,centerzero,scale=1]
\draw[<-] (-0.2,-0.2) -- (-0.2,0) arc(180:0:0.2) -- (0.2,-0.2) \botlabel{i};
\end{tikzpicture}\ $ and 
$\ \begin{tikzpicture}[H,centerzero,scale=1]
\draw[<-] (-0.2,0.2) -- (-0.2,0) arc(180:360:0.2) -- (0.2,0.2) \toplabel{i};
\end{tikzpicture}$.
The various inclusions and projections are compatible with these colored cups/caps, in the sense that
\begin{equation}
\begin{gathered}
\begin{tikzpicture}[H,centerzero={0,-0.2},scale=1.1]
\draw[->] (-0.2,-0.3) \botlabel{i} -- (-0.2,-0.1) arc(180:0:0.2) -- (0.2,-0.3);
\notch{-0.2,-0.1};
\end{tikzpicture}
\ =
\begin{tikzpicture}[H,centerzero={0,-0.2},scale=1.1]
\draw[->] (-0.2,-0.3) \botlabel{i} -- (-0.2,-0.1) arc(180:0:0.2) -- (0.2,-0.3);
\notch{0.2,-0.1};
\end{tikzpicture}
\ ,\qquad
\begin{tikzpicture}[H,centerzero={0,-0.2},scale=1.1]
\draw[<-] (-0.2,-0.3) \botlabel{i} -- (-0.2,-0.1) arc(180:0:0.2) -- (0.2,-0.3);
\notch{-0.2,-0.1};
\end{tikzpicture}
\ =
\begin{tikzpicture}[H,centerzero={0,-0.2},scale=1.1]
\draw[<-] (-0.2,-0.3) \botlabel{i} -- (-0.2,-0.1) arc(180:0:0.2) -- (0.2,-0.3);
\notch{0.2,-0.1};
\end{tikzpicture}
\ ,\qquad
\begin{tikzpicture}[H,centerzero={0,-0.2},scale=1.1]
\draw[->] (-0.2,-0.3) -- (-0.2,-0.1) arc(180:0:0.2) -- (0.2,-0.3) \botlabel{i};
\notch{-0.2,-0.1};
\end{tikzpicture}
=\
\begin{tikzpicture}[H,centerzero={0,-0.2},scale=1.1]
\draw[->] (-0.2,-0.3) -- (-0.2,-0.1) arc(180:0:0.2) -- (0.2,-0.3) \botlabel{i};
\notch{0.2,-0.1};
\end{tikzpicture}
\ ,\qquad
\begin{tikzpicture}[H,centerzero={0,-0.2},scale=1.1]
\draw[<-] (-0.2,-0.3) -- (-0.2,-0.1) arc(180:0:0.2) -- (0.2,-0.3) \botlabel{i};
\notch{-0.2,-0.1};
\end{tikzpicture}
=\
\begin{tikzpicture}[H,centerzero={0,-0.2},scale=1.1]
\draw[<-] (-0.2,-0.3) -- (-0.2,-0.1) arc(180:0:0.2) -- (0.2,-0.3) \botlabel{i};
\notch{0.2,-0.1};
\end{tikzpicture}
\ ,
\\
\begin{tikzpicture}[H,centerzero={0,0.2},scale=1.1]
\draw[->] (-0.2,0.3) \toplabel{i} -- (-0.2,0.1) arc(180:360:0.2) -- (0.2,0.3);
\notch{-0.2,0.1};
\end{tikzpicture}
\ =
\begin{tikzpicture}[H,centerzero={0,0.2},scale=1.1]
\draw[->] (-0.2,0.3) \toplabel{i} -- (-0.2,0.1) arc(180:360:0.2) -- (0.2,0.3);
\notch{0.2,0.1};
\end{tikzpicture}
\ ,\qquad
\begin{tikzpicture}[H,centerzero={0,0.2},scale=1.1]
\draw[<-] (-0.2,0.3) \toplabel{i} -- (-0.2,0.1) arc(180:360:0.2) -- (0.2,0.3);
\notch{-0.2,0.1};
\end{tikzpicture}
\ =
\begin{tikzpicture}[H,centerzero={0,0.2},scale=1.1]
\draw[<-] (-0.2,0.3) \toplabel{i} -- (-0.2,0.1) arc(180:360:0.2) -- (0.2,0.3);
\notch{0.2,0.1};
\end{tikzpicture}
\ ,\qquad
\begin{tikzpicture}[H,centerzero={0,0.2},scale=1.1]
\draw[->] (-0.2,0.3) -- (-0.2,0.1) arc(180:360:0.2) -- (0.2,0.3) \toplabel{i};
\notch{-0.2,0.1};
\end{tikzpicture}
=\
\begin{tikzpicture}[H,centerzero={0,0.2},scale=1.1]
\draw[->] (-0.2,0.3) -- (-0.2,0.1) arc(180:360:0.2) -- (0.2,0.3) \toplabel{i};
\notch{0.2,0.1};
\end{tikzpicture}
\ ,\qquad
\begin{tikzpicture}[H,centerzero={0,0.2},scale=1.1]
\draw[<-] (-0.2,0.3) -- (-0.2,0.1) arc(180:360:0.2) -- (0.2,0.3) \toplabel{i};
\notch{-0.2,0.1};
\end{tikzpicture}
=\
\begin{tikzpicture}[H,centerzero={0,0.2},scale=1.1]
\draw[<-] (-0.2,0.3) -- (-0.2,0.1) arc(180:360:0.2) -- (0.2,0.3) \toplabel{i};
\notch{0.2,0.1};
\end{tikzpicture}
\ .
\end{gathered}
\end{equation}
Regardless of the color, dots and Clifford tokens slide over colored cups and caps in the same way as in \cref{ruby,ruby2}.

\subsection{Projected crossings}

For $i,j,i',j' \in \kk$, define
\begin{equation} \label{tassie}
\begin{tikzpicture}[H,centerzero]
\draw[->] (-0.3,-0.3) \botlabel{i} -- (0.3,0.3) \toplabel{j'};
\draw[->] (0.3,-0.3) \botlabel{j} -- (-0.3,0.3) \toplabel{i'};
\projcr{0,0};
\end{tikzpicture}
:=
\begin{tikzpicture}[H,centerzero]
\draw[->] (-0.3,-0.3) \botlabel{i} -- (0.3,0.3) \toplabel{j'};
\draw[->] (0.3,-0.3) \botlabel{j} -- (-0.3,0.3) \toplabel{i'};
\notch[-45]{0.15,0.15};
\notch[45]{0.15,-0.15};
\notch[135]{-0.15,-0.15};
\notch[45]{-0.15,0.15};
\end{tikzpicture}.
\end{equation}
Thus, this supernatural transformation is defined by first including the summand $P_{i} \circ P_{j}$ into $P \circ P$, then applying the supernatural transformation $P\circ P \Rightarrow P\circ P$ defined by the upward crossing, then projecting $P \circ P$ onto the summand $P_{i'} \circ P_{j'}$.
By \cref{moreslipperyness,sergeev2,sergeev12}, we have that
\begin{align}\label{tie}
\begin{tikzpicture}[H,centerzero,scale=1.1]
\draw[->] (-0.3,-0.3) \botlabel{i} -- (0.3,0.3) \toplabel{j'};
\draw[->] (0.3,-0.3) \botlabel{j} -- (-0.3,0.3) \toplabel{i'};
\projcr{0,0};
\token{-.17,-.17};
\end{tikzpicture}
&=
\begin{tikzpicture}[H,centerzero,scale=1.1]
\draw[->] (-0.3,-0.3) \botlabel{i} -- (0.3,0.3) \toplabel{j'};
\draw[->] (0.3,-0.3) \botlabel{j} -- (-0.3,0.3) \toplabel{i'};
\projcr{0,0};
\token{.17,.17};
\end{tikzpicture}\ ,&
\begin{tikzpicture}[H,centerzero,scale=1.1]
\draw[->] (-0.3,-0.3) \botlabel{i} -- (0.3,0.3) \toplabel{j'};
\draw[->] (0.3,-0.3) \botlabel{j} -- (-0.3,0.3) \toplabel{i'};
\projcr{0,0};
\token{-.17,.17};
\end{tikzpicture}
&=
\begin{tikzpicture}[H,centerzero,scale=1.1]
\draw[->] (-0.3,-0.3) \botlabel{i} -- (0.3,0.3) \toplabel{j'};
\draw[->] (0.3,-0.3) \botlabel{j} -- (-0.3,0.3) \toplabel{i'};
\projcr{0,0};
\token{.17,-.17};;
\end{tikzpicture}
\end{align}
for any $i,j,i',j' \in \kk$.
Relations \cref{affsergeev,sergeev12,sizzle} imply that
\begin{align} \label{thai1}
\begin{tikzpicture}[H,centerzero]
\draw[->] (-0.3,-0.3) \botlabel{i} -- (0.3,0.3) \toplabel{j'};
\draw[->] (0.3,-0.3) \botlabel{j} -- (-0.3,0.3) \toplabel{i'};
\projcr{0,0};
\singdot{-0.17,-0.17};
\end{tikzpicture}
-
\begin{tikzpicture}[H,centerzero]
\draw[->] (-0.3,-0.3) \botlabel{i} -- (0.3,0.3) \toplabel{j'};
\draw[->] (0.3,-0.3) \botlabel{j} -- (-0.3,0.3) \toplabel{i'};
\projcr{0,0};
\singdot{0.17,0.17};
\end{tikzpicture}
&= \delta_{i=i'} \delta_{j=j'}
\begin{tikzpicture}[H,centerzero]
\draw[->] (-0.2,-0.3) \botlabel{i} -- (-0.2,0.3);
\draw[->] (0.2,-0.3) \botlabel{j}  -- (0.2,0.3);
\end{tikzpicture}
- \delta_{i=-i'} \delta_{j=-j'}
\begin{tikzpicture}[H,centerzero]
\draw[->] (-0.2,-0.3) \botlabel{i} -- (-0.2,0.3);
\draw[->] (0.2,-0.3) \botlabel{j}  -- (0.2,0.3);
\token{-0.2,0};
\token{0.2,0};
\end{tikzpicture}
,
\\ \label{thai2}
\begin{tikzpicture}[H,centerzero]
\draw[->] (-0.3,-0.3) \botlabel{i} -- (0.3,0.3) \toplabel{j'};
\draw[->] (0.3,-0.3) \botlabel{j} -- (-0.3,0.3) \toplabel{i'};
\projcr{0,0};
\singdot{-0.17,0.17};
\end{tikzpicture}
-
\begin{tikzpicture}[H,centerzero]
\draw[->] (-0.3,-0.3) \botlabel{i} -- (0.3,0.3) \toplabel{j'};
\draw[->] (0.3,-0.3) \botlabel{j} -- (-0.3,0.3) \toplabel{i'};
\projcr{0,0};
\singdot{0.17,-0.17};
\end{tikzpicture}
&= \delta_{i=i'} \delta_{j=j'}
\begin{tikzpicture}[H,centerzero]
\draw[->] (-0.2,-0.3) \botlabel{i} -- (-0.2,0.3);
\draw[->] (0.2,-0.3) \botlabel{j}  -- (0.2,0.3);
\end{tikzpicture}
+ \delta_{i=-i'} \delta_{j=-j'}
\begin{tikzpicture}[H,centerzero]
\draw[->] (-0.2,-0.3) \botlabel{i} -- (-0.2,0.3);
\draw[->] (0.2,-0.3) \botlabel{j}  -- (0.2,0.3);
\token{-0.2,0};
\token{0.2,0};
\end{tikzpicture}.
\end{align}
In particular, these show that
\begin{align}\label{thai}
\begin{tikzpicture}[H,centerzero,scale=1.1]
\draw[->] (-0.3,-0.3) \botlabel{i} -- (0.3,0.3) \toplabel{i};
\draw[->] (0.3,-0.3) \botlabel{j} -- (-0.3,0.3) \toplabel{j};
\projcr{0,0};
\pin{-.17,-.17}{-1,-.17}{f(x)};
\end{tikzpicture}
&=
\begin{tikzpicture}[H,centerzero,scale=1.1]
\draw[->] (-0.3,-0.3) \botlabel{i} -- (0.3,0.3) \toplabel{i};
\draw[->] (0.3,-0.3) \botlabel{j} -- (-0.3,0.3) \toplabel{j};
\projcr{0,0};
\pin{.17,.17}{1,.17}{f(x)};
\end{tikzpicture}\ ,&
\begin{tikzpicture}[H,centerzero,scale=1.1]
\draw[->] (-0.3,-0.3) \botlabel{i} -- (0.3,0.3) \toplabel{i};
\draw[->] (0.3,-0.3) \botlabel{j} -- (-0.3,0.3) \toplabel{j};
\projcr{0,0};
\pin{-.17,.17}{-1,.17}{f(x)};
\end{tikzpicture}
&=
\begin{tikzpicture}[H,centerzero,scale=1.1]
\draw[->] (-0.3,-0.3) \botlabel{i} -- (0.3,0.3) \toplabel{i};
\draw[->] (0.3,-0.3) \botlabel{j} -- (-0.3,0.3) \toplabel{j};
\projcr{0,0};
\pin{.17,-.17}{1,-.17}{f(x)};
\end{tikzpicture}
\end{align}
for $i \neq \pm j$ and any $f(x) \in \kk[x]$.
More succinctly,
\begin{equation}\label{shortthai}
\begin{tikzpicture}[H,centerzero,scale=1.1]
\draw[->] (-0.3,-0.3) \botlabel{i} -- (0.3,0.3) \toplabel{i};
\draw[->] (0.3,-0.3) \botlabel{j} -- (-0.3,0.3) \toplabel{j};
\projcr{0,0};
\pinpin{-.17,-.17}{.17,-.17}{1.1,-.17}{f(x,y)};
\end{tikzpicture}
=
\begin{tikzpicture}[H,centerzero,scale=1.1]
\draw[->] (-0.3,-0.3) \botlabel{i} -- (0.3,0.3) \toplabel{i};
\draw[->] (0.3,-0.3) \botlabel{j} -- (-0.3,0.3) \toplabel{j};
\projcr{0,0};
\pinpin{.17,.17}{-.17,.17}{-1.1,.17}{f(y,x)};
\end{tikzpicture}
\end{equation}
for $i \neq \pm j$ and any $f(x,y) \in \kk[x,y]$.

\begin{lem} \label{electric}
If the supernatural transformation
\cref{tassie} is non-zero then  either
$i=i'$ and $j=j'$, or $i=-i'$ and $j=-j'$,
or $i = j'$ and $j=i'$.
\end{lem}

\begin{proof}
This argument is similar to the proof of \cite[Lem.~4.1]{BSW-HKM}.
Suppose not ($i=i'$ and $j=j'$)
and not ($i=-i'$ and $j=-j'$)
and not ($i=j'$ and $j=i'$).
We must prove that \cref{tassie} is 0.
We either have that $i \neq i'$ or $j \neq j'$,
so the first terms on the right-hand sides of \cref{thai1,thai2} are 0.
We either have that $i \neq -i'$ or $j \neq -j'$,
so the second terms on the right-hand sides of
\cref{thai1,thai2} are 0.
Thus, we have that
\begin{align*}
\begin{tikzpicture}[H,centerzero]
\draw[->] (-0.3,-0.3) \botlabel{i} -- (0.3,0.3) \toplabel{j'};
\draw[->] (0.3,-0.3) \botlabel{j} -- (-0.3,0.3) \toplabel{i'};
\projcr{0,0};
\singdot{-0.17,-0.17};
\end{tikzpicture}
&=
\begin{tikzpicture}[H,centerzero]
\draw[->] (-0.3,-0.3) \botlabel{i} -- (0.3,0.3) \toplabel{j'};
\draw[->] (0.3,-0.3) \botlabel{j} -- (-0.3,0.3) \toplabel{i'};
\projcr{0,0};
\singdot{0.17,0.17};
\end{tikzpicture},&
\begin{tikzpicture}[H,centerzero]
\draw[->] (-0.3,-0.3) \botlabel{i} -- (0.3,0.3) \toplabel{j'};
\draw[->] (0.3,-0.3) \botlabel{j} -- (-0.3,0.3) \toplabel{i'};
\projcr{0,0};
\singdot{-0.17,0.17};
\end{tikzpicture}
&=
\begin{tikzpicture}[H,centerzero]
\draw[->] (-0.3,-0.3) \botlabel{i} -- (0.3,0.3) \toplabel{j'};
\draw[->] (0.3,-0.3) \botlabel{j} -- (-0.3,0.3) \toplabel{i'};
\projcr{0,0};
\singdot{0.17,-0.17};
\end{tikzpicture}.
\end{align*}
Now we assume that $i \neq j'$ and show that \cref{tassie} is 0.
It suffices to show that
$$
\begin{tikzpicture}[H,centerzero]
\draw[->] (-0.3,-0.3) \botlabel{i} -- (0.3,0.3) \toplabel{j'};
\draw[->] (0.3,-0.3) \botlabel{j} -- (-0.3,0.3) \toplabel{i'};
\projcr{0,0};
\draw[-,gcolor,thick] (.7,-.3)\botlabel{V} to (.7,.3);
\end{tikzpicture}=0
$$
for any finitely generated object $V \in \catR$.
Since $i \neq j'$, we have that $b(i) \neq b(j')$, so the polynomials $(x-b(i))^{\eps_i(P_j V)}$ and $(x-b(j'))^{\eps_{j'}(V)}$
are relatively prime.
So we can find $f(x), g(x) \in \kk[x]$ such that
$$
f(x) (x-b(i))^{\eps_i(P_j V)}+
g(x)(x-b(j'))^{\eps_{j'}(V)}=1.
$$
We deduce that
$$
\begin{tikzpicture}[H,centerzero]
\draw[->] (-0.3,-0.3) \botlabel{i} -- (0.3,0.3) \toplabel{j'};
\draw[->] (0.3,-0.3) \botlabel{j} -- (-0.3,0.3) \toplabel{i'};
\projcr{0,0};
\draw[-,gcolor,thick] (.7,-.3)\botlabel{V} to (.7,.3);
\end{tikzpicture}
= \begin{tikzpicture}[H,centerzero]
\draw[->] (-0.3,-0.3) \botlabel{i} -- (0.3,0.3) \toplabel{j'};
\draw[->] (0.3,-0.3) \botlabel{j} -- (-0.3,0.3) \toplabel{i'};
\projcr{0,0};
\pin{-.16,-.16}{-2,-.16}{g(x)(x-b(j'))^{\eps_{j'}(V)}};
\draw[-,gcolor,thick] (.7,-.3)\botlabel{V} to (.7,.3);
\end{tikzpicture}
\stackrel{\cref{thai1}}{=}
\begin{tikzpicture}[H,centerzero]
\draw[->] (-0.3,-0.3) \botlabel{i} -- (0.3,0.3) \toplabel{j'};
\draw[->] (0.3,-0.3) \botlabel{j} -- (-0.3,0.3) \toplabel{i'};
\projcr{0,0};
\pin{.14,.14}{2,.14}{g(x)(x-b(j'))^{\eps_{j'}(V)}};
\draw[-,gcolor,thick] (3.6,-.3)\botlabel{V} to (3.6,.3);
\end{tikzpicture}
=0,
$$
as claimed.
A similar argument shows that \cref{tassie} is 0 if $i' \neq j$.
\end{proof}

Next, we introduce an important diagrammatic convention,
also used in \cite[Sec.~4]{BSW-HKM}.  On any finitely generated $V \in \catR$, the endomorphisms
$\begin{tikzpicture}[H,centerzero={0,-0.1}]
\draw[->] (-0.15,-0.2) \botlabel{i} -- (-0.15,0.2);
\draw[gcolor,thick] (0.15,-0.2) \botlabel{V} -- (0.15,0.2);
\pin{-0.15,0}{-0.9,0}{x-b(i)};
\end{tikzpicture}$
$\begin{tikzpicture}[H,centerzero={0,-0.1}]
\draw[<-] (-0.15,-0.2) \botlabel{i} -- (-0.15,0.2);
\draw[gcolor,thick] (0.15,-0.2) \botlabel{V} -- (0.15,0.2);
\pin{-0.15,0}{-0.9,0}{x-b(i)};
\end{tikzpicture}$
are nilpotent, so the notations
$\begin{tikzpicture}[H,centerzero={0,-0.1}]
\draw[->] (-0.15,-0.2) \botlabel{i} -- (-0.15,0.2);
\draw[gcolor,thick] (0.15,-0.2) \botlabel{V} -- (0.15,0.2);
\pin{-0.15,0}{-0.8,0}{f(x)};
\end{tikzpicture}$ and
$\begin{tikzpicture}[H,centerzero={0,-0.1}]
\draw[<-] (-0.15,-0.2) \botlabel{i} -- (-0.15,0.2);
\draw[gcolor,thick] (0.15,-0.2) \botlabel{V} -- (0.15,0.2);
\pin{-0.15,0}{-0.8,0}{f(x)};
\end{tikzpicture}$
makes sense for power series $f(x) \in \kk \llbracket x-b(i) \rrbracket$ rather than merely for polynomials. It follows that there are well-defined supernatural transformations
\begin{align*}
\begin{tikzpicture}[H,centerzero={0,-0.1}]
\draw[->] (0,-0.3) \botlabel{i} -- (0,0.3);
\pin{0,0}{-0.7,0}{f(x)};
\end{tikzpicture}
&   : P_i \Rightarrow P_i,&
\begin{tikzpicture}[H,centerzero={0,-0.1}]
\draw[<-] (0,-0.3) \botlabel{i} -- (0,0.3);
\pin{0,0}{-0.7,0}{f(x)};
\end{tikzpicture}
&   : Q_i \Rightarrow Q_i,
\end{align*}
for any $i \in \kk$ and any $f(x) \in \kk \llbracket x-b(i) \rrbracket$.
This generalizes in the obvious way to pins attached to two or more strings. For example,
suppose that $i \neq j$, hence, $b(i) \neq b(j)$. Let $\gamma := (b(i)-b(j))^{-1}$ so that $$
\frac{1}{x-y} = (b(i)-b(j) + (x-b(i)) - (y-b(j)))^{-1}
=
\gamma -\gamma^2(x-b(i))+\gamma^2(y-b(j)) +
\cdots \in \kk \llbracket x-b(i),y-b(j) \rrbracket.
$$
Then there is a supernatural transformation
\begin{equation}
\begin{tikzpicture}[H,centerzero]
\draw[->] (-0.2,-0.4) \botlabel{i} -- (-0.2,0.4);
\draw[->] (0.3,-0.4) \botlabel{j} -- (0.3,0.4);
\pinpin{0.3,0}{-0.2,0}{-1,0}{\frac{1}{x-y}};
\end{tikzpicture}
= \gamma
\begin{tikzpicture}[H,centerzero]
\draw[->] (-0.3,-0.4) \botlabel{i} -- (-0.3,0.4);
\draw[->] (0.2,-0.4) \botlabel{j} -- (0.2,0.4);
\end{tikzpicture}
- \gamma^2\
\begin{tikzpicture}[H,centerzero]
\draw[->] (-0.3,-0.4) \botlabel{i} -- (-0.3,0.4);
\draw[->] (0.2,-0.4) \botlabel{j} -- (0.2,0.4);
\pin{-0.3,0}{-1.15,0}{x-b(i)};
\end{tikzpicture}
+ \gamma^2
\begin{tikzpicture}[H,centerzero]
\draw[->] (-0.2,-0.4) \botlabel{i} -- (-0.2,0.4);
\draw[->] (0.3,-0.4) \botlabel{j} -- (0.3,0.4);
\pin{0.3,0}{1.15,0}{x-b(j)};
\end{tikzpicture}
+ \dots
\in \End(P_i\circ P_j)\llbracket u^{-1} \rrbracket.
\end{equation}
Of course, this is a two-sided inverse of
$\ \begin{tikzpicture}[H,baseline=-2mm]
\draw[->] (-0.15,-0.3) \botlabel{i} -- (-0.15,0.3);
\draw[->] (0.25,-0.3) \botlabel{j} -- (0.25,0.3);
\pinpin{0.25,0}{-0.15,0}{-.9,0}{x-y};
\end{tikzpicture}\ $, hence, the latter supernatural transformation is invertible when $i \neq j$.
Similarly, $\ \begin{tikzpicture}[H,baseline=-2mm]
\draw[->] (-0.15,-0.3) \botlabel{i} -- (-0.15,0.3);
\draw[->] (0.25,-0.3) \botlabel{j} -- (0.25,0.3);
\pinpin{0.25,0}{-0.15,0}{-.9,0}{x+y};
\end{tikzpicture}\ $ is invertible when $i \neq -j$.

\begin{lem}\label{power}
For $i,j \in \kk$,
we have that
\begin{align}\label{noodle}
\begin{tikzpicture}[H,centerzero]
\draw[->] (-0.3,-0.3) \botlabel{i} -- (0.3,0.3) \toplabel{j};
\draw[->] (0.3,-0.3) \botlabel{j} -- (-0.3,0.3) \toplabel{i};
\projcr{0,0};
\end{tikzpicture}
&=
\begin{tikzpicture}[H,centerzero]
\draw[->] (-0.3,-0.3) \botlabel{i} -- (-0.3,0.3);
\draw[->] (0.3,-0.3) \botlabel{j} -- (0.3,0.3);
\pinpin{0.3,0}{-0.3,0}{-1.1,0}{\frac{1}{x-y}};
\end{tikzpicture}
\quad \text{if } i \ne j,&
\begin{tikzpicture}[H,centerzero]
\draw[->] (-0.3,-0.3) \botlabel{i} -- (0.3,0.3) \toplabel{-j};
\draw[->] (0.3,-0.3) \botlabel{j} -- (-0.3,0.3) \toplabel{-i};
\projcr{0,0};
\end{tikzpicture}
&=-\ 
\begin{tikzpicture}[H,centerzero]
\draw[->] (-0.3,-0.3) \botlabel{i} -- (-0.3,0.3);
\draw[->] (0.3,-0.3) \botlabel{j} -- (0.3,0.3);
\pinpin{0.3,0.05}{-0.3,-0.15}{-1.1,-.15}{\frac{1}{x-y}};
\token{-0.3,0.05};
\token{0.3,-0.15};
\end{tikzpicture}
\quad \text{if } i \ne -j.
\end{align}
\end{lem}

\begin{proof}
The method of proof is analogous to that of \cite[Lem.~4.2]{BSW-HKM}. For the first one, we must show for $V \in \catR$ that
\[
\psi
:=
\begin{tikzpicture}[H,centerzero]
\draw[->] (-0.3,-0.3) \botlabel{i} -- (0.3,0.3) \toplabel{j};
\draw[->] (0.3,-0.3) \botlabel{j} -- (-0.3,0.3) \toplabel{i};
\projcr{0,0};
\draw[gcolor,thick] (0.75,-0.3) \botlabel{V} -- (0.75,0.3);
\end{tikzpicture}
-
\begin{tikzpicture}[H,centerzero]
\draw[->] (-0.3,-0.3) \botlabel{i} -- (-0.3,0.3);
\draw[->] (0.3,-0.3) \botlabel{j} -- (0.3,0.3);
\pinpin{0.3,0}{-0.3,0}{-1.1,0}{\frac{1}{x-y}};
\draw[gcolor,thick] (.75,-0.3) \botlabel{V} -- (.75,0.3);
\end{tikzpicture}
\]
is 0 in the finite-dimensional superalgebra $\End_\catR(P_{i} P_{j} V)$.  Let $L: \End_\catR(P_{i} P_{j} V)\rightarrow \End_\catR(P_{i} P_{j} V)$ be the linear map
defined by left multiplication  by $
\begin{tikzpicture}[H,centerzero={0,-0.15}]
\draw[->] (-0.3,-0.2) \botlabel{i} -- (-0.3,0.2);
\draw[->] (0,-0.2) \botlabel{j} -- (0,0.2);
\singdot{-0.3,0};
\draw[gcolor,thick] (0.3,-0.2) \botlabel{V}-- (0.3,0.2);
\end{tikzpicture}
\,
$(diagrammatically, this is vertical composition on the top),
let $R:\End_\catR(P_{i} P_{j} V)\rightarrow \End_\catR(P_{i} P_{j} V)$ be the linear map defined by right multiplication by
$\begin{tikzpicture}[H,centerzero={0,-0.15}]
\draw[->] (-0.3,-0.2) \botlabel{i} -- (-0.3,0.2);
\draw[->] (0,-0.2) \botlabel{j} -- (0,0.2);
\singdot{0,0};
\draw[gcolor,thick] (0.3,-0.2) \botlabel{V}-- (0.3,0.2);
\end{tikzpicture}
\,$ (diagrammatically, this is vertical composition on the bottom).
We have that $(L - b(i))^{\eps_i(V)} = 0$ and $(R-b(j))^{\eps_j(V)} = 0$.  Hence, for sufficiently large $N$, we have that
\[
\big( (L-R) + (b(j)-b(i)) \big)^N
= \big( (L-b(i)) - (R-b(j)) \big)^N
= 0.
\]
The assumption $i \neq j$ implies that one of $i$ or $j$ is non-zero, hence, either $i \neq -i$ or $j \neq -j$. Using this, the relation \cref{thai2} implies that
$(L-R) (\psi) = 0$. So the equation just displayed implies that $(b(j)-b(i))^N \psi = 0$.  Since $i \ne j$, this implies that $\psi=0$, as desired.

For the second equality in \cref{noodle}, we must show instead that
\[
\psi
:=
\begin{tikzpicture}[H,centerzero]
\draw[->] (-0.3,-0.3) \botlabel{i} -- (0.3,0.3) \toplabel{-j};
\draw[->] (0.3,-0.3) \botlabel{j} -- (-0.3,0.3) \toplabel{-i};
\projcr{0,0};
\draw[gcolor,thick] (0.75,-0.3) \botlabel{V}-- (0.75,0.3);
\end{tikzpicture}
+
\begin{tikzpicture}[H,centerzero]
\draw[->] (-0.3,-0.3) \botlabel{i} -- (-0.3,0.3);
\draw[->] (0.3,-0.3) \botlabel{j} -- (0.3,0.3);
\pinpin{0.3,0.05}{-0.3,-.15}{-1.1,-.15}{\frac{1}{x-y}};
\token{-0.3,0.05};
\token{0.3,-0.15};
\draw[gcolor,thick] (0.75,-0.3) \botlabel{V}-- (0.75,0.3);
\end{tikzpicture}
\]
is 0 in $\Hom_{\catR}(P_i P_j V, P_{-i}P_{-j}V)$.  The proof proceeds as before, using also the first relation in \cref{affsergeev}.
One first shows that
$\big( (L-R) + (b(j)-b(-i)) \big)^N
= \big( (L-b(-i)) - (R-b(j)) \big)^N
= 0$.
\end{proof}

Recall that the {\em Demazure operator}
$\partial_{xy}:\kk[x,y] \rightarrow \kk[x,y]$ is the linear map defined by
\begin{equation}\label{demazure}
\partial_{xy} f(x,y) := \frac{f(x,y)-f(y,x)}{x-y}.
\end{equation}
In fact, this formula defines
a linear map
$\partial_{xy}: \kk\llbracket x-b(i),y-b(i)\rrbracket \rightarrow
\kk\llbracket x-b(i),y-b(i)\rrbracket$
for any $i \in \kk$.
Note that $\partial_{xy} = -\partial_{yx}$.

\begin{lem}\label{grid}
For any $f(x,y) \in \kk[x,y]$ and $i,j \in \kk$, we have that
\begin{align} \label{naha}
\begin{tikzpicture}[H,centerzero]
\draw[->] (-0.4,-0.4) \botlabel{i} -- (0.4,0.4) \toplabel{j};
\draw[->] (0.4,-0.4) \botlabel{j} -- (-0.4,0.4) \toplabel{i};
\pinpin{0.25,0.25}{-0.25,0.25}{-1.2,.25}{f(x,y)};
\projcr{0,0};
\end{tikzpicture}\ -\      
\begin{tikzpicture}[H,centerzero]
\draw[->] (-0.4,-0.4) \botlabel{i} -- (0.4,0.4) \toplabel{j};
\draw[->] (0.4,-0.4) \botlabel{j} -- (-0.4,0.4) \toplabel{i};
\pinpin{-0.25,-0.25}{0.25,-0.25}{1.2,-.25}{f(y,x)};
\projcr{0,0};
\end{tikzpicture}
&=
\begin{tikzpicture}[H,centerzero]
\draw[->] (-0.3,-0.4) \botlabel{i} -- (-0.3,0.4);
\draw[->] (0.3,-0.4) \botlabel{j} -- (0.3,0.4);
\pinpin{0.3,0}{-0.3,0}{-1.2,0}{\partial_{xy} f(x,y)};
\end{tikzpicture}
- \delta_{i=j=0}\
\begin{tikzpicture}[H,centerzero]
\draw[->] (-0.3,-0.4) \botlabel{0} -- (-0.3,0.4);
\draw[->] (0.3,-0.4) \botlabel{0} -- (0.3,0.4);
\token{-0.3,0.1};
\token{0.3,-0.15};
\pinpin{0.3,0.1}{-0.3,-0.15}{-1.4,-.15}{\partial_{xy} f(-x,y)};
\end{tikzpicture}\ .    
\end{align}
When $i = j$, this identity holds more generally for any
$f(x,y)\in \kk \llbracket x-b(i), y-b(i) \rrbracket$.
\end{lem}

\begin{proof}
When $i \neq j$, this follows from \cref{noodle}.
Now suppose that $i=j$. We may assume that $f(x,y) = x^a y^b$
for $a,b \geq 0$, and proceed by induction on the degree $a+b$.
The degree 0 case is trivial. For the induction step,
we assume that \cref{naha} holds
for the monomial $f(x,y) = x^a y^b$
and prove it for the monomials $x^{a+1} y^b$
and $x^a y^{b+1}$. Both cases are similar, so we just explain the argument in the first case.
Using the induction hypothesis and \cref{thai2}, we have that
\begin{align*}
\begin{tikzpicture}[H,centerzero]
\draw[->] (-0.4,-0.4) \botlabel{i} -- (0.4,0.4) \toplabel{i};
\draw[->] (0.4,-0.4) \botlabel{i} -- (-0.4,0.4) \toplabel{i};
\pinpin{0.25,0.25}{-0.25,0.25}{-1.2,.25}{x^{a+1} y^b};
\projcr{0,0};
\end{tikzpicture}\!
&=
\begin{tikzpicture}[H,centerzero]
\draw[->] (-0.4,-0.4) \botlabel{i} -- (0.4,0.4) \toplabel{i};
\draw[->] (0.4,-0.4) \botlabel{i} -- (-0.4,0.4) \toplabel{i};
\pinpin{0.25,-0.25}{-0.25,-0.25}{-.95,-.25}{y^a x^b};
\singdot{-.25,.25};
\projcr{0,0};
\end{tikzpicture}
+
\begin{tikzpicture}[H,centerzero]
\draw[->] (-0.3,-0.4) \botlabel{i} -- (-0.3,0.4);
\draw[->] (0.3,-0.4) \botlabel{i} -- (0.3,0.4);
\singdot{-.3,.1};
\pinpin{0.3,-.2}{-0.3,-.2}{-1.2,-.2}{\partial_{xy}(x^a y^b)};
\end{tikzpicture}
-
\delta_{i=0}\ \begin{tikzpicture}[H,centerzero]
\draw[->] (-0.3,-0.4) \botlabel{0} -- (-0.3,0.4);
\draw[->] (0.3,-0.4) \botlabel{0} -- (0.3,0.4);
\singdot{-.3,.2};
\token{-0.3,-.025};
\token{0.3,-0.25};
\pinpin{0.3,-0.025}{-0.3,-0.25}{-1.6,-.25}{(-1)^a\partial_{xy}(x^a y^b)};
\end{tikzpicture}
\\
&=
\begin{tikzpicture}[H,centerzero]
\draw[->] (-0.4,-0.4) \botlabel{i} -- (0.4,0.4) \toplabel{i};
\draw[->] (0.4,-0.4) \botlabel{i} -- (-0.4,0.4) \toplabel{i};
\pinpin{0.25,-0.25}{-0.25,-0.25}{-1.1,-.25}{y^{a+1} x^b};
\projcr{0,0};
\end{tikzpicture}\!\!
+
\begin{tikzpicture}[H,centerzero]
\draw[->] (-0.3,-0.4) \botlabel{i} -- (-0.3,0.4);
\draw[->] (0.3,-0.4) \botlabel{i} -- (0.3,0.4);
\pinpin{0.3,0}{-0.3,0}{-1.8,0}{y^a x^b+\partial_{xy}(x^a y^b)x};
\end{tikzpicture}\!
+\delta_{i=0}\
\begin{tikzpicture}[H,centerzero]
\draw[->] (-0.3,-0.4) \botlabel{0} -- (-0.3,0.4);
\draw[->] (0.3,-0.4) \botlabel{0} -- (0.3,0.4);
\token{-0.3,.1};
\token{0.3,0.1};
\pinpin{0.3,-.2}{-0.3,-0.2}{-1,-.2}{y^a x^b};
\end{tikzpicture}
\!+\delta_{i=0}\
\begin{tikzpicture}[H,centerzero]
\draw[->] (-0.3,-0.4) \botlabel{0} -- (-0.3,0.4);
\draw[->] (0.3,-0.4) \botlabel{0} -- (0.3,0.4);
\token{-0.3,.1};
\token{0.3,-0.2};
\pinpin{0.3,.1}{-0.3,-0.2}{-1.7,-.2}{(-1)^{a}\partial_{xy}(x^a y^b)x};
\end{tikzpicture}\\
&=
\begin{tikzpicture}[H,centerzero]
\draw[->] (-0.4,-0.4) \botlabel{i} -- (0.4,0.4) \toplabel{i};
\draw[->] (0.4,-0.4) \botlabel{i} -- (-0.4,0.4) \toplabel{i};
\pinpin{0.25,-0.25}{-0.25,-0.25}{-1.1,-.25}{y^{a+1} x^b};
\projcr{0,0};
\end{tikzpicture}\!
+
\begin{tikzpicture}[H,centerzero]
\draw[->] (-0.3,-0.4) \botlabel{i} -- (-0.3,0.4);
\draw[->] (0.3,-0.4) \botlabel{i} -- (0.3,0.4);
\pinpin{0.3,0}{-0.3,0}{-1.8,0}{y^a x^b+\partial_{xy}(x^a y^b)x};
\end{tikzpicture}\!
-\delta_{i=0}\
\begin{tikzpicture}[H,centerzero]
\draw[->] (-0.3,-0.4) \botlabel{0} -- (-0.3,0.4);
\draw[->] (0.3,-0.4) \botlabel{0} -- (0.3,0.4);
\token{-0.3,.1};
\token{0.3,-0.2};
\pinpin{0.3,.1}{-0.3,-0.2}{-2.4,-.2}{(-1)^{a+1}\left(y^a x^b+\partial_{xy}(x^a y^b)x\right)};
\end{tikzpicture}.
\end{align*}
It remains to observe that
$\partial_{xy}(x^{a+1} y^b) = \partial_{xy}(x^a y^b \cdot x) = y^a x^b+\partial_{xy}(x^a y^b)x$.
\end{proof}

In the remaining lemmas in this subsection, we restrict attention to the eigenfunctors $P_i$ and $Q_i$ for $i \in I$ (rather than all of $\kk$). It is sufficient to do this because $\kk = I \cup (-I)$, and
there are odd isomorphisms $P_i \cong P_{-i}$ and $Q_i \cong Q_{-i}$ defined by the Clifford tokens
\cref{cliffordisos}.

\begin{lem}\label{upwardsinvertibility}
For $i,j \in I$, we have
\begin{equation} \label{cold}
\begin{tikzpicture}[H,centerzero]
\draw[->] (-0.2,-0.6) \botlabel{i} \braidup (0.2,0) \braidup (-0.2,0.6) \toplabel{i};
\draw[->] (0.2,-0.6) \botlabel{j} \braidup (-0.2,0) \braidup (0.2,0.6) \toplabel{j};
\projcr{0,-0.3};
\projcr{0,0.3};
\strand{0.35,0}{i};
\strand{-0.35,0}{j};
\end{tikzpicture}
=
\begin{dcases}
\begin{tikzpicture}[H,centerzero]
\draw[->] (-0.2,-0.3) \botlabel{i} -- (-0.2,0.3);
\draw[->] (0.2,-0.3) \botlabel{j} -- (0.2,0.3);
\pinpin{-0.2,0}{0.2,0}{1.9,0}{1 - \frac{1}{(x-y)^{2}} - \frac{1}{(x+y)^{2}}};
\node at (0,-.34) {$\phantom{X}$};
\end{tikzpicture}
&\text{if } i \ne \pm j\\
\begin{tikzpicture}[H,centerzero]
\draw[->] (-0.2,-0.3) \botlabel{i} -- (-0.2,0.3);
\node at (0,-.34) {$\phantom{X}$};
\draw[->] (0.2,-0.3) \botlabel{i} -- (0.2,0.3);
\pinpin{-0.2,0}{0.2,0}{1.3,0}{1-\frac{1}{(x+y)^{2}}};
\end{tikzpicture}&\text{if $i=j \neq 0$}\\
\begin{tikzpicture}[H,centerzero]
\draw[->] (-0.2,-0.3) \botlabel{0} -- (-0.2,0.3);
\draw[->] (0.2,-0.3) \botlabel{0} -- (0.2,0.3);
\end{tikzpicture}&\text{if $i=j=0$.}       
\end{dcases}
\end{equation}
\end{lem}

\begin{proof}
If $i = j \neq 0$, we have that
\[
\begin{tikzpicture}[H,centerzero]
\draw[->] (-0.2,-0.6) \botlabel{i} \braidup (0.2,0) \braidup (-0.2,0.6) \toplabel{i};
\draw[->] (0.2,-0.6) \botlabel{i} \braidup (-0.2,0) \braidup (0.2,0.6) \toplabel{i};
\projcr{0,-0.3};
\projcr{0,0.3};
\strand{0.3,0}{i};
\strand{-0.3,0}{i};
\end{tikzpicture}
=
\begin{tikzpicture}[H,centerzero]
\draw[->] (-0.2,-0.6) \botlabel{i} \braidup (0.2,0) \braidup (-0.2,0.6) \toplabel{i};
\draw[->] (0.2,-0.6) \botlabel{i} \braidup (-0.2,0) \braidup (0.2,0.6) \toplabel{i};
\notch[40]{-0.12,0.4};
\notch[-40]{0.12,0.4};
\notch[-35]{-0.16,-0.45};
\notch[35]{0.16,-0.45};
\end{tikzpicture}
-
\begin{tikzpicture}[H,centerzero]
\draw[->] (-0.2,-0.6) \botlabel{i} \braidup (0.2,0) \braidup (-0.2,0.6) \toplabel{i};
\draw[->] (0.2,-0.6) \botlabel{i} \braidup (-0.2,0) \braidup (0.2,0.6) \toplabel{i};
\projcr{0,-0.3};
\projcr{0,0.3};
\strand{0.4,0}{-i};
\strand{-0.4,0}{-i};
\end{tikzpicture}\
\overset{\cref{sergeev1}}
{\underset{\cref{noodle}}{=}}\
\begin{tikzpicture}[H,centerzero]
\draw[->] (-0.2,-0.6) \botlabel{i} -- (-0.2,0.6);
\draw[->] (0.2,-0.6) \botlabel{i} -- (0.2,0.6);
\end{tikzpicture}
-
\begin{tikzpicture}[H,centerzero]
\draw[->] (-0.2,-0.6) \botlabel{i} -- (-0.2,0.6);
\draw[->] (0.2,-0.6) \botlabel{i} -- (0.2,0.6);
\pinpin{0.2,0.43}{-0.2,0.13}{-.9,.13}{\frac{1}{y-x}};
\token{-0.2,0.43};
\token{0.2,0.13};
\pinpin{0.2,-0.18}{-0.2,-0.48}{-.9,-.48}{\frac{1}{y-x}};
\token{-0.2,-0.18};
\token{0.2,-0.48};
\end{tikzpicture}
\overset{\cref{sergeev2}}{\underset{\cref{affsergeev}}{=}}\
\begin{tikzpicture}[H,centerzero]
\draw[->] (-0.2,-0.6) \botlabel{i} -- (-0.2,0.6);
\draw[->] (0.2,-0.6) \botlabel{i} -- (0.2,0.6);
\pinpin{0.2,0}{-0.2,0}{-1.3,0}{1-\frac{1}{(x+y)^{2}}};
\end{tikzpicture}
,
\]
where the first equality follows from \cref{electric}.
If $i \ne j$, we also have that $i \neq -j$ by the definition of $I$, and the same argument gives that
\[
\begin{tikzpicture}[H,centerzero]
\draw[->] (-0.2,-0.6) \botlabel{i} \braidup (0.2,0) \braidup (-0.2,0.6) \toplabel{i};
\draw[->] (0.2,-0.6) \botlabel{j} \braidup (-0.2,0) \braidup (0.2,0.6) \toplabel{j};
\projcr{0,-0.3};
\projcr{0,0.3};
\strand{0.35,0}{i};
\strand{-0.35,0}{j};
\end{tikzpicture} =
\begin{tikzpicture}[H,centerzero]
\draw[->] (-0.2,-0.6) \botlabel{i} \braidup (0.2,0) \braidup (-0.2,0.6) \toplabel{i};
\draw[->] (0.2,-0.6) \botlabel{j} \braidup (-0.2,0) \braidup (0.2,0.6) \toplabel{j};
\notch[40]{-0.12,0.4};
\notch[-40]{0.12,0.4};
\notch[-35]{-0.16,-0.45};
\notch[35]{0.16,-0.45};
\end{tikzpicture}
-
\begin{tikzpicture}[H,centerzero]
\draw[->] (-0.2,-0.6) \botlabel{i} \braidup (0.2,0) \braidup (-0.2,0.6) \toplabel{i};
\draw[->] (0.2,-0.6) \botlabel{j} \braidup (-0.2,0) \braidup (0.2,0.6) \toplabel{j};
\notch[40]{-0.12,0.4};
\notch[-40]{0.12,0.4};
\notch[-35]{-0.16,0.15};
\notch[35]{0.16,0.15};
\notch[35]{-0.16,-0.15};
\notch[-35]{0.16,-0.15};
\notch[-35]{-0.16,-0.45};
\notch[35]{0.16,-0.45};
\strand{0.35,0}{j};
\strand{-0.35,0}{i};
\end{tikzpicture}
-
\begin{tikzpicture}[H,centerzero]
\draw[->] (-0.2,-0.6) \botlabel{i} \braidup (0.2,0) \braidup (-0.2,0.6) \toplabel{i};
\draw[->] (0.2,-0.6) \botlabel{j} \braidup (-0.2,0) \braidup (0.2,0.6) \toplabel{j};
\notch[40]{-0.12,0.4};
\notch[-40]{0.12,0.4};
\notch[-35]{-0.16,0.15};
\notch[35]{0.16,0.15};
\notch[35]{-0.16,-0.15};
\notch[-35]{0.16,-0.15};
\notch[-35]{-0.16,-0.45};
\notch[35]{0.16,-0.45};
\strand{0.48,0}{-j};
\strand{-0.48,0}{-i};
\end{tikzpicture}=\
\begin{tikzpicture}[H,centerzero]
\draw[->] (-0.2,-0.6) \botlabel{i} -- (-0.2,0.6);
\draw[->] (0.2,-0.6) \botlabel{j} -- (0.2,0.6);
\pinpin{0.2,0}{-0.2,0}{-1.9,0}{1-\frac{1}{(x-y)^{2}}-\frac{1}{(x+y)^{2}}};
\end{tikzpicture}\ .
\]
The remaining case $i=j=0$ is similar.
\end{proof}

We point out that the
power series
\begin{equation}\label{pxyagain}
p(x,y) = 1 - \frac{1}{(x-y)^{2}} - \frac{1}{(x+y)^{2}}
\in \kk\llbracket x-b(i),y-b(j) \rrbracket
\end{equation}
occurring in the $i \neq \pm j$ case of \cref{cold}
is the rational function
seen before in \cref{divergent}.
It also appears in the next lemma.
Note for this that $p(x,y)-p(z,y)$ is divisible by $x-z$,
so we have that
$\frac{p(x,y)-p(z,y)}{x-z}\in\kk\llbracket x-b(i),y-b(j),z-b(i)\rrbracket$.

\begin{lem} \label{projbraid}
For $i,j,k \in I$, we have that
\begin{multline}
\begin{tikzpicture}[H,centerzero,scale=.9]
\draw[->] (-0.5,-1) \botlabel{i} \braidup (0.5,1) \toplabel{i};
\draw[->] (0,-1) \botlabel{j} \braidup (-0.5,0) \braidup (0,1) \toplabel{j};
\draw[->] (0.5,-1) \botlabel{k} \braidup (-0.5,1) \toplabel{k};
\projcr{-0.32,0.42};
\projcr{0,0};
\projcr{-0.32,-0.42};
\node at (-0.07,-0.28) {\strandlabel{i}};
\node at (-0.05,0.31) {\strandlabel{k}};
\node at (-0.63,0) {\strandlabel{j}};
\end{tikzpicture}
-
\begin{tikzpicture}[H,centerzero,scale=.9]
\draw[->] (-0.5,-1) \botlabel{i} \braidup (0.5,1) \toplabel{i};
\draw[->] (0,-1) \botlabel{j} \braidup (0.5,0) \braidup (0,1) \toplabel{j};
\draw[->] (0.5,-1) \botlabel{k} \braidup (-0.5,1) \toplabel{k};
\projcr{0.33,0.42};
\projcr{0,0};
\projcr{0.33,-0.42};
\node at (0.07,-0.29) {\strandlabel{k}};
\node at (0.05,0.3) {\strandlabel{i}};
\node at (0.63,0) {\strandlabel{j}};
\end{tikzpicture}
= \delta_{i = k \ne j}
\left(
\begin{tikzpicture}[H,centerzero,scale=.9]
\draw[->] (-0.5,-1) \botlabel{i} -- (-0.5,1);
\draw[->] (0,-1) \botlabel{j} -- (0,1);
\draw[->] (0.5,-1) \botlabel{i} -- (0.5,1);
\pinpinpin{-0.5,0}{0,0}{0.5,0}{1.9,0}{\frac{p(x,y)-p(z,y)}{x-z}};
\end{tikzpicture}
- \delta_{i=0}
\begin{tikzpicture}[H,centerzero,scale=.9]
\draw[->] (-0.5,-1) \botlabel{0} -- (-0.5,1);
\draw[->] (0,-1) \botlabel{j} -- (0,1);
\draw[->] (0.5,-1) \botlabel{0} -- (0.5,1);
\pinpinpin{-0.5,-0.2}{0,0}{0.5,0.2}{1.9,0.2}{\frac{p(x,y)-p(z,y)}{x-z}};
\token{-0.5,0.2};
\token{0.5,-0.2};
\end{tikzpicture}
\right)
\\
- \delta_{i=j=k\neq 0}
\begin{tikzpicture}[H,centerzero,scale=.9]
\draw[->] (-0.5,-1) \botlabel{i} -- (-0.5,1);
\draw[->] (0,-1) \botlabel{i} -- (0,1);
\draw[->] (0.5,-1) \botlabel{i} -- (0.5,1);
\pinpinpin{-0.5,0}{0,0}{0.5,0}{2.5,0}{\frac{1}{x-z}\left(\frac{1}{(x+y)^2}-\frac{1}{(y+z)^2}\right)};
\end{tikzpicture}
\ .
\end{multline}
\end{lem}

\begin{proof}
Recall that $I \cap (-I) = \{0\}$.
Using this, it follows from \cref{tassie,electric} for $i,j,k \in I$
that
\begin{multline*}
\begin{tikzpicture}[H,centerzero,scale=.9]
\draw[->] (-0.5,-1) \botlabel{i} \braidup (0.5,1) \toplabel{i};
\draw[->] (0,-1) \botlabel{j} \braidup (-0.5,0) \braidup (0,1) \toplabel{j};
\draw[->] (0.5,-1) \botlabel{k} \braidup (-0.5,1) \toplabel{k};
\notch{-0.48,-0.8};
\notch{-0.04,-0.8};
\notch{0.48,-0.8};
\notch{-0.48,0.8};
\notch{-0.04,0.8};
\notch{0.48,0.8};
\end{tikzpicture}
=
\begin{tikzpicture}[H,centerzero,scale=.9]
\draw[->] (-0.5,-1) \botlabel{i} \braidup (0.5,1) \toplabel{i};
\draw[->] (0,-1) \botlabel{j} \braidup (-0.5,0) \braidup (0,1) \toplabel{j};
\draw[->] (0.5,-1) \botlabel{k} \braidup (-0.5,1) \toplabel{k};
\projcr{-0.32,0.42};
\projcr{0,0};
\projcr{-0.32,-0.42};
\node at (-0.07,-0.28) {\strandlabel{i}};
\node at (-0.05,0.3) {\strandlabel{k}};
\node at (-0.63,0) {\strandlabel{j}};
\end{tikzpicture}
+ \delta_{i=k \ne j}
\begin{tikzpicture}[H,centerzero,scale=.9]
\draw[->] (-0.5,-1) \botlabel{i} \braidup (0.5,1) \toplabel{i};
\draw[->] (0,-1) \botlabel{j} \braidup (-0.5,0) \braidup (0,1) \toplabel{j};
\draw[->] (0.5,-1) \botlabel{i} \braidup (-0.5,1) \toplabel{i};
\projcr{-0.32,0.42};
\projcr{0,0};
\projcr{-0.32,-0.42};
\node at (-0.07,-0.28) {\strandlabel{j}};
\node at (-0.05,0.3) {\strandlabel{j}};
\node at (-0.63,0) {\strandlabel{i}};
\end{tikzpicture}
+ \delta_{i=k \ne -j}
\begin{tikzpicture}[H,centerzero,scale=.9]
\draw[->] (-0.5,-1) \botlabel{i} \braidup (0.5,1) \toplabel{i};
\draw[->] (0,-1) \botlabel{j} \braidup (-0.5,0) \braidup (0,1) \toplabel{j};
\draw[->] (0.5,-1) \botlabel{i} \braidup (-0.5,1) \toplabel{i};
\projcr{-0.32,0.42};
\projcr{0,0};
\projcr{-0.32,-0.42};
\node at (-0.05,-0.28) {\strandlabel{-j}};
\node at (-0.04,0.29) {\strandlabel{-j}};
\node at (-0.7,0) {\strandlabel{-i}};
\end{tikzpicture}
+ \delta_{i=k=0 \ne j}
\left(
\begin{tikzpicture}[H,centerzero,scale=.9]
\draw[->] (-0.5,-1) \botlabel{0} \braidup (0.5,1) \toplabel{0};
\draw[->] (0,-1) \botlabel{j} \braidup (-0.5,0) \braidup (0,1) \toplabel{j};
\draw[->] (0.5,-1) \botlabel{0} \braidup (-0.5,1) \toplabel{0};
\projcr{-0.32,0.42};
\projcr{0,0};
\projcr{-0.32,-0.42};
\node at (-0.05,-0.29) {\strandlabel{-j}};
\node at (-0.05,0.3) {\strandlabel{j}};
\node at (-0.63,0) {\strandlabel{0}};
\end{tikzpicture}+
\begin{tikzpicture}[H,centerzero,scale=.9]
\draw[->] (-0.5,-1) \botlabel{0} \braidup (0.5,1) \toplabel{0};
\draw[->] (0,-1) \botlabel{j} \braidup (-0.5,0) \braidup (0,1) \toplabel{j};
\draw[->] (0.5,-1) \botlabel{0} \braidup (-0.5,1) \toplabel{0};
\projcr{-0.32,0.42};
\projcr{0,0};
\projcr{-0.32,-0.42};
\node at (-0.05,-0.28) {\strandlabel{j}};
\node at (-0.05,0.28) {\strandlabel{-j}};
\node at (-0.63,0) {\strandlabel{0}};
\end{tikzpicture}\right)
\\\!\!\!\overset{\cref{noodle}}{=}
\begin{tikzpicture}[H,centerzero,scale=.9]
\draw[->] (-0.5,-1) \botlabel{i} \braidup (0.5,1) \toplabel{i};
\draw[->] (0,-1) \botlabel{j} \braidup (-0.5,0) \braidup (0,1) \toplabel{j};
\draw[->] (0.5,-1) \botlabel{k} \braidup (-0.5,1) \toplabel{k};
\projcr{-0.32,0.42};
\projcr{0,0};
\projcr{-0.32,-0.42};
\node at (-0.07,-0.28) {\strandlabel{i}};
\node at (-0.05,0.3) {\strandlabel{k}};
\node at (-0.63,0) {\strandlabel{j}};
\end{tikzpicture}\!\!
+
\begin{tikzpicture}[H,centerzero,scale=.9]
\draw[->] (-0.5,-1) \botlabel{i} -- (-0.5,1);
\draw[->] (0,-1) \botlabel{j} -- (0,1);
\draw[->] (0.5,-1) \botlabel{i} -- (0.5,1);
\pinpinpin{-0.5,0}{0,0}{0.5,0}{2.8,0}{\frac{\delta_{i=k \ne j}}{(x-y)^2(y-z)}-\frac{\delta_{i=k \ne -j}}{(x+y)^2(y+z)}};
\end{tikzpicture}
- \delta_{i=k=0 \ne j}
\begin{tikzpicture}[H,centerzero,scale=.9]
\draw[->] (-0.5,-1) \botlabel{0} -- (-0.5,1);
\draw[->] (0,-1) \botlabel{j} -- (0,1);
\draw[->] (0.5,-1) \botlabel{0} -- (0.5,1);
\pinpinpin{-0.5,-.4}{0,-.2}{0.5,0}{2.8,0}{\frac{1}{(x-y)^2(y-z)}-\frac{1}{(x+y)^2(y+z)}};
\token{-0.5,0};
\token{0.5,-0.4};
\end{tikzpicture}\ .
\end{multline*}
Similarly,
\begin{multline*}
\begin{tikzpicture}[H,centerzero,scale=.9]
\draw[->] (-0.5,-1) \botlabel{i} \braidup (0.5,1) \toplabel{i};
\draw[->] (0,-1) \botlabel{j} \braidup (0.5,0) \braidup (0,1) \toplabel{j};
\draw[->] (0.5,-1) \botlabel{k} \braidup (-0.5,1) \toplabel{k};
\notch{-0.48,-0.8};
\notch{0.04,-0.8};
\notch{0.48,-0.8};
\notch{-0.48,0.8};
\notch{0.04,0.8};
\notch{0.48,0.8};
\end{tikzpicture}
=\!\!\!\begin{tikzpicture}[H,centerzero,scale=.9]
\draw[->] (-0.5,-1) \botlabel{i} \braidup (0.5,1) \toplabel{i};
\draw[->] (0,-1) \botlabel{j} \braidup (0.5,0) \braidup (0,1) \toplabel{j};
\draw[->] (0.5,-1) \botlabel{k} \braidup (-0.5,1) \toplabel{k};
\projcr{0.33,0.42};
\projcr{0,0};
\projcr{0.33,-0.42};
\node at (0.07,-0.28) {\strandlabel{k}};
\node at (0.06,0.28) {\strandlabel{i}};
\node at (0.63,0) {\strandlabel{j}};
\end{tikzpicture}\!
+
\begin{tikzpicture}[H,centerzero,scale=.9]
\draw[->] (-0.5,-1) \botlabel{i} -- (-0.5,1);
\draw[->] (0,-1) \botlabel{j} -- (0,1);
\draw[->] (0.5,-1) \botlabel{i} -- (0.5,1);
\pinpinpin{-0.5,0}{0,0}{0.5,0}{2.8,0}{\frac{ \delta_{i=k \ne j}}{(x-y)(y-z)^2}+\frac{ \delta_{i=k \ne -j}}{(x+y)(y+z)^2}};
\end{tikzpicture}
- \delta_{i=k=0 \ne j}
\begin{tikzpicture}[H,centerzero,scale=.9]
\draw[->] (-0.5,-1) \botlabel{0} -- (-0.5,1);
\draw[->] (0.5,-1) \botlabel{0} -- (0.5,1);
\pinpinpin{-0.5,-.4}{0,-.2}{0.5,0}{2.8,0}{ \frac{1}{(x-y)(y-z)^2}+\frac{1}{(x+y)(y+z)^2}};
\token{-0.5,0};
\token{0.5,-0.4};
\end{tikzpicture}\ .
\end{multline*}
The result now follows from the second relation in \cref{sergeev1}, using also the elementary identities
\begin{gather*}
\frac{1}{(x-y)(y-z)^2}+\frac{1}{(x+y)(y+z)^2}
-\frac{1}{(x-y)^2(y-z)}+\frac{1}{(x+y)^2(y+z)}
=
\frac{p(x,y)-p(z,y)}{x-z},\\
\frac{1}{(x+y)^2(y+z)}
+\frac{1}{(x+y)(y+z)^2}
=-\frac{1}{x-z}\left(\frac{1}{(x+y)^2}-\frac{1}{(y+z)^2}\right).
\qedhere
\end{gather*}
\end{proof}

In a similar way to \cref{tassie},
we define supernatural transformations represented by the rightward and leftward crossings
$\begin{tikzpicture}[H,centerzero,scale=.9]
\draw[->] (-0.3,-0.3) \botlabel{j} -- (0.3,0.3) \toplabel{i'};
\draw[<-] (0.3,-0.3) \botlabel{i} -- (-0.3,0.3) \toplabel{j'};
\projcr{0,0};
\end{tikzpicture}$ and
$\begin{tikzpicture}[H,centerzero,scale=.9]
\draw[<-] (-0.3,-0.3) \botlabel{j} -- (0.3,0.3) \toplabel{i'};
\draw[->] (0.3,-0.3) \botlabel{i} -- (-0.3,0.3) \toplabel{j'};
\projcr{0,0};
\end{tikzpicture}$.
These can also be obtained by attaching appropriate cups and caps
to rotate the upward crossings from before, and in this way analogous results to \cref{tie,thai1,thai2,thai,electric,power,grid} can be deduced for the other sorts of crossing.
For example, from \cref{noodle}, we get
\begin{align}\label{doggydreams1}
\begin{tikzpicture}[H,centerzero,scale=1.4]
\draw[->] (-0.3,-0.3) \botlabel{j} -- (0.3,0.3) \toplabel{i};
\draw[<-] (0.3,-0.3) \botlabel{j} -- (-0.3,0.3) \toplabel{i};
\projcr{0,0};
\end{tikzpicture}
&=
\begin{tikzpicture}[centerzero,H,scale=1.4]
\draw[->] (-.2,.3) \toplabel{i} to[out=down,in=down,looseness=2.2] (.2,.3);
\draw[->] (-.2,-.3) \botlabel{j} to[out=up,in=up,looseness=2.2] (.2,-.3);
\pinpin{-.095,-.07}{.095,.07}{.7,.07}{\frac{1}{y-x}};
\end{tikzpicture}\qquad\text{if $i \neq j$},&
\begin{tikzpicture}[H,centerzero,scale=1.4]
\draw[->] (-0.3,-0.3) \botlabel{j} -- (0.3,0.3) \toplabel{-i};
\draw[<-] (0.3,-0.3) \botlabel{-j} -- (-0.3,0.3) \toplabel{i};
\projcr{0,0};
\end{tikzpicture}
&=  \begin{tikzpicture}[centerzero,H,scale=1.4]
\draw[->] (-.2,.3) \toplabel{i} to[out=down,in=down,looseness=2.2] (.2,.3);
\draw[->] (-.2,-.3) \botlabel{j} to[out=up,in=up,looseness=2.2] (.2,-.3);
\pinpin{-.095,-.07}{.095,.07}{.7,.07}{\frac{1}{y-x}};
\token{-.18,-.19};\token{.19,.17};
\end{tikzpicture}\qquad\text{if $i \neq -j$},\\\label{doggydreams2}
\begin{tikzpicture}[H,centerzero,scale=1.4]
\draw[<-] (-0.3,-0.3) \botlabel{i} -- (0.3,0.3) \toplabel{j};
\draw[->] (0.3,-0.3) \botlabel{i} -- (-0.3,0.3) \toplabel{j};
\projcr{0,0};
\end{tikzpicture} &=
\begin{tikzpicture}[centerzero,H,scale=1.4]
\draw[<-] (-.2,.3)  to[out=down,in=down,looseness=2.2] (.2,.3)\toplabel{j};
\draw[<-] (-.2,-.3)to[out=up,in=up,looseness=2.2] (.2,-.3) \botlabel{i};
\pinpin{.095,-.07}{-.095,.07}{-.7,.07}{\frac{1}{y-x}};
\end{tikzpicture}
\qquad\text{if } i \ne j,
&
\begin{tikzpicture}[H,centerzero,scale=1.4]
\draw[<-] (-0.3,-0.3) \botlabel{-i} -- (0.3,0.3) \toplabel{j};
\draw[->] (0.3,-0.3) \botlabel{i} -- (-0.3,0.3) \toplabel{-j};
\projcr{0,0};
\end{tikzpicture}&=
\begin{tikzpicture}[centerzero,H,scale=1.4]
\draw[<-] (-.2,.3)  to[out=down,in=down,looseness=2.2] (.2,.3)\toplabel{j};
\draw[<-] (-.2,-.3)to[out=up,in=up,looseness=2.2] (.2,-.3) \botlabel{i};
\pinpin{.095,-.07}{-.095,.07}{-.7,.07}{\frac{1}{y-x}};
\token{-.19,.17};\token{.18,-.19};
\end{tikzpicture}
\qquad\text{if $i \neq -j$.}
\end{align}
There is one more useful lemma about sideways crossings, which needs to be proved from scratch.

\begin{lem}\label{sidewaysinvertibility}
For $i, j \in I$, we have that
\begin{align}\label{tendays}
\begin{tikzpicture}[H,anchorbase,scale=1]
\draw[<-] (-0.25,-0.6)\botlabel{i} \braidup (0.25,0) \braidup (-0.25,0.6) \toplabel{i};
\draw[->] (0.25,-0.6)\botlabel{j} \braidup (-0.25,0) \braidup (0.25,0.6)\toplabel{j};
\projcr{0,.3};\projcr{0,-.3};
\node at (.35,0) {\strandlabel{i}};
\node at (-.35,0) {\strandlabel{j}};
\end{tikzpicture}\!
&=
\begin{tikzpicture}[H,anchorbase,scale=1]
\draw[<-] (-0.15,-0.6) to (-0.15,0.6)\toplabel{j};
\draw[->] (0.15,-0.6)\botlabel{i} to (0.15,0.6);
\end{tikzpicture}\!
+\delta_{i=j}\!\! \left[\!
\begin{tikzpicture}[H,anchorbase,scale=1]
\draw[->] (-0.2,0.6)\toplabel{i} to (-0.2,0.3) arc(180:360:0.2) to (0.2,0.6);
\draw[<-] (-0.2,-0.6) to (-0.2,-0.3) arc(180:0:0.2) to (0.2,-0.6)\botlabel{i};
\rightbubgen{1.1,0};
\circled{0.2,0.35}{u};
\circled{0.2,-0.35}{u};
\end{tikzpicture}\!\!-\!\!\!\sum_{i \neq k \in \kk} \!\!
\begin{tikzpicture}[H,anchorbase,scale=1]
\draw[->] (-0.2,0.6)\toplabel{i} to (-0.2,0.3) arc(180:360:0.2) to (0.2,0.6);
\draw[<-] (-0.2,-0.6) to (-0.2,-0.3) arc(180:0:0.2) to (0.2,-0.6)\botlabel{i};
\draw[->] (1,0)++(-0.25,0) arc(0:-360:0.2);
\node at (.45,.32) {\strandlabel{k}};
\pinPinpin{.19,.25}{.35,0}{.19,-.25}{-1.15,0}{\!\frac{u^{-1}}{(x-z)(y-z)}\!};
\end{tikzpicture}\!\!\!\right]_{\!u:-1}\!\!\!\!\!\!\!\!\!-\delta_{i=-j}\!\left[\!
\begin{tikzpicture}[H,anchorbase,scale=1]
\draw[->] (-0.2,0.7)\toplabel{i} to (-0.2,0.3) arc(180:360:0.2) to (0.2,0.7);
\draw[<-] (-0.2,-0.7) to (-0.2,-0.3) arc(180:0:0.2) to (0.2,-0.7)\botlabel{-i};
\rightbubgen{1.1,0};
\token{0.2,0.53};
\token{0.2,-0.53};
\circled{0.2,0.3}{u};
\circled{0.2,-0.3}{u};
\end{tikzpicture}\!\!-\!\!\!\sum_{i \neq k \in \kk} \!\!
\begin{tikzpicture}[H,anchorbase,scale=1]
\draw[->] (-0.2,0.6) \toplabel{i}to (-0.2,0.3) arc(180:360:0.2) to (0.2,0.6);
\token{0.2,0.43};
\token{0.2,-0.43};
\draw[<-] (-0.2,-0.6) to (-0.2,-0.3) arc(180:0:0.2) to (0.2,-0.6)\botlabel{-i};
\draw[->] (1,0)++(-0.25,0) arc(0:-360:0.2);
\node at (.45,.32) {\strandlabel{k}};
\pinPinpin{.19,.25}{.35,0}{.19,-.25}{-1.15,0}{\!\frac{u^{-1}}{(x-z)(y-z)}\!};
\end{tikzpicture}\!\!\!
\right]_{\!u:-1}\!\!\!\!\!\!,\\  \label{ninedays}
\begin{tikzpicture}[H,anchorbase,scale=1]
\draw[->] (-0.25,-0.6)\botlabel{i} \braidup (0.25,0) \braidup (-0.25,0.6) \toplabel{i};
\draw[<-] (0.25,-0.6)\botlabel{j} \braidup (-0.25,0) \braidup (0.25,0.6)\toplabel{j};
\projcr{0,.3};\projcr{0,-.3};
\node at (.35,0) {\strandlabel{i}};
\node at (-.35,0) {\strandlabel{j}};
\end{tikzpicture}\!
&= \begin{tikzpicture}[H,anchorbase,scale=1]
\draw[->] (-0.15,-0.6)\botlabel{i} to (-0.15,0.6);
\draw[<-] (0.15,-0.6) to (0.15,0.6)\toplabel{j};
\end{tikzpicture}\!+
\delta_{i=j}\!\! \left[
\begin{tikzpicture}[H,anchorbase,scale=1]
\draw[<-] (-0.2,0.6) to (-0.2,0.3) arc(180:360:0.2) to (0.2,0.6)\toplabel{i};
\draw[->] (-0.2,-0.6)\botlabel{i} to (-0.2,-0.3) arc(180:0:0.2) to (0.2,-0.6);
\leftbubgen{-0.5,0};
\circled{-0.2,0.35}{u};
\circled{-0.2,-0.35}{u};
\end{tikzpicture}
\!-\!\!\!\sum_{i \neq k \in \kk}\!\!\! \!
\begin{tikzpicture}[H,anchorbase,scale=1]
\draw[<-] (-0.2,0.6) to (-0.2,0.3) arc(180:360:0.2) to (0.2,0.6)\toplabel{i};
\draw[->] (-0.2,-0.6)\botlabel{i} to (-0.2,-0.3) arc(180:0:0.2) to (0.2,-0.6);
\draw[->] (-.5,0)++(-0.25,0) arc(-180:180:0.2);
\node at (-.45,.32) {\strandlabel{k}};
\pinPinpin{-.2,.3}{-.35,0}{-.2,-.3}{1.15,0}{\!\frac{u^{-1}}{(x-y)(x-z)}\!};
\end{tikzpicture}\right]_{\!u:-1}
\!\!\!\!\!\!\!\!\!-\delta_{i=-j}\!\left[
\begin{tikzpicture}[H,anchorbase,scale=1]
\draw[<-] (-0.2,0.7) to (-0.2,0.3) arc(180:360:0.2) to (0.2,0.7)\toplabel{-i};
\draw[->] (-0.2,-0.7)\botlabel{i} to (-0.2,-0.3) arc(180:0:0.2) to (0.2,-0.7);
\leftbubgen{-0.6,0};
\token{-0.2,0.53};
\token{-0.2,-0.53};
\circled{-0.2,0.3}{u};
\circled{-0.2,-0.3}{u};
\end{tikzpicture}\!\!
-\!\!\!\sum_{i \neq k \in \kk} \!\!\!\!
\begin{tikzpicture}[H,anchorbase,scale=1]
\draw[<-] (-0.2,0.6) to (-0.2,0.3) arc(180:360:0.2) to (0.2,0.6)\toplabel{-i};
\token{-0.2,0.43};
\token{-0.2,-0.43};
\draw[->] (-0.2,-0.6)\botlabel{i} to (-0.2,-0.3) arc(180:0:0.2) to (0.2,-0.6);
\draw[->] (-.5,0)++(-0.25,0) arc(-180:180:0.2);
\node at (-.45,.32) {\strandlabel{k}};
\pinPinpin{-.19,.25}{-.35,0}{-.19,-.25}{1.15,0}{\!\frac{u^{-1}}{(x-y)(x-z)}\!};
\end{tikzpicture}
\right]_{\!u:-1}\!\!\!\!\!\!.
\end{align}
\end{lem}

\begin{proof}
We first explain how to interpret the infinite sums appearing in the statement. On any finitely generated object $V \in \catR$, these make sense since $P_k V = Q_k V = 0$ for all but finitely many $k \in \kk$. On $V \in\catR$ that is not finitely generated, the supernatural transformations in the sums on the right-hand side should be interpreted by taking
the direct limit of their restrictions to all finitely generated subobjects of $V$.

To derive the equations, the second one follows from the first by applying the Chevalley involution.
For the first one, \cref{electric} implies that
\begin{align*}
\begin{tikzpicture}[anchorbase,H]
\draw[<-] (-0.23,-.6)\botlabel{i} to[out=90,in=-90] (0.23,0);
\draw[-] (0.23,0) to[out=90,in=-90] (-0.23,.6)\toplabel{i};
\draw[->] (-0.23,0) to[out=90,in=-90] (0.23,.6)\toplabel{j};
\draw[-] (0.23,-.6)\botlabel{j} to[out=90,in=-90] (-0.23,0);
\projcr{0,.3};
\projcr{0,-.3};
\node at (.33,0) {\strandlabel{i}};
\node at (-.33,0) {\strandlabel{j}};
\end{tikzpicture}
&=
\begin{tikzpicture}[anchorbase,H]
\draw[<-] (-0.23,-.6)\botlabel{i} to[out=90,in=-90] (0.23,0);
\draw[-] (0.23,0) to[out=90,in=-90] (-0.23,.6)\toplabel{i};
\draw[->] (-0.23,0) to[out=90,in=-90] (0.23,.6)\toplabel{j};
\draw[-] (0.23,-.6)\botlabel{j} to[out=90,in=-90] (-0.23,0);
\notch[45]{-.14,.4};
\notch[-45]{-.14,-.4};
\notch[-45]{.14,.4};
\notch[45]{.14,-.4};
\end{tikzpicture}
-
\delta_{i=j}
\sum_{i \neq k \in \kk}
\begin{tikzpicture}[anchorbase,H]
\draw[<-] (-0.23,-.6)\botlabel{i} to[out=90,in=-90] (0.23,0);
\draw[-] (0.23,0) to[out=90,in=-90] (-0.23,.6)\toplabel{i};
\draw[->] (-0.23,0) to[out=90,in=-90] (0.23,.6)\toplabel{i};
\draw[-] (0.23,-.6)\botlabel{i} to[out=90,in=-90] (-0.23,0);
\projcr{0,.3};
\projcr{0,-.3};
\node at (.35,0) {\strandlabel{k}};
\node at (-.35,0) {\strandlabel{k}};
\end{tikzpicture}
-
\delta_{i=-j}
\sum_{i \neq k \in \kk}
\begin{tikzpicture}[anchorbase,H]
\draw[<-] (-0.23,-.6)\botlabel{i} to[out=90,in=-90] (0.23,0);
\draw[-] (0.23,0) to[out=90,in=-90] (-0.23,.6)\toplabel{i};
\draw[->] (-0.23,0) to[out=90,in=-90] (0.23,.6)\toplabel{-i};
\draw[-] (0.23,-.6)\botlabel{-i} to[out=90,in=-90] (-0.23,0);
\projcr{0,.3};
\projcr{0,-.3};
\node at (.35,0) {\strandlabel{k}};
\node at (-.43,0) {\strandlabel{-k}};
\end{tikzpicture}\ .
\end{align*}
The two summations on the right-hand side can now be simplified using \cref{doggydreams2}, yielding the two summations in the formula we are trying to prove.
For remaining first term,
\cref{pizza1} implies that
\begin{align*}
\begin{tikzpicture}[anchorbase,H]
\draw[<-] (-0.23,-.6)\botlabel{i} to[out=90,in=-90] (0.23,0);
\draw[-] (0.23,0) to[out=90,in=-90] (-0.23,.6)\toplabel{i};
\draw[->] (-0.23,0) to[out=90,in=-90] (0.23,.6)\toplabel{j};
\draw[-] (0.23,-.6)\botlabel{j} to[out=90,in=-90] (-0.23,0);
\notch[45]{-.14,.4};
\notch[-45]{-.14,-.4};
\notch[-45]{.14,.4};
\notch[45]{.14,-.4};
\end{tikzpicture}=  \begin{tikzpicture}[H,anchorbase]
\draw[<-] (-0.2,-0.6) to (-0.2,0.6)\toplabel{i};
\draw[->] (0.2,-0.6)\botlabel{j} to (0.2,0.6);
\end{tikzpicture}
\ +
\left[
\begin{tikzpicture}[H,anchorbase,scale=1]
\draw[->] (-0.2,0.6)\toplabel{i} to (-0.2,0.3) arc(180:360:0.2) to (0.2,0.6)\toplabel{j};
\draw[<-] (-0.2,-0.6)\botlabel{i} to (-0.2,-0.3) arc(180:0:0.2) to (0.2,-0.6)\botlabel{j};
\rightbubgen{1.3,0};
\notch{-.2,.4};\notch{.2,.4};
\notch{-.2,-.4};\notch{.2,-.4};
\circled{0.2,0.2}{u};
\circled{0.2,-0.2}{u};
\end{tikzpicture}
-
\begin{tikzpicture}[H,anchorbase,scale=1]
\draw[->] (-0.2,0.7)\toplabel{i} to (-0.2,0.3) arc(180:360:0.2) to (0.2,0.7)\toplabel{j};
\draw[<-] (-0.2,-0.7)\botlabel{i} to (-0.2,-0.3) arc(180:0:0.2) to (0.2,-0.7)\botlabel{j};
\rightbubgen{1.3,0};
\circled{0.2,0.15}{u};
\circled{0.2,-0.15}{u};
\token{0.2,0.37};
\token{0.2,-0.37};
\notch{-.2,.5};\notch{.2,.5};
\notch{-.2,-.5};\notch{.2,-.5};
\end{tikzpicture}
\right]_{u:-1}.
\end{align*}
The first term in square brackets
is 0 unless $i=j$, and the second term is 0 unless $i=-j$.
\end{proof}

\subsection{Weight space decomposition}

For $V \in \catR$, we let $Z_V = Z_{V,\0}\oplus Z_{V,\1}$ be the supercenter of
the superalgebra $\End_\catR(V)$ as in \cref{supercenter}.
The {\em supercenter} $Z(\catR)$ of $\catR$ is the commutative superalgebra consisting of all supernatural transformations $z:\id_\catR \Rightarrow \id_\catR$.
An element $z \in Z(\catR)$ evaluates to $z_V \in Z_V$
for each $V \in \catR$. The image of the dotted bubble 
$\ \begin{tikzpicture}[H,centerzero]
\leftbub{0,0};
\multdot{0.24,0}{west}{n};
\end{tikzpicture}$
under the superfunctor $\Psi$ from \cref{Psi}
is an element of 
$Z(\catR)$.
We put all of these central elements together into the generating function
\begin{equation}\label{allzusammen}
\OO(u) = (\OO_V(u))_{V \in \catR} \in Z(\catR)\lround u^{-1}\rround
\end{equation}
where
\begin{equation} \label{heat}
\OO_V(u) :=
\begin{tikzpicture}[H,centerzero={0,-0.06}]
\leftbubgen{-0.4,0};
\draw[gcolor,thick] (0,-0.3) \botlabel{V} -- (0,0.3);
\end{tikzpicture}
\overset{\cref{eyes}}{=}
-\left(
\begin{tikzpicture}[H,centerzero={0,-0.06}]
\rightbubgen{-0.4,0};
\draw[gcolor,thick] (0,-0.3) \botlabel{V} -- (0,0.3);
\end{tikzpicture}
\right)^{-1} \in u^\kappa + u^{\kappa-2} Z_V \llbracket u^{-2} \rrbracket.
\end{equation}
Extending the notation of pins from before,
for a polynomial $f(x) = \sum_{r=0}^n z_r x^r \in Z_V[x]$, we let
\begin{equation}
\begin{tikzpicture}[H,centerzero]
\draw[->] (-0.2,-0.3) -- (-0.2,0.3);
\draw[gcolor,thick] (0.2,-0.3) \botlabel{V} -- (0.2,0.3);
\pin{-0.2,0}{-0.9,0}{f(x)};
\end{tikzpicture}
:=
\sum_{r=0}^n
\begin{tikzpicture}[H,centerzero]
\draw[->] (-0.2,-0.3) -- (-0.2,0.3);
\multdot{-0.2,0}{east}{r};
\draw[gcolor,thick] (0.2,-0.3) \botlabel{V} -- (0.2,0.3);
\node[draw,gcolor,thick,fill=white!90!green,inner sep=1.5pt,rounded corners] at (.2,0) {\objlabel{z_r}};
\end{tikzpicture}
\ ,\qquad
\begin{tikzpicture}[H,centerzero]
\draw[<-] (-0.2,-0.3) -- (-0.2,0.3);
\draw[gcolor,thick] (0.2,-0.3) \botlabel{V}-- (0.2,0.3);
\pin{-0.2,0}{-0.9,0}{f(x)};
\end{tikzpicture}
:=
\sum_{r=0}^n
\begin{tikzpicture}[H,centerzero]
\draw[<-] (-0.2,-0.3) -- (-0.2,0.3);
\multdot{-0.2,0}{east}{r};
\draw[gcolor,thick] (0.2,-0.3) \botlabel{V} -- (0.2,0.3);
\node[draw,gcolor,thick,fill=white!90!green,inner sep=1.5pt,rounded corners] at (.2,0) {\objlabel{z_r}};
\end{tikzpicture}
\ .
\end{equation}

\begin{lem}\label{wednesday}
Let $V \in \catR$ be any object.
\begin{enumerate}
\item If $f(x) \in Z_V[x]$ is a monic polynomial such that
$\begin{tikzpicture}[H,baseline = -2mm]
\draw[->] (0.2,-.25) to (0.2,.25);
\pin{.2,0}{-.5,0}{f(x)};
\draw[-,gcolor,thick] (0.5,.25) to (0.5,-.25)\botlabel{V};
\end{tikzpicture} = 0$,
then there is a monic polynomial $g(x) \in Z_V[x]$ of degree $\deg f(x)+\kappa$ such that
\begin{equation}\label{latenight}
\OO_V(u) = \frac{g(u)}{f(u)}.
\end{equation}
This has the property that
$\begin{tikzpicture}[H,baseline = -1mm]
\draw[<-] (0.2,-.25) to (0.2,.25);
\pin{.2,0}{-.5,0}{g(x)};
\draw[-,gcolor,thick] (0.5,.25) to (0.5,-.25)\botlabel{V};
\end{tikzpicture} = 0$.\label{wednesday-a}
\item If $g(x) \in Z_V[x]$ is a monic polynomial such that
$\begin{tikzpicture}[H,baseline = -2mm]
\draw[<-] (0.2,-.25) to (0.2,.25);
\pin{.2,0}{-.5,0}{g(x)};
\draw[-,gcolor,thick] (0.5,.25) to (0.5,-.25)\botlabel{V};
\end{tikzpicture} = 0$,
then there is a monic polynomial
$f(x) \in \kk[x]$ of degree $\deg g(x)-\kappa$ such that \cref{latenight} holds.
This has the property that
$\begin{tikzpicture}[H,baseline = -1mm]
\draw[->] (0.2,-.25) to (0.2,.25);
\pin{.2,0}{-.5,0}{f(x)};
\draw[-,gcolor,thick] (0.5,.25) to (0.5,-.25)\botlabel{V};
\end{tikzpicture} = 0$.\label{wednesday-b}
\end{enumerate}
\end{lem}

\begin{proof}
This proof is similar to that of \cite[Lem.~4.3]{BSW-HKM}.
We just consider \cref{wednesday-a}, since \cref{wednesday-b} is similar.
We define $g(u) := f(u) \OO_V(u)  \in \kk\lround u^{-1}\rround$. To show that $g(u)$ is a polynomial,
we must show that $[g(u)]_{u:-r-1}=0$ for $r \geq 0$.
We have that
$$
\big[
g(u)\big]_{u^{-r-1}}
=\big[
f(u) \OO_V (u) \big]_{u^{-r-1}}
=
\left[
\begin{tikzpicture}[H,anchorbase]
\leftbubgen{0.05,.2};
\draw[-,gcolor,thick] (.8,.6) to (.8,-.2);
\node at (.8,-.35) {\strandlabel{V}};
\node[draw,gcolor,thick,fill=white!90!green,inner sep=1.5pt,rounded corners] at (.8,.2) {\objlabel{f(u)}};
\end{tikzpicture}
\right]_{u:-r-1}\!\!\!\stackrel{\cref{tadpole}}{=}
\begin{tikzpicture}[H,anchorbase]
\leftbub{0,.2};
\pin{.25,.2}{1.2,.2}{x^r f(x)};
\draw[-,gcolor,thick] (2,.6) to (2,-.2)\botlabel{V};
\end{tikzpicture}\ .
$$
This is 0 as
$
\begin{tikzpicture}[H,baseline = -1mm]
\draw[->] (0.2,-.25) to (0.2,.25);
\pin{.2,0}{-.5,0}{f(x)};
\draw[-,gcolor,thick] (0.5,.25) to (0.5,-.25)\botlabel{V};
\end{tikzpicture}
=0$.
Hence, $g(u)$ is a polynomial in $u$ satisfying \cref{latenight}.
It is clear from \cref{heat} that it is
monic of degree $\deg f(x)+\kappa$.

It remains to show that
$
\begin{tikzpicture}[H,baseline = -1mm]
\draw[<-] (0.2,-.25) to (0.2,.25);
\pin{.2,0}{-.5,0}{g(x)};
\draw[-,gcolor,thick] (0.5,.25) to (0.5,-.25)\botlabel{V};
\end{tikzpicture}
=0$.
This follows by \cref{junesun}:
\begin{align*}
0
=
\begin{tikzpicture}[H,anchorbase]
\draw[<-] (0,-0.5) to[out=up,in=180] (0.3,0.2) to[out=0,in=up] (0.45,0) to[out=down,in=0] (0.3,-0.2) to[out=180,in=down] (0,0.5);
\pin{.45,0}{1.1,0}{f(x)};
\draw[-,gcolor,thick] (1.7,.4) to (1.7,-.4)  \botlabel{V};
\end{tikzpicture}=
\left[
\begin{tikzpicture}[H,anchorbase]
\draw[<-] (-0.4,-.4) to (-0.4,.4);
\circled{-.4,0}{u};
\leftbubgen{.65,0};
\draw[-,gcolor,thick] (1.5,.4) to (1.5,-.4)\botlabel{V};
\node[draw,gcolor,thick,fill=white!90!green,inner sep=1.5pt,rounded corners] at (1.5,0) {\objlabel{f(u)}};
\end{tikzpicture}
\right]_{u:-1}=
\left[
\begin{tikzpicture}[H,anchorbase]
\draw[<-] (0.08,-.4) to (0.08,.4);
\circled{0.08,0}{u};
\draw[-,gcolor,thick] (.8,.4) to (.8,-.4) \botlabel{V};
\node[draw,gcolor,thick,fill=white!90!green,inner sep=1.5pt,rounded corners] at (.8,0) {\objlabel{g(u)}};
\end{tikzpicture}
\right]_{u:-1}
&\stackrel{\cref{trick}}{=}\begin{tikzpicture}[H,anchorbase]
\draw[<-] (0.2,-.3) to (0.2,.3);
\pin{.2,0}{-.5,0}{g(x)};
\draw[-,gcolor,thick] (0.6,.3) to (0.6,-.3)\botlabel{V};
\end{tikzpicture}.
\qedhere
\end{align*}
\end{proof}

If $L \in \catR$ is an irreducible object
then, by the superalgebra version of Schur's Lemma, it is either the case that $Z_L$ is one-dimensional, 
or that $Z_L$ is isomorphic to the Clifford superalgebra $C$. In both cases, $Z_{L,\0} = \kk$.
Since all of the coefficients of 
$\OO_L(u)$ are even by the odd bubble relation, it follows that
$\OO_L(u) \in \kk\lround u^{-1}\rround$. 
In fact, $\OO_L(u)$ is a rational function:

\begin{cor}\label{ben}
For an irreducible object $L \in \catR$, we have
\[
\OO_L(u) = \frac{n_L(u)}{m_L(u)}.
\]
Hence, $\deg n_L(x) = \deg m_L(x)+\kappa$.
\end{cor}

\begin{proof}
We first apply \cref{wednesday}\ref{wednesday-a}
to deduce that $n_L(x)$ divides a monic polynomial $g(x)$ of degree $\deg m_L(x)+ \kappa$
such that $g(u) = \OO_L(u) m_L(u)$.
Hence, $\deg n_L(x) \leq \deg m_L(x)+ \kappa$.
Then we apply \cref{wednesday}\ref{wednesday-b}
to deduce that $m_L(x)$ divides a monic polynomial
of degree $\deg n_L(x)-\kappa$.
Hence, $\deg m_L(x) \leq \deg n_L(x) - \kappa$.
Comparing the two inequalities, we deduce that equality holds in both cases. Hence $n_L(x)$ has the same degree as $g(x)$. Since $n_L(x) | g(x)$ and both are monic, it follows that $n_L(x) = g(x)$.
So $n_L(u) = \OO_L(u) m_L(u)$.
\end{proof}

Now we can properly explain the significance of the bubble slide relations from \cref{bubbleslide}
and the role of the
bijection $b$ from
\cref{importantfunction}.  Recall
for each $i \in \kk$ that $P_i \cong P_{-i}$
and $Q_i \cong Q_{-i}$ via odd isomorphisms defined by the Clifford tokens.
This means that
we do not lose any information
by restricting attention to the
eigenfunctors $P_i$ and $Q_i$
indexed just
by the elements $i\in I$.

\begin{lem} \label{crunchy}
Suppose that $L \in \catR$ is an irreducible object and
$i \in I$.
\begin{enumerate}
\item If $K$ is an irreducible subquotient of $P_i L$ then
\begin{equation*}
\OO_K(u)=
\OO_L(u) \times
\frac{\big( u^2 - i(i+1) \big)^2}{\big( u^2 - (i-1)i \big) \big( u^2 - (i+1)(i+2) \big)}.
\end{equation*}\label{crunchy-a}
\item If
$K$ is an irreducible subquotient of $Q_i L$ then
\begin{equation*}
\OO_K(u)=
\OO_L(u) \times
\frac{\big( u^2 - (i-1)i \big) \big( u^2 - (i+1)(i+2) \big)}{\big( u^2 - i(i+1) \big)^2}.
\end{equation*}\label{crunchy-b}
\end{enumerate}
\end{lem}

\begin{proof}
(1) Let $p(u,x)$ be as in \cref{divergent}.
By the first bubble slide in \cref{bubslide}, we have
\[
\begin{tikzpicture}[H,centerzero]
\leftbubgen{-0.5,0};
\draw[->] (0,-0.3) \botlabel{i} -- (0,0.3);
\draw[gcolor,thick] (0.5,-0.3) \botlabel{L} -- (0.5,0.3);
\end{tikzpicture}
=
\begin{tikzpicture}[H,centerzero]
\leftbubgen{-0.1,0};
\draw[->] (-1.2,-0.3) \botlabel{i} -- (-1.2,0.3);
\draw[gcolor,thick] (0.3,-0.3) \botlabel{L} -- (0.3,0.3);
\pin{-1.2,0}{-2.2,0}{p(u,x)^{-1}};
\end{tikzpicture}
\overset{\cref{heat}}{=}
\begin{tikzpicture}[H,centerzero]
\draw[->] (0,-0.3) \botlabel{i} -- (0,0.3);
\draw[gcolor,thick] (0.5,-0.3) \botlabel{L} -- (0.5,0.3);
\pin{0,0}{-1.4,0}{\OO_L(u) p(u,x)^{-1}};
\end{tikzpicture}.
\]
Now consider the filtration
$$
0 \subset
\ker (x_L - b(i)) \subset \ker (x_L-b(i))^2
\subset\cdots\subset \ker (x_L - b(i))^{\eps_i(L)} = P_{i} L.
$$
The subquotient $K$ is isomorphic to a subquotient of one of the sections of this filtration. Since $x_L$ acts as $b(i)$ on each section, we can replace $x$ by $b(i)$ in the above to deduce that
$$
\OO_K(u)=\OO_L(u) p(u,b(i))^{-1}.
$$
The formula for $\OO_K(u)$ follows from this together with \cref{whenisitzero}.

\vspace{2mm}
\noindent
(2) This is similar, starting instead from the identity
\[
\begin{tikzpicture}[H,centerzero]
\leftbubgen{-0.5,0};
\draw[<-] (0,-0.3) \botlabel{i} -- (0,0.3);
\draw[gcolor,thick] (0.5,-0.3) \botlabel{L} -- (0.5,0.3);
\end{tikzpicture}
=
\begin{tikzpicture}[H,centerzero]
\leftbubgen{-0.1,0};
\draw[<-] (-1.2,-0.3) \botlabel{i} -- (-1.2,0.3);
\draw[gcolor,thick] (0.3,-0.3) \botlabel{L} -- (0.3,0.3);
\pin{-1.2,0}{-2,0}{p(u,x)};
\end{tikzpicture}
\overset{\cref{heat}}{=}
\begin{tikzpicture}[H,centerzero]
\draw[<-] (0,-0.3) \botlabel{i} -- (0,0.3);
\draw[gcolor,thick] (0.5,-0.3)\botlabel{L}  -- (0.5,0.3);
\pin{0,0}{-1.2,0}{\OO_L(u) p(u,x)};
\end{tikzpicture},
\]
which is the second bubble slide in \cref{bubslide2}.
\end{proof}

For an irreducible object $L \in \catR$,
let $h_i(L) \in \Z$ be the multiplicity of $b(i)$ as a zero or a pole of the rational function $\OO_L(u)$.
Equivalently, by \cref{ben},
\begin{equation}\label{crystalweight}
h_i(L) = \phi_i(L) - \eps_i(L).
\end{equation}
Then, for $\lambda \in X$, we let $\catR_\lambda$ be the Serre subcategory of $\catR$ generated by the irreducible objects $L$ with 
$h_i(L) = h_i(\lambda)$ for all $i \in I$.
We refer to $\catR_\lambda$ as a {\em weight subcategory}.

\begin{theo}\label{lostboys}
Every object $V$ of $\catR$ decomposes as 
$V = \bigoplus_{\lambda \in X} V_\lambda$
for $V_\lambda \in \catR_\lambda$, with $V_\lambda$ being zero for all but finitely many $\lambda \in X$.
Also there are no non-zero morphisms between objects of $\catR_\lambda$
and $\catR_{\mu}$ for $\lambda \neq \mu$.
So we have that
\begin{equation}\label{wtdef}
\catR = \bigoplus_{\lambda \in X} \catR_\lambda.
\end{equation}
Moreover, for each $i \in I$, $P_i$ restricts to a functor
$\catR_\lambda \rightarrow \catR_{\lambda+\alpha_i}$,
and $Q_i$ restricts to a functor
$\catR_\lambda \rightarrow \catR_{\lambda-\alpha_i}$.
\end{theo}

\begin{proof}
For irreducible objects $K$ and $L$,
we have that $\OO_K(u) = \OO_L(u)$
if and only if $h_i(K) = h_i(L)$ for all $i \in I$.
When $\OO_K(u) \neq \OO_L(u)$, the irreducible objects
have different central characters. All of the theorem except for the last assertion follows from these observations.

Now we prove for $i \in I$
that $P_i$ takes an object of $\catR_\lambda$ to an object of $\catR_{\lambda+\alpha_i}$.
It suffices to show that $h_j(K) = h_j(L) + h_j(\alpha_i)$
for an irreducible object $L \in \catR_\lambda$,
an irreducible subquotient $K$ of $P_i L$, and all
$j \in I$.
There are various cases.
We will use the observation that $b(I) \cap (-b(I)) = \{0\}$ several times, which follows from \cref{goodintersections2}.
\begin{itemize}
\item
Suppose first that $i=0$. We have that $0=i(i+1)=(i-1)i\neq (i+1)(i+2)$.
So \cref{crunchy}\ref{crunchy-a} implies that
$$
\OO_K(u) = \OO_L(u) \times \frac{(u-b(0))^2}{u-b(1)}\times \frac{1}{u+b(1)}.
$$
We deduce that $h_0(K) = h_0(L)+2$,
$h_1(K) = h_1(L)-1$, and $h_j(K) = h_j(L)$
for all other $j \in I - \{0,1\}$.
This is what we want since from \cref{cartan} we have that $h_j(\alpha_0)=c_{j0}$ is 2 if $j=0$, $-1$ if $j=1$, and $0$ for all other $j$.
\item
Next suppose that $i = \hbar$ and $p \neq 3$.
Then $0 \neq i(i+1) \neq (i-1)i=(i+1)(i+2) \neq 0$.
So \cref{crunchy}\ref{crunchy-a} gives that
$$
\OO_K(u) = \OO_L(u) \times \frac{\left(u-b(\hbar)\right)^2
}{\Big(u-b\big(-\frac{3}{2}\big)\Big)^2}
\times\frac{\left(u+b(\hbar)\right)^2}{\Big(u+b\big(-\frac{3}{2}\big)\Big)^2}.
$$
We deduce that
$h_{\hbar}(K) = h_{\hbar}(L)+2, h_{-\frac{3}{2}}(K) = h_{-\frac{3}{2}}(L)-2$,
and $h_j(K) = h_j(L)$ for all other $j \in I - \left\{\hbar,-\frac{3}{2}\right\}$.
Again this is right because $c_{j(\hbar)}$ is 2 if $j = \hbar$, $-2$ if $j =-\frac{3}{2}$, and 0 otherwise.
\item
If $i = 1$ and $p=3$ then
$i(i+1) \neq (i-1)i=(i+1)(i+2) = 0$.
So \cref{crunchy}\ref{crunchy-a} gives that
$$
\OO_K(u) = \OO_L(u) \times \frac{(u-b(1))^2
}{(u-b(0))^4}\times \frac{1}{(u+b(1))^2}.
$$
So
$h_{0}(K) = h_{0}(L)-4$ and $h_{1}(K) = h_{1}(L)+2$.
This is what we wanted.
\item
In the remaining cases, $i(i+1), (i-1)i$ and $(i+1)(i+2)$ are all different and all are non-zero. So
$$
\OO_K(u) = \OO_L(u) \times \frac{(u-b(i))^2}{(u-b(i-1))(u-b(i+1))}
\times\frac{(u+b(i))^2}{(u+b(i-1))(u+b(i+1))}
$$
implies that $h_i(K) = h_i(L)+2$,
$h_{i+1}(K) = h_{i+1}(L)-1, h_{i-1}(K) = h_{i-1}(L)-1$, and $h_j(K) = h_j(L)$ for all other $j$. This is right.
\end{itemize}

Finally,
it is easy to deduce
that $Q_i$
takes an object of $\catR_\lambda$ to an object of $\catR_{\lambda-\alpha_i}$ using the result for $P_i$ just proved
and the biadjointness of $P_i$ and $Q_i$ (or one can prove it for $Q_i$
with a similar argument using \cref{crunchy}\ref{crunchy-b}).
\end{proof}

\subsection{The spectrum}

We define the {\em spectrum} of $\catR$ to be the
set $I(\catR) \subseteq I$
consisting of all $i \in I$ such that any of the following equivalent conditions hold:
\begin{itemize}
\item
$b(i)$ is a root of $m_L(x)$ for some irreducible $L \in \catR$;
\item $\eps_i(L) \neq 0$ for some irreducible $L \in \catR$;
\item
$P_i L \neq \{0\}$ for some irreducible $L \in \catR$;
\item
$P_i V \neq \{0\}$ for some $V \in \catR$;
\item
the functor $P_i:\catR \rightarrow \catR$ is non-zero;
\item
the functor $Q_i:\catR \rightarrow \catR$ is non-zero;
\item
$Q_i V \neq \{0\}$ for some $V \in \catR$;
\item
$Q_i L \neq \{0\}$ for some irreducible $L \in \catR$;
\item $\phi_i(L) \neq 0$ for some irreducible $L \in \catR$;
\item
$b(i)$ is a root of $n_L(x)$ for some irreducible $L \in \catR$.
\end{itemize}
The equivalence of these properties is easy to see
from the definitions, using also
the biadjointness (hence, exactness) of $P_i$ and $Q_i$.

\begin{lem}
The spectrum $I(\catR)$ is a union of
connected components $I_k\:(k \in A)$.
\end{lem}

\begin{proof}
We must show for $i \in I(\catR)$ and $j = i \pm 1\in I$
that $j \in I(\catR)$.
As $i \in I(\catR)$, there is an irreducible $L \in \catR$
such that $P_i L \neq \{0\}$.
Let $K$ be an irreducible subquotient of $P_i L$.
By \cref{crunchy}\ref{crunchy-a}, we have that
$$
\OO_K(u) (u^2-(i-1)i)(u^2-(i+1)(i+2))
= \OO_L(u) (u^2-i(i+1))^2.
$$
Using \cref{ben}, we deduce that
$$
n_K(u) m_L(u)
(u^2-(i-1)i)(u^2-(i+1)(i+2))
=
m_K(u) n_L(u)(u^2-i(i+1))^2.
$$
Since $u-b(j)$ divides the left hand side,
we deduce either that $(u-b(j)) \vert m_K(u)$
or that $(u-b(j)) \vert n_L(u)$ or that $(u-b(j)) \vert
(u-b(i))(u+b(i))$.
In the first case, $\eps_j(K) \neq 0$
so $j \in I(\catR)$. In the second case
$\phi_j(L) \neq 0$ so $j \in I(\catR)$ again.
In the third case, we have that $b(j)^2 = b(i)^2$,
hence, $i = j$ by \cref{rainisback}\ref{rainisback-a}, so this actually never happens.
\end{proof}

\subsection{Inversion relations}

The combinatorics of weights also underpins the next theorem.

\begin{theo}\label{theworstplace}
Suppose that $i \in I$, $\lambda \in X$ and $L\in \catR_\lambda$ is an irreducible object.
Let $\eps := \eps_i(L)$ and $\phi := \phi_i(L)$ for short.
Let $x_i := (x-b(i)) \xi(x)$
for some given $\xi(x) \in \kk[x]$ such that
$\xi(b(i)) \neq 0$,
and let $r(x,y), s(x,y) \in \kk[x,y]$ be some other polynomials.
\begin{enumerate}
\item \label{twp1}
If $i \neq 0$ and $h_i(\lambda) \leq 0$ then the matrix
$$
\begin{bmatrix}
\begin{tikzpicture}[centerzero,H,scale=1.3]
\draw[->] (-.2,-.35)\botlabel{i} to[out=up,in=down,looseness=1] (.2,.35) \toplabel{i};
\draw[<-] (.2,-.35)\botlabel{i} to[out=up,in=down,looseness=1] (-.2,.35) \toplabel{i};
\projcr{0,0};
\draw[gcolor,thick] (0.45,-0.35) \botlabel{L} -- (0.45,0.35);
\end{tikzpicture}
+\begin{tikzpicture}[centerzero,H,scale=1.3]
\draw[->] (-.2,.4) \toplabel{i} to[out=down,in=down,looseness=2.5] (.2,.4);
\draw[->] (-.2,-.3) \botlabel{i} to[out=up,in=up,looseness=2.5] (.2,-.3);
\pinpin{-.1,-.04}{.1,.13}{.7,.13}{r(x,y)};
\draw[gcolor,thick] (1.2,-0.3) \botlabel{L} -- (1.2,0.4);
\end{tikzpicture}\ \ &
\begin{tikzpicture}[H,anchorbase,scale=1.2]
\draw[->] (-0.2,0.3)\toplabel{i} -- (-0.2,0) arc(180:360:0.2) -- (0.2,0.3);
\draw[gcolor,thick] (0.45,-0.3) \botlabel{L} -- (0.45,0.3);
\end{tikzpicture}\ \
&
\begin{tikzpicture}[anchorbase,H,scale=1.2]
\draw[->] (-0.2,0.3) \toplabel{i} -- (-0.2,0) arc(180:360:0.2) -- (0.2,0.3);
\pin{0.2,0.07}{.7,.07}{x_i};
\draw[gcolor,thick] (1,-.3) \botlabel{L} -- (1,0.3);
\end{tikzpicture}
&\!\!\cdots\!\!&
\begin{tikzpicture}[anchorbase,H,scale=1.2]
\draw[->] (-0.2,0.3)\toplabel{i}  -- (-0.2,0) arc(180:360:0.2) -- (0.2,0.3);
\pin{0.2,0.07}{1.27,.07}{x_i^{-h_i(\lambda)-1}};
\draw[gcolor,thick] (2.05,-.3) \botlabel{L} -- (2.05,0.3);
\end{tikzpicture}\
\end{bmatrix}
$$
defines an even
isomorphism $P_i Q_i L \oplus L^{\oplus (-h_i(\lambda))}
\stackrel{\sim}{\rightarrow} Q_i P_i L$.
\item \label{twp2}
If $i=0$ and $h_0(\lambda) \leq 0$ then the matrix
$$
\begin{bmatrix}
\begin{tikzpicture}[centerzero,H,scale=1.3]
\draw[->] (-.2,-.35)\botlabel{0} to[out=up,in=down,looseness=1] (.2,.35) \toplabel{0};
\draw[<-] (.2,-.35)\botlabel{0} to[out=up,in=down,looseness=1] (-.2,.35) \toplabel{0};
\projcr{0,0};
\draw[gcolor,thick] (0.45,-0.35) \botlabel{L} -- (0.45,0.35);
\end{tikzpicture}
+\begin{tikzpicture}[centerzero,H,scale=1.3]
\draw[->] (-.2,.4) \toplabel{0} to[out=down,in=down,looseness=2.5] (.2,.4);
\draw[->] (-.2,-.3) \botlabel{0} to[out=up,in=up,looseness=2.5] (.2,-.3);
\pinpin{-.1,-.04}{.1,.13}{.7,.13}{r(x,y)};
\draw[gcolor,thick] (1.2,-0.3) \botlabel{L} -- (1.2,0.4);
\end{tikzpicture}
+\!\!\begin{tikzpicture}[centerzero,H,scale=1.3]
\draw[->] (-.2,.4) \toplabel{0} to[out=down,in=down,looseness=2.5] (.2,.4);
\draw[->] (-.2,-.3) \botlabel{0} to[out=up,in=up,looseness=2.5] (.2,-.3);
\token{-.19,-.18};
\token{.19,.27};
\pinpin{-.1,-.04}{.1,.13}{.7,.13}{s(x,y)};
\draw[gcolor,thick] (1.2,-0.3) \botlabel{L} -- (1.2,0.4);
\end{tikzpicture}\ \ &
\begin{tikzpicture}[H,anchorbase,scale=1.2]
\draw[->] (-0.2,0.3)\toplabel{0} -- (-0.2,0) arc(180:360:0.2) -- (0.2,0.3);
\draw[gcolor,thick] (0.45,-0.3) \botlabel{L} -- (0.45,0.3);
\token{.2,.14};
\end{tikzpicture}\ \ &
\begin{tikzpicture}[H,anchorbase,scale=1.2]
\draw[->] (-0.2,0.3)\toplabel{0} -- (-0.2,0) arc(180:360:0.2) -- (0.2,0.3);
\draw[gcolor,thick] (0.45,-0.3) \botlabel{L} -- (0.45,0.3);
\end{tikzpicture}\ \ &
\begin{tikzpicture}[anchorbase,H,scale=1.2]
\draw[->] (-0.2,0.3) \toplabel{0} -- (-0.2,0) arc(180:360:0.2) -- (0.2,0.3);
\pin{0.2,0}{.7,0}{x_0};
\token{.2,.14};
\draw[gcolor,thick] (1,-.3) \botlabel{L} -- (1,0.3);
\end{tikzpicture}
&\!\!\cdots\!\!&
\begin{tikzpicture}[anchorbase,H,scale=1.2]
\draw[->] (-0.2,0.3)\toplabel{0}  -- (-0.2,0) arc(180:360:0.2) -- (0.2,0.3);
\pin{0.2,0}{1.27,0}{x_0^{-h_i(\lambda)-1}};
\draw[gcolor,thick] (2.05,-.3) \botlabel{L} -- (2.05,0.3);
\end{tikzpicture}\
\end{bmatrix}
$$
defines an even
isomorphism $P_0 Q_0 L \oplus (\Pi L \oplus L)^{\oplus (-h_0(\lambda))}
\stackrel{\sim}{\rightarrow} Q_0 P_0 L$.
\item \label{twp3}
If $i \neq 0$ and $h_i(\lambda) \geq 0$ then the matrix
$$
\begin{bmatrix}
\begin{tikzpicture}[centerzero,H,scale=1.3]
\draw[->] (-.2,-.35)\botlabel{i} to[out=up,in=down,looseness=1] (.2,.35) \toplabel{i};
\draw[<-] (.2,-.35)\botlabel{i} to[out=up,in=down,looseness=1] (-.2,.35) \toplabel{i};
\projcr{0,0};
\draw[gcolor,thick] (0.45,-0.35) \botlabel{L} -- (0.45,0.35);
\end{tikzpicture}
+\begin{tikzpicture}[centerzero,H,scale=1.3]
\draw[->] (-.2,.4) \toplabel{i} to[out=down,in=down,looseness=2.5] (.2,.4);
\draw[->] (-.2,-.3) \botlabel{i} to[out=up,in=up,looseness=2.5] (.2,-.3);
\pinpin{-.1,-.04}{.1,.13}{.7,.13}{r(x,y)};
\draw[gcolor,thick] (1.2,-0.3) \botlabel{L} -- (1.2,0.4);
\end{tikzpicture}\ \ &
\begin{tikzpicture}[H,anchorbase,scale=1.2]
\draw[->] (-0.2,-0.3)\botlabel{i} -- (-0.2,0) arc(-180:-360:0.2) -- (0.2,-0.3);
\draw[gcolor,thick] (0.45,-0.3) \botlabel{L} -- (0.45,0.3);
\end{tikzpicture}\ \
&
\begin{tikzpicture}[anchorbase,H,scale=1.2]
\draw[->] (-0.2,-0.3) \botlabel{i} -- (-0.2,0) arc(-180:-360:0.2) -- (0.2,-0.3);
\pin{0.2,-0.07}{.7,-.07}{x_i};
\draw[gcolor,thick] (1,-.3) \botlabel{L} -- (1,0.3);
\end{tikzpicture}
&\!\!\cdots\!\!&
\begin{tikzpicture}[anchorbase,H,scale=1.2]
\draw[->] (-0.2,-0.3)\botlabel{i}  -- (-0.2,0) arc(-180:-360:0.2) -- (0.2,-0.3);
\pin{0.2,-0.07}{1.1,-.07}{x_i^{h_i(\lambda)-1}};
\draw[gcolor,thick] (1.8,-.3) \botlabel{L} -- (1.8,0.3);
\end{tikzpicture}\
\end{bmatrix}^\transpose
$$
defines an even
isomorphism $P_i Q_i L\stackrel{\sim}{\rightarrow} Q_i P_i L \oplus L^{\oplus h_i(\lambda)}$.
\item \label{twp4}
If $i=0$ and $h_0(\lambda) \geq 0$ then the matrix
$$
\begin{bmatrix}
\begin{tikzpicture}[centerzero,H,scale=1.3]
\draw[->] (-.2,-.35)\botlabel{0} to[out=up,in=down,looseness=1] (.2,.35) \toplabel{0};
\draw[<-] (.2,-.35)\botlabel{0} to[out=up,in=down,looseness=1] (-.2,.35) \toplabel{0};
\projcr{0,0};
\draw[gcolor,thick] (0.45,-0.35) \botlabel{L} -- (0.45,0.35);
\end{tikzpicture}
+\begin{tikzpicture}[centerzero,H,scale=1.3]
\draw[->] (-.2,.4) \toplabel{0} to[out=down,in=down,looseness=2.5] (.2,.4);
\draw[->] (-.2,-.3) \botlabel{0} to[out=up,in=up,looseness=2.5] (.2,-.3);
\pinpin{-.1,-.04}{.1,.13}{.7,.13}{r(x,y)};
\draw[gcolor,thick] (1.2,-0.3) \botlabel{L} -- (1.2,0.4);
\end{tikzpicture}
+\!\!\begin{tikzpicture}[centerzero,H,scale=1.3]
\draw[->] (-.2,.4) \toplabel{0} to[out=down,in=down,looseness=2.5] (.2,.4);
\draw[->] (-.2,-.3) \botlabel{0} to[out=up,in=up,looseness=2.5] (.2,-.3);
\token{-.19,-.18};
\token{.19,.27};
\pinpin{-.1,-.04}{.1,.13}{.7,.13}{s(x,y)};
\draw[gcolor,thick] (1.2,-0.3) \botlabel{L} -- (1.2,0.4);
\end{tikzpicture}\ \ &
\begin{tikzpicture}[H,anchorbase,scale=1.2]
\draw[->] (-0.2,-0.3)\botlabel{0} -- (-0.2,0) arc(-180:-360:0.2) -- (0.2,-0.3);
\draw[gcolor,thick] (0.42,-0.3) \botlabel{L} -- (0.42,0.3);
\token{.2,-.1};
\end{tikzpicture}\ \ &
\begin{tikzpicture}[H,anchorbase,scale=1.2]
\draw[->] (-0.2,-0.3)\botlabel{0} -- (-0.2,0) arc(-180:-360:0.2) -- (0.2,-0.3);
\draw[gcolor,thick] (0.42,-0.3) \botlabel{L} -- (0.42,0.3);
\end{tikzpicture}\ \ &
\begin{tikzpicture}[anchorbase,H,scale=1.2]
\draw[->] (-0.2,-0.3) \botlabel{0} -- (-0.2,0) arc(-180:-360:0.2) -- (0.2,-0.3);
\pin{0.2,.05}{.7,.05}{x_0};
\draw[gcolor,thick] (1.04,-.3) \botlabel{L} -- (1.04,0.3);
\token{.2,-.1};
\end{tikzpicture}
&\!\!\cdots\!\!&
\begin{tikzpicture}[anchorbase,H,scale=1.2]
\draw[->] (-0.2,-0.3)\botlabel{0}  -- (-0.2,0) arc(-180:-360:0.2) -- (0.2,-0.3);
\pin{0.2,0.05}{1.1,.05}{x_0^{h_0(\lambda)-1}};
\draw[gcolor,thick] (1.8,-.3) \botlabel{L} -- (1.8,0.3);
\end{tikzpicture}\
\end{bmatrix}^\transpose
$$
defines an even
isomorphism $P_0 Q_0 L\stackrel{\sim}{\rightarrow} Q_0 P_0 L \oplus (\Pi L \oplus L)^{\oplus h_0(\lambda)}$.
\end{enumerate}
\end{theo}

\begin{proof}
(1) This part of the proof is similar to the proof of \cite[Lem.~4.9]{BSW-HKM}.
We repeat it in full since the extra $r(x,y)$ term was not present in \cite{BSW-HKM}, and also we will refer to this argument again in the proof of \cref{twp2}.

Let $\minush:= \eps-\phi=-
\langle h_i,\lambda\rangle \geq 0$.
Let $A := \kk[x] / (m_L(x))$ and
$B := \kk[x] / (n_L(x))$.
Let $A_i$ and $B_i$ be the subalgebras of $A$ and $B$ that are
isomorphic to
$\kk[x] \big/ (x-b(i))^{\eps}$
and
$\kk[x] \big/ (x-b(i))^{\phi}$
in the CRT decomposition from (\ref{switchybitchy})
(with $V$ replaced by $L$).
To be explicit, let $e(x) := m_L(x) / (x-b(i))^\eps$ and $f(x) := n_L(x) / (x-b(i))^\phi$.
Then $A_i$ is the ideal of $A$ generated by $e(x) \in A_i^\times$, and $B_i$ is the ideal of $B$ generated by $f(x) \in B_i^\times$.
The algebra $A_i$ has basis
$e(x), (x-b(i)) e(x), \dots, (x-b(i))^{\eps-1} e(x)$,
and $B_i$ has basis
$f(x), (x-b(i)) f(x), \dots, (x-b(i))^{\phi-1} f(x)$.
Multiplication by $\frac{m_L(x)}{n_L(x)}$
defines an
injective $\kk[x]$-module homomorphism
$$
\mu:B_i \hookrightarrow A_i,
\qquad
(x-b(i))^r f(x) \mapsto (x-b(i))^{\minush+r} e(x).
$$
Since $e(x)$ and $\xi(x)$ are units in $A_i$,
$x_i^r e_i$ is equal to a non-zero multiple of $(x-b(i))^r e(x)$ plus a linear combination of
$(x-b(i))^s e(x)$ for $r < s \leq \eps-1$.
We deduce that
$e_i, \dots,x_i^{\minush-1} e_i,
(x-b(i))^\minush e(x),\dots, (x-b(i))^{\eps-1} e(x)$
is another basis for $A_i$.

The algebra embeddings \cref{cradle1,cradle2}
restrict to
$A_i \hookrightarrow \End_{\catR}(P_i L)$
and $B_i \hookrightarrow \End_{\catR}(Q_i L)$.
Composing with adjunction
isomorphisms, we obtain
injective linear maps
\begin{align*}
\vec{\beta}:A_i &\hookrightarrow \Hom_{\catR}(L, Q_i P_i L),
& p(x) &\mapsto
\begin{tikzpicture}[anchorbase,H]
\draw[<-] (0.4,.4) to[out=-90, in=0] (0.1,0);
\draw[-] (0.1,0) to[out = 180, in = -90] (-0.2,.4)\toplabel{i};
\pin{.36,.17}{1,.17}{p(x)};
\draw[gcolor,thick] (1.6,-0.2) \botlabel{L} -- (1.6,0.4);
\end{tikzpicture},\\
\cev{\beta}:B_i &\hookrightarrow \Hom_{\catR}(L, P_i Q_i L),
& p(x) &\mapsto
\begin{tikzpicture}[anchorbase,H]
\draw[-] (0.4,.4)\toplabel{i} to[out=-90, in=0] (0.1,0);
\draw[->] (0.1,0) to[out = 180, in = -90] (-0.2,.4);
\pin{.36,.18}{1,.18}{p(x)};
\draw[gcolor,thick] (1.6,-0.2) \botlabel{L} -- (1.6,0.4);
\end{tikzpicture}\ .
\end{align*}
Viewing $A_i$ and $B_i$ as purely even superalgebras,
these linear maps are even.
Since the morphisms $\vec{\beta}\big(e_i\big),\dots
\vec{\beta}\big(x_i^{\minush-1} e_i\big),
\vec{\beta}\big((x-b(i))^\minush e(x)\big), \dots,\vec{\beta}\big((x-b(i))^{\eps-1} e(x)\big)$
are linearly independent,
Schur's Lemma implies that the completely reducible subobject
$$
\vec{L} := \vec{\beta}\big(e_i\big)(L)
+
\cdots+\vec{\beta}\big(x_i^{\minush-1} e_i\big)(L)
+ \vec{\beta}\big((x-b(i))^\minush e(x)\big)(L) +\cdots+\vec{\beta}\big((x-b(i))^{\eps-1} e(x)\big)(L)
\leq Q_i P_i L$$
is a direct sum of $\eps$ copies of $L$.
Similarly,
$$
\cev{L} := \cev{\beta}\big(f(x)\big)(L) +
\cdots+
\cev{\beta}\big((x-b(i))^{\phi-1} f(x)\big)(L)
\leq P_i Q_i L$$
is a direct sum of $\phi$ copies of $L$. Finally, let
\begin{align*}
\vec{\rho}&:=\!
\begin{tikzpicture}[centerzero,H,scale=1.2]
\draw[->] (-.2,-.35)\botlabel{i} to[out=up,in=down,looseness=1] (.2,.35) \toplabel{i};
\draw[<-] (.2,-.35)\botlabel{i} to[out=up,in=down,looseness=1] (-.2,.35) \toplabel{i};
\projcr{0,0};
\draw[gcolor,thick] (0.45,-0.35) \botlabel{L} -- (0.45,0.35);
\end{tikzpicture}:P_i Q_i L \rightarrow Q_i P_i L,
&\cev{\rho}&:=\!
\begin{tikzpicture}[centerzero,H,scale=1.2]
\draw[<-] (-.2,-.35)\botlabel{i} to[out=up,in=down,looseness=1] (.2,.35) \toplabel{i};
\draw[->] (.2,-.35)\botlabel{i}
to[out=up,in=down,looseness=1] (-.2,.35) \toplabel{i};
\projcr{0,0};
\draw[gcolor,thick] (0.45,-0.35) \botlabel{L} -- (0.45,0.35);
\end{tikzpicture}
:Q_i P_i L \rightarrow P_i Q_i L,
\end{align*}
\begin{align*}
\vec{\sigma}
&:=\!\begin{tikzpicture}[centerzero,H,scale=1.2]
\draw[->] (-.2,.4) \toplabel{i} to[out=down,in=down,looseness=2.5] (.2,.4);
\draw[->] (-.2,-.3) \botlabel{i} to[out=up,in=up,looseness=2.5] (.2,-.3);
\pinpin{-.1,-.04}{.1,.13}{.7,.13}{r(x,y)};
\draw[gcolor,thick] (1.27,-0.3) \botlabel{L} -- (1.27,0.4);
\end{tikzpicture}:P_i Q_i L \rightarrow Q_i P_i L.
\end{align*}
We are trying to prove that the morphism
\begin{equation*}
\theta := \begin{bmatrix}
\vec{\rho}+\vec{\sigma}&
\vec{\beta}\left(e_i\right)&
\vec{\beta}\left(x_ie_i\right)&\cdots&
\vec{\beta}\big(x_i^{\minush-1} e_i\big)
\end{bmatrix}
:P_i Q_i L \oplus L^{\oplus \minush} \rightarrow Q_i P_i L
\end{equation*}
is an isomorphism.
This follows from the following series of claims.

\vspace{2mm}
\noindent
\underline{Claim 1}:
{\em $\cev{\rho}(\vec{L}) \leq \cev{L}$.}
To justify this, take any $p(x) \in A_i$.
We have that
$$
\cev{\rho}\circ \vec{\beta}(p(x)) =
\begin{tikzpicture}[anchorbase,H]
\draw[-] (0.25,0) to[out=-90, in=0] (0,-0.25);
\draw[-] (0.25,.6)\toplabel{i} to[out=240,in=90] (-0.25,0);
\draw[-] (0,-0.25) to[out = 180, in = -90] (-0.25,0);
\draw[<-] (-0.25,.6)\toplabel{i} to[out=300,in=90] (0.25,0);
\node at (-0.3,.2) {\strandlabel{i}};
\pin{.26,0}{.85,0}{p(x)};
\projcr{0,.35};
\draw[gcolor,thick] (1.4,0.6)  -- (1.4,-0.25)\botlabel{L};
\end{tikzpicture}
=
\begin{tikzpicture}[anchorbase,H]
\draw[-] (0.25,-0) to[out=-90, in=0] (0,-0.25);
\draw[-] (0,-0.25) to[out = 180, in = -90] (-0.25,-0);
\draw[-] (0.25,.6)\toplabel{i} to[out=240,in=90] (-0.25,-0);
\draw[<-] (-0.25,.6)\toplabel{i}to[out=300,in=90] (0.25,-0);
\notch[-45]{.15,.47};
\notch[45]{-.15,.47};
\pin{.26,0}{.85,0}{p(x)};
\draw[gcolor,thick] (1.4,0.6)  -- (1.4,-0.25)\botlabel{L};
\end{tikzpicture}
\overset{\cref{junesun}}{\underset{\cref{heat}}{=}}
\left[
p(u)\OO_L(u)\
\begin{tikzpicture}[anchorbase,scale=1.2,H]
\draw[-] (0.3,.4)\toplabel{i} to[out=-90, in=0] (0.1,0);
\draw[->] (0.1,0) to[out = 180, in = -90] (-0.1,.4);
\circled{0.3,0.18}{u};
\draw[gcolor,thick] (.55,0.4)  -- (.55,0)\botlabel{L};
\end{tikzpicture}
\right]_{u^{-1}}
=
\begin{tikzpicture}[anchorbase,scale=1.2,H]
\draw[-] (0.3,.4)\toplabel{i} to[out=-90, in=0] (0.1,0);
\draw[->] (0.1,0) to[out = 180, in = -90] (-0.1,.4);
\pin{0.29,.18}{.8,.18}{q(x)};
\draw[gcolor,thick] (1.25,0.4)  -- (1.25,0)\botlabel{L};
\end{tikzpicture}
$$
where $q(x) := \left[p(u) \OO_L(u) (u-x)^{-1}\right]_{u:-1} \in \kk[x]$.
The image of this morphism is contained in
$\cev{L}$, indeed, this is true for {\em any} polynomial $q(x)$.

\vspace{2mm}
\noindent
\underline{Claim 2}:
{\em We have that $\vec{\beta}\big((x-b(i))^{\minush+r} e(x)\big)
=\big(\vec{\rho}+\vec{\sigma}\big)\circ \cev{\beta}\big((x-b(i))^r f(x)\big)
$ for any $r \geq 0$.}
We prove a
stronger statement, namely, that $\vec{\rho}\circ \cev{\beta}(p(x)) = \vec{\beta}(q(x))$ and $\vec{\sigma} \circ \cev{\beta}(p(x)) = 0$,
where
$p(x) := (x-b(i))^r f(x) \in B_i$
and $q(x) := (x-b(i))^{\minush+r} e(x) \in A_i$.
Note $q(x) = \mu(p(x))$ where
$\mu:B_i \hookrightarrow A_i$ is the $\kk[x]$-module homomorphism defined earlier, that is, multiplication by $\frac{m_L(x)}{n_L(x)}$.
\cref{ben} implies that that $q(u) = p(u) \OO_L(u)^{-1}$.
Now we calculate:
\begin{align*}
\vec{\rho}\circ \cev{\beta}\big(p(x)\big) &=
\begin{tikzpicture}[anchorbase,H]
\draw[-] (-0.25,.6)\toplabel{i} to[out=300,in=90] (0.25,-0);
\draw[-] (0.25,-0) to[out=-90, in=0] (0,-0.25) to [out=180,in=-90] (-.25,0);
\draw[<-] (0.25,.6)\toplabel{i} to[out=240,in=90] (-0.25,-0);
\node at (-0.3,.2) {\strandlabel{i}};
\pin{.25,-.03}{.85,-.03}{p(x)};
\projcr{0,.35};
\draw[gcolor,thick] (1.4,0.6)  -- (1.4,-0.25)\botlabel{L};
\end{tikzpicture}
=
\begin{tikzpicture}[anchorbase,H]
\draw[-] (-0.25,.6)\toplabel{i} to[out=300,in=90] (0.25,-0);
\draw[-] (0.25,-0) to[out=-90, in=0] (0,-0.25);
\draw[-] (0,-0.25) to[out = 180, in = -90] (-0.25,-0);
\draw[<-] (0.25,.6)\toplabel{i} to[out=240,in=90] (-0.25,-0);
\pin{.25,.03}{.85,.03}{p(x)};
\draw[gcolor,thick] (1.4,0.6)  -- (1.4,-0.25)\botlabel{L};
\notch[-45]{.15,.47};
\notch[45]{-.15,.47};
\end{tikzpicture}
\overset{\cref{junesun}}{\underset{\cref{heat}}{=}}
\left[
p(u)\OO_L(u)^{-1}\
\begin{tikzpicture}[anchorbase,scale=1.2,H]
\draw[<-] (0.3,.4) to[out=-90, in=0] (0.1,0);
\draw[-] (0.1,0) to[out = 180, in = -90] (-0.1,.4)\toplabel{i};
\circled{0.3,0.18}{u};
\draw[gcolor,thick] (.55,0.4)  -- (.55,0)\botlabel{L};
\end{tikzpicture}
\right]_{u^{-1}}
\\&=
\left[
q(u)\
\begin{tikzpicture}[anchorbase,scale=1.2,H]
\draw[<-] (0.3,.4) to[out=-90, in=0] (0.1,0);
\draw[-] (0.1,0) to[out = 180, in = -90] (-0.1,.4)\toplabel{i};
\circled{0.3,0.18}{u};
\draw[gcolor,thick] (.55,0.4)  -- (.55,0)\botlabel{L};
\end{tikzpicture}
\right]_{u^{-1}}
\stackrel{\cref{trick}}{=}\begin{tikzpicture}[anchorbase,scale=1.2,H]
\draw[<-] (0.3,.4) to[out=-90, in=0] (0.1,0);
\draw[-] (0.1,0) to[out = 180, in = -90] (-0.1,.4)\toplabel{i};
\pin{0.3,.18}{.8,.18}{q(x)};
\draw[gcolor,thick] (1.25,0.4)  -- (1.25,0)\botlabel{L};
\end{tikzpicture}
=\vec{\beta}\big(q(x)\big),\\
\vec{\sigma}\circ\cev{\beta}\big(p(x)\big)
&=\!
\begin{tikzpicture}[anchorbase,H]
\draw[<-] (0.35,.4) to[out=-90, in=0] (0.1,0) to [out=180,in=-90] (-.15,.4) \toplabel{i};
\rightbub{.15,-.4};
\pinpin{0,-.17}{.25,.05}{1,.05}{r(x,y)};
\pin{.34,-.5}{1,-.5}{p(x)};
\draw[gcolor,thick] (1.7,0.4)  -- (1.7,-0.7)\botlabel{L};
\end{tikzpicture}
\stackrel{\cref{tadpole}}{=}-\left[
p(u)\OO_L(u)^{-1}
\begin{tikzpicture}[anchorbase,H]
\draw[<-] (0.35,.4) to[out=-90, in=0] (0.1,0) to [out=180,in=-90] (-.15,.4) \toplabel{i};
\pin{.3,.15}{1,.15}{r(u,x)};
\draw[gcolor,thick] (1.7,0.4)  -- (1.7,-0.2)\botlabel{L};
\end{tikzpicture}
\right]_{u:-1}\!\!\!\!
=-\left[
q(u)
\begin{tikzpicture}[anchorbase,H]
\draw[<-] (0.35,.4) to[out=-90, in=0] (0.1,0) to [out=180,in=-90] (-.15,.4) \toplabel{i};
\pin{.31,.15}{1,.15}{r(u,x)};
\draw[gcolor,thick] (1.7,0.4)  -- (1.7,-0.2)\botlabel{L};
\end{tikzpicture}
\right]_{u:-1}\!\!=0.
\end{align*}

\vspace{2mm}
\noindent
\underline{Claim 3}: {\em We have that $\big(\vec{\rho} +\vec{\sigma}\big)\circ \cev{\rho} = \id_{Q_i P_i L} + \vec{\alpha}$ for some morphism $\vec{\alpha}:Q_i P_i L \rightarrow Q_i P_i L$ whose image is contained in $\vec{L}$.}
The composition $\big(\vec{\rho}+\vec{\sigma}\big) \circ \cev{\rho}$ equals
\begin{align*}
\begin{tikzpicture}[anchorbase,H]
\draw[-,gcolor,thick] (0.6,.6) to (0.6,-.6)\botlabel{L};
\draw[<-] (-0.23,-.6)\botlabel{i} to[out=90,in=-90] (0.23,0);
\draw[-] (0.23,0) to[out=90,in=-90] (-0.23,.6)\toplabel{i};
\draw[->] (-0.23,0) to[out=90,in=-90] (0.23,.6)\toplabel{i};
\draw[-] (0.23,-.6)\botlabel{i} to[out=90,in=-90] (-0.23,0);
\projcr{0,.3};
\projcr{0,-.3};
\node at (.33,0) {\strandlabel{i}};
\node at (-.33,0) {\strandlabel{i}};
\end{tikzpicture}
+
\begin{tikzpicture}[anchorbase,H]
\draw[-,gcolor,thick] (1.4,.6) to (1.4,-.6)\botlabel{L};
\draw[<-] (-0.23,-.6)\botlabel{i} to[out=90,in=-90] (0.23,0);
\draw[-] (0.23,0) to[out=90,in=0] (0,.2);
\draw[-] (-0.23,0) to[out=90,in=180] (0,.2);
\draw[-] (0.23,-.6)\botlabel{i} to[out=90,in=-90] (-0.23,0);
\projcr{0,-.3};
\draw[<-] (0.2,0.6) to[out=-90, in=0] (0,.3);
\draw[-] (0,.3) to[out = 180, in = -90] (-0.2,0.6)\toplabel{i};
\node at (.33,-.03) {\strandlabel{i}};
\pinpin{-.1,.18}{.1,.32}{.8,.32}{r(x,y)};
\end{tikzpicture}&\stackrel{\cref{tendays}}{=}
\begin{tikzpicture}[H,anchorbase,scale=1]
\draw[<-] (-0.2,-0.6) to (-0.2,0.6)\toplabel{j};
\draw[->] (0.2,-0.6)\botlabel{i} to (0.2,0.6);
\end{tikzpicture}
+\left[\!
\begin{tikzpicture}[H,anchorbase,scale=1]
\draw[->] (-0.2,0.6)\toplabel{i} to (-0.2,0.3) arc(180:360:0.2) to (0.2,0.6);
\draw[<-] (-0.2,-0.6) to (-0.2,-0.3) arc(180:0:0.2) to (0.2,-0.6)\botlabel{i};
\rightbubgen{1,0};
\circled{0.2,0.35}{u};
\circled{0.2,-0.35}{u};
\end{tikzpicture}\!\right]_{\!u:-1}
-\sum_{i \neq j \in \kk} \!\!\begin{tikzpicture}[H,anchorbase,scale=1]
\draw[->] (-0.2,0.6)\toplabel{i} to (-0.2,0.3) arc(180:360:0.2) to (0.2,0.6);
\draw[<-] (-0.2,-0.6) to (-0.2,-0.3) arc(180:0:0.2) to (0.2,-0.6)\botlabel{i};
\draw[->] (1,0)++(-0.25,0) arc(0:-360:0.2);
\node at (.45,.32) {\strandlabel{j}};
\pinPinpin{.19,.25}{.35,0}{.19,-.25}{-1.15,0}{\frac{1}{(x-z)(y-z)}};
\end{tikzpicture}
+
\begin{tikzpicture}[anchorbase,H]
\draw[-,gcolor,thick] (1.4,.6) to (1.4,-.6)\botlabel{L};
\draw[<-] (-0.23,-.6)\botlabel{i} to[out=90,in=-90] (0.23,0);
\draw[-] (0.23,0) to[out=90,in=0] (0,.2);
\draw[-] (-0.23,0) to[out=90,in=180] (0,.2);
\draw[-] (0.23,-.6)\botlabel{i} to[out=90,in=-90] (-0.23,0);
\projcr{0,-.3};
\draw[<-] (0.2,0.6) to[out=-90, in=0] (0,.3);
\draw[-] (0,.3) to[out = 180, in = -90] (-0.2,0.6)\toplabel{i};
\pinpin{-.1,.18}{.1,.32}{.8,.32}{r(x,y)};
\node at (.33,-0.03) {\strandlabel{i}};
\end{tikzpicture}\ .
\end{align*}
All of the morphisms on the right-hand side
have image
contained in $\vec{L}$ except for the first one, which is the desired identity.

\vspace{2mm}
\noindent
\underline{Claim 4}: {\em We have that $\cev{\rho}\circ \big(\vec{\rho} +\vec{\sigma}\big) = \id_{P_i Q_i L} + \cev{\alpha}$ for some morphism $\cev{\alpha}:P_i Q_i L \rightarrow P_i Q_i L$ whose image is contained in $\cev{L}$.}
This is a similar calculation to the one used to prove Claim 3.

\vspace{2mm}
\noindent
\underline{Claim 5}: {\em $\theta$ is an epimorphism.}
We have that $\vec{L} = \vec{L}_{\operatorname{lo}}\oplus\vec{L}_{\operatorname{hi}}$ where
\begin{align*}
\vec{L}_{\operatorname{lo}} &:=  \vec{\beta}\big(e_i\big)(L)+\cdots+ \vec{\beta}\big(x_i^{\minush-1} e_i\big)(L),\\
\vec{L}_{\operatorname{hi}} &:= \vec{\beta}\big((x-b(i))^\minush e(x)\big)(L)+\cdots+\vec{\beta}\big((x-b(i))^{\eps-1} e(x)\big)(L).
\end{align*}
Claim 2 implies that
$\vec{L}_{\operatorname{hi}} \leq
\big(\vec{\rho}+\vec{\sigma}\big)(P_i Q_i L)$.
Using this plus Claim 3 for the first containment, we deduce that
$$Q_i P_i L \leq \big(\vec{\rho}+\vec{\sigma}\big)(P_i Q_i L) + \vec{L}
= \big(\vec{\rho}+\vec{\sigma}\big)(P_i Q_i L)
+ \vec{L}_{\operatorname{hi}}+
\vec{L}_{\operatorname{lo}}
= \big(\vec{\rho}+\vec{\sigma}\big)(P_i Q_i L)
+\vec{L}_{\operatorname{lo}}
= \theta(P_i Q_i L).
$$

\vspace{2mm}
\noindent
\underline{Claim 6}: {\em $\theta$ is a monomorphism.}
Let
$\pr_1:P_i Q_i L \oplus L^{\oplus \minush}
\twoheadrightarrow P_i Q_i L$
and $\pr_2:P_i Q_i L \oplus L^{\oplus \minush}
\twoheadrightarrow L^{\oplus \minush}$ be the projections.
There is an isomorphism
$$
\vec\gamma := \begin{bmatrix}
\vec{\beta}\left(e_i\right)&
\vec{\beta}\left(x_ie_i\right)&\cdots&
\vec{\beta}\big(x_i^{\minush-1} e_i\big)\end{bmatrix}
: L^{\oplus \minush} \stackrel{\sim}{\rightarrow} \vec{L}_{\operatorname{lo}}.
$$
By Claim 2,
$\big(\vec{\rho}+\vec{\sigma}\big)|_{\cev{\scriptstyle L}}:\cev{L}
\stackrel{\sim}{\rightarrow} \vec{L}_{\operatorname{hi}}$
is an isomorphism.
Hence,
since
$\theta =
\big(\vec{\rho}+\vec{\sigma}\big)\circ\pr_1+
\vec{\gamma} \circ \pr_2$, the restriction
$\theta|_{\cev{\scriptstyle L}\oplus L^{\oplus \minush}}:\cev{L}\oplus L^{\oplus \minush} \stackrel{\sim}{\rightarrow}
\vec{L}$ is an isomorphism.
Therefore, to show that $\theta$ itself is a monomorphism, it is enough to show that
$\ker \theta \leq \cev{L}\oplus L^{\oplus \minush}$.
By Claim 4, we have that
$$
\cev{\rho}\circ\theta =
\cev{\rho} \circ \big(\vec{\rho}+\vec{\sigma}\big)
\circ \pr_1 +
\cev{\rho} \circ \vec{\gamma} \circ \pr_2
= \pr_1 + \cev{\alpha}\circ \pr_1
+ \cev{\rho} \circ \vec{\gamma} \circ \pr_2.
$$
Since the images of $\cev{\alpha} \circ\pr_1$
and $\cev{\rho}\circ \vec{\gamma}\circ\pr_2$
are both contained in $\cev{L}$, the latter following from Claim 1, we deduce that $\pr_1\big(\ker(\cev{\rho}\circ\theta)\big) \leq \cev{L}$ too.
We deduce that
$\ker \theta \leq \ker(\cev{\rho}\circ\theta)
\leq \cev{L}\oplus L^{\oplus \minush}$,
completing the proof.

\vspace{2mm}
\noindent
(2)
We adopt exactly the same setup as in the first paragraph of the proof of (1), now taking $i = 0$ everywhere, of course,
so that $b(i) = 0$.
Let $M := \Pi L \oplus L$.
Recalling the shorthand \cref{gnocchi},
we replace the injective even linear maps $\vec{\beta}$ and $\cev{\beta}$ from (1) with
the injective even linear maps
\begin{align*}
\vec{\beta}: A_0 &\hookrightarrow \Hom_{\catR}(M, Q_0 P_0 L),
& p(x) &\mapsto
\begin{bmatrix}
\begin{tikzpicture}[anchorbase,H]
\draw[<-] (0.4,.4) to[out=-90, in=0] (0.1,0);
\draw[-] (0.1,0) to[out = 180, in = -90] (-0.2,.4)\toplabel{0};
\pin{.28,.07}{1,.07}{p(x)};
\token{.37,.22};
\draw[gcolor,thick] (1.6,-0.2) \botlabel{\Pi L} -- (1.6,0.4)\toplabel{L};
\gnotch{1.6,0};
\end{tikzpicture}&
\begin{tikzpicture}[anchorbase,H]
\draw[<-] (0.4,.4) to[out=-90, in=0] (0.1,0);
\draw[-] (0.1,0) to[out = 180, in = -90] (-0.2,.4)\toplabel{0};
\pin{.36,.17}{1,.17}{p(x)};
\draw[gcolor,thick] (1.6,-0.2) \botlabel{L} -- (1.6,0.4);
\end{tikzpicture}\end{bmatrix},\\
\cev{\beta}:B_0 &\hookrightarrow \Hom_{\catR}(M, P_0 Q_0 L),
& p(x) &\mapsto
\begin{bmatrix}
\begin{tikzpicture}[anchorbase,H]
\draw[->] (0.4,.4)\toplabel{0} to[out=-90, in=0] (0.1,0) to [out=180,in=-90] (-.2,.4);
\pin{.39,.26}{1,.26}{p(x)};
\token{.28,.07};
\draw[gcolor,thick] (1.6,-0.2) \botlabel{\Pi L} -- (1.6,0.4)\toplabel{L};
\gnotch{1.6,-.05};
\end{tikzpicture}&
\begin{tikzpicture}[anchorbase,H]
\draw[-] (0.4,.4)\toplabel{0} to[out=-90, in=0] (0.1,0);
\draw[->] (0.1,0) to[out = 180, in = -90] (-0.2,.4);
\pin{.36,.18}{1,.18}{p(x)};
\draw[gcolor,thick] (1.6,-0.2) \botlabel{L} -- (1.6,0.4);
\end{tikzpicture}
\end{bmatrix}
\ .
\end{align*}
Again,
Schur's Lemma implies that the completely reducible subobject
$$
\vec{M} := \vec{\beta}\big(e_0\big)(M)
+
\cdots+\vec{\beta}\big(x_0^{\minush-1} e_i\big)(M)
+ \vec{\beta}\big(x^\minush e(x)\big)(M) +\cdots+\vec{\beta}\big(x)^{\eps-1} e(x)\big)(M)
\leq Q_i P_i L$$
is a direct sum of $\eps$ copies of $M$.
Similarly,
$$
\cev{M} := \cev{\beta}\big(f(x)\big)(M) +
\cdots+
\cev{\beta}\big((x-b(i))^{\phi-1} f(x)\big)(L)
\leq P_i Q_i L$$
is a direct sum of $\phi$ copies of $M$.
We define $\vec{\rho}$ and $\cev{\rho}$ as before
but modify the definition of $\vec{\sigma}$:
\begin{align*}
\vec{\rho}&:=\!
\begin{tikzpicture}[centerzero,H,scale=1.2]
\draw[->] (-.2,-.35)\botlabel{0} to[out=up,in=down,looseness=1] (.2,.35) \toplabel{0};
\draw[<-] (.2,-.35)\botlabel{0} to[out=up,in=down,looseness=1] (-.2,.35) \toplabel{0};
\projcr{0,0};
\draw[gcolor,thick] (0.45,-0.35) \botlabel{L} -- (0.45,0.35);
\end{tikzpicture}:P_0 Q_0 L \rightarrow Q_0 P_0 L,
&\cev{\rho}&:=\!
\begin{tikzpicture}[centerzero,H,scale=1.2]
\draw[<-] (-.2,-.35)\botlabel{0} to[out=up,in=down,looseness=1] (.2,.35) \toplabel{0};
\draw[->] (.2,-.35)\botlabel{0} to[out=up,in=down,looseness=1] (-.2,.35) \toplabel{0};
\projcr{0,0};
\draw[gcolor,thick] (0.45,-0.35) \botlabel{L} -- (0.45,0.35);
\end{tikzpicture}
:Q_0 P_0 L \rightarrow P_0 Q_0 L,
\end{align*}
\begin{align*}
\vec{\sigma}
&:=
\begin{tikzpicture}[centerzero,H,scale=1.2]
\draw[->] (-.2,.4) \toplabel{0} to[out=down,in=down,looseness=2.5] (.2,.4);
\draw[->] (-.2,-.3) \botlabel{0} to[out=up,in=up,looseness=2.5] (.2,-.3);
\pinpin{-.1,-.04}{.1,.13}{.7,.13}{r(x,y)};
\draw[gcolor,thick] (1.27,-0.3) \botlabel{L} -- (1.27,0.4);
\end{tikzpicture}+ \begin{tikzpicture}[centerzero,H,scale=1.3]
\draw[->] (-.2,.4) \toplabel{0} to[out=down,in=down,looseness=2.5] (.2,.4);
\draw[->] (-.2,-.3) \botlabel{0} to[out=up,in=up,looseness=2.5] (.2,-.3);
\token{-.19,-.18};
\token{.19,.27};
\pinpin{-.1,-.04}{.1,.13}{.7,.13}{s(x,y)};
\draw[gcolor,thick] (1.2,-0.3) \botlabel{L} -- (1.2,0.4);
\end{tikzpicture}: P_0 Q_0 L \rightarrow Q_0 P_0 L.
\end{align*}
The goal is to prove that the morphism
\begin{equation*}
\theta := \left[
\vec{\rho}+\vec{\sigma}\qquad
\vec{\beta}\left(e_0\right) \qquad
\vec{\beta}\left(x_0 e_0\right)\quad\cdots\quad
\vec{\beta}\big(x_0^{\minush-1} e_0\big)
\right]
:P_0 Q_0 L \oplus M^{\oplus \minush} \rightarrow Q_0 P_0 L
\end{equation*}
is an isomorphism.
This follows from a series of claims which are similar to the ones in \cref{twp1}.

\vspace{2mm}
\noindent
\underline{Claim 1$'$}:
{\em $\cev{\rho}(\vec{M}) \leq \cev{M}$.}
Take any $p(x) \in A_0$ and let
$q(x) := \left[p(u)\OO_L(u) (u-x)^{-1}\right]_{u:-1} \in \kk[x]$.
The calculation from the proof of Claim 1 above is exactly what is needed to see that $\cev{\rho}$ composed with
the second entry of the matrix
$\vec\beta(p(x))$ has image contained in $\cev{M}$.
The following analogous calculation does the job for the first entry of $\vec\beta(p(x))$:
$$
\begin{tikzpicture}[anchorbase,H,scale=1.1]
\draw[-] (0.25,0) to[out=-90, in=0] (0,-0.25);
\draw[-] (0.25,.6)\toplabel{0} to[out=240,in=90] (-0.25,0);
\draw[-] (0,-0.25) to[out = 180, in = -90] (-0.25,0);
\draw[<-] (-0.25,.6)\toplabel{0} to[out=300,in=90] (0.25,0);
\node at (-0.3,.2) {\strandlabel{0}};
\pin{.24,-.04}{.85,-.04}{p(x)};
\token{.22,0.13};
\projcr{0,.35};
\gnotch{1.4,0};
\draw[gcolor,thick] (1.4,0.6)\toplabel{L}  -- (1.4,-0.25)\botlabel{\Pi L};
\end{tikzpicture}
\stackrel{\cref{tie}}{=}
\begin{tikzpicture}[anchorbase,H,scale=1.1]
\draw[-] (0.25,-0) to[out=-90, in=0] (0,-0.25);
\draw[-] (0,-0.25) to[out = 180, in = -90] (-0.25,-0);
\draw[-] (0.25,.6)\toplabel{0} to[out=240,in=90] (-0.25,-0);
\draw[<-] (-0.25,.6)\toplabel{0}to[out=300,in=90] (0.25,-0);
\notch[50]{-.07,.4};
\token{-.15,.47};
\notch[-45]{.15,.47};
\pin{.26,0}{.85,0}{p(x)};
\draw[gcolor,thick] (1.4,0.6)\toplabel{L}  -- (1.4,-0.25)\botlabel{\Pi L};
\gnotch{1.4,0};
\end{tikzpicture}
\overset{\cref{junesun}}{\underset{\cref{heat}}{=}}
\left[
p(u)\OO_L(u)\
\begin{tikzpicture}[anchorbase,scale=1.3,H]
\draw[-] (0.3,.4)\toplabel{0} to[out=-90, in=0] (0.1,0);
\draw[->] (0.1,0) to[out = 180, in = -90] (-0.1,.4);
\circled{0.26,0.13}{u};
\token{-.09,.25};
\draw[gcolor,thick] (.55,0.4)\toplabel{L}  -- (.55,-.05)\botlabel{\Pi L};
\gnotch{.55,.05};
\end{tikzpicture}
\right]_{u^{-1}}
=
\begin{tikzpicture}[anchorbase,scale=1.3,H]
\draw[-] (0.3,.4)\toplabel{i} to[out=-90, in=0] (0.1,0);
\draw[->] (0.1,0) to[out = 180, in = -90] (-0.1,.4);
\pin{0.3,.3}{.8,.3}{q(x)};
\token{.28,.15};
\draw[gcolor,thick] (1.25,0.4)\toplabel{L}  -- (1.25,-.05)\botlabel{\Pi L};
\gnotch{1.25,.05};
\end{tikzpicture}\ .
$$
The image of this morphism is contained in
$\cev{M}$.

\vspace{2mm}
\noindent
\underline{Claim 2$'$}:
{\em We have that $\vec{\beta}\big(x^{\minush+r} e(x)\big)
=\big(\vec{\rho}+\vec{\sigma}\big)\circ \cev{\beta}\big(x^r f(x)\big)
$ for any $r \geq 0$.}
Let
$p(x) := x^r f(x) \in B_0$
and $q(x) := x^{\minush+r} e(x) \in A_0$.
We prove that $\vec{\rho}\circ \cev{\beta}(p(x)) = \vec{\beta}(q(x))$ and $\vec{\sigma} \circ \cev{\beta}(p(x)) = 0$.
Again, we apply $\vec{\rho}$ and $\vec{\sigma}$ to the second and first entries
of the matrix $\cev{\beta}(p(x))$ separately.
For the second entry,
the two calculations made in the proof of Claim 2
together with the fact that any odd bubble is 0
does the job. To see the necessary for the first entry, we need two more calculations,
also using that any odd bubble is 0 one more time:
\begin{align*}
\begin{tikzpicture}[anchorbase,H,scale=1.1]
\draw[-] (-0.25,.6)\toplabel{0} to[out=300,in=90] (0.25,-0);
\draw[-] (0.25,-0) to[out=-90, in=0] (0,-0.25) to [out=180,in=-90] (-.25,0);
\draw[<-] (0.25,.6)\toplabel{0} to[out=240,in=90] (-0.25,-0);
\node at (-0.3,.2) {\strandlabel{0}};
\pin{.21,.18}{.85,.18}{p(x)};
\token{.25,0};
\projcr{0,.35};
\gnotch{1.4,-.13};
\draw[gcolor,thick] (1.4,0.6)\toplabel{L}  -- (1.4,-0.25)\botlabel{\Pi L};
\end{tikzpicture}\!\!
&\stackrel{\cref{tie}}{=}
\begin{tikzpicture}[anchorbase,H,scale=1.1]
\draw[-] (-0.25,.6)\toplabel{i} to[out=300,in=90] (0.25,-0);
\draw[-] (0.25,-0) to[out=-90, in=0] (0,-0.25);
\draw[-] (0,-0.25) to[out = 180, in = -90] (-0.25,-0);
\draw[<-] (0.25,.6)\toplabel{0} to[out=240,in=90] (-0.25,-0);
\notch[-50]{.07,.4};
\token{.16,.47};
\pin{.25,.03}{.85,.03}{p(x)};
\draw[gcolor,thick] (1.4,0.6)\toplabel{L}  -- (1.4,-0.25)\botlabel{\Pi L};
\gnotch{1.4,-.05};
\notch[45]{-.15,.47};
\end{tikzpicture}
\overset{\cref{junesun}}{\underset{\cref{heat}}{=}}
\left[
p(u)\OO_L(u)^{-1}\
\begin{tikzpicture}[anchorbase,scale=1.2,H]
\draw[<-] (0.3,.4) to[out=-90, in=0] (0.1,0);
\draw[-] (0.1,0) to[out = 180, in = -90] (-0.1,.4)\toplabel{0};
\token{.29,.26};
\circled{0.24,0.06}{u};
\draw[gcolor,thick] (.55,0.4)\toplabel{L}  -- (.55,-.05)\botlabel{\Pi L};
\gnotch{.55,.05};
\end{tikzpicture}
\right]_{u^{-1}}\!\!\!\!\!
=
\left[
q(u)\
\begin{tikzpicture}[anchorbase,scale=1.2,H]
\draw[<-] (0.3,.4) to[out=-90, in=0] (0.1,0);
\draw[-] (0.1,0) to[out = 180, in = -90] (-0.1,.4)\toplabel{0};
\token{.29,.26};
\circled{0.24,0.06}{u};
\gnotch{.55,.05};
\draw[gcolor,thick] (.55,0.4)\toplabel{L}  -- (.55,-0.05)\botlabel{\Pi L};
\end{tikzpicture}
\right]_{u^{-1}}\!\!\!\!\!\!
\stackrel{\cref{trick}}{=}
\begin{tikzpicture}[anchorbase,scale=1.2,H]
\draw[<-] (0.3,.4) to[out=-90, in=0] (0.1,0);
\draw[-] (0.1,0) to[out = 180, in = -90] (-0.1,.4)\toplabel{0};
\token{.3,.25};
\pin{0.26,.1}{.78,.1}{q(x)};
\gnotch{1.25,0.05};
\draw[gcolor,thick] (1.25,0.4)\toplabel{L}  -- (1.25,-.05)\botlabel{\Pi L};
\end{tikzpicture},\\
\begin{tikzpicture}[anchorbase,H,scale=1.1]
\draw[<-] (0.35,.4) to[out=-90, in=0] (0.1,0) to [out=180,in=-90] (-.15,.4) \toplabel{0};
\rightbub{.15,-.4};
\pinpin{0,-.17}{.25,.05}{1,.05}{s(x,y)};
\token{.34,.22};
\pin{.33,-.5}{1,-.5}{p(x)};
\draw[gcolor,thick] (1.7,0.4)\toplabel{L}  -- (1.7,-0.7)\botlabel{\Pi L};
\gnotch{1.7,-.5};
\end{tikzpicture}\!\!
&\stackrel{\cref{tadpole}}{=}-\left[
p(u)\OO_L(u)^{-1}\!
\begin{tikzpicture}[anchorbase,scale=1.2,H]
\draw[<-] (0.3,.4) to[out=-90, in=0] (0.1,0);
\draw[-] (0.1,0) to[out = 180, in = -90] (-0.1,.4)\toplabel{0};
\token{.3,.25};
\pin{0.26,.1}{.9,.1}{s(u,x)};
\gnotch{1.5,0};
\draw[gcolor,thick] (1.5,0.4)\toplabel{L}  -- (1.5,-.1)\botlabel{\Pi L};
\end{tikzpicture}
\right]_{u:-1}\!\!\!\!\!
=-\left[
q(u)\!
\begin{tikzpicture}[anchorbase,scale=1.2,H]
\draw[<-] (0.3,.4) to[out=-90, in=0] (0.1,0);
\draw[-] (0.1,0) to[out = 180, in = -90] (-0.1,.4)\toplabel{0};
\token{.3,.25};
\pin{0.26,.1}{.9,.1}{s(u,x)};
\gnotch{1.5,0};
\draw[gcolor,thick] (1.5,0.4)\toplabel{L}  -- (1.5,-.1)\botlabel{\Pi L};
\end{tikzpicture}
\right]_{u:-1}\!\!=0.
\end{align*}

\vspace{2mm}
\noindent
\underline{Claim 3$'$}: {\em We have that $\big(\vec{\rho} +\vec{\sigma}\big)\circ \cev{\rho} = \id_{Q_0 P_0 L} + \vec{\alpha}$ for some morphism $\vec{\alpha}:Q_0 P_0 L \rightarrow Q_0 P_0 L$ whose image is contained in $\vec{M}$.}
This follows by almost the same calculation as was used to prove Claim 3. There are some extra terms arising from the $\delta_{i=-j}$ part of \cref{tendays}, and there is one more term
$$
\begin{tikzpicture}[anchorbase,H,scale=1.2]
\draw[-,gcolor,thick] (1.4,.6) to (1.4,-.6)\botlabel{L};
\draw[<-] (-0.23,-.6)\botlabel{0} to[out=90,in=-90] (0.23,0);
\draw[-] (0.23,0) to[out=90,in=0] (0,.2);
\draw[-] (-0.23,0) to[out=90,in=180] (0,.2);
\draw[-] (0.23,-.6)\botlabel{0} to[out=90,in=-90] (-0.23,0);
\projcr{0,-.3};
\token{.18,.44};\token{-.22,.05};
\draw[<-] (0.2,0.6) to[out=-90, in=0] (0,.3);
\draw[-] (0,.3) to[out = 180, in = -90] (-0.2,0.6)\toplabel{0};
\pinpin{-.1,.18}{.1,.32}{.8,.32}{s(x,y)};
\node at (.33,-0.03) {\strandlabel{0}};
\end{tikzpicture}
$$
coming from the extra term in the definition of $\vec{\sigma}$ compared to earlier.  All of the extra terms have image contained in $\vec{M}$ so the argument goes through as before.

\vspace{2mm}
\noindent
\underline{Claim 4$'$}: {\em We have that $\cev{\rho}\circ \big(\vec{\rho} +\vec{\sigma}\big) = \id_{P_0 Q_0 L} + \cev{\alpha}$ for some morphism $\cev{\alpha}:P_0 Q_0 L \rightarrow P_0 Q_0 L$ whose image is contained in $\cev{M}$.}
Similar to Claim 3$'$.

\vspace{2mm}
\noindent
\underline{Claim 5$'$}: {\em $\theta$ is an epimorphism.}
We just have to repeat the argument used to prove Claim 5 earlier with minor modifications:
we have that $\vec{M} = \vec{M}_{\operatorname{lo}}\oplus\vec{M}_{\operatorname{hi}}$ where
\begin{align*}
\vec{M}_{\operatorname{lo}} &:=  \vec{\beta}\big(e_0\big)(M)+\cdots+ \vec{\beta}\big(x_0^{\minush-1} e_0\big)(M),\\
\vec{M}_{\operatorname{hi}} &:= \vec{\beta}\big(x^\minush e(x)\big)(M)+\cdots+\vec{\beta}\big(x^{\eps-1} e(x)\big)(L).
\end{align*}
Claim 2$'$ implies that
$\vec{M}_{\operatorname{hi}} \leq
\big(\vec{\rho}+\vec{\sigma}\big)(P_0 Q_0 L)$.
Using this plus Claim 3$'$ for the first containment, we deduce that
$$Q_0 P_0 L \leq \big(\vec{\rho}+\vec{\sigma}\big)(P_0 Q_0 L) + \vec{M}
= \big(\vec{\rho}+\vec{\sigma}\big)(P_0 Q_0 L)
+ \vec{M}_{\operatorname{hi}}+
\vec{M}_{\operatorname{lo}}
= \big(\vec{\rho}+\vec{\sigma}\big)(P_0 Q_0 L)
+\vec{M}_{\operatorname{lo}}
= \theta(P_0 Q_0 L).
$$

\vspace{2mm}
\noindent
\underline{Claim 6$'$}: {\em $\theta$ is a monomorphism.}
This follows by a similarly modified version of the proof of Claim 6 earlier.

\vspace{2mm}
\noindent
(3),(4)
These follow from \cref{twp1,twp2} by an argument involving the Chevalley involution $\tT$ from \cref{mirror}; see the similar proof of \cite[Lem.~4.10]{BSW-HKM}.
\end{proof}

\setcounter{section}{4}
\section{Isomeric Kac--Moody categorifications}\label{s5-ikm}

Next, we introduce the isomeric Kac--Moody 2-category $\fV(\fg)$,
and the notion of an isomeric Kac--Moody categorification.
As will be explained in Part II, 
isomeric Kac--Moody 2-categories are closely related to the 
super Kac--Moody 2-categories $\fU(\fg)$ of \cite{BE-SKM}.
The definition of
$\fV(\fg)$ involves defining relations of the quiver Hecke--Clifford superalgebras from \cite{KKT16}, whereas $\fU(\fg)$ involves relations of quiver Hecke superalgebras.
We will also record some consequences of the defining relations of $\fV(\fg)$, but omit the proofs since the arguments used to derive them are very similar to the arguments 
in \cite{B-KM,BE-SKM,Sav19}.
Then, in \cref{s6-iheis2ikm}, we will show that any isomeric Heisenberg categorification can be made into an isomeric Kac--Moody categorification for the 
particular super Cartan datum defined in \cref{seccd}.

\subsection{Parameters}\label{data}

Let $(c_{i,j})_{i,j \in I}$ be a Cartan matrix
symmetrized by $(d_i)_{i \in I}$, with parity function $\p:I \rightarrow \Z/2$ satisfying \cref{constraint}. Fix also
a realization in the sense of \cref{seccd}.
Let $\fg$ be the Kac--Moody algebra associated to this Cartan datum.
It will not play any direct role in this paper, but it is used in our
notation $\fV(\fg)$ for the the isomeric Kac--Moody 2-category.

We need a matrix of parameters
$Q = (Q_{ij}(x,y))_{i,j \in I}$ such that $Q_{ii}(x,y)=0$,
and the following hold when $i \neq j$:
\begin{itemize}
\item $Q_{ij}(x,y)=Q_{ji}(y,x)$ is a homogeneous
polynomial in $\kk[x,y]$ of degree $-2d_i c_{ij}$
when $x$ is of degree $2d_i$ and $y$ is of degree $2d_j$.
\item $\p(i) = \1\Rightarrow Q_{ij}(x,y) \in \kk[x^2,y]$ (this
is only possible because $c_{ij}$ is even).
\item $Q_{ij}(1,0) \in \kk^\times$.
\end{itemize}
We also let 
\begin{align}
t_{ij} &:= \begin{dcases}
Q_{ij}(1,0)\hspace{23.5mm}&\text{if $i \neq j$}\\
1&\text{if $i=j$,}
\end{dcases}\\\intertext{then define rational functions
$R_{ij}(x,y) \in \kk(x,y)$ by}
R_{ij}(x,y) &:=
\begin{dcases}
\frac{Q_{ij}(x,y)}{t_{ij}}&\text{if $i \neq j$}\\
\frac{1}{(x-y)^2}&\text{if $i=j$ and $\p(i) = \0$}\\
\frac{1}{2(x-y)^2}+
\frac{1}{2(x+y)^2}
&\text{if $i=j$ and $\p(i) = \1$.}
\end{dcases}
\end{align}

\subsection{Definition of isomeric Kac--Moody 2-category}\label{ikmdef}

The \emph{isomeric Kac--Moody 2-category} $\fV(\fg)$
is the 2-supercategory with objects $X$, generating $1$-morphisms
$P_i \one_\lambda =\one_{\lambda+\alpha_i} P_i : \lambda \to \lambda + \alpha_i$ and
$\one_{\lambda-\alpha_i} Q_i=  Q_i \one_{\lambda} : \lambda \to \lambda - \alpha_i$
$(i \in I,\ \lambda \in X)$,
whose identity 2-endomorphisms are denoted by
$\begin{tikzpicture}[IKM,anchorbase]
\draw[->] (0,-0.2)\botlabel{i} -- (0,0.2);
\region{0.2,0}{\lambda};
\end{tikzpicture}$ and
$\begin{tikzpicture}[IKM,anchorbase]
\draw[<-] (0,-0.2) \botlabel{i}-- (0,0.2);
\region{0.2,0.1}{\lambda};
\end{tikzpicture}$,
and generating 2-morphisms
\begin{align*}
\begin{tikzpicture}[IKM,centerzero]
\draw[->] (0,-0.3) \botlabel{k} -- (0,0.3)\toplabel{k};
\token{0,0};
\region{0.25,0}{\lambda};
\end{tikzpicture}
&: P_k \one_\lambda \Rightarrow P_{k} \one_\lambda,&
\begin{tikzpicture}[IKM,centerzero]
\draw[->] (0,-0.3) \botlabel{i} -- (0,0.3)\toplabel{i};
\singdot{0,0};
\region{0.25,0}{\lambda};
\end{tikzpicture}
&: P_i \one_\lambda \Rightarrow P_i \one_\lambda,
&
\begin{tikzpicture}[IKM,centerzero]
\draw[->] (-0.3,-0.3) \botlabel{i} -- (0.3,0.3)\toplabel{i};
\draw[->] (0.3,-0.3) \botlabel{j} -- (-0.3,0.3)\toplabel{j};
\region{0.4,0}{\lambda};
\end{tikzpicture}
&: P_{i} P_{j} \one_\lambda \Rightarrow P_j P_i \one_\lambda,
\end{align*}

\vspace{-5mm}

\begin{align*}
\begin{tikzpicture}[IKM,centerzero]
\draw[->] (-0.25,0.25) \toplabel{i} -- (-0.25,0) arc(180:360:0.25) -- (0.25,0.25)\toplabel{i};
\region{0.45,0}{\lambda};
\end{tikzpicture}
&: \one_\lambda \Rightarrow Q_i P_i \one_\lambda,&
\begin{tikzpicture}[IKM,centerzero]
\draw[->] (-0.25,-0.25) \botlabel{i} -- (-0.25,0) arc(180:0:0.25) -- (0.25,-0.25)\botlabel{i};
\region{0.45,0}{\lambda};
\end{tikzpicture}
&: P_i Q_i \one_\lambda \Rightarrow \one_\lambda,
\end{align*}
for all $\lambda \in X$, $i,j,k \in I$ with $\p(k)=\1$.
The $\Z/2$-grading on 2-morphisms is defined so that the generating 2-morphisms represented by the solid dots, which we call {\em Clifford tokens},
are odd, and all of the
other generating 2-morphisms are even.

From now on, we will only write the string label strings at one place on the string, and we may omit 2-cell labels
when writing something which is true for all possible labels.
Also, whenever a string is decorated with a Clifford token, it is implicit that the string label is odd so that it makes sense.
Like in \cref{hrightpivot}, we use the following
to denote the composite 2-morphisms obtained by ``rotating'' the generating 2-morphisms:
\begin{align}
\begin{tikzpicture}[IKM,centerzero,scale=1.1]
\draw[<-] (0,-0.4) \botlabel{i}-- (0,0.4);
\token{0,0};
\end{tikzpicture}
&:=
\begin{tikzpicture}[IKM,centerzero,scale=1.1]
\draw[<-] (0.3,-0.4) \botlabel{i}-- (0.3,0) arc(0:180:0.15) arc(360:180:0.15) -- (-0.3,0.4);
\token{0,0};
\end{tikzpicture}\ ,&       
\begin{tikzpicture}[IKM,centerzero,scale=1.1]
\draw[<-] (0,-0.4)\botlabel{i} -- (0,0.4);
\singdot{0,0};
\end{tikzpicture}
&:=
\begin{tikzpicture}[IKM,centerzero,scale=1.1]
\draw[<-] (0.3,-0.4) \botlabel{i}-- (0.3,0) arc(0:180:0.15) arc(360:180:0.15) -- (-0.3,0.4);
\singdot{0,0};
\end{tikzpicture}\ ,&
\begin{tikzpicture}[IKM,centerzero]
\draw[<-] (0.3,-0.3)\botlabel{j}  -- (-0.3,0.3);
\draw[->] (-0.3,-0.3) \botlabel{i} -- (0.3,0.3);
\end{tikzpicture}
&:=
\begin{tikzpicture}[IKM,centerzero,scale=1.2]
\draw[->] (0.1,-0.3) \botlabel{i} \braidup (-0.1,0.3);
\draw[->] (-0.4,0.3) -- (-0.4,0.1) to[out=down,in=left] (-0.2,-0.2) to[out=right,in=left] (0.2,0.2) to[out=right,in=up] (0.4,-0.1) -- (0.4,-0.3)\botlabel{j};
\end{tikzpicture}\ ,&
\begin{tikzpicture}[IKM,centerzero]
\draw[<-] (0.3,-0.3) \botlabel{j}-- (-0.3,0.3) ;
\draw[<-] (-0.3,-0.3) \botlabel{i}-- (0.3,0.3) ;
\end{tikzpicture}
&:=
\begin{tikzpicture}[IKM,centerzero,scale=1.2]
\draw[<-] (0.1,-0.3) \botlabel{i}\braidup (-0.1,0.3);
\draw[->] (-0.4,0.3)  -- (-0.4,0.1) to[out=down,in=left] (-0.2,-0.2) to[out=right,in=left] (0.2,0.2) to[out=right,in=up] (0.4,-0.1) -- (0.4,-0.3)\botlabel{j};
\end{tikzpicture}
\ .\label{irightpivot}
\end{align}
We will use the pin notation and \cref{jonisdotty} just like \cref{singlepin,doublepin}.
There are four families of relations.
First, we have the \emph{zig-zag relations} for all $\lambda \in X$ and $i \in I$:
\begin{align}\label{KMrightadj}
\begin{tikzpicture}[IKM,centerzero,scale=1.2]
\draw[->] (-0.3,0.4) -- (-0.3,0) arc(180:360:0.15) arc(180:0:0.15) -- (0.3,-0.4) \botlabel{i};
\end{tikzpicture}
&=
\begin{tikzpicture}[IKM,centerzero,scale=1.2]
\draw[<-] (0,-0.4) \botlabel{i} -- (0,0.4);
\end{tikzpicture}
\ ,&
\begin{tikzpicture}[IKM,centerzero,scale=1.2]
\draw[->] (-0.3,-0.4) \botlabel{i}-- (-0.3,0) arc(180:0:0.15) arc(180:360:0.15) -- (0.3,0.4);
\end{tikzpicture}
&=
\begin{tikzpicture}[IKM,centerzero,scale=1.2]
\draw[->] (0,-0.4) \botlabel{i} -- (0,0.4);
\end{tikzpicture}\ .
\end{align}
Next, the \emph{quiver Hecke--Clifford superalgebra} relations:
\begin{align} \label{QHC1}
\begin{tikzpicture}[IKM,centerzero]
\draw[->] (0,-0.3)\botlabel{i} -- (0,0.3);
\token{0,-0.1};
\token{0,0.1};
\end{tikzpicture}
&= -\ 
\begin{tikzpicture}[IKM,centerzero]
\draw[->] (0,-0.3)\botlabel{i} -- (0,0.3);
\end{tikzpicture}
\ ,&
\begin{tikzpicture}[IKM,centerzero,scale=1.1]
\draw[->] (0,-0.3)\botlabel{i} -- (0,0.3);
\token{0,-0.1};
\singdot{0,0.1};
\end{tikzpicture}
&= -\,
\begin{tikzpicture}[IKM,centerzero,scale=1.1]
\draw[->] (0,-0.3)\botlabel{i} -- (0,0.3);
\token{0,0.1};
\singdot{0,-0.1};
\end{tikzpicture}
\ ,\\ \label{QHC2}
\begin{tikzpicture}[IKM,centerzero]
\draw[->] (0.3,-0.3)\botlabel{j} -- (-0.3,0.3);
\draw[->] (-0.3,-0.3)\botlabel{i} -- (0.3,0.3);
\token{-0.15,-0.15};
\end{tikzpicture}
&=
\begin{tikzpicture}[IKM,centerzero]
\draw[->] (0.3,-0.3)\botlabel{j} -- (-0.3,0.3);
\draw[->] (-0.3,-0.3)\botlabel{i} -- (0.3,0.3);
\token{0.15,0.15};
\end{tikzpicture}
\ ,
&
\begin{tikzpicture}[IKM,centerzero]
\draw[->] (0.3,-0.3)\botlabel{j} -- (-0.3,0.3);
\draw[->] (-0.3,-0.3)\botlabel{i} -- (0.3,0.3);
\token{-0.15,0.15};
\end{tikzpicture}
&=
\begin{tikzpicture}[IKM,centerzero]
\draw[->] (0.3,-0.3)\botlabel{j} -- (-0.3,0.3);
\draw[->] (-0.3,-0.3)\botlabel{i} -- (0.3,0.3);
\token{0.15,-0.15};
\end{tikzpicture}
\ ,
\end{align}
\begin{align} \label{QHC3a}
\begin{tikzpicture}[IKM,centerzero]
\draw[->] (-0.3,-0.3) \botlabel{i} -- (0.3,0.3);
\draw[->] (0.3,-0.3) \botlabel{j} -- (-0.3,0.3);
\singdot{-0.15,-0.15};
\end{tikzpicture}
-
\begin{tikzpicture}[IKM,centerzero]
\draw[->] (-0.3,-0.3) \botlabel{i} -- (0.3,0.3);
\draw[->] (0.3,-0.3) \botlabel{j} -- (-0.3,0.3);
\singdot{0.15,0.15};
\end{tikzpicture}
&=
\begin{dcases}
\begin{tikzpicture}[IKM,centerzero]
\draw[->] (-0.2,-0.3) \botlabel{i} -- (-0.2,0.3);
\draw[->] (0.2,-0.3) \botlabel{i} -- (0.2,0.3);
\end{tikzpicture}&\text{if $i=j$, $\p(i)=\0$}\\
\begin{tikzpicture}[IKM,centerzero]
\draw[->] (-0.2,-0.3) \botlabel{i} -- (-0.2,0.3);
\draw[->] (0.2,-0.3) \botlabel{i} -- (0.2,0.3);
\end{tikzpicture}\ -\
\begin{tikzpicture}[IKM,centerzero]
\draw[->] (-0.2,-0.3) \botlabel{i} -- (-0.2,0.3);
\draw[->] (0.2,-0.3) \botlabel{i} -- (0.2,0.3);
\token{-.2,0};\token{.2,0};
\end{tikzpicture}\:\:
&\text{if $i=j$, $\p(i)=\1$}
\\
0&\text{if $i \neq j$},
\end{dcases}
\\ \label{QHC3b}
\begin{tikzpicture}[IKM,centerzero]
\draw[->] (-0.3,-0.3) \botlabel{i} -- (0.3,0.3);
\draw[->] (0.3,-0.3) \botlabel{j} -- (-0.3,0.3);
\singdot{-0.15,0.15};
\end{tikzpicture}
-
\begin{tikzpicture}[IKM,centerzero]
\draw[->] (-0.3,-0.3) \botlabel{i} -- (0.3,0.3);
\draw[->] (0.3,-0.3) \botlabel{j} -- (-0.3,0.3);
\singdot{0.15,-0.15};
\end{tikzpicture}
&=
\begin{dcases}
\begin{tikzpicture}[IKM,centerzero]
\draw[->] (-0.2,-0.3) \botlabel{i} -- (-0.2,0.3);
\draw[->] (0.2,-0.3) \botlabel{i} -- (0.2,0.3);
\end{tikzpicture}
&\text{if $i=j$, $\p(i)=\0$}
\\
\begin{tikzpicture}[IKM,centerzero]
\draw[->] (-0.2,-0.3) \botlabel{i} -- (-0.2,0.3);
\draw[->] (0.2,-0.3) \botlabel{i} -- (0.2,0.3);
\end{tikzpicture}
\ +\
\begin{tikzpicture}[IKM,centerzero]
\draw[->] (-0.2,-0.3) \botlabel{i} -- (-0.2,0.3);
\draw[->] (0.2,-0.3) \botlabel{i} -- (0.2,0.3);
\token{-.2,0};\token{.2,0};
\end{tikzpicture}\:\:
&\text{if $i=j$, $\p(i)=\1$}
\\
0 & \text{if $i \neq j$},
\end{dcases}
\\ \label{QHC4}
\begin{tikzpicture}[IKM,centerzero]
\draw[->] (-0.2,-0.4) \botlabel{i} to[out=45,in=down] (0.15,0) to[out=up,in=-45] (-0.2,0.4);
\draw[->] (0.2,-0.4) \botlabel{j} to[out=135,in=down] (-0.15,0) to[out=up,in=225] (0.2,0.4);
\end{tikzpicture}
&=
\begin{tikzpicture}[IKM,centerzero]
\draw[->] (-0.2,-0.3) \botlabel{i} -- (-0.2,0.3);
\draw[->] (0.2,-0.3) \botlabel{j} -- (0.2,0.3);
\pinpin{0.2,0}{-0.2,0}{-1.2,0}{Q_{ij}(x,y)};
\end{tikzpicture}\ ,
\end{align}
\begin{equation} \label{QHC5}
\begin{tikzpicture}[IKM,centerzero,scale=1.1]
\draw[->] (-0.4,-0.4) \botlabel{i} -- (0.4,0.4);
\draw[->] (0,-0.4) \botlabel{j} to[out=135,in=down] (-0.32,0) to[out=up,in=225] (0,0.4);
\draw[->] (0.4,-0.4) \botlabel{k} -- (-0.4,0.4);
\end{tikzpicture}
\ -\
\begin{tikzpicture}[IKM,centerzero,scale=1.1]
\draw[->] (-0.4,-0.4) \botlabel{i} -- (0.4,0.4);
\draw[->] (0,-0.4) \botlabel{j} to[out=45,in=down] (0.32,0) to[out=up,in=-45] (0,0.4);
\draw[->] (0.4,-0.4) \botlabel{k} -- (-0.4,0.4);
\end{tikzpicture}
=
\begin{dcases}
\begin{tikzpicture}[IKM,centerzero,scale=1.2]
\draw[->] (-0.3,-0.3) \botlabel{i} -- (-0.3,0.3);
\draw[->] (0,-0.3) \botlabel{j} -- (0,0.3);
\draw[->] (0.3,-0.3) \botlabel{i} -- (0.3,0.3);
\pinpinpin{.3,0}{0,0}{-.3,0}{-1.6,0}{
\frac{Q_{ij}(x,y)-Q_{ij}(z,y)}{x-z}};
\end{tikzpicture}
& \text{if $i=k$, $\p(i)=\0$}
\\
\begin{tikzpicture}[IKM,centerzero,scale=1.2]
\draw[->] (-0.3,-0.3) \botlabel{i} -- (-0.3,0.3);
\draw[->] (0,-0.3) \botlabel{j} -- (0,0.3);
\draw[->] (0.3,-0.3) \botlabel{i} -- (0.3,0.3);
\pinpinpin{.3,0}{0,0}{-.3,0}{-1.6,0}{
\frac{Q_{ij}(x,y)-Q_{ij}(z,y)}{x-z}};
\end{tikzpicture}
\!-\
\begin{tikzpicture}[IKM,centerzero,scale=1.2]
\draw[->] (-0.3,-0.3) \botlabel{i} -- (-0.3,0.3);
\draw[->] (0,-0.3) \botlabel{j} -- (0,0.3);
\draw[->] (0.3,-0.3) \botlabel{i} -- (0.3,0.3);
\pinpinpin{.3,.05}{0,-.05}{-.3,-.15}{-1.6,-.15}{
\frac{Q_{ij}(x,y) - Q_{ij}(z,y)}{x-z}};
\token{.3,-.15};\token{-.3,.05};
\end{tikzpicture}
& \text{if $i=k$, $\p(i)=\1$}
\\
0 & \text{if $i \neq k$.}
\end{dcases}
\end{equation}
The \emph{inversion relations} assert that
\begin{equation}
\begin{tikzpicture}[IKM,centerzero]
\draw[->] (-0.3,-0.3) \botlabel{i}-- (0.3,0.3);
\draw[<-] (0.3,-0.3) \botlabel{j} -- (-0.3,0.3);
\end{tikzpicture}:
P_i Q_j \one_\lambda \Rightarrow Q_j P_i \one_\lambda
\end{equation}
is an isomorphism for all $i,j \in I$ with $i \neq j$, as are the following matrices for all $\lambda$ and $i$:
\begin{equation}\label{iceland}
M_{\lambda;i} :=
\begin{dcases}
\begin{pmatrix}
\begin{tikzpicture}[IKM,centerzero]
\draw[->] (-0.25,-0.25) \botlabel{i}-- (0.25,0.25);
\draw[<-] (0.25,-0.25) \botlabel{i} -- (-0.25,0.25);
\region{0.35,0.02}{\lambda};
\end{tikzpicture} &
\begin{tikzpicture}[IKM,anchorbase]
\draw[->] (-0.2,0.3)\toplabel{i} -- (-0.2,0) arc(180:360:0.2) -- (0.2,0.3);
\region{0.38,0}{\lambda};        
\end{tikzpicture}&
\begin{tikzpicture}[anchorbase,IKM]
\draw[->] (-0.2,0.3) \toplabel{i} -- (-0.2,0) arc(180:360:0.2) -- (0.2,0.3);
\singdot{0.2,0.1};
\region{0.38,0}{\lambda};
\end{tikzpicture}&\!\!\cdots\!\!&
\begin{tikzpicture}[anchorbase,IKM]
\draw[->] (-0.2,0.3)\toplabel{i}  -- (-0.2,0) arc(180:360:0.2) -- (0.2,0.3);
\multdot{0.2,0.1}{west}{-h_i(\lambda)-1};
\region{0.38,-0.1}{\lambda};
\end{tikzpicture}
\end{pmatrix}\phantom{_T}
&\text{if } h_i(\lambda) \leq 0,\p(i)=\0\\
\begin{pmatrix}
\begin{tikzpicture}[IKM,centerzero]
\draw[->] (-0.25,-0.25) \botlabel{i}-- (0.25,0.25);
\draw[<-] (0.25,-0.25) \botlabel{i} -- (-0.25,0.25);
\region{0.35,0.02}{\lambda};
\end{tikzpicture} &
\begin{tikzpicture}[IKM,anchorbase]
\draw[->] (-0.2,0.3) \toplabel{i} -- (-0.2,0) arc(180:360:0.2) -- (0.2,0.3);
\region{0.38,0}{\lambda};   
\token{.2,-.05};     
\end{tikzpicture}
&\begin{tikzpicture}[IKM,anchorbase]
\draw[->] (-0.2,0.3) \toplabel{i} -- (-0.2,0) arc(180:360:0.2) -- (0.2,0.3);
\region{0.38,0}{\lambda};        
\end{tikzpicture}&
\begin{tikzpicture}[anchorbase,IKM]
\draw[->] (-0.2,0.3) \toplabel{i} -- (-0.2,0) arc(180:360:0.2) -- (0.2,0.3);
\singdot{0.2,0.1};\token{.2,-.05};
\region{0.38,0}{\lambda};
\end{tikzpicture}&\!\!\cdots\!\!&
\begin{tikzpicture}[anchorbase,IKM]
\draw[->] (-0.2,0.3) \toplabel{i} -- (-0.2,0) arc(180:360:0.2) -- (0.2,0.3);
\multdot{0.2,0.1}{west}{-h_i(\lambda)-1};
\region{0.38,-0.1}{\lambda};
\token{.2,-.05};
\end{tikzpicture}
\begin{tikzpicture}[anchorbase,IKM]
\draw[->] (-0.2,0.3)\toplabel{i}  -- (-0.2,0) arc(180:360:0.2) -- (0.2,0.3);
\multdot{0.2,0.1}{west}{-h_i(\lambda)-1};
\region{0.38,-0.1}{\lambda};
\end{tikzpicture}
\end{pmatrix}\phantom{^T}
&\text{if } h_i(\lambda) \leq 0,\p(i) = \1\\
\begin{pmatrix}   \begin{tikzpicture}[IKM,centerzero]
\draw[->] (-0.25,-0.25) \botlabel{i}-- (0.25,0.25);
\draw[<-] (0.25,-0.25)\botlabel{i}  -- (-0.25,0.25);
\region{0.38,0.02}{\lambda};
\end{tikzpicture} &
\begin{tikzpicture}[anchorbase,IKM]
\node at (-.2,.17) {$\scriptstyle{\phantom i}$};
\draw[->] (-0.2,-0.3) \botlabel{i}-- (-0.2,0) arc(180:0:0.2) -- (0.2,-0.3);
\region{0.38,0}{\lambda};
\end{tikzpicture}&
\begin{tikzpicture}[anchorbase,IKM]
\node at (-.2,.17) {$\scriptstyle{\phantom i}$};
\draw[->] (-0.2,-0.3)\botlabel{i} -- (-0.2,0) arc(180:0:0.2) -- (0.2,-0.3);
\singdot{-0.2,0};
\region{0.38,0}{\lambda};
\end{tikzpicture}
&\!\!\!\cdots\!\!\!&
\begin{tikzpicture}[anchorbase,IKM]
\node at (-.2,.17) {$\scriptstyle{\phantom i}$};            
\draw[->] (-0.2,-0.3)\botlabel{i} -- (-0.2,0) arc(180:0:0.2) -- (0.2,-0.3);
\multdot{-0.2,0}{east}{h_i(\lambda)-1};
\region{0.38,0}{\lambda};
\end{tikzpicture}
\end{pmatrix}^\transpose\phantom{_T}
&\text{if } h_i(\lambda) > 0, \p(i) =\0\\
\begin{pmatrix}    
\begin{tikzpicture}[IKM,centerzero]
\draw[->] (-0.25,-0.25) \botlabel{i}-- (0.25,0.25);
\draw[<-] (0.25,-0.25) \botlabel{i} -- (-0.25,0.25);
\region{0.38,0.02}{\lambda};
\end{tikzpicture} &
\begin{tikzpicture}[anchorbase,IKM]
\node at (-.2,.17) {$\scriptstyle{\phantom i}$};\draw[->] (-0.2,-0.3) \botlabel{i}-- (-0.2,0) arc(180:0:0.2) -- (0.2,-0.3);\token{-.2,-.15};
\region{0.38,0}{\lambda};
\end{tikzpicture} &
\begin{tikzpicture}[anchorbase,IKM]
\node at (-.2,.17) {$\scriptstyle{\phantom i}$};
\draw[->] (-0.2,-0.3) \botlabel{i}-- (-0.2,0) arc(180:0:0.2) -- (0.2,-0.3);
\region{0.38,0}{\lambda};
\end{tikzpicture}&
\begin{tikzpicture}[anchorbase,IKM]
\node at (-.2,.17) {$\scriptstyle{\phantom i}$};            
\draw[->] (-0.2,-0.3)\botlabel{i} -- (-0.2,0) arc(180:0:0.2) -- (0.2,-0.3);
\singdot{-0.2,0.0};\token{-.2,-.15};
\region{0.38,0}{\lambda};
\end{tikzpicture}
&\!\!\!\cdots\!\!\!&
\begin{tikzpicture}[anchorbase,IKM]
\node at (-.2,.17) {$\scriptstyle{\phantom i}$};
\draw[->] (-0.2,-0.3)\botlabel{i} -- (-0.2,0) arc(180:0:0.2) -- (0.2,-0.3);
\multdot{-0.2,0}{east}{h_i(\lambda)-1};
\region{0.38,0}{\lambda};
\token{-.2,-.15};
\end{tikzpicture}&\!\!\!
\begin{tikzpicture}[anchorbase,IKM]
\node at (-.2,.17) {$\scriptstyle{\phantom i}$};
\draw[->] (-0.2,-0.3)\botlabel{i} -- (-0.2,0) arc(180:0:0.2) -- (0.2,-0.3);
\multdot{-0.2,0}{east}{h_i(\lambda)-1};
\region{0.38,0}{\lambda};
\end{tikzpicture}
\end{pmatrix}^\transpose\!
&\text{if } h_i(\lambda) > 0, \p(i)=\1.
\end{dcases}
\end{equation}
We introduce a few more shorthands:
\begin{itemize}
\item For $i \in I$, let $\gamma_i \in \kk^\times$ be $1$ if $\p(i)=\0$
or a square root of $2$ if $\p(i)=\1$.
\item
Let
$\begin{tikzpicture}[IKM,centerzero]
\draw[->] (0.25,-0.25) \botlabel{i} -- (-0.25,0.25);
\draw[<-] (-0.25,-0.25) \botlabel{j} -- (0.25,0.25);
\region{0.33,0}{\lambda};
\end{tikzpicture}$
be
$\Big(\begin{tikzpicture}[IKM,centerzero]
\draw[<-] (0.25,-0.25)\botlabel{j}  -- (-0.25,0.25);
\draw[->] (-0.25,-0.25) \botlabel{i} -- (0.25,0.25);
\region{0.33,0}{\lambda};
\end{tikzpicture}\Big)^{-1}$ if $i \neq j$, or the first entry of the matrix $-\gamma_i^2 M_{\lambda;i}^{-1}$ if $i=j$.
\item
Let
$\begin{tikzpicture}[IKM,centerzero]
\draw[<-] (-0.2,-0.2) \botlabel{i} -- (-0.2,0) arc(180:0:0.2) -- (0.2,-0.2);
\region{0.4,0}{\lambda};
\end{tikzpicture}\!\!$
be the last entry of $\gamma_iM_{\lambda;i}^{-1}$ if $h_i(\lambda) < 0$ or $-\gamma_i^{-1}\begin{tikzpicture}[anchorbase,scale=1.1,IKM]
\draw[->] (.2,-.2)  to [out=90,in=-90] (-.2,.2) to[out=90,in=180] (0,.4) to[out=0,in=90] (.2,.2) to [out=-90,in=90] (-.2,-.2)\botlabel{i};
\multdot{-0.2,0.2}{east}{h_i(\lambda)};
\region{0.32,.05}{\lambda};
\end{tikzpicture}\!$ if $h_i(\lambda) \geq 0$.
\item Let $\begin{tikzpicture}[IKM,centerzero]
\draw[<-] (-0.2,0.2) \toplabel{i} -- (-0.2,0) arc(180:360:0.2) -- (0.2,0.2);
\region{0.37,0}{\lambda};
\end{tikzpicture}\!\!$
be the last entry of $\gamma_i M_{\lambda;i}^{-1}$ if $h_i(\lambda) > 0$ or
$\gamma_i^{-1}\begin{tikzpicture}[anchorbase,scale=1.1,IKM]
\draw[->] (.2,.2) to [out=-90,in=90] (-.2,-.2) to[out=-90,in=180] (0,-.4) to[out=0,in=-90] (.2,-.2) to [out=90,in=-90] (-.2,.2) \toplabel{i};
\region{0.4,0.1}{\lambda};
\multdot{0.2,-0.2}{west}{-h_i(\lambda)};
\end{tikzpicture}\!\!\!$ if $h_i(\lambda) \leq 0$.
\end{itemize}
All of these morphisms are even.
Finally, there are the {\em odd bubble relations}, which assert that
\begin{equation}\label{ikmoddbubble}
\begin{tikzpicture}[IKM,baseline=-1mm,scale=.8]
\draw[-] (-0.25,0) arc(180:-180:0.25);
\draw[-] (-.18,.18) to (.18,-.18);
\draw[-] (-.18,-.18) to (.18,.18);
\node at (0,-.4) {\strandlabel{i}};
\region{0.55,0}{\lambda};
\end{tikzpicture} :=
\begin{cases}
\begin{tikzpicture}[baseline=-1mm,IKM]
\draw[<-] (-0.25,0) arc(180:-180:0.25);
\node at (0,-.4) {\strandlabel{i}};
\region{-0.55,0}{\lambda};
\token{.23,.1};
\multdot{0.23,-.1}{west}{-h_i(\lambda)};
\end{tikzpicture}&\text{if $h_i(\lambda) \leq 0$}\\
\begin{tikzpicture}[baseline=-1mm,IKM]
\draw[->] (0.25,0) arc(360:0:0.25);
\node at (0,-.4) {\strandlabel{i}};
\region{0.55,0}{\lambda};
\token{-.23,-.1};
\multdot{-0.23,.1}{east}{h_i(\lambda)};
\end{tikzpicture}&\text{if $h_i(\lambda) > 0$}
\end{cases}
\end{equation}
is 0 for all $\lambda \in X$ and $i \in I$ with $\p(i)=\1$.

\begin{rem}\label{ikmoddbubbleremark}
We refer to the 2-supercategory $\widehat{\fV}(\fg)$
defined in the same way as $\fV(\fg)$ but with the final odd bubble relations omitted as the {\em non-reduced isomeric Kac--Moody 2-category}. Like in \cref{likeinhere}, 
in $\widehat{\fV}(\fg)$,
the odd 2-morphisms 
$\begin{tikzpicture}[IKM,baseline=-1mm,scale=.8]
\draw[-] (-0.25,0) arc(180:-180:0.25);
\draw[-] (-.18,.18) to (.18,-.18);
\draw[-] (-.18,-.18) to (.18,.18);
\node at (0,-.4) {\strandlabel{i}};
\end{tikzpicture}$
on the left hand side of \cref{ikmoddbubble} 
slide freely across other strings up to multiplication by a sign.
\end{rem}

\subsection{Chevalley involution}

There is an isomorphism of 2-supercategories
\begin{equation}\label{invo1}
\tT:\fV(\fg) \rightarrow \fV(\fg)^\op
\end{equation}
defined on objects by $\lambda \mapsto -\lambda$,
on generating 1-morphisms by
$P_i \one_\lambda \mapsto Q_i \one_{-\lambda},
Q_i \one_\lambda \mapsto P_i \one_{-\lambda}$,
and on a generating 2-morphisms by
reflecting string diagrams
in a horizontal axis, negating all weights labelling 2-cells, then
multiplying by $(-1)^{m+\binom{n}{2}}$
where $m$ is the number of crossings and $n$ is the number of Clifford tokens in the diagram. For this recipe to be unambiguous, 
Clifford tokens should
arranged so that they are
all at different horizontal levels. For example, applying $\tT$ to the first relation in 
\cref{QHC1} shows that the Clifford token on a downward string must square to the identity, as may be checked like in \cref{torch}.
The symmetry $\tT$ is very useful when deriving further relations,
which is our next topic.

\subsection{Further relations}

Next, we record some consequences of the defining relations.
The proofs involve some elementary but lengthy calculations
which we are going to omit entirely. The reader familiar with
the arguments given in \cite{B-KM,BE17,Sav19}
should be able to reproduce the details since the overall strategy is identical.

The leftward cups and caps satisfy zig-zag
relations
\begin{align}\label{ikmadjleft}
\begin{tikzpicture}[centerzero,IKM,scale=1.2]
\draw[<-] (-0.3,0.4) -- (-0.3,0) arc(180:360:0.15) arc(180:0:0.15) -- (0.3,-0.4)\botlabel{i};
\end{tikzpicture}
&=
\begin{tikzpicture}[centerzero,IKM,scale=1.2]
\draw[->] (0,-0.4)\botlabel{i} -- (0,0.4);
\end{tikzpicture}
\ ,&
\begin{tikzpicture}[centerzero,IKM,scale=1.2]
\draw[<-] (-0.3,-0.4)\botlabel{i} -- (-0.3,0) arc(180:0:0.15) arc(180:360:0.15) -- (0.3,0.4);
\end{tikzpicture}
&=
\begin{tikzpicture}[centerzero,IKM,scale=1.2]
\draw[<-] (0,-0.4) \botlabel{i} -- (0,0.4);
\end{tikzpicture}\ .
\end{align}
This is far from obvious, and is one of the last relations that gets established when mimicking the arguments from \cite{B-KM,BE17,Sav19}.
We also have that
\begin{align} \label{iruby}
\begin{tikzpicture}[IKM,baseline=-1mm,scale=1.2]
\draw[->] (-0.2,0.2)\toplabel{i} -- (-0.2,0) arc (180:360:0.2) -- (0.2,0.2);
\token{-0.2,0};
\end{tikzpicture}
&=
\begin{tikzpicture}[IKM,baseline=-1mm,scale=1.2]
\draw[->] (-0.2,0.2)\toplabel{i} -- (-0.2,0) arc (180:360:0.2) -- (0.2,0.2);
\token{0.2,0};
\end{tikzpicture} ,&
\begin{tikzpicture}[IKM,anchorbase,scale=1.2]
\draw[->] (-0.2,-0.2)\botlabel{i} -- (-0.2,0) arc (180:0:0.2) -- (0.2,-0.2);
\token{-0.2,0};
\end{tikzpicture}
&=
\begin{tikzpicture}[IKM,anchorbase,scale=1.2]
\draw[->] (-0.2,-0.2)\botlabel{i} -- (-0.2,0) arc (180:0:0.2) -- (0.2,-0.2);
\token{0.2,0};
\end{tikzpicture}
,&
\begin{tikzpicture}[IKM,baseline=-1mm,scale=1.2]
\draw[<-] (-0.2,0.2) -- (-0.2,0) arc (180:360:0.2) -- (0.2,0.2)\toplabel{i};
\token{-0.2,0};
\region{0.35,0}{\lambda};
\end{tikzpicture}
&= (-1)^{h_i(\lambda)}\ 
\begin{tikzpicture}[IKM,baseline=-1mm,scale=1.2]
\draw[<-] (-0.2,0.2) -- (-0.2,0) arc (180:360:0.2) -- (0.2,0.2)\toplabel{i};
\token{0.2,0};
\region{0.4,0}{\lambda};
\end{tikzpicture},&
\begin{tikzpicture}[IKM,anchorbase,scale=1.2]
\draw[<-] (-0.2,-0.2) -- (-0.2,0) arc (180:0:0.2) -- (0.2,-0.2)\botlabel{i};
\token{-0.2,0};
\region{0.35,0}{\lambda};
\end{tikzpicture}
&= (-1)^{h_i(\lambda)}\ 
\begin{tikzpicture}[IKM,anchorbase,scale=1.2]
\draw[<-] (-0.2,-0.2) -- (-0.2,0) arc (180:0:0.2) -- (0.2,-0.2)\botlabel{i};
\token{0.2,0};
\region{0.4,0}{\lambda};
\end{tikzpicture}\end{align}
assuming, of course, that $i$ is odd, and
\begin{align}
\begin{tikzpicture}[IKM,anchorbase,scale=1.2]
\draw[->] (-0.2,0.2)\toplabel{i} -- (-0.2,0) arc (180:360:0.2) -- (0.2,0.2);
\singdot{-0.2,0};
\end{tikzpicture}
&=
\begin{tikzpicture}[IKM,anchorbase,scale=1.2]
\draw[->] (-0.2,0.2)\toplabel{i} -- (-0.2,0) arc (180:360:0.2) -- (0.2,0.2);
\singdot{0.2,0};
\end{tikzpicture}\ ,&
\begin{tikzpicture}[IKM,baseline=-1mm,scale=1.2]
\draw[->] (-0.2,-0.2)\botlabel{i} -- (-0.2,0) arc (180:0:0.2) -- (0.2,-0.2);
\singdot{-0.2,0};
\end{tikzpicture}
&=
\begin{tikzpicture}[IKM,baseline=-1mm,scale=1.2]
\draw[->] (-0.2,-0.2)\botlabel{i} -- (-0.2,0) arc (180:0:0.2) -- (0.2,-0.2);
\singdot{0.2,0};
\end{tikzpicture}\ ,&
\begin{tikzpicture}[IKM,anchorbase,scale=1.2]
\draw[<-] (-0.2,0.2) -- (-0.2,0) arc (180:360:0.2) -- (0.2,0.2)\toplabel{i};
\singdot{-0.2,0};
\end{tikzpicture}
&=
\begin{tikzpicture}[IKM,anchorbase,scale=1.2]
\draw[<-] (-0.2,0.2) -- (-0.2,0) arc (180:360:0.2) -- (0.2,0.2)\toplabel{i};
\singdot{0.2,0};
\end{tikzpicture}\ ,&
\begin{tikzpicture}[IKM,baseline=-1mm,scale=1.2]
\draw[<-] (-0.2,-0.2) -- (-0.2,0) arc (180:0:0.2) -- (0.2,-0.2)\botlabel{i};
\singdot{-0.2,0};
\end{tikzpicture}
&=
\begin{tikzpicture}[IKM,baseline=-1mm,scale=1.2]
\draw[<-] (-0.2,-0.2) -- (-0.2,0) arc (180:0:0.2) -- (0.2,-0.2)\botlabel{i};
\singdot{0.2,0};
\end{tikzpicture}\ ,\label{iwax1}
\\
\begin{tikzpicture}[IKM,baseline=1mm,scale=1.2]
\draw[->] (-0.2,0.3)\toplabel{i} -- (-0.2,0.1) arc(180:360:0.2) -- (0.2,0.3);
\draw[->] (-0.3,-0.3) to[out=up,in=down] (0,0.3)\toplabel{j};
\end{tikzpicture}
&=
\begin{tikzpicture}[IKM,baseline=1mm,scale=1.2]
\draw[->] (-0.2,0.3) \toplabel{i}-- (-0.2,0.1) arc(180:360:0.2) -- (0.2,0.3);
\draw[->] (0.3,-0.3)to[out=up,in=down] (0,0.3)\toplabel{j};
\end{tikzpicture}\ ,&
\begin{tikzpicture}[IKM,baseline=1mm,scale=1.2]
\draw[->] (-0.2,0.3)\toplabel{i} -- (-0.2,0.1) arc(180:360:0.2) -- (0.2,0.3);
\draw[<-] (-0.3,-0.3) to[out=up,in=down] (0,0.3)\toplabel{j};
\end{tikzpicture}
&=
\begin{tikzpicture}[IKM,baseline=1mm,scale=1.2]
\draw[->] (-0.2,0.3)\toplabel{i} -- (-0.2,0.1) arc(180:360:0.2) -- (0.2,0.3);
\draw[<-] (0.3,-0.3)to[out=up,in=down] (0,0.3)\toplabel{j};
\end{tikzpicture}\ ,&
\begin{tikzpicture}[IKM,baseline=1mm,scale=1.2]
\draw[<-] (-0.2,0.3)\toplabel{i} -- (-0.2,0.1) arc(180:360:0.2) -- (0.2,0.3);
\draw[->] (-0.3,-0.3) to[out=up,in=down] (0,0.3)\toplabel{j};
\end{tikzpicture}
&=t_{ij}\ 
\begin{tikzpicture}[IKM,baseline=1mm,scale=1.2]
\draw[<-] (-0.2,0.3)\toplabel{i} -- (-0.2,0.1) arc(180:360:0.2) -- (0.2,0.3);
\draw[->] (0.3,-0.3)to[out=up,in=down] (0,0.3)\toplabel{j};
\end{tikzpicture}\ ,&
\begin{tikzpicture}[IKM,baseline=1mm,scale=1.2]
\draw[<-] (-0.2,0.3)\toplabel{i} -- (-0.2,0.1) arc(180:360:0.2) -- (0.2,0.3);
\draw[<-] (-0.3,-0.3) to[out=up,in=down] (0,0.3)\toplabel{j};
\end{tikzpicture}
&=t_{ij}^{-1}\ 
\begin{tikzpicture}[IKM,baseline=1mm,scale=1.2]
\draw[<-] (-0.2,0.3)\toplabel{i} -- (-0.2,0.1) arc(180:360:0.2) -- (0.2,0.3);
\draw[<-] (0.3,-0.3)to[out=up,in=down] (0,0.3)\toplabel{j};
\end{tikzpicture}\ ,\label{iwax2}\\
\begin{tikzpicture}[IKM,baseline=-1mm,scale=1.2]
\draw[->] (-0.2,-0.3)\botlabel{i} -- (-0.2,-0.1) arc(180:0:0.2) -- (0.2,-0.3);
\draw[->] (-0.3,0.3) \braiddown (0,-0.3)\botlabel{j};
\end{tikzpicture}
&=
\begin{tikzpicture}[IKM,baseline=-1mm,scale=1.2]
\draw[->] (-0.2,-0.3)\botlabel{i} -- (-0.2,-0.1) arc(180:0:0.2) -- (0.2,-0.3);
\draw[->] (0.3,0.3) \braiddown (0,-0.3)\botlabel{j};
\end{tikzpicture}\ ,&
\begin{tikzpicture}[IKM,baseline=-1mm,scale=1.2]
\draw[->] (-0.2,-0.3)\botlabel{i} -- (-0.2,-0.1) arc(180:0:0.2) -- (0.2,-0.3);
\draw[<-] (-0.3,0.3) \braiddown (0,-0.3)\botlabel{j};
\end{tikzpicture}
&=
\begin{tikzpicture}[IKM,baseline=-1mm,scale=1.2]
\draw[->] (-0.2,-0.3)\botlabel{i} -- (-0.2,-0.1) arc(180:0:0.2) -- (0.2,-0.3);
\draw[<-] (0.3,0.3) \braiddown (0,-0.3)\botlabel{j};
\end{tikzpicture}\ ,&
\begin{tikzpicture}[IKM,baseline=-1mm,scale=1.2]
\draw[<-] (-0.2,-0.3)\botlabel{i} -- (-0.2,-0.1) arc(180:0:0.2) -- (0.2,-0.3);
\draw[<-] (-0.3,0.3) \braiddown (0,-0.3)\botlabel{j};
\end{tikzpicture}
&=t_{ij}^{-1}\ 
\begin{tikzpicture}[IKM,baseline=-1mm,scale=1.2]
\draw[<-] (-0.2,-0.3)\botlabel{i} -- (-0.2,-0.1) arc(180:0:0.2) -- (0.2,-0.3);
\draw[<-] (0.3,0.3) \braiddown (0,-0.3)\botlabel{j};
\end{tikzpicture}\ ,
&\begin{tikzpicture}[IKM,baseline=-1mm,scale=1.2]
\draw[<-] (-0.2,-0.3)\botlabel{i} -- (-0.2,-0.1) arc(180:0:0.2) -- (0.2,-0.3);
\draw[->] (-0.3,0.3) \braiddown (0,-0.3)\botlabel{j};
\end{tikzpicture}
&=t_{ij}\ 
\begin{tikzpicture}[IKM,baseline=-1mm,scale=1.2]
\draw[<-] (-0.2,-0.3)\botlabel{i} -- (-0.2,-0.1) arc(180:0:0.2) -- (0.2,-0.3);
\draw[->] (0.3,0.3) \braiddown (0,-0.3)\botlabel{j};
\end{tikzpicture}
\label{iwax3}
\end{align}
for any $i,j \in I$.
The relations here involving a rightward cup or cap follow immediately from the definitions \cref{irightpivot}, but the ones involving a leftward cup or cap require a lot more work.
Note also that the dot slides involving leftward cups and caps \cref{iwax1} depend on
the odd bubble relations \cref{ikmoddbubble}.

With \cref{ikmadjleft,iruby,iwax1,iwax2,iwax3} in hand, it is straightforward to deduce analogs of the relations \cref{QHC2,QHC3a,QHC3b} for rightward, downward and leftward crossings. Tokens slide across all types of crossings, and dots slide across all crossings involving strings of two different colors. Dot slides across crossings of strings of the same color are more complicated but are similar to \cref{QHC3a,QHC3b} in all cases.

Using the odd bubble relations, it follows 
that {\em all} odd bubbles are 0, hence,
the superalgebra $\End_{\fV(\fg)}(\one_\lambda)$ is purely even for all $\lambda \in X$; this is similar to \cref{msc}.
We also have that
\begin{equation}\label{ikmmoreodddottedbubbles}
\begin{tikzpicture}[baseline=-1mm,IKM]
\draw[<-] (-0.25,0) arc(180:-180:0.25);
\node at (0,-.4) {\strandlabel{i}};
\region{-0.55,0}{\lambda};
\multdot{0.25,0}{west}{n};
\end{tikzpicture}
= \begin{tikzpicture}[baseline=-1mm,IKM]
\draw[->] (0.25,0) arc(360:0:0.25);
\node at (0,-.4) {\strandlabel{i}};
\region{0.55,0}{\lambda};
\multdot{-0.25,0}{east}{n};
\end{tikzpicture}
 = 0
\end{equation}
for $\lambda \in X$, $i \in I$ with $\p(i)=\1$ and
$n \geq 0$ such that
$n \equiv h_i(\lambda)\pmod{2}$; see \cref{darkness} for an analogous proof.

Next, we have the {\em infinite Grassmannian relation} in $\fV(\fg)$, which asserts that there are unique formal Laurent series
\begin{align}
\begin{tikzpicture}[IKM,baseline=-1mm]
\leftbubgen{-.3,0};
\stringlabel{-.77,-.4}{i};
\region{-1.3,0}{\lambda};
\end{tikzpicture}
&\in  \gamma_i u^{h_i(\lambda)}
\id_{\one_\lambda}+ u^{h_i(\lambda)-2}
\End_{\fV(\fg)}(\one_\lambda)\llbracket u^{-2}\rrbracket,\\
\begin{tikzpicture}[IKM,anchorbase,scale=1.1]
\rightbubgen{-.3,0};
\stringlabel{-.77,-.4}{i};
\region{-1.2,0}{\lambda};
\end{tikzpicture}
&\in \gamma_i u^{-h_i(\lambda)}\id_{\one_\lambda} + u^{-h_i(\lambda)-2} \End_{\fV(\fg)}(\one_\lambda)\llbracket u^{-2}\rrbracket
\end{align}
such that 
\begin{align}
\left[\ 
\begin{tikzpicture}[IKM,centerzero]
\leftbubgen{0,0};
\stringlabel{-.5,-.4}{i};
\end{tikzpicture}
\right]_{u:<0}&=
\sum_{n \geq 0} 
\begin{tikzpicture}[IKM,centerzero]
\leftbubdot{0,0}{n};
\stringlabel{0,-.4}{i};
\end{tikzpicture}
u^{-n-1},&
\left[\ 
\begin{tikzpicture}[IKM,centerzero]
\rightbubgen{0,0};
\stringlabel{-.5,-.4}{i};
\end{tikzpicture}
\right]_{u:<0}&=
\sum_{n \geq 0}
\begin{tikzpicture}[IKM,centerzero]
\rightbubdot{0,0}{n};
\stringlabel{0,-.4}{i};
\end{tikzpicture}
\  u^{-n-1},&
\end{align}
and
\begin{equation} \label{ikmeyes}
\begin{tikzpicture}[IKM,centerzero]
\rightbubgen{0,0};
\stringlabel{-.5,-.4}{i};
\region{.6,0}{\lambda};
\end{tikzpicture}\ 
\begin{tikzpicture}[IKM,centerzero]
\leftbubgen{0,0};
\stringlabel{-.5,-.4}{i};
\end{tikzpicture}
= \gamma_i^2 \id_{\one_\lambda}.
\end{equation}

We will use the analogous dot generating function to \cref{idgf}.
It follows from \cref{iwax1} that these slide over all caps and cups. By \cref{QHC1}, we have that
\begin{align} \label{ifrog}
\begin{tikzpicture}[IKM,centerzero]
\draw[->] (0,-0.4)\botlabel{i} -- (0,0.4);
\token{0,-0.2};
\circled{0,0.1}{u};
\end{tikzpicture}
&=
\begin{tikzpicture}[IKM,centerzero]
\draw[->] (0,-0.4)\botlabel{i} -- (0,0.4);
\token{0,0.13};
\circledbar{0,-0.17}{u};
\end{tikzpicture}\ ,&
\begin{tikzpicture}[IKM,centerzero]
\draw[->] (0,-0.4)\botlabel{i} -- (0,0.4);
\token{0,0.13};
\circled{0,-0.17}{u};
\end{tikzpicture}
&=
\begin{tikzpicture}[IKM,centerzero]
\draw[->] (0,-0.4)\botlabel{i} -- (0,0.4);
\token{0,-0.17};
\circledbar{0,0.13}{u};
\end{tikzpicture}\ .
\end{align}
Like in \cref{trick,tadpole}, 
for a polynomial $f(x) \in \kk[x]$, we have 
that
\begin{align} \label{trick3}
\begin{tikzpicture}[IKM,centerzero,scale=1.1]
\draw[-] (0,-0.25)\botlabel{i} -- (0,0.25);
\pin{0,0}{-0.7,0}{f(x)};
\end{tikzpicture}
&=
\left[f(u)\  
\begin{tikzpicture}[IKM,centerzero,scale=1.1]
\draw[-] (0,-0.25)\botlabel{i}-- (0,0.25);
\circled{0,0}{u};
\end{tikzpicture}
\right]_{u:-1},
&
\begin{tikzpicture}[IKM,centerzero,scale=1.1]
\draw[-] (0,-0.25) -- (0,0.25);
\pin{0,0}{-0.7,0}{f(-x)};
\end{tikzpicture}
&=
\left[
f(u)\
\begin{tikzpicture}[IKM,centerzero,scale=1.1]
\draw[-] (0,-0.25)\botlabel{i} -- (0,0.25);
\circledbar{0,0}{u};
\end{tikzpicture}
\right]_{u:-1},\\
\label{itadpole}
\begin{tikzpicture}[IKM,centerzero]
\rightbub{0,0};
\pin{-0.24,0}{-0.9,0}{f(x)};
\node at (0,-.4) {\strandlabel{i}};
\end{tikzpicture}
&=
\left[f(u)\
\begin{tikzpicture}[IKM,centerzero]
\rightbubgen{0,0};
\stringlabel{-.5,-.4}{i};
\end{tikzpicture}
\right]_{u:-1},
&
\begin{tikzpicture}[IKM,centerzero]
\leftbub{0,0};
\pin{0.25,0}{0.9,0}{f(x)};
\stringlabel{-.5,-.4}{i};
\end{tikzpicture}
&=
\left[f(u)\
\begin{tikzpicture}[IKM,centerzero]
\leftbubgen{0,0};
\node at (-.5,0) {\strandlabel{i}};
\end{tikzpicture}
\right]_{u:-1}.
\end{align}

Here are a couple of particularly important
further relations which exploit the generating function formalism.
First, we have the {\em curl relations}:
\begin{align}
\label{ikmfancycurls}
\begin{tikzpicture}[anchorbase,scale=1.1,IKM]
\draw[->] (0,-0.5)\botlabel{i} to[out=up,in=0] (-0.3,0.2) to[out=180,in=up] (-0.45,0) to[out=down,in=180] (-0.3,-0.2) to[out=0,in=down] (0,0.5);
\circled{-0.45,0}{u};
\end{tikzpicture}
&=
\left[
\begin{tikzpicture}[anchorbase,scale=1.1,IKM]
\draw[->] (0,-0.5)\botlabel{i} -- (0,0.5);
\leftbubgen{-0.5,0};
\stringlabel{-.95,-.35}{i};
\circled{0,0}{u};
\end{tikzpicture}
\right]_{u:<0},
&
\begin{tikzpicture}[anchorbase,scale=1.1,IKM]
\draw[->] (0,-0.5)\botlabel{i} to[out=up,in=180] (0.3,0.2) to[out=0,in=up] (0.45,0) to[out=down,in=0] (0.3,-0.2) to[out=180,in=down] (0,0.5);
\circled{0.45,0}{u};
\end{tikzpicture}
&=
-
\left[
\begin{tikzpicture}[anchorbase,scale=1.1,IKM]
\draw[->] (0,-0.5)\botlabel{i} -- (0,0.5);
\circled{0,0}{u};
\rightbubgen{1,0};
\stringlabel{.55,-.35}{i};
\end{tikzpicture}
\right]_{u:<0}.
\end{align}
From the second of these relations and \cref{trick3}, it follows
that
\begin{align}
\begin{tikzpicture}[anchorbase,scale=1.1,IKM]
\draw[->] (0,-0.5)\botlabel{i} to[out=up,in=180] (0.3,0.2) to[out=0,in=up] (0.45,0) to[out=down,in=0] (0.3,-0.2) to[out=180,in=down] (0,0.5);
\pin{0.45,0}{1.1,0}{f(x)};
\end{tikzpicture}
&= -\left[
f(u)\
\begin{tikzpicture}[anchorbase,scale=1.1,IKM]
\draw[->] (0,-0.5)\botlabel{i} -- (0,0.5);
\circled{0,0}{u};
\rightbubgen{1,0};
\stringlabel{.55,-.4}{i};
\end{tikzpicture}
\right]_{u:-1}
\end{align}
for any polynomial $f(x) \in \kk[x]$.
Applying $\tT$, we get also that
\begin{align}
\begin{tikzpicture}[anchorbase,scale=1.1,IKM]
\draw[<-] (0,-0.5)\botlabel{i} to[out=up,in=180] (0.3,0.2) to[out=0,in=up] (0.45,0) to[out=down,in=0] (0.3,-0.2) to[out=180,in=down] (0,0.5);
\pin{0.45,0}{1.1,0}{f(x)};
\end{tikzpicture}
&=
\left[
f(u)\ \begin{tikzpicture}[anchorbase,scale=1.1,IKM]
\draw[<-] (0,-0.5)\botlabel{i}-- (0,0.5);
\circled{0,0}{u};
\leftbubgen{1,0};
\stringlabel{.55,-.35}{i};
\end{tikzpicture}
\right]_{u:-1}.
\end{align}
Also, there are the {\em bubble slides}:
\begin{align}\label{ikmbubslide}
\begin{tikzpicture}[anchorbase,scale=1.1,IKM]
\draw[->] (0,-0.5)\botlabel{j} to (0,0.5);
\leftbubgen{.9,0};
\stringlabel{.45,-.35}{i};
\end{tikzpicture}
&=
\begin{tikzpicture}[anchorbase,scale=1.1,IKM]
\draw[->] (0,-0.5)\botlabel{j} to (0,0.5);
\leftbubgen{-0.4,0};
\stringlabel{-.85,-.35}{i};
\pin{0,0}{.9,0}{R_{ij}(u,x)};
\end{tikzpicture}\ ,&
\begin{tikzpicture}[anchorbase,scale=1.1,IKM]
\draw[->] (0,-0.5)\botlabel{j} to (0,0.5);
\rightbubgen{-0.4,0};
\stringlabel{-.85,-.35}{i};
\end{tikzpicture}        
&=
\begin{tikzpicture}[anchorbase,scale=1.1,IKM]
\draw[->] (0,-0.5)\botlabel{j} to (0,0.5);
\rightbubgen{.9,0};
\stringlabel{.45,-.35}{i};
\pin{0,0}{-.9,0}{R_{ij}(u,x)};
\end{tikzpicture}.
\end{align}

\subsection{Definition of isomeric Kac--Moody categorification}\label{ikmcatdef}

An {\em isomeric Kac--Moody categorification} with the type and parameters fixed in \cref{data}
is a locally finite 
Abelian supercategory $\catR$ with a given internal direct sum decomposition
\begin{equation}
\catR = \bigoplus_{\lambda \in X} \catR_\lambda
\end{equation}
for Serre subcategories $\catR_\lambda\:(\lambda \in X)$,
the {\em weight subcategories} of $\catR$,
plus adjoint pairs $(P_i, Q_i)$ of endofunctors 
for each $i \in I$,
such that:
\begin{enumerate}
\item[(IKM0)]
The functor $P_i$ takes objects of $\catR_\lambda$ to $\catR_{\lambda+\alpha_i}$; equivalently, $Q_i$ takes objects of $\catR_{\lambda+\alpha_i}$ to $\catR_{\lambda}$.
\item[(IKM1)] The adjoint pair $(P_i, Q_i)$ has a prescribed 
adjunction
with unit and counit of adjunction denoted
 $\begin{tikzpicture}[IKM,centerzero,scale=.8]
\draw[-to] (-0.25,0.15) \toplabel{i} to[out=-90,in=-90,looseness=3] (0.25,0.15);
\node at (0,.2) {$\phantom.$};\node at (0,-.3) {$\phantom.$};
\end{tikzpicture}:\id_\catR \Rightarrow Q_i \circ P_i$
and $\begin{tikzpicture}[IKM,centerzero,scale=.8]
\draw[-to] (-0.25,-0.15) \botlabel{i} to [out=90,in=90,looseness=3](0.25,-0.15);
\node at (0,.3) {$\phantom.$};
\node at (0,-.4) {$\phantom.$};
\end{tikzpicture}:P_i \circ Q_i \Rightarrow \id_\catR$.
These should be even.
\item[(IKM2)] There are
given odd supernatural transformations
$\begin{tikzpicture}[IKM,centerzero]
\draw[-to] (0,-0.2) \botlabel{i} -- (0,0.2);
\token{0,0};
\end{tikzpicture}:P_i \Rightarrow P_i$
for all odd $i \in I$, and
even supernatural transformations
$\begin{tikzpicture}[IKM,centerzero]
\draw[-to] (0,-0.2) \botlabel{i} -- (0,0.2);
\singdot{0,0};
\end{tikzpicture}:P_i \Rightarrow P_i$ and 
$\begin{tikzpicture}[IKM,centerzero,scale=.9]
\draw[-to] (-0.2,-0.2) \botlabel{i} -- (0.2,0.2);
\draw[-to] (0.2,-0.2) \botlabel{j} -- (-0.2,0.2);
\end{tikzpicture}:P_i \circ P_j \Rightarrow P_j \circ P_i$
for all $i,j \in I$,
satisfying the quiver Hecke--Clifford superalgebra relations \cref{QHC1,QHC2,QHC3a,QHC3b,QHC4,QHC5}.
\item[(IKM3)]
Defining $\begin{tikzpicture}[IKM,centerzero]
\draw[to-] (0.2,-0.2) \botlabel{j} -- (-0.2,0.2);
\draw[-to] (-0.2,-0.2) \botlabel{i} -- (0.2,0.2);
\end{tikzpicture}$
as in \cref{irightpivot}, the natural transformations
$\begin{tikzpicture}[IKM,centerzero]
\draw[to-] (0.2,-0.2) \botlabel{j} -- (-0.2,0.2);
\draw[-to] (-0.2,-0.2) \botlabel{i} -- (0.2,0.2);
\end{tikzpicture}:P_i \circ Q_j \Rightarrow Q_j \circ P_i$
are isomorphisms for all $i \neq j$, as are the matrices
of supernatural transformations defined by \cref{iceland}
for all $i \in I$ and $\lambda \in X$.
\item[(IKM4)] 
There exists a family of objects $V \in \catR$
such that the supercenter $Z_V$ of $\End_{\catR}(V)$ from \cref{supercenter}
is purely even for each $V$ in the family, and
the objects obtained from these objects by applying sequences of the functors $P_i$ and $Q_i$ are a generating family for $\catR$.
\end{enumerate}
From these axioms, it follows that there is induced
a strict $\kk$-linear 2-functor
from $\fV(\fg)$ to the 2-supercategory of locally finite 
Abelian supercategories; cf. the discussion at the end of \cref{ihcdef}. One could also say that $\catR$ is a {\em super 2-representation} of $\fV(\fg)$.

\setcounter{section}{5}
\section{The bridge from isomeric Heisenberg to isomeric Kac--Moody}\label{s6-iheis2ikm}

Now we return to the setup of \cref{s4-wsd}. So the super 
Cartan datum is as in \cref{seccd} with the entries $c_{ij}$ of the Cartan matrix given by 
\cref{cartan}, and the weight lattice $X$ is the minimal one from \cref{minreal}. 
Let $\fV(\fg)$ be the isomeric Kac--Moody 2-category 
associated to this super Cartan datum from \cref{ikmdef}
with parameters
\begin{equation}\label{pdub}
Q_{ij}(x,y) :=
\begin{cases}
0&\text{if $i=j$}\\
(i-j) \left(x^{-c_{i j}} - y^{-c_{ji}}\right)
&\text{if $i=j \pm 1$}\\
1&\text{otherwise,}
\end{cases}
\end{equation}
for $i,j \in I$.  
For this choice, the relations \cref{QHC4,QHC5} simplify to
\begin{align}\label{QHC4new}
\begin{tikzpicture}[IKM,centerzero,scale=1.1]
\draw[->] (-0.2,-0.4) \botlabel{i} to[out=45,in=down] (0.15,0) to[out=up,in=-45] (-0.2,0.4);
\draw[->] (0.2,-0.4) \botlabel{j} to[out=135,in=down] (-0.15,0) to[out=up,in=225] (0.2,0.4);
\end{tikzpicture}
&=
\begin{dcases}
0 & \text{if } i=j,\\
(i-j)\!\left(\!
\begin{tikzpicture}[IKM,centerzero,scale=1.1]
\draw[->] (-0.2,-0.3) \botlabel{i} -- (-0.2,0.3);
\draw[->] (0.2,-0.3) \botlabel{j} -- (0.2,0.3);
\multdot{-0.2,0}{east}{-c_{i j}};
\end{tikzpicture}
-
\begin{tikzpicture}[IKM,centerzero,scale=1.1]
\draw[->] (-0.2,-0.3) \botlabel{i} -- (-0.2,0.3);
\draw[->] (0.2,-0.3) \botlabel{j} -- (0.2,0.3);
\multdot{0.2,0}{west}{-c_{j i}};
\end{tikzpicture}\right)\hspace{22mm}
&\text{if $i = j \pm 1$}\\
\begin{tikzpicture}[IKM,centerzero,scale=1.1]
\draw[->] (-0.2,-0.3) \botlabel{i} -- (-0.2,0.3);
\draw[->] (0.2,-0.3) \botlabel{j} -- (0.2,0.3);
\end{tikzpicture}
& \text{otherwise,}
\end{dcases}
\\ \label{QHC5new}
\begin{tikzpicture}[IKM,centerzero,scale=1.1]
\draw[->] (-0.4,-0.4) \botlabel{i} -- (0.4,0.4);
\draw[->] (0,-0.4) \botlabel{j} to[out=135,in=down] (-0.32,0) to[out=up,in=225] (0,0.4);
\draw[->] (0.4,-0.4) \botlabel{k} -- (-0.4,0.4);
\end{tikzpicture}
\ -\
\begin{tikzpicture}[IKM,centerzero,scale=1.1]
\draw[->] (-0.4,-0.4) \botlabel{i} -- (0.4,0.4);
\draw[->] (0,-0.4) \botlabel{j} to[out=45,in=down] (0.32,0) to[out=up,in=-45] (0,0.4);
\draw[->] (0.4,-0.4) \botlabel{k} -- (-0.4,0.4);
\end{tikzpicture}
&=
\begin{dcases}
(i-j) \sum_{\substack{r,s \ge 0 \\ r+s=-c_{i j}-1}}
\begin{tikzpicture}[IKM,centerzero,scale=1.2]
\draw[->] (-0.3,-0.3) \botlabel{i} -- (-0.3,0.3);
\draw[->] (0,-0.3) \botlabel{j} -- (0,0.3);
\draw[->] (0.3,-0.3) \botlabel{i} -- (0.3,0.3);
\multdot{-0.3,0}{east}{r};
\multdot{0.3,0}{west}{s};
\end{tikzpicture} & \text{if $i=k \neq 0$}
\\
(0-j) \sum_{\substack{r,s \ge 0 \\ r+s=-c_{0 j}-1}}\left(
\begin{tikzpicture}[IKM,centerzero,scale=1.2]
\draw[->] (-0.3,-0.3) \botlabel{0} -- (-0.3,0.3);
\draw[->] (0,-0.3) \botlabel{j} -- (0,0.3);
\draw[->] (0.3,-0.3) \botlabel{0} -- (0.3,0.3);
\multdot{-0.3,0}{east}{r};
\multdot{0.3,0}{west}{s};
\end{tikzpicture}
-
\begin{tikzpicture}[IKM,centerzero,scale=1.2]
\draw[->] (-0.3,-0.3) \botlabel{0} -- (-0.3,0.3);
\draw[->] (0,-0.3) \botlabel{j} -- (0,0.3);
\draw[->] (0.3,-0.3) \botlabel{0} -- (0.3,0.3);
\multdot{-0.3,-.1}{east}{r};
\token{-.3,.1};
\token{.3,-.1};
\multdot{0.3,.1}{west}{s};
\end{tikzpicture}\right)&\text{if $i=k=0$}
\\
0 & \text{if $i \neq k$.}
\end{dcases}
\end{align}
We also fix an isomeric Heisenberg categorification $\catR$.
The goal is to make $\catR$ into an isomeric Kac--Moody categorification.
We have already
decomposed the endofunctors $P, Q:\catR \rightarrow \catR$ into eigenfunctors $P_i, Q_i\:(i \in \kk)$ 
with $P_i \cong P_{-i}$ and $Q_i \cong Q_{-i}$ (\cref{piqi,cliffordisos}),
and we have decomposed $\catR$ as 
$\bigoplus_{\lambda \in X} \catR_\lambda$ in such a way that $P_i$ (resp., $Q_i$) takes objects of $\catR_\lambda$
to $\catR_{\lambda+\alpha_i}$ (resp., $\catR_{\lambda-\alpha_i}$)
for $i \in I$ (\cref{cliffordisos}). So we
have in our hands a lot of the required data.
We still need to introduce
supernatural transformations 
$\begin{tikzpicture}[IKM,centerzero,scale=.7]
\draw[->] (0,-0.3) \botlabel{i} -- (0,0.3);
\singdot{0,0};
\end{tikzpicture}$,
$\begin{tikzpicture}[IKM,centerzero,scale=.7]
\draw[->] (0,-0.3) \botlabel{0} -- (0,0.3);
\token{0,0};
\end{tikzpicture}$,
$\begin{tikzpicture}[IKM,centerzero,scale=.7]
\draw[->] (-0.3,-0.3) \botlabel{i} -- (0.3,0.3);
\draw[->] (0.3,-0.3) \botlabel{j} -- (-0.3,0.3);
\end{tikzpicture}$,
$\begin{tikzpicture}[IKM,centerzero,scale=.7]
\draw[->] (-0.25,0.25) \toplabel{i} -- (-0.25,0) arc(180:360:0.25) -- (0.25,0.25);
\end{tikzpicture}$
and $\begin{tikzpicture}[IKM,centerzero,scale=.7]
\draw[->] (-0.25,-0.25) \botlabel{i} -- (-0.25,0) arc(180:0:0.25) -- (0.25,-0.25);
\end{tikzpicture}$
corresponding to the generating 2-morphisms of $\fV(\fg)$.

\subsection{Dots and Clifford tokens}

The next remarkable definition can be traced back to \cite[Sec.~5.3.2]{KKT16}. 
Recall that $\hbar = -\frac{1}{2}$, so $\hbar(\hbar+1) = -\frac{1}{4}$. Also, for $i \in I$, $b(i)$ is the square root 
$\sqrt{i(i+1)}$ as defined just before \cref{importantfunction}, i.e., it is the unique square root belonging to $J \subset \kk$.
For $i \in I$, we define a new variable $x_i \in \kk\llbracket x-b(i)\rrbracket$ by
\begin{align}\label{xi}
x_i &:=
\begin{cases}
\sqrt{x^2+\frac{1}{4}}-i-\frac{1}{2}
&\text{if $i \neq 0, \hbar$}\\
x^2+\frac{1}{4}
&\text{if $i = \hbar$}\\
\sqrt{\sqrt{x^2+\frac{1}{4}}-\frac{1}{2}}
&\text{if $i=0$.}
\end{cases}
\end{align}
The ambiguous signs in the square roots here should be chosen so that
\begin{align}\label{xi2} \textstyle
x_0 &= -x + \frac{x^3}{2} + \cdots,&
x_i &=  \frac{(x^2-i(i+1))}{2i+1}  - \frac{\left(x^2-i(i+1)\right)^2}{(2i+1)^3} + \dotsb\quad (i \ne 0,\hbar).
\end{align}
We recognize that our notation for $x_i$, which is a new variable depending implicitly on the original variable $x$ and $i \in I$, is a little unconventional, but it will appear often subsequently.

Note that $x_0 \in x \kk\llbracket x^2\rrbracket^\times$, and $x_i \in \left(x^2-i(i+1)\right)\kk\llbracket x^2-i(i+1)\rrbracket^\times$ when $i \neq 0$. 
Hence, for any $i \in I$, we have that
$x_i \in (x-b(i))\kk\llbracket x-b(i)\rrbracket^\times$.
In particular, this implies that $x_i = 0$ at $x=b(i)$.

Now we want to rearrange \cref{xi} to make $x$ the subject.
After some obvious rearranging and squaring, one
gets that
\begin{equation}\label{magic}
x^2 = \left(x_i^{1/d_i} + i\right)\left(x_i^{1/d_i}+i+1\right)
=
\begin{cases}
(x_i+i)(x_i+i+1)&\text{if $i \neq 0, \hbar$}\\
x_i+\hbar(\hbar+1)&\text{if $i=\hbar$}\\
x_i^2(x_i^2+1)&\text{if $i = 0$.}
\end{cases}
\end{equation}
It remains to take 
square roots on both sides of this identity. To be clear about signs, we use more non-standard piece of notation:
for a power series $f(x) \in \kk\llbracket x\rrbracket^\times$, we let $\sqrt[(x)]{f(x)}$ denote the unique square root of $f(x)$ in
$\kk\llbracket x\rrbracket$ whose constant term is equal to $\sqrt{f(0)}$ computed according to the square root function on $\kk$ 
specified just before \cref{importantfunction}.
Then, by \cref{magic}, we have that
\begin{equation}\label{magic2}
x = \begin{cases}
\sqrt[(x_i)]{(x_i+i)(x_i+i+1)}
&\text{if $i \neq 0, \hbar$}\\
\sqrt[(x_\hbar)]{x_\hbar+\hbar(\hbar+1)}
&\text{if $i = \hbar$}\\
x_0\sqrt[(x_0)]{x_0^2+1}&\text{if $i = 0$.}
\end{cases}
\end{equation}
This shows that $x-b(i) \in x_i \kk\llbracket x_i\rrbracket^\times$, 
and we have already noted that 
$x_i \in (x-b(i))\kk\llbracket x-b(i)\rrbracket^\times$, 
so $\kk\llbracket x - b(i)\rrbracket = \kk\llbracket x_i\rrbracket$.

Now we start again to use string diagrams. In \cref{s4-wsd}, we used the colored strings 
$\begin{tikzpicture}[H,baseline=-2mm,scale=.75]
\draw[->] (0,-0.2) \botlabel{i} -- (0,0.2);
\end{tikzpicture}$
and
$\begin{tikzpicture}[H,baseline=-2mm,scale=.75]
\draw[<-] (0,-0.2) \botlabel{i} -- (0,0.2);
\end{tikzpicture}$
to denote the identity endomorphisms of $P_i$ and $Q_i$, respectively, and investigated various supernatural transformations derived from the isomeric Heisenberg action. We switch now to denoting these identity endomorphisms by the thin black strings
$\begin{tikzpicture}[IKM,baseline=-2mm,scale=.75]
\draw[->] (0,-0.2) \botlabel{i} -- (0,0.2);
\end{tikzpicture}$
and
$\begin{tikzpicture}[IKM,baseline=-2mm,scale=.75]
\draw[<-] (0,-0.2) \botlabel{i} -- (0,0.2);
\end{tikzpicture}\ $. We are going to introduce new supernatural transformations which will be shown to satisfy the isomeric Kac--Moody relations.
To start with, we define the even supernatural transformations
\(
\begin{tikzpicture}[anchorbase,IKM]
\draw[->] (0,-.2) \botlabel{i} to (0,.25);
\singdot{0,.025};
\end{tikzpicture}
: P_i \Rightarrow P_i
\)
and
\(
\begin{tikzpicture}[anchorbase,IKM]
\draw[<-] (0,-.2) \botlabel{i} to (0,.25);
\singdot{0,.025};
\end{tikzpicture}
:  Q_i \Rightarrow Q_i
\)
by declaring that
\begin{align}\label{dotdef}
\begin{tikzpicture}[anchorbase,IKM]
\draw[->] (0,-.35) \botlabel{i} to (0,.35);
\singdot{0,0};
\end{tikzpicture}
&:=
\begin{tikzpicture}[anchorbase,H]
\draw[->] (0,-.35) \botlabel{i} to (0,.35);
\pin{0,0}{-0.5,0}{x_i};
\end{tikzpicture}\ ,&
\begin{tikzpicture}[anchorbase,IKM]
\draw[<-] (0,-.35) \botlabel{i} to (0,.35);
\singdot{0,0};
\end{tikzpicture}
&:=
\begin{tikzpicture}[anchorbase,H]
\draw[<-] (0,-.35) \botlabel{i} to (0,.35);
\pin{0,0}{-0.5,0}{x_i};
\end{tikzpicture}\:.
\end{align}
Here, we are using the useful diagrammatic convention introduced after \cref{electric}.

\begin{lem}\label{tunic}
For an irreducible object $L \in \catR$ and $i \in I$, the minimal polynomials of the endomorphisms
\(
\begin{tikzpicture}[IKM,anchorbase]
\draw[<-] (0,0.25) -- (0,-0.25)\botlabel{i};
\singdot{0,-0.03};
\draw[gcolor,thick] (0.3,0.25)  -- (0.3,-0.25)\botlabel{L};
\end{tikzpicture}
: P_i L \rightarrow P_i L
\)
and
\(
\begin{tikzpicture}[IKM,anchorbase]
\draw[<-] (0,-0.25) \botlabel{i}-- (0,0.25);
\singdot{0,0.03};
\draw[gcolor,thick] (0.3,-0.25) \botlabel{L} -- (0.3,0.25);
\end{tikzpicture}
: F_i L \rightarrow F_i L
\)
are $x^{\eps_i(L)}$ and $x^{\phi_i(L)}$, respectively.
\end{lem}

\begin{proof}
We just explain how to find the minimal polynomial of
\(
\psi :=
\begin{tikzpicture}[IKM,baseline=-2mm]
\draw[<-] (0,0.25) -- (0,-0.25)\botlabel{i};
\singdot{0,-0.03};
\draw[gcolor,thick] (0.3,0.25)  -- (0.3,-0.25)\botlabel{L};
\end{tikzpicture},
\)
the other case being similar. Let
\(
\theta:=
\begin{tikzpicture}[H,centerzero,scale=1.1]
\draw[->] (-0.15,-0.25) \botlabel{i} -- (-0.15,0.25);
\draw[gcolor,thick] (0.15,-0.25) \botlabel{L} -- (0.15,0.25);
\pin{-.15,0}{-1,0}{x-b(i)};
\end{tikzpicture}.
\)
By \cref{switchybitchy} the minimal polynomial of $\theta$ is $x^{\eps_i(L)}$. We have that $x_i = (x-b(i)) \xi_i(x)$ for $\xi_i(x) \in \kk\llbracket x-b(i)\rrbracket^\times$.
So $\psi = \theta \circ \nu = \nu \circ \theta$ for an automorphism $\nu : P_i L \rightarrow P_i L$.  It follows that the minimal polynomial of $\psi$ is also $x^{\eps_i(L)}$.
\end{proof}

We define odd supernatural transformations
\(
\begin{tikzpicture}[IKM,centerzero]
\draw[->] (0,-0.2) \botlabel{0} -- (0,0.22);
\token{0,0};
\end{tikzpicture}
: P_0 \Rightarrow P_0
\)
and
\(
\begin{tikzpicture}[IKM,centerzero]
\draw[<-] (0,-0.22) \botlabel{0} -- (0,0.2);
\token{0,0};
\end{tikzpicture}
: Q_0 \Rightarrow Q_0
\)
by setting
\begin{align}\label{tokendef}
\begin{tikzpicture}[anchorbase,IKM]
\draw[->] (0,-.35) \botlabel{0} to (0,.35);
\token{0,0};
\end{tikzpicture}
&:=
\begin{tikzpicture}[anchorbase,H]
\draw[->] (0,-.35) \botlabel{0} to (0,.35);
\token{0,0};
\end{tikzpicture}
\ ,&
\begin{tikzpicture}[anchorbase,IKM]
\draw[<-] (0,-.35) \botlabel{0} to (0,.35);
\token{0,0};
\end{tikzpicture}
&:=
\begin{tikzpicture}[anchorbase,H]
\draw[<-] (0,-.35) \botlabel{0} to (0,.35);
\token{0,0};
\end{tikzpicture}\ .
\end{align}

\begin{lem}\label{checkingrels1}
The supernatural transformations \cref{dotdef,tokendef} satisfy the quiver Hecke--Clifford superalgebra relations \cref{QHC1} with $i=0$.
\end{lem}

\begin{proof}
It is clear from \cref{sergeev2,torch} that the Clifford token squares to $-\id$ on an upward string and to $\id$ on a downward string.  The second relation in \cref{QHC1} when $i=0$ follows from the first relation in \cref{affsergeev} since $x_0 \in x \kk\llbracket x^2\rrbracket$.
\end{proof}

\subsection{Crossings}

As well as the notation $x_i$ for the new variable 
in $\kk\llbracket x-b(i)\rrbracket$ obtained from $x$ 
and $i \in I$ that satisfies \cref{xi,magic2}, we use
$y_i$ for the element of $\kk\llbracket y-b(i)\rrbracket$ defined in a similar way but replacing all occurrences of $x$ by $y$. 
Recall the rational function
\begin{equation} \label{peon}
p(x,y) = 1-\frac{1}{(x-y)^2}-\frac{1}{(x+y)^2}
= \frac{(x^2-y^2)^2-2(x^2+y^2)}{(x^2-y^2)^2}\in \kk(x,y)
\end{equation}
first seen in \cref{divergent}, which also appeared in \cref{pxyagain}.  Let
\begin{equation}\label{newp}
\Delta\big(x,y\big) := x(x+1) - y(y+1)
= (x-y)(x+y+1)
\in \kk[x,y].
\end{equation}
These polynomials are needed in the next lemma, which gives an explicit formula expressing $p(x,y)$ as a rational function in $\kk(x_i,y_j)$.

\begin{lem}\label{gunky}
For $i,j \in I$, we have that
\begin{equation}\label{sillyness}
p(x,y) = \frac{\Delta\big(x_i^{1/d_i}+i,y_j^{1/d_j}+j-1\big)
\Delta\big(x_i^{1/d_i}+i,y_j^{1/d_j}+j+1\big)}
{\Delta\big(x_i^{1/d_i}+i,y_j^{1/d_j}+j\big)^2}.
\end{equation}
If $j \neq \hbar$, then $\Delta\big(x_i^{1/d_i}+i, y_j^{1/d_j}+j\big)$ and $\Delta\big(x_i^{1/d_i}+i, y_j^{1/d_j}+j\pm 1\big)$ are polynomials in $\kk[x_i,y_j]$.  If $j = \hbar$, then $\Delta\big(x_i^{1/d_i}+i, y_j^{1/d_j}+j-1\big)$ and the product $\Delta\big(x_i^{1/d_i}+i, y_j^{1/d_j}+j-1\big) \Delta\big(x_i^{1/d_i}+i, y_j^{1/d_j}+j+1\big)$ are polynomials in $\kk[x_i,y_j]$.
\end{lem}

\begin{proof}
Since $y^2 = \big(y_j^{1/d_j} + j\big)\big(y_j^{1/d_j}+j+1\big)$ by \cref{magic}, we get from \cref{insurgent} that
\[
p(x,y_j) =
\frac{
\left[
x^2-\big(y_j^{1/d_j}+j-1\big)\big(y_j^{1/d_j}+j\big)
\right]
\left[
x^2-\big(y_j^{1/d_j}+j+1\big)\big(y_j^{1/d_j}+j+2\big)
\right]
}{
\left[
x^2-\big(y_j^{1/d_j}+j\big)\big(y_j^{1/d_j}+j+1\big)\right]^2
}.
\]
Replacing $x^2$ by $\big(x_i^{1/d_i}+i\big)\big(x_i^{1/d_i}+i+1\big)$ gives the formula \cref{sillyness}.  For the statement about polynomiality, $\big(x_i^{1/d_i}+i\big)\big(x_i^{1/d_i}+i+1\big)$ is a polynomial in $x_i$ for all values of $i \in I$, hence, $\Delta\big(x_i^{1/d_i}+i, y_j^{1/d_j}+j\big)$ is a polynomial in $x_i$ and $y_j$.  Also if $j \neq \hbar$ then the exponent $1/d_j$ is a positive integer, so  $\Delta\big(x_i^{1/d_i}+i,y_j^{1/d_j}+j\pm 1\big)$ are both polynomials.  It remains to observe by  an elementary calculation that
\begin{align*}
\Delta\big(x_i^{1/d_i}+i,y_\hbar^{1/2}+\hbar-1\big) \Delta\big(x_i^{1/d_i}+i,y_\hbar^{1/2}+\hbar+1\big) &=
\textstyle\left(\big(x_i^{1/d_i}+i\big)\big(x_i^{1/d_i}+i+1\big)
-y_\hbar - \frac{3}{4}\right)^2 - 4 y_\hbar,
\end{align*}
which is a polynomial.
\end{proof}

The next definitions were extracted from \cite[Sec.~5.3.3]{KKT16}. 
Recall the notation $\sqrt[(x)]{f(x)}$ introduced before \cref{magic2}. Similarly, for $f(x,y) \in \kk\llbracket x,y\rrbracket^\times$, 
$\sqrt[(x,y)]{f(x,y)}$ denotes the unique square root of $f(x,y)$
in $\kk\llbracket x,y\rrbracket$ whose constant term is $\sqrt{f(0,0)}$.
For $i \in I$, we let
\begin{equation}\label{fi}
f_i(x,y) :=
\begin{dcases}
\frac{\sqrt[(x)]{(x+i)(x+i+1)}+\sqrt[(y)]{(y+i)(y+i+1)}}{x+y + 2i+1}
&\text{if $i \neq 0,\hbar$}\\
\sqrt[(x)]{x+\hbar(\hbar+1)}+\sqrt[(y)]{y+\hbar(\hbar+1)}
&\text{if $i = \hbar$}\\
\frac{x\sqrt[(x)]{x^2+1}+y\sqrt[(y)]{y^2+1}}{(x+y)\left(x^2 + y^2+1\right)}
&\text{if $i=0$.}
\end{dcases}
\end{equation}
This is a power series in $\kk\llbracket x,y\rrbracket^\times$.  This is immediately clear from the definition when $i \neq 0$. To see it in the case $i=0$, note that $x^{2n+1}+y^{2n+1}$ is divisible by $x+y$ for each $n \in \N$, hence, $x\sqrt[(x)]{x^2+1} + y\sqrt[(y)]{y^2+1}$ is divisible by $x+y$ in $\kk\llbracket x,y\rrbracket$.
For $i, j \in I$, we define
\begin{align}\label{g}
g_{ij}(x,y) :=
\begin{dcases}
\sqrt[(x,y)]{\frac{\Delta\big(x^{1/d_i}+i,y^{1/d_j}+j\big)^2}
{\Delta\big(x^{1/d_i}+i,y^{1/d_j}+j-1\big)\Delta\big(x^{1/d_i}+i,y^{1/d_j}+j+1\big)}}
& \text{if $i \neq j, j \pm 1$}
\\
1 & \text{if $i = j-1$}
\\
\frac{\Delta\big(x^{1/d_i}+i,y^{1/d_j}+j\big)^2}{\big(x+y^{1/d_j}+2i+1\big)^{1-\delta_{i=\hbar}}
\Delta\big(x^{1/d_i}+i,y^{1/d_j}+j-1\big) }&\text{if $i = j+1$}
\\
\frac{\frac{1}{2}(x+y)+i-\hbar}{\sqrt[(x,y)]{\left(1-(x-y)^2\right)\left(\frac{1}{2}(x+y)+i\right)\left(\frac{1}{2}(x+y)+i+1\right)}}&\text{if $i=j \neq 0, \hbar$}
\\
\frac{1}{2\sqrt[(x,y)]{{\frac{1}{2}(x+y)-\frac{1}{4}(x-y)^2+\hbar(\hbar+1)}}}
& \text{if $i=j=\hbar$}\\
\frac{x^2+y^2+1}{\sqrt[(x,y)]{\left(1-\left(x^2-y^2\right)^2\right)\left(\frac{1}{2}\left(x^2+y^2\right)+1\right)}}
& \text{if $i=j=0$,}
\end{dcases}
\end{align}
working in $\kk\llbracket x,y\rrbracket$.
The definition makes sense because all of the power series in the numerators and denominators inside the square roots have non-zero constant term, hence, it makes sense to invert them or take their square roots.  To see this, one just has to set $x = y = 0$ in each of them, then to invoke \cref{rainisback} as needed.  This actually proves that $g_{ij}(x,y) \in \kk\llbracket x,y\rrbracket^\times$.  
Finally, for each $i \in I$, we define
\begin{align}\label{h}
h_i(x_i,y_i) &:=
\frac{1}{x_i-y_i}-\frac{g_{ii}(x_i,y_i)}{x-y}
\in \kk\lround x_i,y_i\rround.
\end{align}

\begin{lem}\label{moremagic}
For any $i \in I$, we have that $(x-y)f_i(x_i,y_i)=x_i-y_i$.
\end{lem}

\begin{proof}
This follows by an elementary calculation using \cref{magic2}.
\end{proof}

\begin{lem}\label{hiduff}
For any $i \in I$, we have that $f_i(x,x) g_{ii}(x,x) = 1$.
\end{lem}

\begin{proof}
This follows from the definitions.
\end{proof}

\begin{cor}\label{hidef}
We have that $h_i(x_i,y_i)
\in \kk\llbracket x_i,y_i\rrbracket$ for each $i \in I$.
\end{cor}

\begin{proof}
By \cref{hiduff}, the power series $1 - f_i(x_i,y_i) g_{ii}(x_i,y_i) \in \kk\llbracket x_i,y_i\rrbracket$ vanishes at $y_i = x_i$, so it is divisible by $x_i - y_i$.
Since $f_i(x_i,y_i) = \frac{x_i-y_i}{x-y}$ by
\cref{moremagic},
the quotient
$\frac{1 - f_i(x_i,y_i) g_{ii}(x_i,y_i)}{x_i-y_i}\in \kk\llbracket x_i,y_i\rrbracket$
is equal to $h_i(x_i,y_i)$.
\end{proof}

For $i,j \in I$, the following properties are easy to check from the definitions:
\begin{gather} \label{cash1}
g_{ii}(x_i,y_i) = g_{ii}(y_i,x_i),\quad
g_{0j}(x_0,y_j) = g_{0j}(-x_0,y_j),\quad
g_{i0}(x_i,y_0)=g_{i0}(x_i,-y_0),
\\ \label{cash3}
h_i(x_i,y_i) = -h_i(y_i,x_i),\quad
h_0(x_i,y_i)=-h_0(-x_i,-y_i).
\end{gather}
Also, the definition \cref{h} implies that
\begin{align}\label{doggydaycare}
\frac{g_{ii}(x_i,y_i)}{x-y} &=
\frac{1}{x_i-y_i} - h_i(x_i,y_i),&
\frac{g_{00}(x_0,y_0)}{x+y} &=
\frac{1}{x_0+y_0} - h_0(x_0,-y_0).
\end{align}
The key property of the power series $g_{ij}(x_i,y_j)$ is established in the next lemma; we imagine that Kang, Kashiwara and Tsuchioka discovered these power series in the first place by solving this equation.

\begin{lem}\label{echo}
For any $i,j \in I$, we have that
\begin{align}\label{money}
g_{ij}(x_i,y_j)g_{ji}(y_j,x_i)
&= (-1)^{\delta_{i=j}} \frac{q_{ij}(x_i,y_j)}{p(x,y)},
\end{align}
\end{lem}

\begin{proof}
We consider various cases.

\vspace{2mm}
\noindent
\underline{Case 1}: $i \neq j, j \pm 1$.
Applying \cref{sillyness} twice, we have that
\begin{align*}
\frac{1}{p(x,y)}
&= \frac{\Delta\big(x_i^{1/d_i}+i,y_j^{1/d_j}+j\big)^2}{
\Delta\big(x_i^{1/d_i}+i,y_j^{1/d_j}+j-1\big)\Delta\big(x_i^{1/d_i}+i,y_j^{1/d_j}+j+1\big)
},\\
\frac{1}{p(y,x)}
&= \frac{\Delta\big(y_j^{1/d_j}+j,x_i^{1/d_i}+i\big)^2}{
\Delta\big(y_j^{1/d_j}+j,x_i^{1/d_i}+i-1\big)\Delta\big(y_j^{1/d_j}+j,x_i^{1/d_i}+i+1\big)
}.
\end{align*}
Since $p(x,y) = p(y,x)$, these two power series are equal.  Since $g_{ij}(x_i,y_j) g_{ji}(y_j,x_i)$ is the product of their square roots by the definition \cref{g}, it follows that it equals $\frac{1}{p(x,y)}$, which is what we want because $q_{ij}(x_i,y_j) = 1$.

\vspace{2mm} 
\noindent
\underline{Case 2}: $i = j+1$.
When $j=i-1$, the second form of the definition \cref{newp} implies that
\[
\Delta\big(x_i^{1/d_i}+i,y_j^{1/d_j}+j+1\big)
= \big(x_i^{1/d_i}-y_j^{1/d_j}\big)\big( x_i^{1/d_i}+y_j^{1/d_j}+2i+1 \big).
\]
The assumption $i=j+1$ means that $i \neq 0$ and $j \neq \hbar$. So, by \cref{cartan}, we have that $-c_{ij} = 1$ and $-c_{ji} = 2^{\delta_{i=1}+\delta_{i=\hbar}}$.
Using this and considering the four cases $1 \neq i \neq \hbar$, $1=i\neq\hbar,1\leq i=\hbar$ and $1=i=\hbar$ separately, the above formula is equivalent to
\[
\Delta\big(x_i^{1/d_i}+i,y_j^{1/d_j}+j+1\big)
= \big(x_i^{-c_{ij}}-y_j^{-c_{ji}}\big)\big(x_i+y_j^{1/d_j}+2i+1\big)^{1-\delta_{i=\hbar}}.
\]
From this and \cref{sillyness}, it follows that
\[
\frac{x_i^{-c_{ij}}-y_j^{-c_{ji}}}{p(x,y)}
=
\frac{\Delta\big(x_i^{1/d_i}+i,y_j^{1/d_j}+j)^2}{\Delta\big(x_i^{1/d_i}+i,y_j^{1/d_j}+j-1\big) \big(x_i+y_j^{1/d_j}+2i+1\big)^{1-\delta_{i=\hbar}}},
\]
which is $g_{ij}(x_i,y_j)$ by the definition \cref{g}.  Since $g_{ji}(y_j,x_i)=(i-j)=1$, this implies the required formula $g_{ij}(x_i,y_j) g_{ji}(y_j,x_i) = \frac{(i-j)\big(x_i^{-c_{ij}}-y_j^{-c_{ji}}\big)}{p(x,y)} = \frac{q_{ij}(x_i,y_j)}{p(x,y)}$.

\vspace{2mm} 
\noindent
\underline{Case 3}: $i = j-1$. By Case 2, we have that $g_{ji}(y_j,x_i) g_{ij}(x_i,y_j) = \frac{(j-i) \big(y_j^{-c_{ji}}-x_i^{-c_{ij}}\big)}{p(y,x)}=\frac{q_{ij}(x_i,y_j)}{p(x,y)}$.

\vspace{2mm} 
\noindent
\underline{Case 4}: $i = j \neq 0, \hbar$.
In this case, $q_{ii}(x_i,y_i) = \frac{1}{(x_i-y_i)^2}$.  Again we start from the identity \cref{sillyness}.  Using the second form of the definition \cref{newp}, noting that $d_i=d_j=1$, it implies that
\[
\frac{p(x,y)}{q_{ii}(x_i,y_i)}
=
\frac{(x_i-y_i-1)(x_i-y_i+1)(x_i+y_i+2i)(x_i+y_i+2i+2)}{(x_i+y_i+2i+1)^2}.
\]
So, starting from \cref{g} with its numerator and denominator doubled, we have that
\begin{align*}
g_{ii}(x_i,y_i)g_{ii}(y_i,x_i)
= \frac{(x_i+y_i+2i+1)^2}{\left(1-(x_i-y_i)^2\right)\left(x_i+y_i+2i\right)\left(x_i+y_i+2i+2\right)}
= -\frac{q_{ii}(x_i,y_i)}{p(x,y)}.
\end{align*}

\vspace{2mm} 
\noindent
\underline{Case 5}: $i=j=\hbar$.
Again $q_{\hbar\hbar}(x_\hbar,y_\hbar) = \frac{1}{(x_\hbar-y_\hbar)^2}$ We start by noting from the second form of \cref{newp} that
\begin{align*}
\Delta\big(x_\hbar^{1/d_\hbar}+\hbar,
y_\hbar^{1/d_\hbar} + \hbar-1\big)
&\Delta\big(x_\hbar^{1/d_\hbar}+\hbar,
y_\hbar^{1/d_\hbar} + \hbar+1\big)
\\
&= \left(\left(x_\hbar^{1/2}-y_\hbar^{1/2}\right)^2-1\right)
\left(\left(x_\hbar^{1/2}+y_\hbar^{1/2}\right)^2-1\right)
\\
&= \left(x_\hbar+y_\hbar-1-2x_\hbar^{1/2}y_\hbar^{1/2}\right)
\left(x_\hbar+y_\hbar-1+2x_\hbar^{1/2}y_\hbar^{1/2}\right)
\\
&= (x_\hbar+y_\hbar-1)^2 - 4 x_\hbar y_\hbar
= 1 - 2 (x_\hbar+y_\hbar) + (x_\hbar-y_\hbar)^2
\end{align*}
and $\Delta\big(x_\hbar^{1/d_\hbar}+\hbar, y_\hbar^{1/d_\hbar} + \hbar\big) \Delta\big(x_\hbar^{1/d_\hbar}+\hbar, y_\hbar^{1/d_\hbar} + \hbar\big) = x_\hbar-y_\hbar$.  So \cref{sillyness} gives that
\[
\frac{p(x,y)}{q_{\hbar,\hbar}(x_\hbar,y_\hbar)}
 = 1 - 2 (x_\hbar+y_\hbar) + (x_\hbar-y_\hbar)^2.
\]
Now we can use the definition \cref{g} to deduce that
\[
g_{\hbar\hbar}(x_\hbar,y_\hbar)g_{\hbar\hbar}(y_\hbar,x_\hbar) =
-\frac{1}{1-2(x_\hbar+y_\hbar)+(x_\hbar-y_\hbar)^2}
= -\frac{q_{\hbar\hbar}(x_\hbar,y_\hbar)}{p(x,y)}.
\]

\vspace{2mm} 
\noindent
\underline{Case 6}: $i=j=0$.
Now
\[
q_{00}(x_0,y_0) = \frac{1}{(x_0-y_0)^2}+\frac{1}{(x_0+y_0)^2}
= \frac{2\big(x_0^2+y_0^2\big)}{\big(x_0^2-y_0^2\big)^2}.
\]
From \cref{sillyness,newp}, we have
\[
\frac{p(x,y)}{q_{00}(x_0,y_0)}
=
\frac{\left(\left(x_0^2-y_0^2\right)^2-1\right)\left(x_0^2+y_0^2+2\right)}{2\left(x_0^2+y_0^2+1\right)^2}.
\]
Using this for the last equality, the definition \cref{g} gives
\begin{align*}
g_{00}(x_0,y_0)g_{00}(y_0,x_0)
&= \frac{2\big(x_0^2+y_0^2+1\big)^2}{\left(1-\big(x_0^2-y_0^2\big)^2\right)\big(x_0^2+y_0^2+2\big)}=
-\frac{q_{00}(x_0,y_0)}{p(x,y)}.
\qedhere
\end{align*}
\end{proof}

\begin{cor}
For $i \in I$, we have that
\begin{equation}\label{newmoney}
g_{ii}(x_i,y_i)^2 \left(1-\frac{\delta_{i \neq 0}}{(x+y)^2}\right)
=h_i(x_i,y_i)^2 - \frac{2h_i(x_i,y_i)}{x_i-y_i}
+ \delta_{i=0}\left(h_0(x_0,-y_0)^2 - \frac{2h_0(x_0,-y_0)}{x_0+y_0}\right).
\end{equation}
\end{cor}

\begin{proof}
We note for any $i \in I$ that
\begin{align*}
\frac{g_{ii}(x_i,y_i)^2}{(x-y)^2}-\frac{1}{(x_i-y_i)^2}
&=
\left[\frac{g_{ii}(x_i,y_i)}{x-y}-\frac{1}{x_i-y_i}\right]
\left[\frac{g_{ii}(x_i,y_i)}{x-y}+\frac{1}{x_i-y_i}\right]
\\&\!\!
\overset{\cref{h}}{\underset{\cref{doggydaycare}}{=}} h_i(x_i,y_i) \left[h_i(x_i,y_i)-\frac{2}{x_i-y_i}\right].
\end{align*}
This shows that
\begin{equation}\label{dinnertime}
\frac{g_{ii}(x_i,y_i)^2}{(x-y)^2}-\frac{1}{(x_i-y_i)^2}
= h_i(x_i,y_i)^2-\frac{2h_i(x_i,y_i)}{x_i-y_i}.
\end{equation}
Taking $i=0$ and replacing $y$ by $-y$, hence, $y_0$ by $-y_0$, gives also that
\begin{equation}\label{dinnertime2}
\frac{g_{00}(x_0,-y_0)^2}{(x+y)^2}-\frac{1}{(x_0+y_0)^2}
= h_i(x_0,-y_0)^2-\frac{2h_0(x_0,-y_0)}{x_0+y_0}.
\end{equation}

Now, to prove \cref{newmoney},
we first treat the case that $i \neq 0$. By \cref{money,cash1} with $i=j \neq 0$, we have that
\[
g_{ii}(x_i,y_i)^2 p(x,y) = -\frac{1}{(x_i-y_i)^2}.
\]
Hence, using the definition of $p(x,y)$ from \cref{peon}, we have that
\[
g_{ii}(x_i,y_i)^2 \left(1-\frac{1}{(x+y)^2} \right)
= \frac{g_{ii}(x_i,y_i)^2}{(x-y)^2}-\frac{1}{(x_i-y_i)^2}
\overset{\cref{dinnertime}}{=} h_i(x_i,y_i)^2 -  \frac{2h_{i}(x_i,y_i)}{x_i-y_i}.
\]
To prove \cref{newmoney} when $i=0$, \cref{money,cash1} with $i=j=0$ imply
\[
g_{00}(x_0,y_0)^2 p(x,y)
= -\frac{1}{(x_0-y_0)^2} - \frac{1}{(x_0+y_0)^2}.
\]
Hence, using the definition of $p(x,y)$, we get
\begin{align*}
g_{00}(x_0,y_0)^2
&= \frac{g_{00}(x_0,y_0)^2}{(x-y)^2}-\frac{1}{(x_0-y_0)^2}+
\frac{g_{00}(x_0,y_0)^2}{(x+y)^2}-\frac{1}{(x_0+y_0)^2}
\\
&\overset{\cref{dinnertime}}{\underset{\cref{dinnertime2}}=}
h_0(x,y)^2 -  \frac{2h_0(x_0,y_0)}{x_0-y_0}
+
h_0(x,-y)^2 - \frac{2h_0(x_0,-y_0)}{x_0+y_0}.
\qedhere
\end{align*}
\end{proof}

For $i,j \in I$, we define an even supernatural transformation
\(
\begin{tikzpicture}[IKM,centerzero,scale=.8]
\draw[->] (-0.3,-0.3) \botlabel{i} -- (0.3,0.3);
\draw[->] (0.3,-0.3) \botlabel{j} -- (-0.3,0.3);
\end{tikzpicture}
: P_i \circ P_j \Rightarrow P_j \circ P_i
\)
by setting
\begin{align}\label{crossingdef}
\begin{tikzpicture}[IKM,centerzero,scale=1.2]
\draw[->] (-0.3,-0.3) \botlabel{i} -- (0.3,0.3);
\draw[->] (0.3,-0.3) \botlabel{j} -- (-0.3,0.3);
\end{tikzpicture}\
&:=\
\begin{tikzpicture}[centerzero,H,scale=1.4]
\draw[->] (-.2,-.4)\botlabel{i} to (-.2,-.3) to [out=up,in=-135,looseness=1] (.2,.3)\toplabel{i};
\draw[->] (.2,-.4)\botlabel{j} to (.2,-.3) to[out=up,in=-45,looseness=1] (-.2,.3) \toplabel{j};
\pinpin{.19,-.2}{-.19,-.2}{-.9,-.2}{g_{ij}(x_i,y_j)};
\projcr{0,0.1};
\end{tikzpicture}
+\delta_{i=j}\
\begin{tikzpicture}[centerzero,H,scale=1.4]
\draw[->] (-0.2,-0.3) \botlabel{i} -- (-0.2,0.3);
\draw[->] (0.2,-0.3) \botlabel{i}  -- (0.2,0.3);
\pinpin{.2,0}{-.2,0}{-.9,0}{h_i(x_i,y_i)};
\end{tikzpicture}-\delta_{i=-j}\
\begin{tikzpicture}[centerzero,H,scale=1.4]
\draw[->] (-0.2,-0.3) \botlabel{0} -- (-0.2,0.3);
\draw[->] (0.2,-0.3) \botlabel{0}  -- (0.2,0.3);
\token{.2,-.1};\token{-.2,.1};
\pinpin{.2,.1}{-.2,-.1}{-.9,-.1}{h_0(x_0,y_0)};
\end{tikzpicture}.
\end{align}
We remind the reader again of \cref{jonisdotty}.
We also let
\begin{align}\label{di}
t_i(x,y) &:= \frac{f_i(x,x) - f_i(x,y)}{x-y} \in \kk\llbracket x,y\rrbracket.
\end{align}
The divisibility here is obvious.
This variant is needed due to our next lemma.

\begin{lem}\label{doingsomerotation}
For $i,j \in I$, we have that
\begin{align}\label{crossingalt}
\begin{tikzpicture}[IKM,centerzero,scale=1.2]
\draw[->] (-0.3,-0.3) \botlabel{i} -- (0.3,0.3);
\draw[->] (0.3,-0.3) \botlabel{j} -- (-0.3,0.3);
\end{tikzpicture}\
&=\
\begin{tikzpicture}[centerzero,H,scale=1.6]
\draw[->] (-.2,-.3)\botlabel{i} to [out=up,in=down,looseness=1] (.2,.3)\toplabel{i};
\draw[->] (.2,-.3)\botlabel{j}to[out=up,in=down,looseness=1] (-.2,.3) \toplabel{j};
\pinpin{-.18,.18}{-.18,-.18}{-.8,-.18}{g_{ij}(x_i,y_j)};
\projcr{0,0};
\end{tikzpicture}
+\delta_{i=j}\
\begin{tikzpicture}[centerzero,H,scale=1.6]
\draw[->] (-0.2,-0.3) \botlabel{i} -- (-0.2,0.3);
\draw[->] (0.2,-0.3) \botlabel{i}  -- (0.2,0.3);
\pinpin{-.2,0}{.2,0}{1.1,0}{g_{ii}(x_i,x_i)t_i(x_i,y_i)};
\end{tikzpicture}-\delta_{i=-j}\
\begin{tikzpicture}[centerzero,H,scale=1.6]
\draw[->] (-0.2,-0.3) \botlabel{0} -- (-0.2,0.3);
\draw[->] (0.2,-0.3) \botlabel{0}  -- (0.2,0.3);
\token{.2,-.1};\token{-.2,.1};
\pinpin{-.2,-.1}{.2,.1}{1.2,.1}{g_{00}(x_0,x_0)t_0(x_0,y_0)};
\end{tikzpicture}
\\ \label{crossingalt2}
&=\
\begin{tikzpicture}[centerzero,H,scale=1.6]
\draw[->] (-.2,-.3)\botlabel{i} to [out=up,in=down,looseness=1] (.2,.3)\toplabel{i};
\draw[->] (.2,-.3)\botlabel{j}to[out=up,in=down,looseness=1] (-.2,.3) \toplabel{j};
\pinpin{.18,.18}{.18,-.18}{.8,-.18}{g_{ij}(y_i,x_j)};
\projcr{0,0};
\end{tikzpicture}
-\delta_{i=j}\
\begin{tikzpicture}[centerzero,H,scale=1.6]
\draw[->] (-0.2,-0.3) \botlabel{i} -- (-0.2,0.3);
\draw[->] (0.2,-0.3) \botlabel{i}  -- (0.2,0.3);
\pinpin{-.2,0}{.2,0}{1.1,0}{g_{ii}(y_i,y_i)t_i(y_i,x_i)};
\end{tikzpicture}+\delta_{i=-j}\
\begin{tikzpicture}[centerzero,H,scale=1.6]
\draw[->] (-0.2,-0.3) \botlabel{0} -- (-0.2,0.3);
\draw[->] (0.2,-0.3) \botlabel{0}  -- (0.2,0.3);
\token{.2,-.1};\token{-.2,.1};
\pinpin{-.2,-.1}{.2,.1}{1.2,.1}{g_{00}(y_0,y_0)t_0(y_0,x_0)};
\end{tikzpicture}
\ .
\end{align}
\end{lem}

\begin{proof}
If $i \neq j$ this follows from \cref{thai}.  If $i=j$, it follows instead from \cref{naha} using
\begin{gather*}
\frac{g_{ii}(x_i,x_i) - g_{ii}(x_i,y_i)}{x-y}
= h_{i}(x_i,y_i)- g_{ii}(x_i,x_i)t_i(x_i,y_i),
\\
\frac{g_{ii}(x_i,y_i) - g_{ii}(y_i,y_i)}{x-y}
=  -h_{i}(x_i,y_i)-g_{ii}(y_i,y_i)t_i(y_i,x_i) .
\end{gather*}
These two formulae follow from the definition \cref{h} and \cref{moremagic,hiduff}.
\end{proof}

\subsection{The Kang--Kashiwara--Tsuchioka theorem}

The following
is a version of \cite[Th.~5.4]{KKT16}.

\begin{theo}\label{checkingrels2}
The supernatural transformations represented by dots, Clifford tokens and upward crossings defined in \cref{dotdef,tokendef,crossingdef} satisfy the quiver Hecke--Clifford superalgebra relations \cref{QHC1,QHC2,QHC3a,QHC3b,QHC4,QHC5} for all admissible $i,j \in I$.
\end{theo}

\begin{proof}
See \cite[Sec.~5]{KKT16}.
Since we have a slightly different setup to \cite{KKT16}, due diligence dictates that we should give some more details. The relations \cref{QHC1} have already been checked in \cref{checkingrels1}, the relations \cref{QHC2} are straightforward, and the relations \cref{QHC3a,QHC3b} when $i \neq j$ follow from \cref{thai}.
In the next paragraph, we prove \cref{QHC3a} when $i=j$, and a similar argument establishes \cref{QHC3b} in this case.
Then, in the paragraph after that, we give a proof of \cref{QHC4}.
This just leaves the braid relation \cref{QHC5}. We were able to verify this by equally naive direct calculations using Lemma~\ref{projbraid}, \cref{doggydaycare,newmoney}, but the approach taken in \cite{KKT16},
exploiting the favorable properties of the intertwiners from \cite[(3.4)]{Naz} (see also \cite[Sec.~14.8]{Kbook}), is more efficient. It still involves some tedious calculations in order to establish those properties in the first place, working in an intermediate algebra defined by localizing at certain morphisms.

To prove \cref{QHC3a} when $i=j$, suppose first that
$i=j\neq 0$. Then we have that
\begin{align*}
\begin{tikzpicture}[IKM,centerzero]
\draw[->] (-0.3,-0.3) \botlabel{i} -- (0.3,0.3);
\draw[->] (0.3,-0.3) \botlabel{i} -- (-0.3,0.3);
\singdot{-0.15,-0.15};
\end{tikzpicture}
\ \ &\overset{\cref{crossingdef}}{=}\
\begin{tikzpicture}[H,centerzero]
\draw[->] (-0.2,-0.6) \botlabel{i} to (-0.2,-0.4) to [out=up,in=down,looseness=1] (0.2,0.4) \toplabel{i};
\draw[->] (0.2,-0.6) \botlabel{i} to (0.2,-0.4) to[out=up,in=down,looseness=1] (-0.2,0.4) \toplabel{i};
\pinpin{-0.2,-0.4}{0.2,-0.4}{1.4,-0.4}{g_{ii}(x_i,y_i)x_i};
\projcr{0,0};
\end{tikzpicture}
+
\begin{tikzpicture}[H,centerzero]
\draw[->] (-0.2,-0.5) \botlabel{i} -- (-0.2,0.5);
\draw[->] (0.2,-0.5) \botlabel{i}  -- (0.2,0.5);
\pinpin{-0.2,0}{0.2,0}{1.4,0}{h_i(x_i,y_i)x_i};
\end{tikzpicture}
\\
&\overset{\cref{naha}}{=}\
\begin{tikzpicture}[H,centerzero]
\draw[->] (-0.2,-0.6) \botlabel{i} to (-0.2,-0.4) to [out=up,in=down,looseness=1] (0.2,0.4) -- (0.2,0.6) \toplabel{i};
\draw[->] (0.2,-0.6) \botlabel{i} to (0.2,-0.4) to[out=up,in=down,looseness=1] (-0.2,0.4) -- (-0.2,0.6) \toplabel{i};
\pinpin{-0.2,-0.4}{0.2,-0.4}{1.3,-0.4}{g_{ii}(x_i,y_i)};
\pinpin{-.2,.4}{0.2,0.4}{0.9,0.4}{y_i};
\projcr{0,0};
\end{tikzpicture}
+
\begin{tikzpicture}[H,centerzero]
\draw[->] (-0.2,-0.5) \botlabel{i} -- (-0.2,0.5);
\draw[->] (0.2,-0.5) \botlabel{i}  -- (0.2,0.5);
\pinpin{-0.2,0}{0.2,0}{2.6,0}{h_i(x_i,y_i)x_i + \frac{g_{ii}(x_i,y_i)(x_i-y_i)}{x-y}};
\end{tikzpicture}
\\
&\overset{\cref{h}}{=}\
\begin{tikzpicture}[H,centerzero]
\draw[->] (-0.2,-0.6)\botlabel{i} to (-0.2,-0.4) to [out=up,in=down,looseness=1] (0.2,0.4) -- (0.2,0.6) \toplabel{i};
\draw[->] (0.2,-0.6)\botlabel{i} to (0.2,-0.4) to[out=up,in=down,looseness=1] (-0.2,0.4) -- (-0.2,0.6) \toplabel{i};
\pinpin{-0.2,-0.4}{0.2,-0.4}{1.3,-0.4}{g_{ii}(x_i,y_i)};
\pinpin{-.2,.4}{0.2,0.4}{0.9,0.4}{y_i};
\projcr{0,0};
\end{tikzpicture}
+
\begin{tikzpicture}[H,centerzero]
\draw[->] (-0.2,-0.5) \botlabel{i} -- (-0.2,0.5);
\draw[->] (0.2,-0.5) \botlabel{i}  -- (0.2,0.5);
\pinpin{-0.2,0}{0.2,0}{1.6,0}{h_i(x_i,y_i)y_i+1};
\end{tikzpicture}
\overset{\cref{crossingdef}}{=}
\begin{tikzpicture}[IKM,centerzero]
\draw[->] (-0.3,-0.3) \botlabel{i} -- (0.3,0.3);
\draw[->] (0.3,-0.3) \botlabel{i} -- (-0.3,0.3);
\singdot{0.15,0.15};
\end{tikzpicture}
+
\begin{tikzpicture}[IKM,centerzero]
\draw[->] (-0.2,-0.3) \botlabel{i} -- (-0.2,0.3);
\draw[->] (0.2,-0.3) \botlabel{i} -- (0.2,0.3);
\end{tikzpicture}
.
\end{align*}
Instead, if $i=j=0$, we have that
\begin{align*}
\begin{tikzpicture}[IKM,centerzero]
\draw[->] (-0.3,-0.3) \botlabel{i} -- (0.3,0.3);
\draw[->] (0.3,-0.3) \botlabel{i} -- (-0.3,0.3);
\singdot{-0.15,-0.15};
\end{tikzpicture}
\ &\overset{\cref{crossingdef}}{=}\
\begin{tikzpicture}[H,centerzero]
\draw[->] (-0.2,-0.6) \botlabel{0} to (-0.2,-0.4) to [out=up,in=down,looseness=1] (0.2,0.4) \toplabel{0};
\draw[->] (0.2,-0.6) \botlabel{0} to (0.2,-0.4) to[out=up,in=down,looseness=1] (-0.2,0.4) \toplabel{0};
\pinpin{-0.2,-0.4}{0.2,-0.4}{1.5,-0.4}{g_{00}(x_0,y_0)x_0};
\projcr{0,0};
\end{tikzpicture}
+
\begin{tikzpicture}[H,centerzero]
\draw[->] (-0.2,-0.5) \botlabel{0} -- (-0.2,0.5);
\draw[->] (0.2,-0.5) \botlabel{0}  -- (0.2,0.5);
\pinpin{0.2,0}{-0.2,0}{-1.4,0}{h_0(x_0,y_0)x_0};
\end{tikzpicture}
-
\begin{tikzpicture}[H,centerzero]
\draw[->] (-0.2,-0.5) \botlabel{0} -- (-0.2,0.5);
\draw[->] (0.2,-0.5) \botlabel{0}  -- (0.2,0.5);
\token{0.2,-0.15};
\token{-0.2,0.15};
\pinpin{0.2,0.15}{-0.2,-0.15}{-1.4,-0.15}{h_0(x_0,y_0)x_0};
\end{tikzpicture}
\\
&\overset{\cref{naha}}{\underset{\cref{cash1}}{=}}
\begin{tikzpicture}[H,centerzero]
\draw[->] (-0.2,-0.6) \botlabel{0} to (-0.2,-0.4) to [out=up,in=down,looseness=1] (0.2,0.4) -- (0.2,0.6) \toplabel{0};
\draw[->] (0.2,-0.6) \botlabel{0} to (0.2,-0.4) to[out=up,in=down,looseness=1] (-0.2,0.4) -- (-0.2,0.6) \toplabel{0};
\pinpin{-0.2,-0.4}{0.2,-0.4}{1.3,-0.4}{g_{00}(x_0,y_0)};
\pinpin{-.2,.4}{0.2,0.4}{0.8,0.4}{y_0};
\projcr{0,0};
\end{tikzpicture}
+
\begin{tikzpicture}[H,centerzero]
\draw[->] (-0.2,-0.5) \botlabel{0} -- (-0.2,0.5);
\draw[->] (0.2,-0.5) \botlabel{0}  -- (0.2,0.5);
\pinpin{.2,0}{-.2,0}{-2.7,0}{h_0(x_0,y_0)x_0 + \frac{g_{00}(x_0,y_0)(x_0-y_0)}{x-y}};
\end{tikzpicture}
-
\begin{tikzpicture}[H,centerzero]
\draw[->] (-0.2,-0.5) \botlabel{0} -- (-0.2,0.5);
\draw[->] (0.2,-0.5) \botlabel{0} -- (0.2,0.5);
\token{-0.2,0.1};
\token{0.2,-0.2};
\pinpin{0.2,0.1}{-0.2,-0.2}{-2.8,-0.2}{h_0(x_0,y_0)x_0+\frac{g_{00}(x_0,y_0)(x_0-y_0)}{x-y}};
\end{tikzpicture}
\\
&\overset{\cref{h}}{=}\ \
\begin{tikzpicture}[H,centerzero]
\draw[->] (-0.2,-0.6) \botlabel{0} to (-0.2,-0.4) to [out=up,in=down,looseness=1] (0.2,0.4) -- (0.2,0.6) \toplabel{0};
\draw[->] (0.2,-0.6) \botlabel{0} to (0.2,-0.4) to[out=up,in=down,looseness=1] (-0.2,0.4) -- (-0.2,0.6) \toplabel{0};
\pinpin{-0.2,-0.4}{0.2,-0.4}{1.3,-0.4}{g_{00}(x_0,y_0)};
\pinpin{-.2,.4}{0.2,0.4}{0.8,0.4}{y_0};
\projcr{0,0};
\end{tikzpicture}
+
\begin{tikzpicture}[H,centerzero]
\draw[->] (-0.2,-0.5) \botlabel{0} -- (-0.2,0.5);
\draw[->] (0.2,-0.5) \botlabel{0}  -- (0.2,0.5);
\pinpin{.2,0}{-.2,0}{-1.6,0}{h_0(x_0,y_0)y_0 + 1};
\end{tikzpicture}
-
\begin{tikzpicture}[H,centerzero]
\draw[->] (-0.2,-0.5) \botlabel{0} -- (-0.2,0.5);
\draw[->] (0.2,-0.5) \botlabel{0} -- (0.2,0.5);
\token{-0.2,0.1};
\token{0.2,-0.2};
\pinpin{0.2,0.1}{-0.2,-0.2}{-1.6,-0.2}{h_0(x_0,y_0)y_0+1};
\end{tikzpicture}\:\:
\overset{\cref{crossingdef}}{=}\:\:
\begin{tikzpicture}[IKM,centerzero]
\draw[->] (-0.3,-0.3) \botlabel{i} -- (0.3,0.3);
\draw[->] (0.3,-0.3) \botlabel{i} -- (-0.3,0.3);
\singdot{0.15,0.15};
\end{tikzpicture}
+
\begin{tikzpicture}[IKM,centerzero]
\draw[->] (-0.2,-0.3) \botlabel{i} -- (-0.2,0.3);
\draw[->] (0.2,-0.3) \botlabel{i} -- (0.2,0.3);
\end{tikzpicture}
-
\begin{tikzpicture}[IKM,centerzero]
\draw[->] (-0.2,-0.3) \botlabel{i} -- (-0.2,0.3);
\draw[->] (0.2,-0.3) \botlabel{i} -- (0.2,0.3);
\token{-0.2,0};
\token{0.2,0};
\end{tikzpicture}.
\end{align*}

Now we prove \cref{QHC4}. Suppose first that $i \neq j$.
Since $i,j \in I$, we have $i \ne -j$. We compute:
\begin{align*}
\begin{tikzpicture}[IKM,centerzero]
\draw[->] (-0.2,-0.4) \botlabel{i} to[out=45,in=down] (0.15,0) to[out=up,in=-45] (-0.2,0.4);
\draw[->] (0.2,-0.4) \botlabel{j} to[out=135,in=down] (-0.15,0) to[out=up,in=225] (0.2,0.4);
\end{tikzpicture}&
\overset{\cref{crossingdef}}{=}
\begin{tikzpicture}[H,centerzero]
\draw[->] (-0.2,-1) \botlabel{i} -- (-0.2,-0.8) to[out=up,in=down] (0.2,0) to[out=up,in=down] (-0.2,0.8) \toplabel{i};
\draw[->] (0.2,-1) \botlabel{j} -- (0.2,-0.8) to[out=up,in=down] (-0.2,0) to[out=up,in=down] (0.2,0.8) \toplabel{j};
\node at (-0.25,0.25) {\strandlabel{j}};
\node at (0.25,0.25) {\strandlabel{i}};
\pinpin{-0.2,0}{0.2,0}{1.3,0}{g_{ji}(x_j,y_i)};
\pinpin{-0.2,-0.8}{0.2,-0.8}{1.3,-0.8}{g_{ij}(x_i,y_j)};
\projcr{0,-0.4};
\projcr{0,0.4};
\end{tikzpicture}
\overset{\cref{shortthai}}{=}
\begin{tikzpicture}[H,centerzero]
\draw[->] (-0.2,-1) \botlabel{i} -- (-0.2,-0.8) to[out=up,in=down] (0.2,0) to[out=up,in=down] (-0.2,0.8) \toplabel{i};
\draw[->] (0.2,-1) \botlabel{j} -- (0.2,-0.8) to[out=up,in=down] (-0.2,0) to[out=up,in=down] (0.2,0.8) \toplabel{j};
\node at (-0.3,0) {\strandlabel{j}};
\node at (0.3,0) {\strandlabel{i}};
\pinpin{-0.2,-0.8}{0.2,-0.8}{1.9,-0.8}{g_{ij}(x_i,y_j)g_{ji}(y_j,x_i)};
\projcr{0,-0.4};
\projcr{0,0.4};
\end{tikzpicture}
\overset{\cref{cold}}{\underset{\cref{money}}{=}}
\begin{tikzpicture}[H,centerzero]
\draw[->] (-0.2,-0.5) \botlabel{i} -- (-0.2,0.5);
\draw[->] (0.2,-0.5) \botlabel{j}  -- (0.2,0.5);
\pinpin{-0.2,0}{0.2,0}{1.2,0}{q_{ij}(x_i,y_j)};
\end{tikzpicture}
\overset{\cref{dotdef}}{=}
\begin{tikzpicture}[IKM,centerzero]
\draw[->] (-0.2,-0.5) \botlabel{i} -- (-0.2,0.5);
\draw[->] (0.2,-0.5) \botlabel{j}  -- (0.2,0.5);
\pinpin{-0.2,0}{0.2,0}{1.2,0}{q_{ij}(x,y)};
\end{tikzpicture}\ .
\end{align*}
Now suppose that $i=j$. Note that $\partial_{xy} g_{ii}(x_i,y_i) = 0$  by \cref{cash1}, $\partial_{xy}h_0(x_0,-y_0) = 0$ by \cref{cash3},
and $\partial_{xy} h_i(x_i,y_i) = \frac{2 h_i(x_i,y_i)}{x-y}$ by \cref{cash3}.
Using these observations, we calculate:
\begin{align*}
\begin{tikzpicture}[IKM,centerzero]
\draw[->] (-0.2,-0.4) \botlabel{i} to[out=45,in=down] (0.15,0) to[out=up,in=-45] (-0.2,0.4);
\draw[->] (0.2,-0.4) \botlabel{i} to[out=135,in=down] (-0.15,0) to[out=up,in=225] (0.2,0.4);
\end{tikzpicture}
\ &\overset{\mathclap{\substack{\cref{crossingdef} \\ \cref{tie}}}}{\underset{\cref{affsergeev}}{=}}\
\begin{tikzpicture}[H,anchorbase]
\draw[->] (-0.2,-1) \botlabel{i} -- (-0.2,-0.8) to[out=up,in=down] (0.2,0) to[out=up,in=down] (-0.2,0.8)\toplabel{i};
\draw[->] (0.2,-1) \botlabel{i} -- (0.2,-0.8) to[out=up,in=down] (-0.2,0) to[out=up,in=down] (0.2,0.8)\toplabel{i};
\pinpin{-0.19,-.1}{0.19,-.1}{1.2,-.1}{g_{ii}(x_i,y_i)};
\pinpin{-0.2,-0.8}{0.2,-0.8}{1.2,-0.8}{g_{ii}(x_i,y_i)};
\projcr{0,-0.4};
\projcr{0,0.4};
\node at (-0.3,0.15) {\strandlabel{i}};
\node at (0.3,0.15) {\strandlabel{i}};
\end{tikzpicture}
+
\begin{tikzpicture}[H,anchorbase]
\draw[->] (-0.2,-0.6)\botlabel{i}  to (-0.2,-0.4) to [out=up,in=down,looseness=1] (0.2,0.4)\toplabel{i};
\draw[->] (0.2,-0.6) \botlabel{i} to (0.2,-0.4) to[out=up,in=down,looseness=1] (-0.2,0.4)\toplabel{i};
\pinpin{-0.2,-0.4}{0.2,-0.4}{1.9,-0.4}{g_{ii}(x_i,y_i)h_{i}(x_i,y_i)};
\projcr{0,0};
\end{tikzpicture}
-
\delta_{i=0} \begin{tikzpicture}[H,anchorbase]
\draw[->] (-0.2,-0.8) \botlabel{0} to (-0.2,-0.4) to [out=up,in=down,looseness=1] (0.2,0.4)\toplabel{0};
\draw[->] (0.2,-0.8) \botlabel{0} to (0.2,-0.4) to[out=up,in=down,looseness=1] (-0.2,0.4)\toplabel{0};
\pinpin{-0.2,-0.6}{0.2,-0.4}{2,-0.4}{g_{00}(-x_0,y_0)h_0(x_0,y_0)};
\token{-0.2,-0.4};
\token{0.2,-0.6};
\projcr{0,0};
\end{tikzpicture}\\
& \quad
+
\begin{tikzpicture}[H,anchorbase]
\draw[->] (-0.2,-0.6) \botlabel{i} to (-0.2,-0.4) to [out=up,in=down,looseness=1] (0.2,0.4) -- (0.2,0.6)\toplabel{i};
\draw[->] (0.2,-0.6) \botlabel{i} to (0.2,-0.4) to[out=up,in=down,looseness=1] (-0.2,0.4) -- (-0.2,0.6)\toplabel{i};
\pinpin{-0.2,-0.4}{0.2,-0.4}{1.3,-0.4}{g_{ii}(x_i,y_i)};
\pinpin{-0.2,0.35}{0.2,0.35}{1.3,0.35}{h_i(x_i,y_i)};
\projcr{0,0};
\end{tikzpicture}
+ \begin{tikzpicture}[H,anchorbase]
\draw[->] (-0.2,-0.5)\botlabel{i}  -- (-0.2,0.5);
\draw[->] (0.2,-0.5)\botlabel{i}  -- (0.2,0.5);
\pinpin{-0.2,0}{0.2,0}{1.3,0}{h_i(x_i,y_i)^2};
\end{tikzpicture} -\delta_{i=0} \begin{tikzpicture}[H,anchorbase]
\draw[->] (-0.2,-0.5)\botlabel{0}  -- (-0.2,0.5);
\draw[->] (0.2,-0.5)\botlabel{0}  -- (0.2,0.5);
\pinpin{-0.2,-.2}{0.2,.1}{1.9,.1}{h_0(x_0,y_0)h_0(-x_0,y_0)};
\token{-.2,.1};\token{.2,-.2};
\end{tikzpicture}\\
&\quad-\delta_{i=0}
\begin{tikzpicture}[H,anchorbase]
\draw[->] (-0.2,-0.6) \botlabel{0} to (-0.2,-0.4) to [out=up,in=down,looseness=1] (0.2,0.4) -- (0.2,0.6)\toplabel{0};
\draw[->] (0.2,-0.6) \botlabel{0} to (0.2,-0.4) to[out=up,in=down,looseness=1] (-0.2,0.4) -- (-0.2,0.6)\toplabel{0};
\pinpin{-0.2,-0.35}{0.2,-0.35}{1.2,-0.35}{g_{00}(x_0,y_0)};
\pinpin{-0.15,.2}{0.2,.4}{1.2,.4}{h_0(x_0,y_0)};
\token{-.2,.4};\token{.15,.2};
\projcr{0,0};
\end{tikzpicture}
-\delta_{i=0} \begin{tikzpicture}[H,anchorbase]
\draw[->] (-0.2,-0.5)\botlabel{0}  -- (-0.2,0.5);
\draw[->] (0.2,-0.5)\botlabel{0}  -- (0.2,0.5);
\pinpin{-0.2,-.2}{0.2,.1}{1.9,.1}{h_0(x_0,y_0)h_0(x_0,-y_0)};
\token{-.2,.1};\token{.2,-.2};
\end{tikzpicture}        +
\delta_{i=0}\begin{tikzpicture}[H,anchorbase]
\draw[->] (-0.2,-0.5)\botlabel{0}  -- (-0.2,0.5);
\draw[->] (0.2,-0.5)\botlabel{0}  -- (0.2,0.5);
\pinpin{-0.2,.1}{0.2,.3}{1.4,.3}{h_0(x_0,y_0)};
\token{-.2,.3};\token{.2,.1};
\pinpin{-0.2,-.4}{0.2,-.2}{1.4,-.2}{h_0(x_0,y_0)};
\token{-.2,-.2};\token{.2,-.4};
\end{tikzpicture}\\
& \overset{\mathclap{\substack{\cref{sergeev2} \\ \cref{cash1}}}}{\underset{\cref{cash3}}{=}}
\begin{tikzpicture}[H,anchorbase]
\draw[->] (-0.2,-1) \botlabel{i} -- (-0.2,-0.8) to[out=up,in=down] (0.2,0) to[out=up,in=down] (-0.2,0.8)\toplabel{i};
\draw[->] (0.2,-1) \botlabel{i} -- (0.2,-0.8) to[out=up,in=down] (-0.2,0) to[out=up,in=down] (0.2,0.8)\toplabel{i};
\pinpin{-0.19,-.1}{0.19,-.1}{1.2,-.1}{g_{ii}(x_i,y_i)};
\pinpin{-0.2,-0.8}{0.2,-0.8}{1.2,-0.8}{g_{ii}(x_i,y_i)};
\projcr{0,-0.4};
\projcr{0,0.4};
\node at (-0.3,0.15) {\strandlabel{i}};
\node at (0.3,0.15) {\strandlabel{i}};
\end{tikzpicture}
+
\begin{tikzpicture}[H,anchorbase]
\draw[->] (-0.2,-0.6) \botlabel{i} to (-0.2,-0.4) to [out=up,in=down,looseness=1] (0.2,0.4) -- (0.2,0.6)\toplabel{i};
\draw[->] (0.2,-0.6) \botlabel{i} to (0.2,-0.4) to[out=up,in=down,looseness=1] (-0.2,0.4) -- (-0.2,0.6)\toplabel{i};
\pinpin{-0.2,-0.4}{0.2,-0.4}{1.3,-0.4}{g_{ii}(x_i,y_i)};
\pinpin{-0.2,0.35}{0.2,0.35}{1.3,0.35}{h_i(x_i,y_i)};
\projcr{0,0};
\end{tikzpicture}+\delta_{i=0}
\begin{tikzpicture}[H,anchorbase]
\draw[->] (-0.2,-0.6) \botlabel{0} to (-0.2,-0.4) to [out=up,in=down,looseness=1] (0.2,0.4) -- (0.2,0.6)\toplabel{0};
\draw[->] (0.2,-0.6) \botlabel{0} to (0.2,-0.4) to[out=up,in=down,looseness=1] (-0.2,0.4) -- (-0.2,0.6)\toplabel{0};
\pinpin{-0.2,-0.43}{0.15,-0.23}{1.3,-0.23}{g_{00}(x_0,y_0)};
\pinpin{-0.2,.3}{0.2,.3}{1.3,.3}{h_0(-x_0,y_0)};
\token{.2,-.43};\token{-.15,-.23};
\projcr{0,0};
\end{tikzpicture}
 +
\begin{tikzpicture}[H,anchorbase]
\draw[->] (-0.2,-0.6)\botlabel{i}  to (-0.2,-0.4) to [out=up,in=down,looseness=1] (0.2,0.4)\toplabel{i};
\draw[->] (0.2,-0.6) \botlabel{i} to (0.2,-0.4) to[out=up,in=down,looseness=1] (-0.2,0.4)\toplabel{i};
\pinpin{-0.2,-0.4}{0.2,-0.4}{1.9,-0.4}{g_{ii}(x_i,y_i)h_{i}(x_i,y_i)};
\projcr{0,0};
\end{tikzpicture}\\&   \quad  -
\delta_{i=0} \begin{tikzpicture}[H,anchorbase]
\draw[->] (-0.2,-0.8) \botlabel{0} to (-0.2,-0.4) to [out=up,in=down,looseness=1] (0.2,0.4)\toplabel{0};
\draw[->] (0.2,-0.8) \botlabel{0} to (0.2,-0.4) to[out=up,in=down,looseness=1] (-0.2,0.4)\toplabel{0};
\pinpin{-0.2,-0.6}{0.2,-0.4}{2,-0.4}{g_{00}(x_0,y_0)h_0(x_0,y_0)};
\token{-0.2,-0.4};
\token{0.2,-0.6};
\projcr{0,0};
\end{tikzpicture}
+ \begin{tikzpicture}[H,anchorbase]
\draw[->] (-0.2,-0.5)\botlabel{i}  -- (-0.2,0.5);
\draw[->] (0.2,-0.5)\botlabel{i}  -- (0.2,0.5);
\pinpin{-0.2,0}{0.2,0}{1.3,0}{h_i(x_i,y_i)^2};
\end{tikzpicture}
+
\delta_{i=0}\begin{tikzpicture}[H,anchorbase]
\draw[->] (-0.2,-0.5)\botlabel{0}  -- (-0.2,0.5);
\draw[->] (0.2,-0.5)\botlabel{0}  -- (0.2,0.5);
\pinpin{-.2,0}{.2,0}{1.4,0}{h_0(x_0,-y_0)^2};
\end{tikzpicture}\\
&\:\:\:\,\overset{\mathclap{\substack{\cref{affsergeev} \\ \cref{naha} \\ \cref{cold}}}}{\underset{\mathclap{\substack{\cref{cash1}\\\cref{cash3}}}}{=}}\:\:\:
\begin{tikzpicture}[H,anchorbase]
\draw[->] (-0.2,-0.5)\botlabel{i}  -- (-0.2,0.5);
\draw[->] (0.2,-0.5)\botlabel{i}  -- (0.2,0.5);
\pinpin{-.2,0}{.2,0}{2.05,0}{g_{ii}(x_i,y_i)^2\big(1-\frac{\delta_{i \neq 0}}{(x+y)^2}\big)};
\end{tikzpicture}-
\begin{tikzpicture}[H,anchorbase]
\draw[->] (-0.2,-0.6) \botlabel{i} to (-0.2,-0.4) to [out=up,in=down,looseness=1] (0.2,0.4) -- (0.2,0.6)\toplabel{i};
\draw[->] (0.2,-0.6) \botlabel{i} to (0.2,-0.4) to[out=up,in=down,looseness=1] (-0.2,0.4) -- (-0.2,0.6)\toplabel{i};
\pinpin{-0.2,-0.4}{0.2,-0.4}{1.7,-0.4}{g_{ii}(x_i,y_i)h_i(x_i,y_i)};
\projcr{0,0};
\end{tikzpicture}
+ \begin{tikzpicture}[H,anchorbase]
\draw[->] (-0.2,-0.5)\botlabel{i}  -- (-0.2,0.5);
\draw[->] (0.2,-0.5)\botlabel{i}  -- (0.2,0.5);
\pinpin{-0.2,0}{0.2,0}{1.8,0}{\frac{2g_{ii}(x_i,y_i)h_i(x_i,y_i)}{x-y}};
\end{tikzpicture}\\
&\quad+\delta_{i=0}
\begin{tikzpicture}[H,anchorbase]
\draw[->] (-0.2,-0.6) \botlabel{0} to (-0.2,-0.4) to [out=up,in=down,looseness=1] (0.2,0.4) -- (0.2,0.6)\toplabel{0};
\draw[->] (0.2,-0.6) \botlabel{0} to (0.2,-0.4) to[out=up,in=down,looseness=1] (-0.2,0.4) -- (-0.2,0.6)\toplabel{0};
\pinpin{-0.2,-0.45}{0.18,-0.25}{2,-0.25}{g_{00}(x_0,y_0)h_0(x_0,y_0)};
\token{.2,-.45};\token{-.18,-.25};
\projcr{0,0};
\end{tikzpicture}
-\delta_{i=0} \begin{tikzpicture}[H,anchorbase]
\draw[->] (-0.2,-0.5)\botlabel{0}  -- (-0.2,0.5);
\draw[->] (0.2,-0.5)\botlabel{0}  -- (0.2,0.5)\toplabel{\phantom{0}};
\pinpin{-0.2,.1}{0.2,.3}{1.4,.3}{\frac{2h_0(x_0,y_0)}{x-y}};
\token{-.2,.3};\token{.2,.1};
\pinpin{-0.2,-.4}{0.2,-.2}{1.4,-.3}{g_{00}(x_0,y_0)};
\token{-.2,-.2};\token{.2,-.4};
\end{tikzpicture}
+
\begin{tikzpicture}[H,anchorbase]
\draw[->] (-0.2,-0.6)\botlabel{i}  to (-0.2,-0.4) to [out=up,in=down,looseness=1] (0.2,0.4)\toplabel{i};
\draw[->] (0.2,-0.6) \botlabel{i} to (0.2,-0.4) to[out=up,in=down,looseness=1] (-0.2,0.4)\toplabel{i};
\pinpin{-0.2,-0.4}{0.2,-0.4}{1.9,-0.4}{g_{ii}(x_i,y_i)h_{i}(x_i,y_i)};
\projcr{0,0};
\end{tikzpicture}\\&\quad
-
\delta_{i=0} \begin{tikzpicture}[H,anchorbase]
\draw[->] (-0.2,-0.8) \botlabel{0} to (-0.2,-0.4) to [out=up,in=down,looseness=1] (0.2,0.4)\toplabel{0};
\draw[->] (0.2,-0.8) \botlabel{0} to (0.2,-0.4) to[out=up,in=down,looseness=1] (-0.2,0.4)\toplabel{0};
\pinpin{-0.2,-0.6}{0.2,-0.4}{2,-0.4}{g_{00}(x_0,y_0)h_0(x_0,y_0)};
\token{-0.2,-0.4};
\token{0.2,-0.6};
\projcr{0,0};
\end{tikzpicture} + \begin{tikzpicture}[H,anchorbase]
\draw[->] (-0.2,-0.5)\botlabel{i}  -- (-0.2,0.5);
\draw[->] (0.2,-0.5)\botlabel{i}  -- (0.2,0.5);
\pinpin{-0.2,0}{0.2,0}{1.3,0}{h_i(x_i,y_i)^2};
\end{tikzpicture}  +
\delta_{i=0}\begin{tikzpicture}[H,anchorbase]
\draw[->] (-0.2,-0.5)\botlabel{0}  -- (-0.2,0.5);
\draw[->] (0.2,-0.5)\botlabel{0}  -- (0.2,0.5);
\pinpin{-.2,0}{.2,0}{1.4,0}{h_0(x_0,-y_0)^2};
\end{tikzpicture}.
\end{align*}
The terms with a crossing obviously cancel. So, after simplifying the remaining term with Clifford tokens using \cref{affsergeev},  we are left with the identity endomorphism
pinned with the polynomial
$$\textstyle
g_{ii}(x_i,y_i)^2\left(1-\frac{\delta_{i \neq 0}}{(x+y)^2}\right)+h_i(x_i,y_i)^2 +\frac{2g_{ii}(x_i,y_i)h_i(x_i,y_i)}{x-y} + \delta_{i=0} \left(h_0(x_0,-y_0)^2+\frac{2 g_{00}(x_0,y_0)h_0(x_0,-y_0)}{x+y}\right).
$$
It remains to observe that this polynomial is 0. This follows on expanding the first term using \cref{newmoney} then using \cref{doggydaycare} to replace $\frac{g_{ii}(x_i,y_i)}{x-y}$ in the second term
by $\frac{1}{x_i-y_i}-h_i(x_i,y_i)$
and $\frac{g_{00}(x_i,-y_i)}{x+y}$ in the third term by $\frac{1}{x_0+y_0}-h_i(x_0,-y_0)$.
\end{proof}

\subsection{Main theorem}

Continue with $\catR$ being an isomeric Heisenberg categorification.

\begin{theo}\label{maintheorem2}
The isomeric Heisenberg categorification $\catR$ 
can be made into an isomeric Kac--Moody categorification
for the Cartan datum described in \cref{seccd} and the parameters \cref{pdub}. The required data is as follows:
\begin{enumerate}
\item The superfunctors $P_i, Q_i\:(i \in I)$ are the eigenfunctors from \cref{piqi}.
\item The weight subcategories $\catR_\lambda\:(\lambda \in X)$
are as defined just before \cref{lostboys}.
\item The unit and counit of the adjunction $(P_i,Q_i)$ are the natural transformations
$\begin{tikzpicture}[IKM,centerzero,scale=.8]
\draw[-to] (-0.25,0.15) \toplabel{i} to[out=-90,in=-90,looseness=3] (0.25,0.15);
\node at (0,.2) {$\phantom.$};\node at (0,-.3) {$\phantom.$};
\end{tikzpicture}:=\begin{tikzpicture}[H,centerzero,scale=.8]
\draw[-to] (-0.25,0.15) \toplabel{i} to[out=-90,in=-90,looseness=3] (0.25,0.15);
\node at (0,.2) {$\phantom.$};\node at (0,-.3) {$\phantom.$};
\end{tikzpicture}\ $
and $\begin{tikzpicture}[IKM,centerzero,scale=.8]
\draw[-to] (-0.25,-0.15) \botlabel{i} to [out=90,in=90,looseness=3](0.25,-0.15);
\node at (0,.3) {$\phantom.$};
\node at (0,-.4) {$\phantom.$};
\end{tikzpicture}:=\begin{tikzpicture}[H,centerzero,scale=.8]
\draw[-to] (-0.25,-0.15) \botlabel{i} to [out=90,in=90,looseness=3](0.25,-0.15);
\node at (0,.3) {$\phantom.$};
\node at (0,-.4) {$\phantom.$};
\end{tikzpicture}\ $, respectively.
\item The supernatural transformations
$\begin{tikzpicture}[IKM,centerzero]
\draw[-to] (0,-0.2) \botlabel{i} -- (0,0.2);
\singdot{0,0};
\end{tikzpicture}
:P_i\rightarrow P_i$,
$\begin{tikzpicture}[IKM,centerzero]
\draw[-to] (0,-0.2) \botlabel{0} -- (0,0.2);
\token{0,0};
\end{tikzpicture}:P_0 \Rightarrow P_0$
and $\begin{tikzpicture}[IKM,centerzero,scale=.9]
\draw[-to] (-0.2,-0.2) \botlabel{i} -- (0.2,0.2);
\draw[-to] (0.2,-0.2) \botlabel{j} -- (-0.2,0.2);
\end{tikzpicture}:P_i \circ P_j \Rightarrow P_j \circ P_i$
are as defined in \cref{tokendef,dotdef,crossingdef}.
\end{enumerate}
\end{theo}

\begin{proof}
We must check the conditions (IKM0)--(IKM4) from \cref{ikmcatdef}.
The condition (IKM0) follows from \cref{lostboys},
the zig-zag relations required for (IKM1)
follow from \cref{adjright},
(IKM2) follows from \cref{checkingrels2}, and (IKM4) follows from (IH4). It just remains to check (IKM3).
From the definition \cref{irightpivot},
the supernatural transformations
represented by rightward crossings
in $\fV(\fg)$
are given explicitly by
\begin{align}\label{rightwardbeast}
\begin{tikzpicture}[IKM,centerzero,scale=1.2]
\draw[->] (-0.3,-0.3) \botlabel{j} -- (0.3,0.3);
\draw[<-] (0.3,-0.3) \botlabel{i} -- (-0.3,0.3);
\end{tikzpicture}
&=
\begin{tikzpicture}[centerzero,H,scale=1.4]
\draw[->] (-.2,-.3)\botlabel{j} to[out=up,in=down,looseness=1] (.2,.3) to (.2,.4)\toplabel{j};
\draw[<-] (.2,-.3)\botlabel{i} to[out=up,in=down,looseness=1] (-.2,.3) to (-.2,.4) \toplabel{i};
\pinpin{.18,.2}{-.18,.2}{-.9,.2}{g_{ij}(x_i,y_j)};
\projcr{0,0};
\end{tikzpicture}
+\delta_{i=j}
\begin{tikzpicture}[centerzero,H,scale=1.4]
\draw[->] (-.2,.4) \toplabel{i} to[out=down,in=down,looseness=2.5] (.2,.4);
\draw[->] (-.2,-.3) \botlabel{i} to[out=up,in=up,looseness=2.5] (.2,-.3);
\pinpin{-.1,-.04}{.1,.13}{1.2,.13}{g_{ii}(y_i,y_i) t_i(y_i,x_i)};
\end{tikzpicture}
-\delta_{i=-j}
\begin{tikzpicture}[centerzero,H,scale=1.4]
\draw[->] (-.2,.4) \toplabel{0} to[out=down,in=down,looseness=2.5] (.2,.4);
\draw[->] (-.2,-.3) \botlabel{0} to[out=up,in=up,looseness=2.5] (.2,-.3);
\token{-.19,-.18};
\token{.19,.27};
\pinpin{-.1,-.04}{.1,.13}{1.3,.13}{g_{00}(y_0,y_0) t_0(y_0,x_0)};
\end{tikzpicture}
\\\label{rightwardbeast2}
&=
\begin{tikzpicture}[centerzero,H,scale=1.4]
\draw[->] (-.2,-.3)\botlabel{j} to (-.2,-.2) to[out=up,in=down,looseness=1] (.2,.4)\toplabel{j};
\draw[<-] (.2,-.3)\botlabel{i} to (.2,-.2) to[out=up,in=down,looseness=1] (-.2,.4) \toplabel{i};
\pinpin{.18,-.1}{-.18,-.1}{-.9,-.1}{g_{ij}(y_i,x_j)};
\projcr{0,0.1};
\end{tikzpicture}
-\delta_{i=j}
\begin{tikzpicture}[centerzero,H,scale=1.4]
\draw[->] (-.2,.4) \toplabel{i} to[out=down,in=down,looseness=2.5] (.2,.4);
\draw[->] (-.2,-.3) \botlabel{i} to[out=up,in=up,looseness=2.5] (.2,-.3);
\pinpin{-.1,-.04}{.1,.13}{1.2,.13}{g_{ii}(x_i,x_i)t_i(x_i,y_i)};
\end{tikzpicture}
+\delta_{i=-j}
\begin{tikzpicture}[centerzero,H,scale=1.4]
\draw[->] (-.2,.4) \toplabel{0} to[out=down,in=down,looseness=2.5] (.2,.4);
\draw[->] (-.2,-.3) \botlabel{0} to[out=up,in=up,looseness=2.5] (.2,-.3);
\token{-.19,-.18};
\token{.19,.27};
\pinpin{-.1,-.04}{.1,.13}{1.3,.13}{g_{00}(x_0,x_0)t_0(x_0,y_0)};
\end{tikzpicture}
\ ,
\end{align}
where $t_i(x,y)$ is as in \cref{di}.
This follows by rotating  \cref{doingsomerotation},
i.e., by adding rightward cups and caps in the
appropriate places.
Now we consider various cases:
\begin{itemize}
\item
For $i,j \in I$ with $i \neq j$,
the invertibility of
 $\begin{tikzpicture}[IKM,centerzero,scale=.9]
\draw[->] (-0.3,-0.3) \botlabel{i} -- (0.3,0.3);
\draw[<-] (0.3,-0.3) \botlabel{j} -- (-0.3,0.3);
\end{tikzpicture}$
follows because the rightward crossing
$\begin{tikzpicture}[H,centerzero,scale=.9]
\draw[->] (-0.3,-0.3) \botlabel{i} -- (0.3,0.3);
\draw[<-] (0.3,-0.3) \botlabel{j} -- (-0.3,0.3);
\projcr{0,0};
\end{tikzpicture}$ is invertible thanks to \cref{sidewaysinvertibility}, and the power series $g_{ij}(x_i,y_j)$ is invertible too.
\item
Suppose that $i=j \neq 0$ and consider $\lambda \in X$ such that $h_i(\lambda) \leq 0$. We need to show that
$\begin{pmatrix}
\begin{tikzpicture}[IKM,centerzero]
\draw[->] (-0.25,-0.25) \botlabel{i}-- (0.25,0.25);
\draw[<-] (0.25,-0.25) \botlabel{i} -- (-0.25,0.25);
\region{0.35,0.02}{\lambda};
\end{tikzpicture}\ \  &
\begin{tikzpicture}[IKM,anchorbase]
\draw[->] (-0.2,0.3)\toplabel{i} -- (-0.2,0) arc(180:360:0.2) -- (0.2,0.3);
\region{0.38,0}{\lambda};        \end{tikzpicture}
&
\begin{tikzpicture}[anchorbase,IKM]
\draw[->] (-0.2,0.3) \toplabel{i} -- (-0.2,0) arc(180:360:0.2) -- (0.2,0.3);
\singdot{0.2,0.1};
\region{0.38,0}{\lambda};
\end{tikzpicture}
&\!\!\cdots\!\!&
\begin{tikzpicture}[anchorbase,IKM]
\draw[->] (-0.2,0.3)\toplabel{i}  -- (-0.2,0) arc(180:360:0.2) -- (0.2,0.3);
\multdot{0.2,0.1}{west}{-h_i(\lambda)-1};
\region{0.38,-0.1}{\lambda};
\end{tikzpicture}
\end{pmatrix}$ is invertible on any object of
$\catR_\lambda$.
Composing the definition with the invertible matrix
$\operatorname{diag}\left(
\begin{tikzpicture}[centerzero,H]
\draw[->] (-.2,-.2)\botlabel{i} to (-.2,.2);
\draw[<-] (.2,-.2)\botlabel{i} to (.2,.2);
\pinpin{-.2,0}{.2,0}{1.4,0}{g_{ii}(y_i,x_i)^{-1}};
\end{tikzpicture}\ ,
\id_{\catR_\lambda},\dots,\id_{\catR_\lambda}\right)$, we are reduced to showing that the matrix
of supernatural transformations      
$$
\begin{pmatrix}
\begin{tikzpicture}[centerzero,H,scale=1.3]
\draw[->] (-.2,-.35)\botlabel{i} to[out=up,in=down,looseness=1] (.2,.35) \toplabel{i};
\draw[<-] (.2,-.35)\botlabel{i} to[out=up,in=down,looseness=1] (-.2,.35) \toplabel{i};
\projcr{0,0};
\end{tikzpicture}
-
\begin{tikzpicture}[centerzero,H,scale=1.3]
\draw[->] (-.2,.4) \toplabel{i} to[out=down,in=down,looseness=2.5] (.2,.4);
\draw[->] (-.2,-.3) \botlabel{i} to[out=up,in=up,looseness=2.5] (.2,-.3);
\pinpin{-.1,-.04}{.1,.13}{1,.13}{t_i(x_i,y_j)};
\end{tikzpicture}\ \ \ &
\begin{tikzpicture}[H,anchorbase,scale=1.2]
\draw[->] (-0.2,0.3)\toplabel{i} -- (-0.2,0) arc(180:360:0.2) -- (0.2,0.3);
\region{0.38,0}{\lambda};        \end{tikzpicture}
&
\begin{tikzpicture}[anchorbase,H,scale=1.2]
\draw[->] (-0.2,0.3) \toplabel{i} -- (-0.2,0) arc(180:360:0.2) -- (0.2,0.3);
\pin{0.2,0.1}{.7,.1}{x_i};
\end{tikzpicture}
&\!\!\cdots\!\!&
\begin{tikzpicture}[anchorbase,H,scale=1.2]
\draw[->] (-0.2,0.3)\toplabel{i}  -- (-0.2,0) arc(180:360:0.2) -- (0.2,0.3);
\pin{0.2,0.1}{1.3,.1}{x_i^{-h_i(\lambda)-1}};
\end{tikzpicture}\
\end{pmatrix}
$$
is invertible when evaluated on any object of $\catR_\lambda$.
By naturality, it suffices to check this just on each irreducible $L \in \catR_\lambda$.
It remains to apply \cref{theworstplace}(1), taking
$\xi(x)$ and $r(x,y)$ there to be some choice of polynomials which have the same images in $\kk[x] / (x-b(i))^{\eps_i(L)}$ and $\kk[x,y] /
\big((x-b(i)^{\phi_i(L)}, (y-b(i))^{\eps_i(L)}\big)$
as the power series $\xi_i(x) \in \kk\llbracket x-b(i)\rrbracket^\times$ and
$t_i(x_i,y_i) \in \kk\llbracket x-b(i),y-b(i)\rrbracket$, respectively.
\item
Suppose that $i=j=0$ and $\lambda \in X$
satisfies $h_0(\lambda) \leq 0$.
Like in the previous case, the proof of the inversion relation reduces to showing that the matrix of supernatural transformations
$$
\begin{pmatrix}
\begin{tikzpicture}[centerzero,H,scale=1.3]
\draw[->] (-.2,-.35)\botlabel{0} to[out=up,in=down,looseness=1] (.2,.35) \toplabel{0};
\draw[<-] (.2,-.35)\botlabel{0} to[out=up,in=down,looseness=1] (-.2,.35) \toplabel{0};
\projcr{0,0};
\end{tikzpicture}
-\begin{tikzpicture}[centerzero,H,scale=1.3]
\draw[->] (-.2,.4) \toplabel{0} to[out=down,in=down,looseness=2.5] (.2,.4);
\draw[->] (-.2,-.3) \botlabel{0} to[out=up,in=up,looseness=2.5] (.2,-.3);
\pinpin{-.1,-.04}{.1,.13}{.9,.13}{t_i(x_i,y_i)};
\end{tikzpicture}
+\!\!\begin{tikzpicture}[centerzero,H,scale=1.3]
\draw[->] (-.2,.4) \toplabel{0} to[out=down,in=down,looseness=2.5] (.2,.4);
\draw[->] (-.2,-.3) \botlabel{0} to[out=up,in=up,looseness=2.5] (.2,-.3);
\token{-.19,-.18};
\token{.19,.27};
\pinpin{-.1,-.04}{.1,.13}{.9,.13}{t_i(x,y)};
\end{tikzpicture}\ &
\begin{tikzpicture}[H,anchorbase,scale=1.2]
\draw[->] (-0.2,0.3)\toplabel{0} -- (-0.2,0) arc(180:360:0.2) -- (0.2,0.3);
\token{.2,0};
\end{tikzpicture}\ \
&
\begin{tikzpicture}[H,anchorbase,scale=1.2]
\draw[->] (-0.2,0.3)\toplabel{0} -- (-0.2,0) arc(180:360:0.2) -- (0.2,0.3);
\end{tikzpicture}\ \
&
\begin{tikzpicture}[anchorbase,H,scale=1.2]
\draw[->] (-0.2,0.3) \toplabel{0} -- (-0.2,0) arc(180:360:0.2) -- (0.2,0.3);
\pin{0.2,.14}{.7,.14}{x_i};
\token{.2,0};
\end{tikzpicture}
&
\!\!\cdots\!\!
&
\begin{tikzpicture}[anchorbase,H,scale=1.2]
\draw[->] (-0.2,0.3)\toplabel{0}  -- (-0.2,0) arc(180:360:0.2) -- (0.2,0.3);
\pin{0.2,0.1}{1.27,.1}{x_0^{-h_0(\lambda)-1}};
\end{tikzpicture}\
\end{pmatrix}
$$
is invertible on any irreducible $L \in \catR_\lambda$.
This follows in a similar way to the previous case using \cref{theworstplace}(2) instead of (1).
\item Finally suppose that $i=j$ and $\lambda \in X$ satisfies $h_i(\lambda) > 0$. Then the inversion relation follows from
\cref{theworstplace}(3)--(4) by similar considerations.
\end{itemize}
\end{proof}

\begin{rem}
A shortcoming of \cref{maintheorem2} is that we do not give explicit formulae for the leftward cups and caps in the isomeric Kac-Moody 2-category in terms of the leftward cups and caps from the isomeric Heisenberg action, or for dotted bubbles, although they are uniquely determined by the information provided. The analogous problem in the ordinary Heisenberg setting was solved in \cite[Sec.~7]{BSW-update}.
\end{rem}

\bibliographystyle{alphaurl}
\bibliography{isomericI}

\end{document}